\newcommand{\pointsize}{11pt}

\documentclass[oneside, \pointsize]{amsbook}

\usepackage[
   includehead,
   includefoot,
     left = 1.5in, 
      top = 0.5in, 
    right = 1in,
   bottom = 1in
]{geometry}
\usepackage{fancyhdr}
\usepackage{setspace}
\usepackage{calc}

\makeatletter
\renewcommand\section{\@startsection
  {section}{1}{0mm}
  {-0.6\baselineskip}
  {0.8\baselineskip}
  {\bfseries\Large}}
\makeatother

\makeatletter
\renewcommand\subsection{\@startsection
  {subsection}{2}{0mm}
  {-\baselineskip}
  {0.5\baselineskip}
  {\bfseries\large}}
\makeatother

\makeatletter
\renewcommand\subsubsection{\@startsection
  {subsubsection}{3}{0mm}
  {-\baselineskip}
  {0.5\baselineskip}
  {\bfseries}}
\makeatother

\fancyheadoffset[R]{0.5in} 

\fancypagestyle{prelim}{%
   \renewcommand{\headrulewidth}{0pt} 
   \fancyhf{}           
   \pagenumbering{roman}    
   \cfoot{-\thepage-}       
}

\fancypagestyle{maintext}{%
   \renewcommand{\headrulewidth}{0.4pt}
   \pagenumbering{arabic}
   \fancyhf{}
   \fancyhead[L]{\rightmark} 
   \rhead{\thepage}
}

\numberwithin{figure}{chapter} 
\numberwithin{table}{chapter}
\numberwithin{equation}{chapter}
\numberwithin{section}{chapter}

\usepackage{amsfonts}
\usepackage{amssymb}
\usepackage{amsthm}
\usepackage{amsmath}
\usepackage{epsfig}
\usepackage{graphicx}
\usepackage[colorlinks=false, linkbordercolor={1 0 0}, citebordercolor={0 1 0}]{hyperref}
\usepackage{longtable}
\usepackage{microtype}
\usepackage{psfrag}
\usepackage{url}
\usepackage{xcolor}

\graphicspath{{support/}}

\def\defn#1{{\color{blue}\emph{#1}}}

\def\N{\mathbb{N}}

\def\R{\mathbb{R}}
\def\Z{\mathbb{Z}}

\def\aff{\operatorname{aff}}
\def\ast{\operatorname{ast}}
\def\cone{\operatorname{cone}}
\def\conv{\operatorname{conv}}
\def\diam{\operatorname{diam}}
\def\dist{\operatorname{dist}}
\def\gcd{\operatorname{gcd}}
\def\hamm{\operatorname{hamm}}
\def\int{\operatorname{int}}
\def\lk{\operatorname{lk}}
\def\ops{\operatorname{S}}
\def\star{\operatorname{star}}
\def\supp{\operatorname{supp}}
\def\vert{\operatorname{vert}}
\def\wed{\operatorname{W}}

\def\given#1{} 

\def\@begintheorem#1#2{\par\bgroup{\sc #1\ #2. }\it\ignorespaces}
\def\@opargbegintheorem#1#2#3{\par\bgroup{\sc #1\ #2\ (#3). } \it\ignorespaces}
\def\@endtheorem{\egroup}

\newtheoremstyle{ekimcustom}
{6pt}
{3pt}
{\sl}
{}
{\sc}
{.}
{.5em}
{}

\theoremstyle{ekimcustom}
\newtheorem{theorem}{Theorem}[section]
\newtheorem{corollary}[theorem]{Corollary}
\newtheorem{lemma}[theorem]{Lemma}
\newtheorem{proposition}[theorem]{Proposition}

\newtheorem{example}[theorem]{Example}
\newtheorem{remark}[theorem]{Remark}

\newtheorem{definition}[theorem]{Definition}
\newtheorem{conjecture}[theorem]{Conjecture}
\newtheorem{openproblem}[theorem]{Open Problem}

\begin{document}
   \frontmatter

   \pagestyle{prelim}
   
   \fancypagestyle{plain}{%
      \fancyhf{}
      \cfoot{-\thepage-}
   }
\begin{center}
   \null\vfill
   \textbf{%
      Geometric Combinatorics of Transportation Polytopes\\and the Behavior of the Simplex Method
   }%
   \\
   \bigskip
   By \\
   \bigskip
   EDWARD DONG HUHN KIM \\ 
   \bigskip
   B.A. (University of California, Berkeley) 2004 \\ 
   M.A. (University of California, Davis) 2007 \\ 
   \bigskip
   DISSERTATION \\
   \bigskip
   Submitted in partial satisfaction of the requirements for the
   degree of \\
   \bigskip
   DOCTOR OF PHILOSOPHY \\
   \bigskip
   in \\
   \bigskip
   Mathematics \\
   \bigskip
   in the \\
   \bigskip
   OFFICE OF GRADUATE STUDIES \\
   \bigskip        
   of the \\
   \bigskip
   UNIVERSITY OF CALIFORNIA \\
   \bigskip
   DAVIS \\
   \bigskip
   Approved: \\
   \bigskip
   \bigskip
   \makebox[3in]{\hrulefill} \\
   Jes\'us A. De Loera (Chair) \\ 
   \bigskip
   \bigskip
   \makebox[3in]{\hrulefill} \\
   Francisco Santos \\ 
   \bigskip
   \bigskip
   \makebox[3in]{\hrulefill} \\
   Nina Amenta \\
   \bigskip
   Committee in Charge \\
   \bigskip
   2010 \\
   \vfill
\end{center}
   

   \newpage
   
~\\[7.75in] 
\centerline{
\copyright\ 2010, Edward Dong Huhn Kim. All rights reserved.
}
\thispagestyle{empty}
\addtocounter{page}{-1}

   \newpage
   
~\\[1.35in] 
\singlespacing{
\begin{center}
{\sl To Harabuji, \\
who placed the pencil in my hand}
\end{center}


   \newpage
   
   \doublespacing

   \setcounter{tocdepth}{1} 
   \tableofcontents


   \newpage
   
\addcontentsline{toc}{chapter}{Abstract}

{\singlespacing
   \begin{flushright}
      Edward D. Kim \\
      June 2010 \\
      Mathematics \\
   \end{flushright}
}

\bigskip

\begin{center}
Geometric Combinatorics of Transportation Polytopes\\ 
and the Behavior of the Simplex Method \\
\end{center}


\begin{center}
{\Large\bfseries Abstract}
\end{center}

This dissertation investigates the geometric combinatorics of convex polytopes and connections to the behavior of the simplex method for linear programming. We focus our attention on transportation polytopes, which are sets of all tables of non-negative real numbers satisfying certain summation conditions. Transportation problems are, in many ways, the simplest kind of                                                                                             
linear programs and thus have a rich combinatorial structure. First, we give new results on the diameters of certain classes of transportation polytopes and their relation to the Hirsch Conjecture, which asserts that the diameter of every $d$-dimensional convex polytope with $n$ facets is bounded above by $n-d$. In particular, we prove a new quadratic upper bound on the diameter of $3$-way axial transportation polytopes defined by $1$-marginals. We also show that the Hirsch Conjecture holds for $p \times 2$ classical transportation polytopes, but that there are infinitely-many Hirsch-sharp classical transportation polytopes.

Second, we present new results on subpolytopes of transportation polytopes. We investigate, for example, a non-regular triangulation of a subpolytope of the fourth Birkhoff polytope $B_4$. This implies the existence of non-regular triangulations of all Birkhoff polytopes $B_n$ for $n \geq 4$. We also study certain classes of network flow polytopes and prove new linear upper bounds for their diameters.

The thesis is organized as follows: Chapter~\ref{chapter:polytopes} introduces polytopes and polyhedra and discusses their connection to optimization. A survey on transportation polytopes is presented here. We close the chapter by discussing the theory behind the software {\tt transportgen}. Chapter~\ref{chapter:hirsch} surveys the state of the art on the Hirsch Conjecture and its variants. Chapter~\ref{chapter:transportation} presents our new results on the geometric combinatorics of transportation polytopes. 
Finally, a summary of the computational results of the software package {\tt transportgen} are presented in Appendix~\ref{appendix:catalogTP}.



   \newpage
   
\addcontentsline{toc}{chapter}{Acknowledgements}

\begin{center}
{\Large\bfseries Acknowledgements}
\end{center}  


There are many people I must acknowledge, since so many people helped me in so many ways. It is almost unfair that only one degree can be awarded. This thesis is the culmination of a group effort, and this degree really belongs to all of you, because I could not have made it to this point without your help. (I apologize a little for its length, but I hope that no one faults me for being too thankful.) It is very obvious whom I must thank first, the one who has taught me so much about mathematics and how to handle every aspect of it.

First and foremost, my deepest thanks go to my advisor, Jes\'us A. De Loera. You have been my greatest teacher, challenger, and cheerleader. You have been patient to teach while demanding nothing short of excellence. Thank you for always informing me of great special programs and conferences so that I could have the opportunity to stay abreast of the research frontier. Thank you for imparting your knowledge and advice on all things, mathematical and non-mathematical. You taught me so much mathematics, but you also taught me so much about how to be patient with myself as I \emph{do} mathematics. If researching mathematics were a high-dimensional polytope, then you have described for me every facet. As I research at future academic institutions, I will strive to visit the right vertices of the research polytope using the De Loera pivot rule! (Then again, I suppose that it is a polyhedron, not a polytope!) I cannot think of a better academic mentor. I hope that I can show my deepest thanks only in the years to come by trying to emulate you as an advisor one day. I also wish to thank Jes\'us' family for opening up their home for events, whether it was a celebration day or a research day. Thanks Jes\'us, for your help and support in every pivot step on the path so far. I could not have asked for a better advisor.

Next, I want to thank Francisco Santos. Paco, your mathematical discussions have always been extremely fruitful in helping me understand the essence of any problem or theorem. Even in a short discussion, I have always left our conversations feeling like I have a more intuitive understanding of polytopes. I hope that, in time, I will learn how to understand mathematics the way that you do. Thank you for your mathematical discussions, our collaboration on the survey (see~\cite{KimSantos:HirschSurvey}), and the opportunity to visit you in Santander.

Thank you, Nina Amenta, for serving on my thesis committee and explaining to me so much about computational geometry. Your example in research and in your teaching style has been a big influence for me. I appreciate the conversations that we have had for ideas on future projects and I look forward to the prospect of future collaboration.

I am grateful for the service of Roger J-B Wets. Thanks for being willing to chair my qualifying examination. You taught me that there is a place where useful and beautiful mathematics intersect. I will take your example and advice every step of the way in the rest of my research career (and hopefully I will converge to your example faster than Newton-Ralphson methods). I am honored to have played a small role in celebrating your distinguished career during the recent celebration at UC Davis.

My thanks also go to Eric Babson and Roman Vershynin for serving on my qualifying examination committee and giving me a firm grounding in algebraic topology and probability theory.

I wish to thank my collaborator Shmuel Onn. Our conversations (see~\cite{Onn:KOboundNotPolynomial}) regarding the Hirsch Conjecture have given me a deeper understanding of the research area.

Bernd Sturmfels, thanks for your many conversations about mathematics. You always provide just the right insight, and I appreciate the many doors to which you have pointed me. I realize that my words here are short, but I know that I join a long list of admirers, and with good reason.

I thank Matthias K\"oppe for wonderfully organizing and nurturing the optimization community at UC Davis. My time here has been enriched by what you have done. Thanks especially for the organization of my exit seminar. To Monica Vazirani, thanks for your advice on so many things, not least of which was helping us organize the Graduate Student Combinatorics Conference. I am very fortunate to have taken a reading course in representation theory with you.

Thank you to Art Duval, Fu Liu, Jeremy Martin, Jay Schweig, and Alex Yong for helping to connect me to more of the geometric combinatorics community. Thank you, Duane Kouba, for always showing me new techniques to becoming a better teacher.

There are several special programs that I must thank. I wish to acknowledge Henry Wolkowicz and workshop assistant Nathan Krislock for teaching the MSRI Workshop on Continuous Optimization and its Applications. I also want to thank my project collaborators Andrej Dudek, Kimia Ghobadi, Shaowei Lin, Richard Spjut, and Jiaping Zhu. I also thank David Cru and Natalie Durgin for their continued friendships after the program. I am extremely grateful to Laura Matusevich, Frank Sottile, and Thorsten Theobald for the IMA program on Applicable Algebraic Geometry at Texas A\&M University. Thanks to Frank, Thorsten, Serkan Ho\,sten, and Seth Sullivant for their lectures. I thank the program tutors Cordian Riener, Reinhard Steffens, Abraham Mart\'in del Campo, and Luis Garcia for their help and friendships. With many happy tears, I thank Marc Noy, Julian Pfeifle, Ferran Hurtado, Francisco Santos, and Antonio Guedes de Oliveira for organizing the 2009 DocCourse in Combinatorics and Geometry at the Centre de Recerca Matem\`atica. I learned an immense amount from our lecturers Ji\v{r}\'i Matou\v{s}ek and G\"unter M. Ziegler. Thanks to our problem session moderators Anastasios Sidiropoulos and Axel Werner. To my fellow program participants David Alonso Gutierrez, Victor \'Alvarez Amaya, Aaron Dall, Connie Dangelmayr, Ragnar Freij, Bernardo Gonz\'alez Merino, Marek Krc\'al, Eva Linke, Mareike Massow, Benjamin Matschke, Silke M\"oser, Noa Nitzan, Arnau Padrol Sureda, Jeong Hyeon Park, Canek Pel\'aez Vald\'es, Juanjo Ru\'e Perna, Maria Saumell, Llu\'is Vena Cros, Birgit Vogtenhuber, Ina Voigt, and Frederik von Heymann, it was an absolute pleasure to live and work with you every day. My special thanks go to Anna Gundert and Daria Schymura, for their collaboration in \cite{GundertKimSchymura:LPLM}, and Anna de Mier for her supervision of our project. I also want to give a special thanks to my wonderful roommate Matthias Henze. I am especially honored to share a role as unofficial student co-organizer with Vincent Pilaud, who is one day ``older'' than me. Vincent, your ability to organize large things is amazing. Finally, I wish to thank Frank Vallentin for encouraging me to apply for the Semidefinite Optimization seminar at the Mathematisches Forschungsinstitut Oberwolfach. Thanks to him and to the co-lecturers Sanjeev Arora, Monique Laurent, Pablo Parrilo, and Franz Rendl for their wonderful lectures.

To Jill Allard, Karen Beverlin, Diana Coombe, Carol Crabill, Connie Dani, Celia Davis, Tina Denena, Richard Edmiston, Perry Gee, Dena Gilday, Jessica Goodall, Zach Johnson, Elena Karn, Phuoc La, Leng Lai, Tracy Ligtenberg, Jessica Potts, DeAnn Roning, Alla Savrasova, and Marianne Waage, I want to express my thanks for all you have done to make the UC Davis math department such a wonderful place to study. Not only have you been wonderful staff members, but you have treated me like a good friend. Thank you.

A number of postdocs, whether here or elsewhere, have helped me in ways both mathematical and personal. Thank you, Andrew Berget, Alex Coward, Moon Duchin, Moto Fukuda, Valerie Hower, Peter Malkin, and Robert Sims.

To my academic siblings Ruriko Yoshida, Tyrrell B. McAllister, Susan Margulies, David Haws, and Mohamed Omar thanks for your advice and imparting your mathematical knowledge with me. In addition, I wish to thank my fellow graduate students in mathematics at UC Davis for their conversations and friendships. I want to especially thank Gabriel Amos, Miranda Antonelli,  Emi Arima, Shinpei Baba, Brandon Barrette, Karl Beutner, Julie Blackwood, McCartney Clark, Tom Denton, Patrick Dragon, Pierre Dueck, Eaman Fattouh, Maria Efthymiou, Creed Erickson, Jeff Ferreira, Galen Ferrel, Ben Fineman, Kristen Freeman, Katia Fuchs, Eli Goldwyn, Joshua Gooding, Ezra Gouvea, Matt Herman, Robert Hildebrand, Andrew Hodge, Thomas Hunt, Blake Hunter, Jesse Johnson, Yvonne Kemper, Corrine Kirkbride, Jaejeong Lee, Kristin Lui, Paul Mach, Leslie Marquez, Roberto Martinez, Spyridon Michalakis, Arpy Mikaelian, Adam Miller, Sonny Mohammadzadeh, Marion Moore, Lola Muldrew, Jaideep Mulherkar, Deanna Needell, Stephen Ng, Alex Papazoglou, Shad Pierson, Ram Puri, Tasia Raymer, Hillel Raz, Matthew Rodrigues, Brad Safnuk, Tami Joy Schlichter, Michael Schwemmer, Chengwu Shao, David Sivakoff, Tyler Skorczewski, Matthew Stamps, John Steinberger, Philip Sternberg, Alice Stevens, Eva Strawbridge, Michelle Stutey, Rohit Thomas, Nick Travers, Diana Webb, Brandy Wiegers, Michael Willams, Robin Wilson, Brian Wissman, Ernest Woei, Yuting Yang, and Juliette Zerick. I want to give extra mention to Yvonne Lai, who has helped me with countless amounts of good advice over the years. My experience in graduate school was enriched by your encouragment. Thanks to fellow co-organizers for the Graduate Student Combinatorics Conference 2008, well, and for being just plain awesome graduate students: Steven Pon, Chris Berg, and Sonya Berg. I thank Isaiah Lankham for teaching me about the Galois Group website. Finally, hanging out with you all has given me the chance to meet your wonderful friends and family, not the least of which include Trueda Gooding, Frances Sivakoff, and Russell Mills Campisi. Last but not least, I thank Matt Rathbun for being an awesome roommate during my final years at Davis. I don't know how it was for you, but I truly enjoyed our conversations about mathematics, politics, faith, and everything in between!

I am very grateful to Luis de la Torre for his collaboration in an undergraduate research program. I was honored to teach in the UC Davis Math Circle, and one of my highlights on a Saturday in MSB was mentoring Haoying Meng and Paul Prue.


The research contained here has been funded in part by NSF grant DMS-0608785, NSF VIGRE grants DMS-0135345 and DMS-0636297, NSF VIGRE Summer grants (see the previously-mentioned grant numbers), the UC Davis GSA Travel Award, block grants, research assistantships, teaching assistantships (thanks to all of my students over the years), and the Centre de Recerca Matem\`atica. I would also like to acknowledge software that has been particularly useful in my research: \cite{Avis:lrs}, \cite{Christof:PORTA}, \cite{Fukuda:cdd}, \cite{polymake}, and \cite{topcom}.


A number of friends have kept asking me over the years how my progress has been going in graduate school. Thanks for maintaining these friendships and having an interest in my work. This list includes Chris Anderson, Shara Anderson, Nicole Andrade, Melissa Andrew, Michael Baggett, Nichole Barlow, Tarah Bass, Jenny Bernstein, Bryan Blythe, Tiffany Bock, Jim Bosch, Brad Brennan, Bob Briggs, John Bruneau, Bob Calonico, Michelle Chan, Ming Cheng, Jason Clark, Tanya Cothran, John Creasey, Kevin Dayaratna, Jenna Dockery, John Essig, Andrew Farris, Val Fraser, Krista Frelinger, Becky Gong, Sally Graglia, Cade Grunst, Cindi Guerrero, Kelly Hamby, Chris Hauth, Rhoda Hauth, Hiro Hiraiwa, Miles Hookey, Shereen Jackson, Megan Kinninger, Sam Lam, Alison Li, Shengxi Liu, Nick Matyas, Duncan McFarland, Janice Mochizuki, Allison Moe, Amy Ng, Mike Nguyen, Chris Perry, James Pfeiffer, Duc Pham, Angela Riedel, Zach Saul, Kristen Schroeder, Allie Scrivener, Daisuke Shiraki, Anya Shyrokova, Alec Stewart, Andrew Sturges, Scott Sutherland, Carol Suveda, Peter Symonds, Travis Taylor, Karina van der Heijden, Adam von Boltenstern, and Brian Wolf. To my housemates who have been the most influential over the years, I thank you for your patience whenever I spoke about mathematics. Thanks for your interest and encouragement in the process. They include Aaron Alcal\'a-Mosley, Aaron Campbell, Helen Fong, Jaime Haletky, Chris Marbach, and Kengo Oishi. Thanks to jam session peers from the math department, including Naoki Saito and Matt Herman.

A significant number of people helped maintain my well-being by making sure that I got to dance every once in a while. Among others, thanks are due to Joan Aubin, Melanie Becker, Susan Bertuleit, Scott Blevens, Catherine Blubaugh, Ali Bollbach, Liz Boswell, Jeff Bowman, Tyler Breaux, Amanda Carlson, Shayna Carp, Mark Carpenter, Sarah Catanio, Patrick Cesarz, Jessica Chan, Sunny Chang, Elisa Chavez, Lena Chervin, Shawn Chiao, Howard Chong, Clark Churchill, Rose Connally, Erin Connolly, Lindsay Conway, Lindsay Cooper, Cortney Copeland, AnneMarie Cordeiro, Christina Crapotta, Nicole Croft, Nate Culpepper, Aurelia Darling, Suma Datta, Sherene Daya, Ria DeBiase, Elissa Dodd, Alyssa Douglas, Solomon Douglas, Melissa Drake, Dave Dranow, Yuxi Duan, Bryna Dunnells, Pep Esp\'igol, Alex Estrada, Mike Fauzy, Richard Flaig, Alice Fong, Brandon Frey, Sarah Froud, Tawnie Gadd, Cid Galicia, Christi Gamage, Glenn Gasner, Megan Getchell, Eliott Gray, Jessica Gross, Jane Halahan, Berlyn Hale, Spiro Halikas, David Hamaker, Amanda Harpold, Hunter Hastings, Patrick Haugen, Carla Heiney, Alexis Hinchcliffe, Melissa Hodson, Amber Hoffman, Joe Hopper, Angeline Huang, Casey Hutchins, Anne Johnson, Rachel Jordan, Scott Kaufman, Sarah Kesecker, Dianne King, Elizabeth Kolodziej, Scott Kraczek, Shoshi Krieger, Heidi Langenbacher, Adrianne Larsen, McLeod Larsen, Linda Lee, Adam Lewis, Lindy Lingren, Dave Madison, Nathan Margoliash, James McBryan, Monica McEldowney, Tim McMahon, Sarah Miller, Clay Mitchell, Meredith Moran, Dallas Morrison, Christine Moser, Kendra Nelson, Rochelle Ng, Kyle O'Brien, Luke Oeding, Jocelyn Price, Dan Printz, Rachel Redler, Michelle Richter, Aimee Ring, Devon Ring, Keely Ring, Karissa Ringel, Kristyn Ringgold, Jessica Riojas, Sonia Robinson, Natasha Roseberry, Rebecca Roston, Valarie Rothfuss, Katy Rullman, Scott Sablan, Rasna Sandhu, Dexter Santos, Jennifer Schmidt, Emily Jo Seminoff, Gary Sharpe, Tia Shelley, Jen ``Skittles'' Sherman, Rachel Silverman, Dennis Simmons, Stephanie Smith, Kristin Sorci, Lora Spencer, Yvetta Surovec, Grace Swickard, Brandon Tearse, Alyssa Teddy, Deanna Thompson, Ron Thompson, Jamie Toulze, Ashley Treece, David Trinh, Dirk Tuell, Barrie Valencia, Michelle Vaughan, Matthew Vicksell, Gina Villagomez, Koren Wake, Naomi Walenta, Bo Wang, Katherine Wessel, Cassie Wicken, Johnna Williams, Blythe Wilson, Chana Winger, Wyatt Winnie, Emily Wolfe, Charlie Yarbrough, Christie Young, and Christine Young. I especially thank Stephanie Jordan, Elizabeth Webb, and Ashley Lambert for being there every step of the way (no pun intended). The three of you were never more than a phone call away.

I want to thank the following non-exhaustive list of people for helping me in the toughest of times. Thanks for serving me and allowing me to serve you. I knew that within this community, I could ask for anything I needed at any moment, and you would help me. I offer my thanks to Barry Abrams, Ellen Abrams, Joey Abrams, Michelle Adams, Jackie Ailes, Emily Allen, Mark Allen, Mark Almlie, Janelle Alvstad-Mattson, Steven Andersen, Kelly Arispe, Sergio Arispe, Dana Armstrong, Kiki Arrasmith, Erica Back, Megan Baer, Michelle Balazs, Carrie Bare, Alyssa Barlow, Brenna Barsam, Alex Barsoom, James Basta, Lauren Basta, Cassie Bauman, Emily Beal, Courtney Beed, Bryan Bell, Christina Benton, Naomi Berg, Christine Berlier, Joe Biggs, Jenny Bjerke, Mark Bjerke, Derek Blevins, Tami Bocash, Caitlyn Bollinger, Jessica Bondi, Jamie Boone, Susie Booth, Dave Boughton, Nico Bouwkamp, Deb Bradbury, Andrea Braunstein, Dan Britts, Hannah Brodersen, Nick Brown, Ray Brown, Teresa Bulich, Erik Busby, Elizabeth Busch, Tara Butterworth, Erica Byrd, Kendra Cavecche, Joe Cech, Mary Cech, Mara Chambers, Ashley Champagne, Andrew Chan, Bryan Chan, Carlton Chan, Daniela Chan, Joey Chapdelaine, Crystal Chau, Andrew Cheng, Jenn Cheng, John Cheng, Luke Cheng, Lisa Chow, Tyler Chuck, Angie Chung, Evan Chung, Kristina Coale, Chris Connell, Lisa Cook, Lisa Corsetto, Michael Corsetto, Kristin Cranmer, Katye Crawford, Rachel Crombie, Jamie Crook, Megan Cser, Tash Dalde, Adam Darbonne, Lauren DaSilva, Diane Davis, Monique De Barruel, Nicole de la Mora, Mois\'es de la Torre, Chris Dietrich, Carol Dillard, Michael Dillard, Chris Dombrowski, Daniel Donnelly, David Dorroh, Jesse Doty, Jason Draut, Amy Duffy, Cassia Edwards, Noah Elhardt, Megan Ellis, Nathan Ely, Nancy Emery, Bryan Enderle, Peggy Enderle, Brandon Ertis, Zach Evans, Rick Fasani, Caitlin Flint, Shannon Flynn, Chuck Foster, Ian Foster, Julie Foster, Ben Fowler, Bryan Fowler, Anda Fox, Andrew Frank, Al Frankeberger, Michelle Freeman, Kira Fuerstenau, Elsie Gabby, Jason Galbraith, Josh Galbraith, Angela Garcia, Cindy Garcia, Amy Gaudard, Rusty Gaudard, David Getchel, Stanford Gibson, Diane Gilmer, Kristi Gladding, Sanford Gladding, Jaime Glahn, Kate Green, Stu Gregson, Lisa Greif, Ryan Greif, Jenna Groom, Dana Gross, Doug Gross, Stephanie Gross, Erin Guerra, Jerry Guerzon, Kelsey Guindon, Erica Guo, Christian Guth, Tapua Gwarada, Bonnie Hammond, Eliza Haney, Jacob Hansen, Erik Hanson, Gail Hatch, Peter Hatch, Erin Hawkes, Carolyn Heinz, Jonnalee Henderson, Andrew Hershberger, Katie Hobbs, Jeff Hodges, Christy Holcomb, David Holcomb, Ralph Holderbein, Farris Holliday, Carole Hom, Robert Hom, Yana Hook, Gail Houck, Peter Houck, Jessica Hsiang, Andy Hsieh, Alicia Hunt, Carla Hunt, Charlie Hunt, Alicia Inn, Jeff Irwin, Jennifer Jeske, Leray Jize, Nichole Jize, Nathan Joe, Diane Johnston, Don Johnston, Arend Jones, Laura Judson, Jocelynn Jurkovich-Hughes, Rutendo Kashambwa, David Kellogg, Clyde Kelly, Molly Kinnier, Ian Kinzel, Adam Kistler, Daniel Kistler, Brooke Kline, Thomas Kline, Kate Kootstra, Bethany Kopriva, Christina Kopriva, Julianne Kopriva, Michael Kopriva, Josh Krage, Sara Krueger, Mark Labberton, Velma Lagerstrom, Michael Lahr, Susan Larock, Devon Latzen, Brian Lawrence, Christina Lawrence, Pat Lazicki, Vanessa Lazo, Bronwyn Lea, Jeremy Lea, Austin Lee, Sharon Lee, Stephen Lee, Marguerite Leoni, Lisa Liang, Angela Liao, Kirsten Lien, Christine Lim, Aleck Lin, Monika Lin, Vicky Lin, Shannon Little, Keith Looney, Stacy Looney, Anna Loscutoff, Bill Loscutoff, Carol Loscutoff, Paul Loscutoff, Emily Loui, Lisa Louie, Mary Lowry, Stephanie Luber, Peter Ludden, Danae Lukis, Jason Lukis, Marissa Lupo, Katie Mack, Paul Mackey, Tyler Mackey, Marc Madrigal, Semra Madrigal, Jerrica Mah, Kate Mallison, Nic Mallory, Thea Mangels, Autumn Martinelli, Dale Mathison, Neil Mattson, Brittany Maxwell, Tom McCabe, Ali McKenna, Katy McLaughlin, Joshua McPaul, Jordan Mee, Brent Meyer, Taryn Micheff, Heidi Mills, Natalia Mogtader, Monica Monari, Daniel Mori, Amy Morice, Chris Muelder, Stacy Muelder, Cuka Muhoro, Joshua Mullins, Bronwyn Murphy, Jarrod Murphy, Liz Myer, Carrie Naylor, Dan Naylor, Matt Naylor, Corey Neu, Donna Neu, Phil Neu, Glen Nielsen, Hannah Nielsen, Karen Nielsen, Stephen Nielsen, Jason O'Brien, Robin Oas, James Obert, Paul Ogden, Betsy Onstad, David Ormont, Paul Otteson, Alan Parkin, Andy Patrick, Bryan Pellissier, Chris Pennello, Kristy Perano, Elizabeth Perry, Liz Perry, Sean Pierce, Lindsey Pitman, Jim Plaskett, David Polanco, Weston Powell, Michael Powers, Annie Prentice, Carrol Quivey, David Quivey, Jessica Radon, Michael Raines, Joshua Ralston, Cory Randolph, Dominic Reisig, Rebecca Reisig, Chic Rey, John Ridenour, Marjorie Ridenour, Joy Robbins, Matt Robbins, Jessica Roberts, Brandt Robinson, Chris Rodgers, Robin Lynn Rodriguez, Ra\'ul Romero, David Ronconi, Ellen Rosenberg, Gary Rossetto, Mary Lou Rossetto, Sarah Rundle, Claire Ruud, Paul Ruud, Lena Rystrom, Peter Rystrom, Moriah Saba, Eddie Sanchez, Alex Sasser, Phil Schaecher, Eddie ``Lemma~\ref{lemma:littleeddielemma}'' Schaff, Sara Schaff, Ulrich Schaff, Samantha Schmidt, Sarah Schnitker, Laura Schoenhoff, Gail Schroeder, Allison Seitz, Dan Seitz, Beth Sekishiro, Angie Sera, Emma Shandy, Tim Shaw, Diane Sherwin, Jean Siemens, John Siemens, Steve Simmonds, Leslie Simmons, Julia Skorczewski, Ben Smith, Warren Smith, Serena Smith-Patten, Kristen ``Pearl'' Snow, Glen Snyder, Chris Solis, Kiho Song, Bryce Spycher, Katie Stabler, Katie Stafford, Claire Stanley, Ali ``currently with a Squeak'' Steele, Greg Steele, Jennifer Stephenson, Jen Sterkel, Sarah Stevenson, Carrie Strand, Erik Strand, Nancy Streeter, Jenna Strong, John Strong, Nancy Strong, Nathan Strong, Mary Stump, Noah Suess, Tom Tandoc, Kristin Taniguchi, Mina Tavatli, Shabnam Tavatli, Farrah Tehrani, Jessica Tekawa, Casey Temple, Drew Temple, Elise ter Haar, Emily Thompson, Kyle Thomsen, Grace Tirapelle, Andrew Tkach, Melly Totman, Lily Toumani, Stephanie Towne, Andromeda Townsend, Paul Tshihamba, Yumi Tuttle, Denh Tuyen, Saskia Van Donk, Sabrina Vigil, Nate Vilain, Catherine Vitt, Jeff Vitt, Luke Wagner, Lewis Waha, Jessica Wald, Erin Wang, Regina Wang, Sammy Wang, Amy Ward, Anna Warde, Mark Webb, Debbie Whaley, Tony Whiteford, Pam Whitney, Dave Wieg, Kate Wieking, Erica Williams, Casie Wilson, Kim Windall, Tom Windall, Libby Wolf, Katelyn Wolfe, Cathy Wolfenden, Joey Wolhaupter, Gabe Wong, Alex Wright, Jackie Wu, Tricia Yarnal, Albert Yee, Angela Yee, David Yoo, John Yoo, Josh Yoon, Danny Yost, Chad Young, Doug Yount, and Jerri Zhang. I especially want to thank Briggs (see~\cite{Sekishiro:Briggs}) for his gift in fostering a community within a community. There are particular individuals who have been there for me, including Kevin Kaub, Laeya Kaufman, Karl Schwartz, Dave Tan, Jeff Tan, and Ben Wilson. Thanks for your love. 

While already mentioned, there are several people who really took care of the little things for me while I was nearing the end of the thesis writing project. While something as small as a simple meal may seem like nothing to you, at the moment, it meant to world to me. Thank you, Christine Berlier, Katie Ridenour, and Ben Wilson, among many others. Your patience and flexibility to take care of others is an example worth emulation. 

To Jan Canfield, I want to thank you for being an early mathematical inspiration to me. In your unique place at the School of Humanities, you had to work hard to convince people who did not like math that math can be enjoyable. As a secondary benefit, I was inspired by your teaching to pursue mathematics for its beauty and enjoyment. I would not have made the choice to pursue degrees in mathematics if it were not for your inspiration. I always have your example in mind when I teach mathematics, and I hope that my influence will be as effective as yours was.

I could not imagine getting to this point without the support of my family. To my cousins Melissa, Andy, Bobby, James, Jennifer, Richard, Stephanie, Tom, Joyce, Chris, and Michelle, thanks for being there all of those years. I thank my aunts and uncles for their interest in my academic career. I only hope that I can continue to deserve the pride they bestow. I am deeply thankful for my grandparents, especially in their service to the community and encouragement to work hard in school. To my brother John, thanks for your encouragement over the years: your own passion to aspire for more has been a good example for me. To my parents, thank you for your confidence in me as I pursued my dreams. Thank you for many years of love, care, and teaching.

Finally, I express my gratitude to Katie Ridenour for being so patient, kind, and loving as I finished this work. Thank you so much for your self-sacrificial character -- I have much to learn and emulate.

\begin{flushright}
Eddie\makebox[1.0in]{}
\end{flushright}

{\singlespacing{\sl\noindent{}Noticing that the hundred-day-old baby didn't grab any of the superstition-filled objects placed before him, Harabuji offered him the pencil, which represents scholarship.}}
\vskip1.0in
\begin{center}
\includegraphics[width=2.5in]{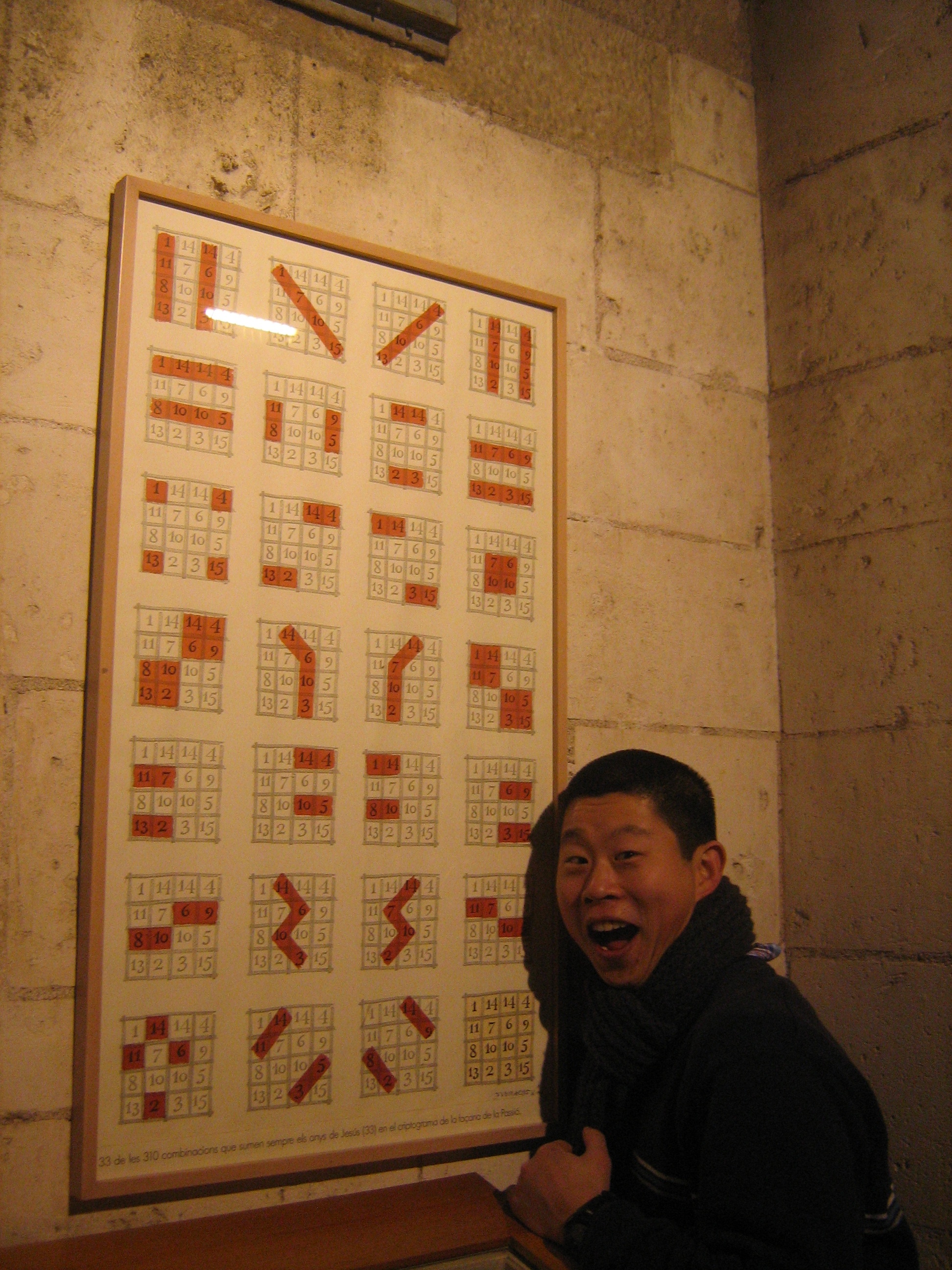}
\end{center}

   
   \mainmatter
   
   \pagestyle{maintext}
   
   \fancypagestyle{plain}{%
      \renewcommand{\headrulewidth}{0pt}
      \fancyhf{}
      \rhead{\thepage}
   }

\chapter[Polytopes, Polyhedra, and Optimization]{Polytopes, Polyhedra, and Applications to Optimization}\label{chapter:polytopes}

\setcounter{theorem}{0} 

A polytope is the generalization of two-dimensional convex polygons and three-dimensional convex polyhedra in higher dimensions. Polytopes appear as the central objects of study in the area of geometric combinatorics, but they also appear in mathematical fields as diverse as algebraic geometry (see, e.g.,~\cite{Oda-book},~\cite{Stanley:CombinatoricsCommutativeAlgebra}, and~\cite{Sturmfels:GrobnerPolytopes}), representation theory (see, e.g.,~\cite{Bjorner-Brenti},~\cite{FominZelevinski-associahedra}, and~\cite{FominZelevinski-clusteralgebrasIV}), lattice point enumeration (see, e.g.,~\cite{Barvinok:Integer-points-in-polyhedra-BOOK},~\cite{Beck-Robins},~\cite{deloera-latticepoints},~\cite{latte}, and~\cite{Yoshida:Barvinoks-Rational}), and optimization (see, e.g.,~\cite{Dantzig:Linear},~\cite{Grotschel:Geometric-Algorithms}, and~\cite{Schrijver:LinearProgramming}). This dissertation discusses connections between optimization and the geometric combinatorics of polytopes. 

Before formally introducing polytopes, we discuss some concepts from affine linear algebra that will appear throughout this dissertation. Fix a positive integer $c \in \Z_{\geq 0}$. We will typically work in the $c$-dimensional Euclidean space $\R^c$. A set $X \subseteq \R^c$ is \defn{convex} if $\lambda x + (1-\lambda)x' \in X$ for all pairs $(x,x') \in X \times X$ and every $\lambda \in [0,1]$. The \defn{convex hull} $\conv(X)$ of a set $X \subseteq \R^c$ is the smallest convex set that contains $X$. That is to say, the convex hull of $X \subseteq \R^c$ is defined as:
\begin{equation*}
\conv(X) = \bigcap_{\genfrac{}{}{0pt}{}{X \subseteq C}{C \text{ convex}}} C.
\end{equation*}
The convex hull of a convex set is the set itself. (For more on convexity, refer to~\cite{Barvinok:CourseConvexity}, Chapter~1 in~\cite{Matousek:LecturesDiscreteGeometry}, or~\cite{Wets:Optimization}.)

Now we give our first examples of convex sets. A set $X \subseteq \R^c$ is called an \defn{affine subspace} if it is a parallel translate of a linear subspace. The dimension $\dim(X)$ of an affine subspace $X$ is the dimension of the parallel linear vector subspace. The zero-dimensional affine subspaces are the \defn{points} and the one-dimensional affine subspaces are the \defn{lines}. The affine subspaces of dimension $c-1$ are called \defn{hyperplanes}. (In other words, we say that hyperplanes are codimension one affine subspaces.) The \defn{affine hull} $\aff(X)$ of a set $X \subseteq \R^c$ is the smallest affine subspace that contains $X$.

Formally, a \defn{polytope} $P$ is the convex hull of a finite number of points in $\R^c$. (According to our definition, a polytope is always convex, so we do not distinguish between \defn{polytopes} and \defn{convex polytopes}.) The integer $c \in \Z_{\geq 0}$ is called the \defn{ambient dimension} of the polytope. This number may differ from the dimension of the polytope: the \defn{dimension} $\dim(P)$ of a polytope $P$ is the dimension $\dim(\aff(P))$ of its affine hull $\aff(P)$.  We typically use $d$ to denote the dimension of a polytope. A $d$-dimensional polytope is also called a \defn{$d$-polytope}. When $c=d$, we say that the polytope $P$ is \defn{full-dimensional}. Otherwise, the number $c-d > 0$ is the \defn{codimension} of $P$.

A \defn{polyhedron} $P$ is the intersection of finitely many closed half-spaces in $\R^c$. A polytope is also a polyhedron: more precisely, polytopes are those subsets of $\R^c$ that are both bounded and polyhedra (see, e.g.,~\cite{Ziegler:Lectures}). The dimension, ambient dimension, and codimension of a polyhedron are defined in the same way as for a polytope. A $d$-dimensional polyhedron is also called a \defn{$d$-polyhedron}.%

\begin{remark}
We have given two equivalent definitions of a polytope. A polytope is either the convex hull of a finite set or the bounded intersection of a finite number of closed half-spaces. Though the definitions are equivalent (due to the Weyl-Minkowski Theorem -- see, e.g.,~\cite{Ziegler:Lectures}), from a computational point of view it makes a difference whether a certain polytope is represented as a convex hull or via linear inequalities: the size of one description cannot be bounded polynomially in the size of the other, if the dimension $d$ is not fixed (see~\cite{Avis:How-good-are-Convex}). For polynomial-time convex hull algorithms on a finite set of points in dimension two, see~\cite{deBerg:ComputationalGeometry} or Chapter~33 of~\cite{Cormen:Algorithms}. (For algorithms in higher dimensions, see, e.g.,~\cite{Chazelle:An-optimal-convex}.)
\end{remark}

A hyperplane is the set of points $x \in \R^c$ satisfying $a_1x_1 + \cdots a_cx_c = m$ for some non-zero vector $a = (a_1,\ldots,a_c) \in \R^c$ and a real number $m$. (A \defn{linear hyperplane} is a hyperplane that contains the origin, i.e., $m=0$.) Since a hyperplane is the intersection of the two half-spaces 
\begin{equation*}
\{x \in \R^c \mid a_1x_1 + \ldots a_cx_c \leq m\}
\text{ and }
\{x \in \R^c \mid a_1x_1 + \ldots a_cx_c \geq m\},
\end{equation*}
we can also say that a polyhedron is the intersection of a finite number (possibly zero) of linear equations and a finite number of linear inequalities.
\begin{remark}
Though we have already given one definition for a hyperplane and two definitions for a polytope, it will be helpful to have in mind a second definition for a hyperplane and a third definition for a polytope. For this, we need to introduce affine and convex combinations.

Let $x_1,\ldots,x_n$ be a finite collection of $n$ points in $\R^c$. A point of the form
\begin{equation}\label{equation:affinecombination}
x = \sum_{i=1}^n \lambda_i x_i
\end{equation}
where $\sum_{i=1}^n \lambda_i = 1$ is called an \defn{affine combination} of the points $x_1,\ldots,x_n$. The collection of points $x_1,\ldots,x_n$ are said to be \defn{affinely independent} if none of the points can be written as an affine combination of the remaining $n-1$ points. A hyperplane is the set of all possible affine combinations of $c$ affinely independent points in $\R^c$. More generally, an affine subspace is the set of all possible affine combinations of $n$ points in $\R^c$.

An affine combination of the points $x_1,\ldots,x_n$ in $\R^c$ of the form \eqref{equation:affinecombination} that meets the additional condition that each $\lambda_i \geq 0$ is called a \defn{convex combination} of the points $x_1,\ldots,x_n$. A polytope is the set of all possible convex combinations of a finite set of points. The collection of points $x_1,\ldots,x_n$ are said to be in \defn{convex position} if none of the points can be written as a convex combination of the remaining $n-1$ points. For example, for two distinct points $x$ and $y$ in $\R^c$, the \defn{interval} between them is the set $[x,y]=\{z \in \R^c \mid z = \lambda x + (1-\lambda)y, 0 \leq \lambda \leq 1\}$ of all possible convex combinations.
\end{remark}
We will show in Section~\ref{section:polytopesandLP} (see Lemma~\ref{lemma:standardform}) that polytopes and polyhedra in optimization are typically presented in the form
$P = \{x \in \R^c \mid Ax = b \text{ and } x \geq 0\}$
where $A$ is a matrix of size $r \times c$ and $b$ is a vector in $\R^r$.
\begin{definition}
Let $A$ be a matrix of size $r \times c$ and let $b$ b a vector in $\R^r$. A polyhedron of the form
\begin{equation}\label{equation:partitionpolyhedron}
P = \{x \in \R^c \mid Ax = b \text{ and } x \geq 0\}
  = \{x \in \R_{\geq 0}^c \mid Ax = b\}
\end{equation}
is called a \defn{partition polyhedron}. The matrix $A$ is called the \defn{constraint matrix} of the partition polyhedron $P$.
\end{definition}
The name \emph{partition polyhedron} is motivated by Lemma~\ref{lemma:partitionvertices} and the discussion in Section~\ref{section:enumerationofpartition}.
\begin{remark}
We will define transportation polytopes in Section~\ref{section:TPintro}, but for now we simply remark that they are of this form.
\end{remark}

\begin{example}\label{example:basic}
Here are examples of polytopes and polyhedra. They will reappear as running examples throughout this chapter:
\begin{itemize}
\item {\bf The $d$-simplex.}
The $d$-simplex generalizes the triangle in dimension two and the tetrahedron in dimension three. The \defn{standard $d$-dimensional simplex} $\Delta_d$ is the convex hull of the standard unit vectors $e_1,\ldots,e_{d+1}$ in $\R^{d+1}$. In this description, the ambient dimension of $\Delta_d$ is $c = d+1$. The $d$-simplex can also be defined as a bounded partition polyhedron: $\Delta_d = \{x \in \R^{d+1} \mid x_1 + \cdots + x_{d+1} = 1, 0 \leq x \leq 1\}$. See Figure~\ref{figure:examplesofsimplices}.
\begin{figure}[hbt]
\begin{center}
\includegraphics[scale=.7]{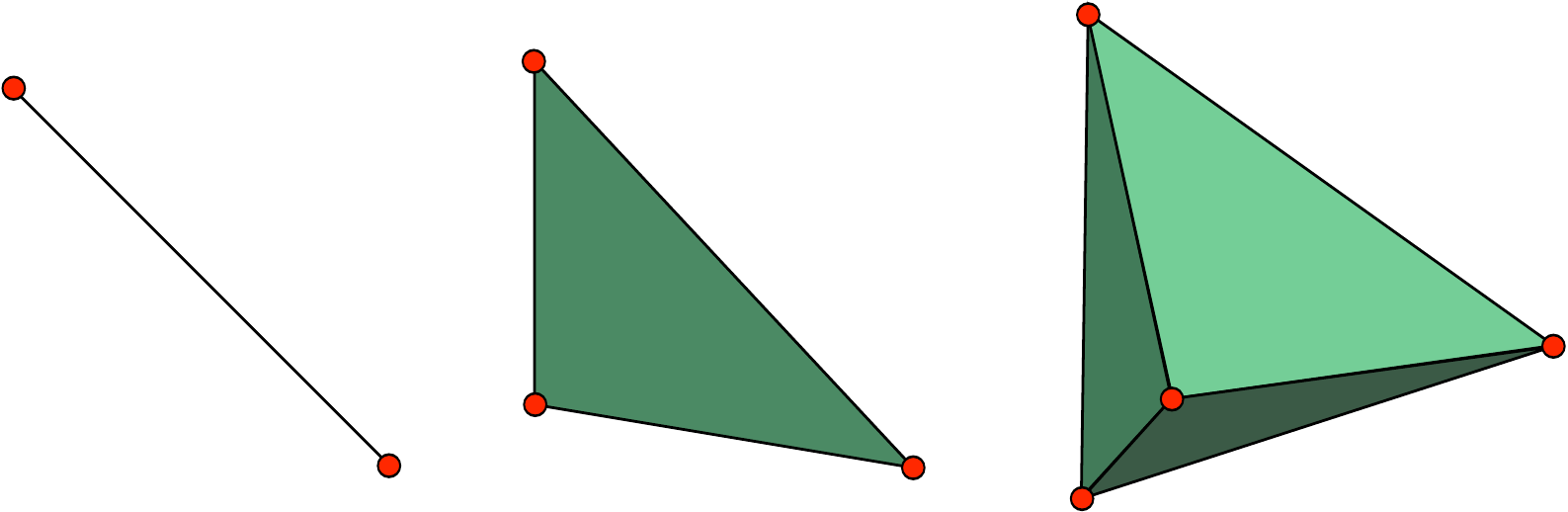}
\caption{The standard $1$-simplex, $2$-simplex, and $3$-simplex.}
\label{figure:examplesofsimplices}
\end{center}
\end{figure}

\item {\bf The $d$-cube.}
The \defn{$d$-dimensional $0$-$1$ cube} $\Box_d$ is the convex hull of the $2^d$ points with coordinates $0$ or $1$. In this description, the $d$-cube is full-dimensional, i.e., $c=d$. The $d$-cube can also be described as the intersection of $2d$ inequalities, i.e., $\Box_d = \{x \in \R^d \mid 0 \leq x_i \leq 1\ \forall i\}$. See Figure~\ref{figure:cubeexample}.
\begin{figure}[hbt]
\begin{center}
\includegraphics[scale=0.85]{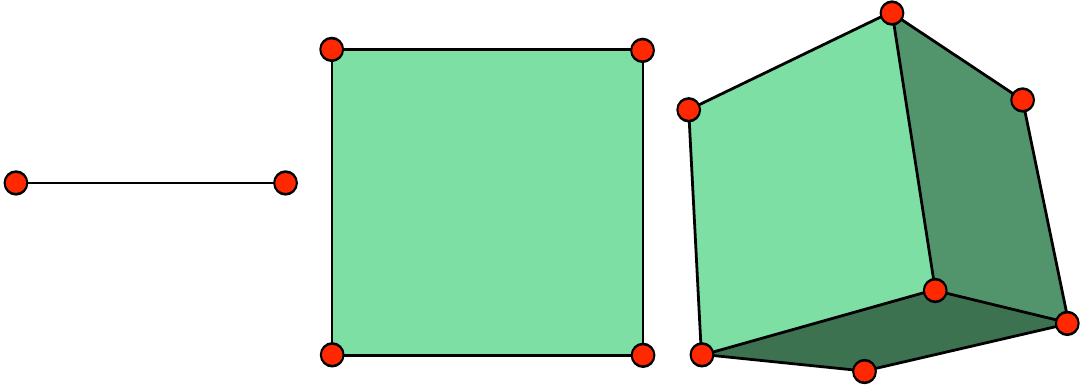}
\caption{The $0$-$1$ cubes of dimensions one, two, and three.}
\label{figure:cubeexample}
\end{center}
\end{figure}

\item {\bf The $d$-dimensional cross polytope.}
The \defn{$d$-dimensional cross polytope} $\Diamond_d$ is the convex hull of the $d$ standard unit vectors $e_1,\ldots,e_d$ and their negatives. For $d=3$, this is an octahedron. See Figure~\ref{figure:crosspolytopeexample}.

\item Let $P_{3 \times 3} = \{ x \in \R^9 \mid Ax = b, x \geq 0 \}$ where
\begin{equation}
A = \left[
\begin{array}{ccccccccc}
1&0&0&1&0&0&1&0&0 \\
0&1&0&0&1&0&0&1&0 \\
0&0&1&0&0&1&0&0&1 \\
1&1&1&0&0&0&0&0&0 \\
0&0&0&1&1&1&0&0&0 
\end{array}
\right]
\quad
\text{\rm and }
\quad
b = \left[
\begin{array}{c}
2 \\ 7 \\ 2 \\ 5 \\ 5 
\end{array}
\right ]
.
\end{equation}
We will prove later that the polyhedron $P_{3 \times 3}$ is bounded and is, therefore, a polytope. (In fact, we will later show that it is a $3 \times 3$ transportation polytope.) Since the $5 \times 9$ matrix $A$ has full row rank, the solution set of the matrix equation $Ax = b$ is an affine subspace of $\R^9$ of dimension $9-5$. Thus, $P_{3 \times 3}$ is a four-dimensional polytope defined in a nine-dimensional ambient space.

\end{itemize}

\begin{figure}[hbt]
\begin{center}
\includegraphics[scale=1.0]{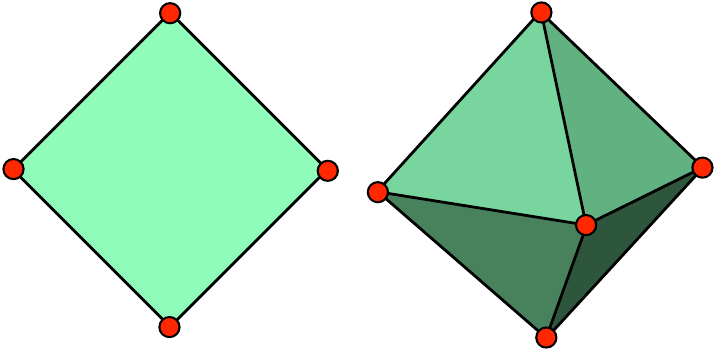}
\caption{The cross polytopes of dimensions two and three.}
\label{figure:crosspolytopeexample}
\end{center}
\end{figure}

\end{example}
The main examples of unbounded polyhedra we will see are the polyhedral cones. A subset of $\R^c$ that is closed under addition and under multiplication by non-negative scalars is called a \defn{cone}. Equivalently, a cone is the intersection of linear half-spaces. (A \defn{linear half-space} is a half-space defined by an equation of the form $a_1x_1 + \cdots + a_cx_c \geq 0$, for some non-zero vector $a \in \R^c$.) For any set $X \subseteq \R^c$, the \defn{cone generated by $X$}, denoted $\cone(X)$, is the set of all points that can be expressed as non-negative linear combinations of vectors in $X$. A cone generated by a finite set $X$ is called a \defn{polyhedral cone}. Equivalently, a cone is polyhedral if it is the intersection of \emph{finitely many} closed linear half-spaces. (Since we are only interested in polyhedral cones, we will use the terms ``cone'' and ``polyhedral cone'' interchangeably. That is to say, all cones we discuss are finitely generated.) A $d$-dimensional cone is called \defn{simple} if it is generated by a set $X$ of cardinality $d$.
\begin{example}
The orthant $\{x \in \R^d \mid x_i \geq 0 \text{ for all }i \in [d]\}$ is a simple $d$-dimensional cone.
\end{example}
We now introduce the primary concepts that will appear throughout this dissertation. For further details, see~\cite{Grunbaum:Polytopes} or~\cite{Ziegler:Lectures}. 

We focus on particular subsets of a polytope (or a polyhedron) $P \subseteq \R^c$ known as faces. A face $F$ of $P$ is obtained in the following way: let $a \in \R^c$ and $m \in \R$ such that 
\begin{equation}\label{equation:valid}
\langle a,x\rangle = a_1 x_1 + \cdots + a_c x_c \leq m \text{ for all } x = (x_1,\ldots,x_c) \in P.
\end{equation}
When this condition is satisfied, we say that the inequality \eqref{equation:valid} is \defn{valid} and determines the face $F = P \cap \{x \in \R^c \mid \langle a,x \rangle = m\}$. Now, suppose that $a \in \R^c$ is non-zero. Then the inequality \eqref{equation:valid} defines a closed half-space $H$ and the equation $\langle a,x\rangle = m$ is the boundary hyperplane of $H$, denoted by $\partial H$. The set $F$ is non-empty only when $\partial H$ intersects the boundary of $P$. In other words, a face is the intersection of $P$ with a \defn{supporting hyperplane}. See Figure~\ref{figure:facedefinition} for an illustration.
\begin{figure}[hbt]
\begin{center}
\includegraphics[scale=1.0]{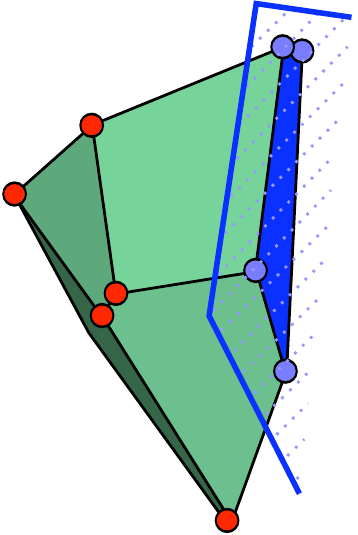}
\caption{A polytope $P$ and one of its non-empty faces. The supporting hyperplane that defines the highlighted face is shown.}
\label{figure:facedefinition}
\end{center}
\end{figure}
Since $F$ is the solution set of linear equations and inequalities, this subset of $P$ is a polyhedron as well. A face of dimension $i$ is called an \defn{$i$-face}. The faces of certain dimensions have special names. The $0$-faces are called \defn{vertices} of $P$, the $1$-faces are called \defn{edges}, the $(d-2)$-faces are called \defn{ridges}, and the $(d-1)$-faces are called \defn{facets}. What happens if we pick $a \in \R^c$ to be the zero vector? By choosing $m \in \{0,1\}$, we see that $P$ itself and the empty set $\emptyset$ are both faces of $P$. These two faces are called the \defn{non-proper} faces and the faces of intermediate dimension are called the \defn{proper} faces. By convention, the empty face $\emptyset$ is said to have dimension $-1$.

\begin{remark}
The definitions above are motivated by what faces of polytopes ``look like.'' For cones and unbounded polyhedra, we typically make a distinction between the bounded and unbounded one-dimensional faces. The \emph{bounded} faces of dimension one are called \defn{bounded edges} while the \emph{unbounded} faces of dimension one are called \defn{rays}. (A cone may have at most one bounded non-empty face, namely the origin. The unbounded faces of a cone are themselves cones. If the origin is a vertex of a cone, then it is called \defn{pointed}. The orthants are examples of pointed cones.)
\end{remark}

The relative interior of a set $X \subseteq \R^c$ is a useful refinement of the concept of the interior of $X$. First, fix any metric on $\R^c$. (Though any metric works, we should have the Euclidean metric in mind.) Recall that the \defn{interior} $\int(X)$ of a set $X \subseteq \R^c$ is
\begin{equation*}
\int(X) = \{x \in \R^c \mid \exists\, \epsilon > 0, N_\epsilon(x)  \subseteq X\},
\end{equation*}
where $N_\epsilon(x)$ is the open ball of radius $\epsilon$ centered at the point $x$. In particular, if the polyhedron $P$ is not full-dimensional, then $\int(P)=\emptyset$. (Intuitively, this is the very thing we want to avoid as we try to discuss the ``inside'' of a polytope or polyhedron. For this, we define the relative interior.) The \defn{relative interior} $\int^*(X)$ of a set $X$ is defined as its interior within its affine hull $\aff(X)$. That is,
\begin{equation*}
\int^*(X) = \{x \in \R^c \mid \exists\, \epsilon > 0, N_\epsilon(x) \cap \aff(X) \subseteq X\}.
\end{equation*}
\begin{figure}[bth]
\begin{center}
\includegraphics[scale=0.85]{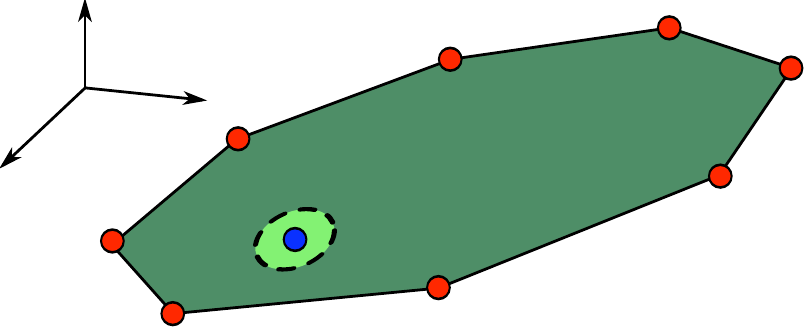}
\caption{A polygon $P$ defined in a three-dimensional ambient space and a point $x$ in the relative interior of $P$. The set $N_\epsilon(x) \cap \aff(P) \subseteq P$ is depicted.}
\label{figure:relativeinterior}
\end{center}
\end{figure}
See Figure~~\ref{figure:relativeinterior} for an example. Note that $\dim(\aff(\int^*(X))) = \dim(\aff(X))$. In particular,
the relative interior of a point $x \in \R^c$ is itself. Unlike the interior of a polyhedron $P$, the relative interior of $P$ remains invariant under embedding in higher-dimensional ambient spaces. The relative interior of a polyhedron $P$ is the set of all points in $P$ that do not belong to any proper face of $P$. The \defn{relative boundary} of a set $X$ is $\partial^*(X) = \overline{X} \setminus \int^*(X)$, where $\overline X$ denotes the closure of $X$. Since polyhedra are closed, the relative boundary of a polytope or polyhedron is the union of all of its proper faces. To make these concepts clear, we state a fundamental decomposition theorem for polytopes and polyhedra. An illustration is given for a tetrahedron in Figure~\ref{figure:polyhedrondecomposition}.
\begin{proposition}[Polyhedron Decomposition Theorem]\label{proposition:polyhedrondecomposition}
Let $P \subseteq \R^c$ be a polyhedron. Then $P$ is the disjoint union of the relative interiors of its faces.
\end{proposition}
\begin{figure}[hbt]
\begin{center}
\includegraphics[scale=0.6]{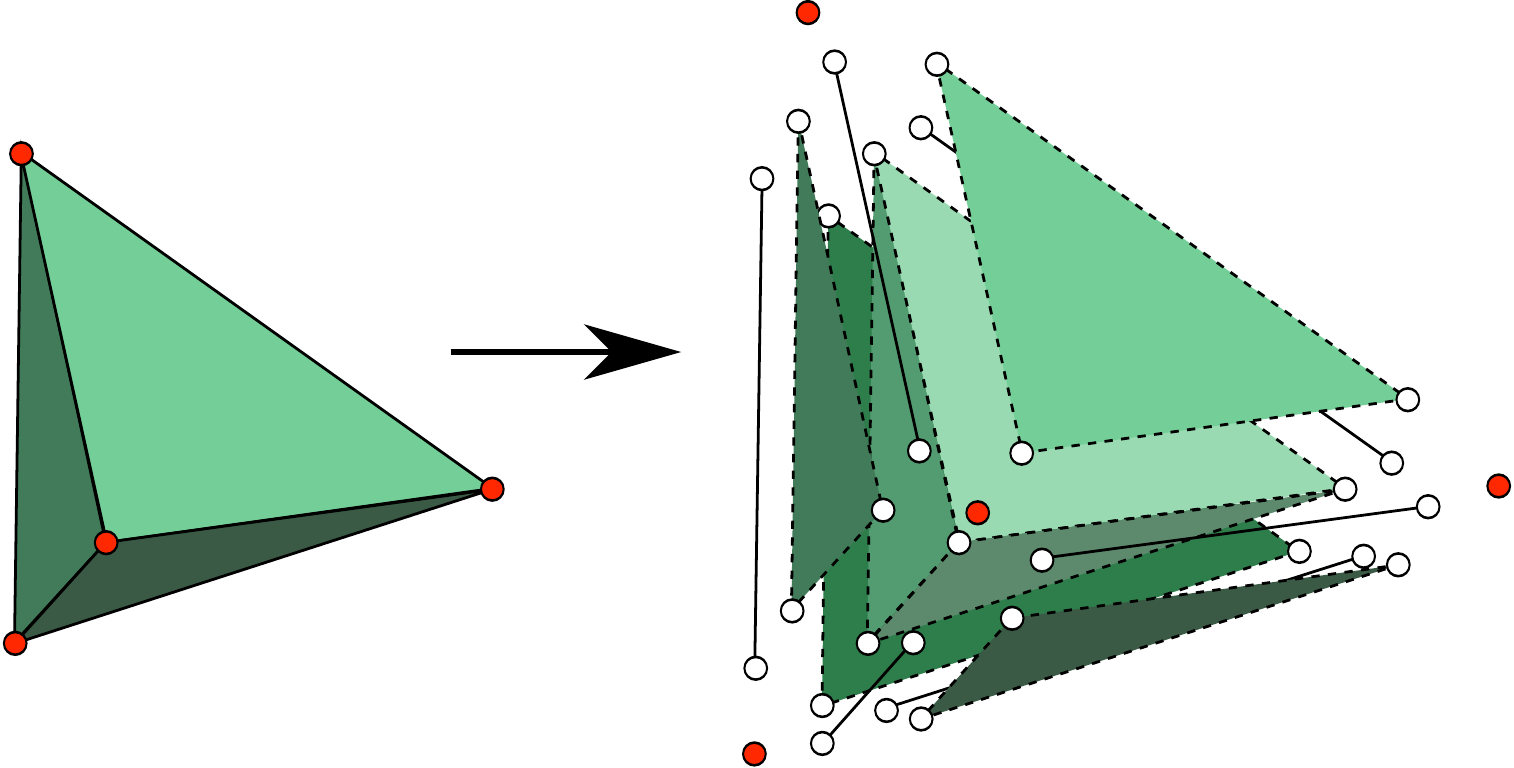}
\caption{A tetrahedron $P$ decomposed into the relative interiors of its faces.}
\label{figure:polyhedrondecomposition}
\end{center}
\end{figure}

\begin{remark}
In defining a polytope $P$ as the convex hull of a finite set $X$ of points, one may give a description of $P$ that is \defn{redundant} if the set $X$ includes points that are ``on the inside of $P$.'' Specifically, if $x \in X$ is in the relative interior of the polytope $P = \conv(X)$ (or even stronger, if $x \in X$ is not a vertex of $P$), then $P = \conv(X) = \conv(X \setminus \{x\})$. In its irredundant description, a polytope is the convex hull of the set of its vertices. (See Figure~\ref{figure:redundant}.) Similarly, in its irredundant description, a polyhedron is the intersection of its facet-defining closed half-spaces. A cone is minimally generated by its rays.
\begin{figure}[hbt]
\begin{center}
\includegraphics[scale=0.9]{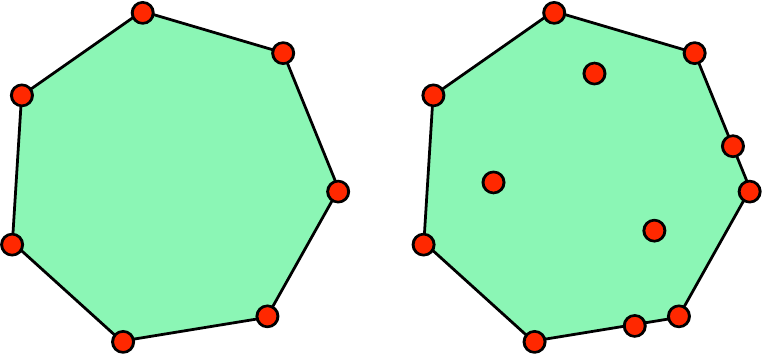}
\caption{The same polygon, in irredundant and redundant descriptions.}
\label{figure:redundant}
\end{center}
\end{figure}
\end{remark}

We present a well-known characterization of vertices of a polytope or polyhedron (see Section 2.4 of~\cite{Grunbaum:Polytopes}).
\begin{lemma}\label{lemma:vertexequivalence}
Let $P \subseteq \R^c$ be a $d$-dimensional polyhedron and suppose $x \in P$. The following conditions are equivalent:
\begin{enumerate}
\item The point $x \in \R^c$ is a vertex of $P$.
\item For every non-zero vector $y \in \R^c$, at most one of $x+y$ or $x-y$ belong to $P$.
\end{enumerate}
\end{lemma}
\begin{proof}
First, suppose that the point $x \in P$ is a vertex of $P$. Suppose, for a contradiction, that there is a non-zero vector $y \in \R^c$ such that both $x+y$ and $x-y$ belong to $P$. Then, the points $x+y$ and $x-y$ are distinct and the set $L = [x-y,x+y] \subseteq \R^c$ defines a line segment. Moreover, since $P$ is convex, it contains $L$. Thus, any supporting hyperplane of $P$ that contains $x$ also must contain $L$. But this means that any face of $P$ that contains $x$ also contains the line segment $L$, and therefore the singleton set $\{x\}$ cannot be a face of $P$, a contradiction.

Conversely, suppose that $x \in P$ is not a vertex of $P$. By Proposition~\ref{proposition:polyhedrondecomposition}, the point $x$ is in the relative interior of some $i$-face $F$ of $P$, with $i \geq 1$. Clearly, $x+y$ and $x-y$ both belong to $F$ for any sufficiently small vector $y \in \R^c$ parallel to the affine hull of $F$.
\end{proof}
\begin{example}\label{example:basicvertices}
Let us revisit some of the polyhedra of Example~\ref{example:basic}:
\begin{itemize}
\item
The $d$-simplex $\Delta_d$ has $d+1$ vertices and $d+1$ facets. Every pair of vertices is contained in an edge.

\item
The vertices of the $d$-cube $\Box_d$ are the $2^d$ points with coordinates $0$ or $1$. Two vertices $x$ and $x'$ of $\Box_d$ are contained in the same edge exactly when all but one of the coordinates of $x$ and $x'$ are the same. The facets of $\Box_d$ are given by the $2d$ inequalities $0\leq x_i\leq 1$.

Let us verify Lemma~\ref{lemma:vertexequivalence} for some points on the $3$-cube $\Box_3$. The point $x = (0,\frac{1}{2},0)$ lies on an edge of $\Box_3$. A vector such as $y = (0, \frac{1}{4},0)$ shows that condition (2) of the lemma holds. See Figure~\ref{figure:non-vertex-on-cube}. For a vertex, the condition is easiest to see at $(0,0,0)$, where one can easily check that for non-zero vectors $y$ belonging to $P$, the vector $-y$ is not in $P$.
\begin{figure}[hbt]
\begin{center}
\includegraphics[scale=1]{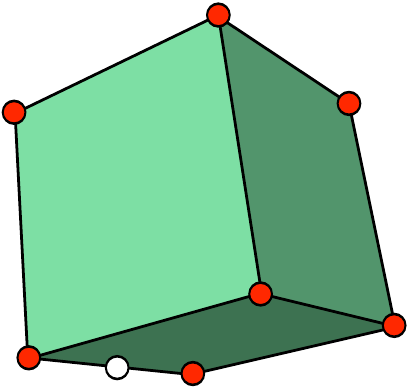}
\caption{The point $x = (0,\frac{1}{2},0)$ is not a vertex of the $3$-cube $\Box_3$. From $x$, one can move in the direction $y = (0, \frac{1}{4},0)$ and in the opposite direction $-y=(0,-\frac{1}{4},0)$.}
\label{figure:non-vertex-on-cube}
\end{center}
\end{figure}

\item
The $d$-dimensional cross polytope $\Diamond_d$ has $2d$ vertices: the $d$ standard unit vectors and their negatives. It has $2^d$ facets, one in each orthant of $\R^d$.

\item
Using the software {\tt polymake} (see~\cite{polymake}) available at~\cite{polymakeurl}, one can compute the twelve vertices of the polytope $P_{3 \times 3}$. They are: 
\begin{align*}
z_1 = (2,2,1,0,5,0,0,0,1),
z_2 = (2,1,2,0,5,0,0,1,0),
z_3 = (1,2,2,0,5,0,1,0,0),\\
z_4 = (1,4,0,0,3,2,1,0,0),
z_5 = (2,3,0,0,3,2,0,1,0),
z_6 = (2,3,0,0,4,1,0,0,1),\\
z_7 = (0,4,1,2,3,0,0,0,1),
z_8 = (0,3,2,2,3,0,0,1,0),
z_9 = (0,3,2,1,4,0,1,0,0),\\
z_{10} = (0,5,0,2,2,1,0,0,1),
z_{11} = (0,5,0,2,1,2,0,1,0),
z_{12} = (0,5,0,1,2,2,1,0,0).
\end{align*}
\end{itemize}
\end{example}

If $P$ is presented as a partition polyhedron of the form~\eqref{equation:partitionpolyhedron}, then there is another characterization for the vertices of $P$, which we describe now. If $P = \{x \in \R^c \mid Ax=b, x \geq 0 \}$ is non-empty, then we can assume that the matrix $A$ has full rank. Indeed, if $A$ is not full rank, then one or more linear equations in $Ax=b$ is redundant. So, let $A$ be an $r \times c$ matrix of full row rank and let $b \in \R^r$. Let $P = \{x \in \R^c \mid Ax=b, x \geq 0 \}$. With a slight abuse of notation, by $\cone(A)$ we mean the cone generated by the set of column vectors of the matrix $A$. A maximal linearly independent subset $\mathcal{A}$ of columns of $A$ is a \defn{basis} of $A$. Geometrically, each basis $\mathcal{A}$ of the matrix $A$ spans a simple cone inside $\cone(A)$. (Recall that a $d$-dimensional cone is simple if it has exactly $d$ rays.) Every basis $\mathcal{A}$ of $A$ defines a \defn{basic solution} of the system as the unique solution of the $r$ linearly independent equations $\mathcal{A} x_\mathcal{A}=b$ and $x_j=0$ for $j$ not in $\mathcal{A}$. A basic solution is \defn{feasible} if, in addition, $x \geq 0$. Geometrically, a basic feasible solution corresponds to a simple cone that contains the vector $b$. In fact, one can see that the polytope $P$ is non-empty if and only if $b \in \cone(A)$. A fundamental fact in linear programming is that, for a given vector $b \in \R^r$, all vertices of the polyhedron $P_b$ are basic feasible solutions (see, e.g.,~\cite{Schrijver:LinearProgramming} or~ \cite{Yemelichev:Polytopes}). So, we have proved the following characterization for vertices of partition polyhedra:
\begin{lemma}\label{lemma:partitionvertices}
Fix an $r \times c$ real matrix $A$ with full row rank $r$. Fix a vector $b \in \R^r$. Let $P$ be the partition polyhedron $P=\{x \in \R^c \mid Ax = b, x \geq 0 \}$. Then $x$ is a vertex of the polyhedron $P$ if and only if $x$ is a basic feasible solution.
\end{lemma}

Polytopes and polyhedra have a finite number of vertices:
\begin{proposition}
Let $P$ be a $d$-dimensional polyhedron. Then, $P$ has a finite number of vertices.
\end{proposition}
\begin{proof}
Clearly, a $d$-polytope $P$ has a finite number of vertices. A $d$-polyhedron $P$ defined by $n$ facets also has a finite number of vertices: since at most one vertex is found at the intersection of $d+1$ facets, the number of vertices is trivially bounded above by $\binom{n}{d+1}$.
\end{proof}
Let $\vert(P)$ denote the set of vertices of $P$. We denote its cardinality by $f_0 = |\vert(P)|$. Polytopes and polyhedra have finitely many faces:
\begin{proposition}
Let $P$ be a $d$-dimensional polyhedron. Then, $P$ has a finite number of $i$-dimensional faces, for $i \in \{-1,\ldots,d\}$.
\end{proposition}
\begin{proof}
Let $f_0 \in \Z_{>0}$ denote the number of vertices of $P$.  Each $i$-face of $P$ contains $i+1$ affinely independent vertices of $P$, and different faces of $P$ have different affine hulls. Thus, the number of different $i$-faces of $P$ is bounded above by $\binom{f_0}{i+1}$.
\end{proof}
Let $f_i = f_i(P) \in \Z_{>0}$ denote the number of $i$-dimensional faces of $P$. The \defn{$f$-vector} of $P$ is $f(P) = (f_{-1}, f_0, f_1, \ldots, f_{d-1}, f_d) \in \Z_{>0}^{d+2}$. Note that $f_{-1} = f_d = 1$, but there is no complete characterization for the values of the other entries: the underlying interaction between combinatorics and algebraic geometry led to a complete characterization of the possible numbers of faces that a simplicial polytope can have (see~\cite{Stanley:The-number-of-faces}). The same question for arbitrary polytopes is open in dimension four and higher (see~\cite{Ziegler-facenumbers}). The $f$-vector of every polytope satisfies the following well-known relation (see, e.g.,~\cite{Hatcher:AlgebraicTopology}):
\begin{proposition}[Euler-Poincar\'e Relation]
Let $P$ be a polytope and let
\begin{equation*}
f(P)=(f_{-1}, f_0, f_1, \ldots, f_{d-1}, f_d)
\end{equation*}
be its $f$-vector. Then, the terms of the $f$-vector satisfy
\[\sum_{i=-1}^d (-1)^i f_i  = -f_{-1} + f_0 - f_1 + \cdots  + (-1)^d f_d = 0.\]
\end{proposition}
In the terminology of~\cite{Hatcher:AlgebraicTopology}, the reduced Euler characteristic of a polytope is zero, since polytopes are contractible.

\section{Graphs of polytopes and polyhedra}\label{section:graphsIntro}

We now introduce the graph of a polytope. Let $P$ be a polytope. The \defn{graph} (or \defn{$1$-skeleton}) of $P$, denoted by $G(P)$, is the following undirected, finite, simple graph:
\begin{itemize}
\item {\bf Vertices of $G(P)$.} The graph $G(P)$ has a vertex for every vertex $v$ of the polytope $P$. Let us denote the vertex of the graph $G(P)$ corresponding to the vertex $v$ of the polytope $P$ by $G(v)$.
\item {\bf Edges of $G(P)$.} Two vertices $G(v)$ and $G(v')$ in $G(P)$ are connected by an edge in $G(P)$ if there is an edge of the polytope $P$ containing the corresponding vertices $v$ and $v'$ of $P$. When this occurs, the vertices $v$ and $v'$ of the polytope $P$ are said to be \defn{neighbors}.
\end{itemize}
Figure~\ref{figure:polytopegraph} gives an intuitive example of the graph of a polytope. Let us review some basic terminology from graph theory. For more information, refer to~\cite{Bollobas:GraphTheory},~\cite{Bondy:Graph-Theory}, or~\cite{Tutte:Graph-Theory}. The \defn{distance} between two vertices in a graph is the minimum number of edges needed to go from one vertex to the other vertex. The \defn{diameter} of a graph is the maximum distance between all pairs of vertices.  The number $\diam(G(P))\in \Z_{\geq 0}$ is the \defn{diameter of the graph of $P$}. This number is also called the \defn{diameter of the polytope $P$}.
\begin{figure}[hbt]
\begin{center}
\includegraphics[scale=1]{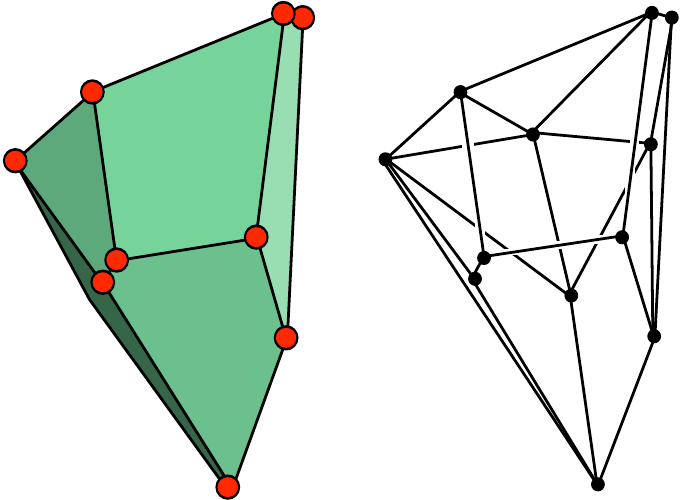}
\caption{A $3$-polytope $P$ and its graph $G(P)$.}
\label{figure:polytopegraph}
\end{center}
\end{figure}
\begin{remark}
For an unbounded polyhedron $P$, the graph $G(P)$ of $P$ is defined in the same way, but $G(P)$ contains only the \emph{bounded} edges. The graph $G(P)$ does not include the rays (the unbounded one-dimensional faces).
\end{remark}
\begin{example}
Let us revisit some of the polytopes and polyhedra of Example~\ref{example:basic}:
\begin{itemize}
\item
The graph of the $d$-simplex $\Delta_d$ is the complete graph $K_{d+1}$ so its diameter is one.

\item
The graph of the $0$-$1$ $d$-cube $\Box_d$ has diameter $d$: the steps needed to go from a vertex to another equals the number of coordinates in which the two vertices differ. 

\item
The graph of the $d$-dimensional cross polytope $\Diamond_d$ is almost complete: the only edges missing from $G(\Diamond_d)$ are those between opposite vertices. Thus, the diameter of a cross polytope is two.

\item
Using the software {\tt polymake} (see~\cite{polymake}), one can compute the graph $G(P_{3 \times 3})$ of the polytope $P_{3 \times 3}$. The graph is shown in Figure~\ref{figure:graphP33}. By inspection, the diameter of the graph is three.
\begin{figure}[hbt]
\begin{center}
\includegraphics[scale=1.0]{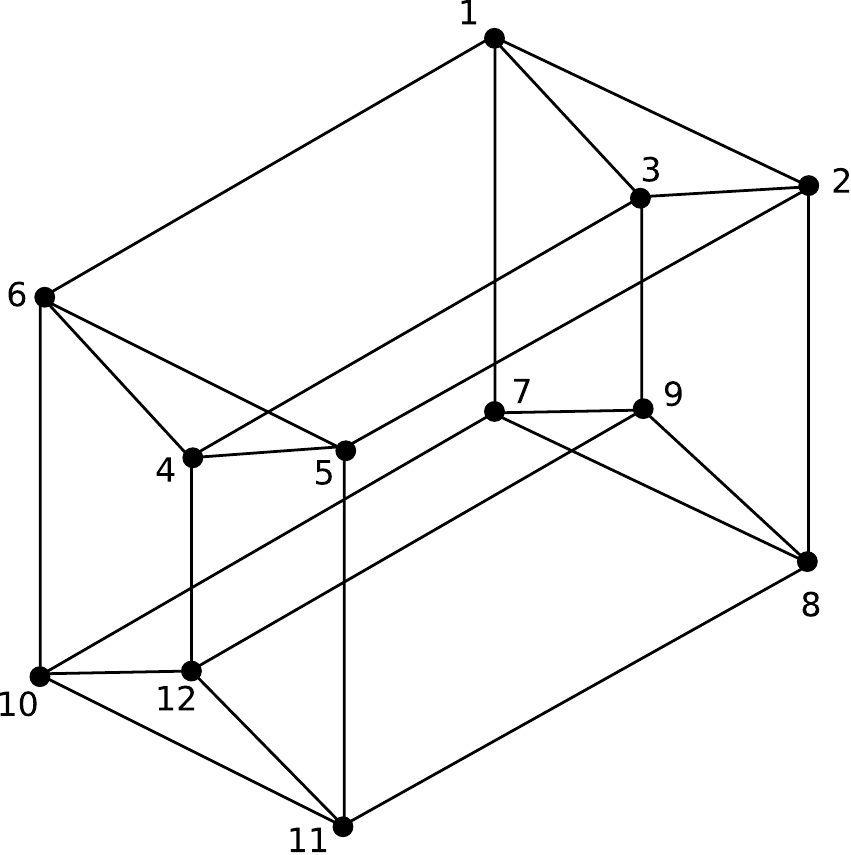}
\caption{The graph $G(P_{3 \times 3})$ of the $4$-polytope $P_{3 \times 3}$.}
\label{figure:graphP33}
\end{center}
\end{figure}

\end{itemize}
\end{example}

A graph is \defn{$d$-connected} if the graph that remains after removing any $d-1$ vertices is still connected. In~\cite{Balinski:On-the-graph-structure}, \given{Michel }Balinski proved that the graph of every $d$-polytope is $d$-connected:
\begin{theorem}[Balinski's Theorem,~\cite{Balinski:On-the-graph-structure}]\label{theorem:balinski}
Let $P$ be a $d$-polytope. Then the graph $G(P)$ is at least $d$-connected.
\end{theorem}

A classic theorem of \given{Ernst }Steinitz characterizes the graphs of $3$-dimensional polytopes. Chapter~4 in~\cite{Ziegler:Lectures} contains a proof of it. No analogous theorem is known for $4$-polytopes (see, e.g.,~\cite{Ziegler-facenumbers}).
\begin{theorem}[Steinitz' Theorem,~\cite{Steinitz:Polyeder-und-Raumeinteilungen,Steinitz:Vorlesungen-uber}]\label{theorem:steinitz}
Let $G$ be a simple graph. Then $G$ is the graph of a $3$-polytope if and only if $G$ is planar and $3$-connected.
\end{theorem}

\section{Geometric combinatorics of polytopes and polyhedra}

We turn now to the combinatorics of the faces of a polytope. The faces of a polytope $P$ form a \defn{poset} (or \defn{partially-ordered set}). For more on posets and lattices, see Chapter 3 in~\cite{Stanley:EnumerativeCombinatorics1}. We describe just what is needed here: for a more complete discussion on the geometric combinatorics of polytopes and their faces, see Section 2.2 in~\cite{Ziegler:Lectures}. The collection of all faces of a polytope $P$ (including the empty face $\emptyset$ and the face $P$ itself, which are the elements $\hat{0}$ and $\hat{1}$, respectively) is a poset where order relation is given by inclusion. This poset is a lattice, called the \defn{face lattice} of $P$, and denoted by $L(P)$. The lattice $L(P)$ is graded by the face dimension. A cover relation exists for two elements of $L(P)$ whenever an $(i+1)$-dimensional face contains an $i$-dimensional face. We often represent the face lattice using a \defn{Hasse diagram}, which has a node for each element of the poset and a vertical edge for each cover relation.
\begin{example}
Figures~\ref{figure:hassediagram2simplex} and~\ref{figure:hassediagram3simplex} (respectively) show examples of Hasse diagrams for the $d$-simplex with $d=2,3$ (respectively).
Figures~\ref{figure:hassediagram2cube} and~\ref{figure:hassediagram3cube} (respectively) show examples of Hasse diagrams for the $d$-cube with $d=2,3$ (respectively).
Figures~\ref{figure:hassediagram2crosspolytope} and~\ref{figure:hassediagram3crosspolytope} (respectively) show examples of Hasse diagrams for the $d$-dimensional cross polytope with $d=2,3$ (respectively).

Figure~\ref{figure:hassediagramP33} shows the Hasse diagram of the four-dimensional polytope $P_{3 \times 3}$.

The largest element $\hat{1}$ in each poset is the polytope $P$ itself, and the smallest element $\hat{0}$ is the empty face $\emptyset$. (These Hasse diagrams are traced copies of the output from the \verb+VISUAL_FACE_LATTICE+ command in {\tt polymake}. See~\cite{polymake}.)
\begin{figure}[hbt]
\begin{center}
\includegraphics{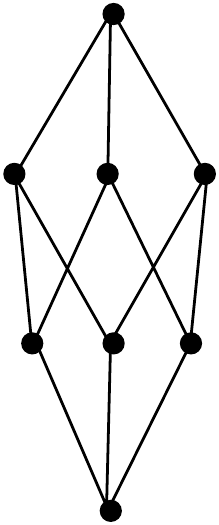}
\caption{Hasse diagram of the $2$-simplex.}
\label{figure:hassediagram2simplex}
\end{center}
\end{figure}
\begin{figure}[hbt]
\begin{center}
\includegraphics{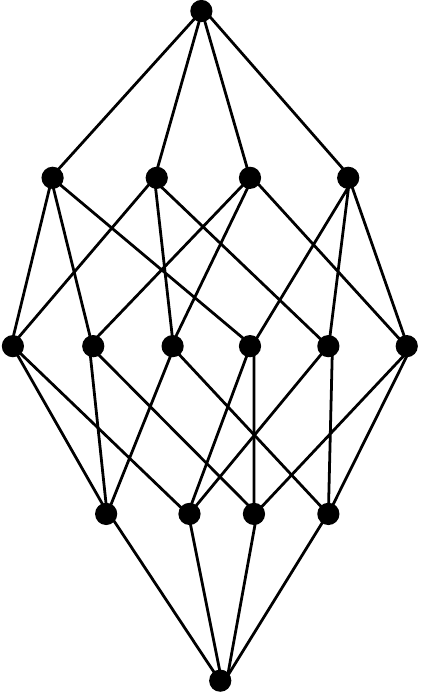}
\caption{Hasse diagram of the $3$-simplex.}
\label{figure:hassediagram3simplex}
\end{center}
\end{figure}
\begin{figure}[hbt]
\begin{center}
\includegraphics{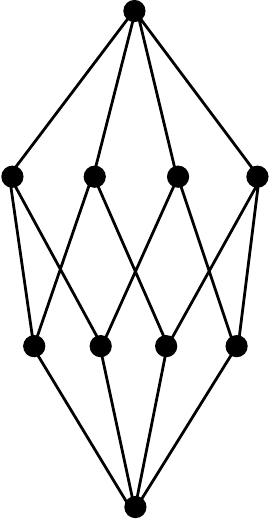}
\caption{Hasse diagram of the $2$-cube.}
\label{figure:hassediagram2cube}
\end{center}
\end{figure}
\begin{figure}[hbt]
\begin{center}
\includegraphics{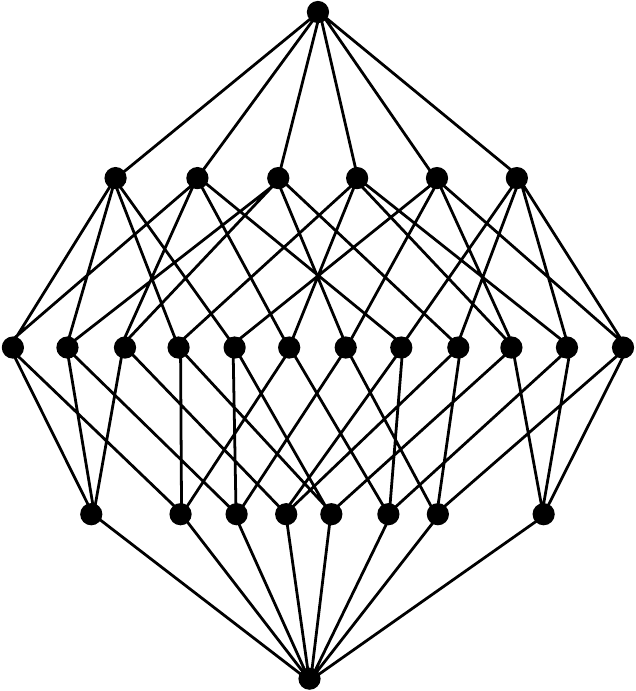}
\caption{Hasse diagram of the $3$-cube.}
\label{figure:hassediagram3cube}
\end{center}
\end{figure}
\begin{figure}[hbt]
\begin{center}
\includegraphics{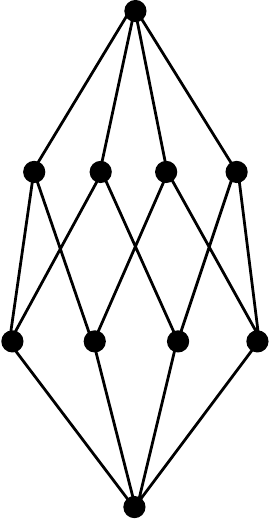}
\caption{Hasse diagram of the $2$-dimensional cross polytope.}\label{figure:hassediagram2crosspolytope}
\end{center}
\end{figure}
\begin{figure}[hbt]
\begin{center}
\includegraphics{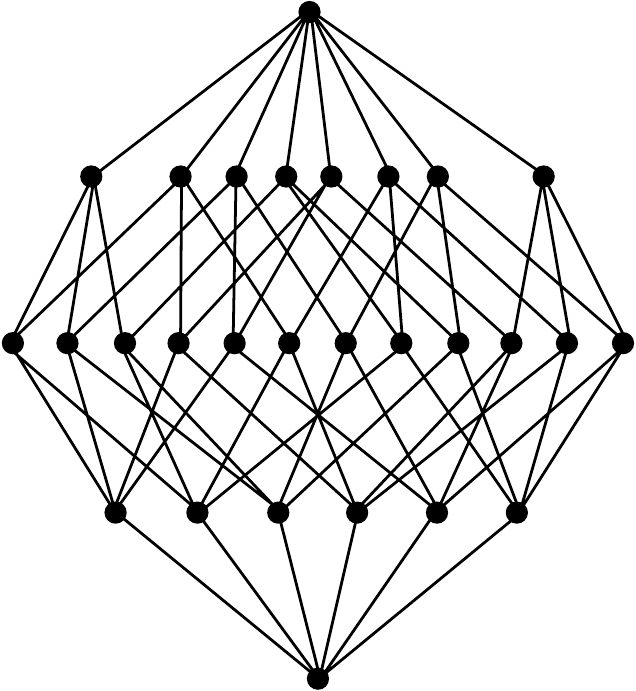}
\caption{Hasse diagram of the $3$-dimensional cross polytope.}
\label{figure:hassediagram3crosspolytope}
\end{center}
\end{figure}
\begin{figure}[hbt]
\begin{center}
\includegraphics{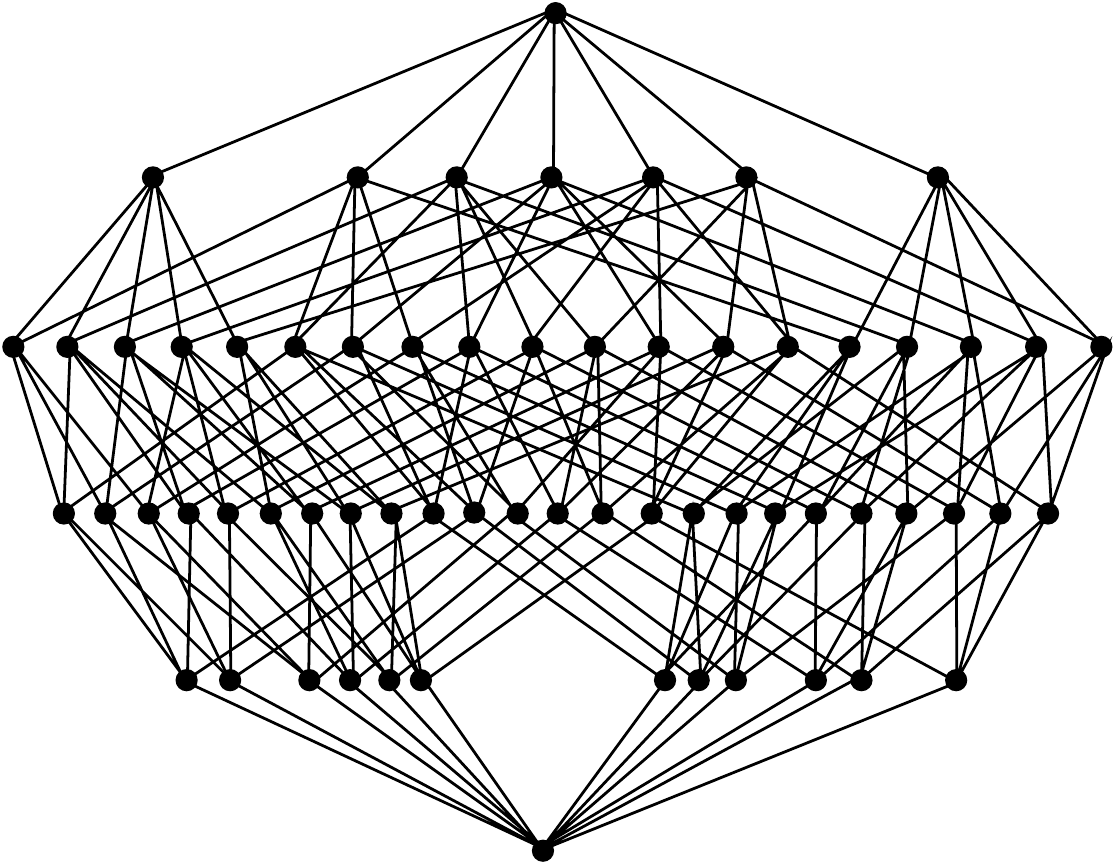}
\caption{Hasse diagram of the polytope $P_{3 \times 3}$.}
\label{figure:hassediagramP33}
\end{center}
\end{figure}
\end{example}
We say that two polytopes $P$ and $P'$ are \defn{combinatorially equivalent} if they have isomorphic face lattices, i.e., $L(P) \cong L(P')$.

\begin{remark}
Terminology for some basic polytopes is often used loosely. For example, any polytope that is combinatorially equivalent to the standard simplex $\Delta_d$ is called a simplex. (In other words, the convex hull of any $d+1$ affinely independent points in $\R^c$ is called a $d$-simplex.) Similarly, a cube is any polytope that is combinatorially equivalent to the standard $0$-$1$ cube. Any polytope that is combinatorially equivalent to a cone is called a cone, or sometimes an affine cone.
\end{remark}

Given a poset $L$, we define the \defn{opposite poset} $L^\Delta$ which has the same elements but the order relation is reversed. The Hasse diagram of the poset $L^\Delta$ is a vertical reflection of the Hasse diagram of the poset $L$.
\begin{example}
Hasse diagrams of simplices are invariant under vertical reflection, so the face lattice of any $d$-simplex is its own opposite. The face lattices of the $d$-cube and the $d$-dimensional cross polytope are opposites of each other. 
\end{example}

Many of the results that we will survey in Chapter~\ref{chapter:hirsch} are more natural in the polar setting. We now describe the polar of a polytope. This is the notion often called the \defn{dual}, but we adopt the use of the term polar to distinguish polarity from duality in the sense of linear programming. (For more on polarity, see Chapter 7 in~\cite{Wets:Optimization} or Section 2.3 in~\cite{Ziegler:Lectures}.)

We introduce the \defn{polar graph} (or \defn{dual graph}) of a polytope $P$. The \defn{polar graph} of a polytope $P$, denoted by $G^\Delta(P)$, is the following undirected, finite, simple graph:
\begin{itemize}
\item {\bf Vertices of $G^\Delta(P)$.} The graph $G^\Delta(P)$ has a vertex for every facet $f$ of the polytope $P$. Let $G^\Delta(f)$ denote the vertex of the graph $G^\Delta(P)$ corresponding to the facet $f$.
\item {\bf Edges of $G^\Delta(P)$.} Two vertices $G^\Delta(f)$ and $G^\Delta(f')$ of $G^\Delta(P)$ are connected by an edge in $G^\Delta(P)$ exactly when their corresponding facets $f$ and $f'$ intersect in a ridge of $P$. (That is to say, the face $f \cap f'$ is a ridge of $P$.) When this occurs, we say the facets $f$ and $f'$ are \defn{neighbors}.
\end{itemize}

The definition of the polar graph is motivated by the polar of a set. The \defn{polar} of a set $P \subseteq \R^c$, denoted by $P^\Delta$ is the set
\begin{equation}\label{equation:polardefiniton}
P^\Delta = \{y \in \R^c \mid \langle x, y \rangle \leq 1 \text{ for all } x \in P\}.
\end{equation}
In the definition of the polar of a set $P \subseteq \R^c$, it will be convenient to assume that $P$ is full-dimensional, and that the origin is in the interior of $P$, which can always be assumed by a suitable translation.  The polar of an arbitrary set $P \subseteq \R^c$ is a convex set. When $P$ is a polytope, then the polar $P^\Delta$ of $P$ has some nice properties:
\begin{lemma}\label{lemma:polar}
Let $P \subseteq \R^c$ be a polytope such that the origin is an interior point of $P$. Then,
\begin{enumerate}
\item\label{item:Pdualpolytope} The set $P^\Delta$ is a polytope.
\item\label{item:polarinvolution} Polarization is an involution: $P^{\Delta\Delta} = P$.
\item\label{item:duallattice} The lattices $L(P^\Delta)$ and $L^\Delta(P)$ are isomorphic.
\item\label{item:dualgraph} The graphs $G(P^\Delta)$ and $G^\Delta(P)$ are isomorphic.
\end{enumerate}
\end{lemma}
Lemma~\ref{lemma:polar} is proved in Section 2.3 of~\cite{Ziegler:Lectures}. If $P$ is a polytope, then $P^\Delta$ is called the \defn{polar polytope}. Any polytope $Q$ that is combinatorially equivalent to $P^\Delta$ is called a \defn{combinatorial polar} of $P$.
\begin{example}
Any $d$-dimensional cube and any $d$-dimensional cross polytope are combinatorial polars of each other. Every polygon is polar to itself. Simplices in any dimension are also self-polar.
\end{example}

Part~(\ref{item:duallattice}) of Lemma~\ref{lemma:polar} says that the face lattices of a polytope $P$ and its polar $P^\Delta$ are opposites. The facets (respectively vertices) of $P^\Delta$ correspond to the vertices (respectively facets) of $P$. More generally, every $(d-i)$-face of $P^\Delta$ corresponds to a face of $P$ of dimension $i-1$, and the incidence relations are reversed. Part~(\ref{item:dualgraph}) of Lemma~\ref{lemma:polar} says that the graph $G(P^\Delta)$ of the polar $P^\Delta$ of a polytope $P$ is isomorphic to the polar graph $G^\Delta(P)$ of $P$. In particular, this means:
\begin{remark}\label{remark:polaritysimplesimplicial}
Via polarity, studying the graphs of polytopes is equivalent to studying the polar graphs of polytopes.
\end{remark}

Of special importance are the simple and simplicial polytopes. A $d$-polytope or $d$-polyhedron is called \defn{simple} if every vertex is the intersection of exactly $d$ facets. Equivalently, a $d$-polyhedron is simple if every vertex in the graph $G(P)$ has degree exactly $d$. (Thus, the graph $G(P)$ of a simple $d$-polyhedron is $d$-regular.)
\begin{example}
Simplices and cubes are simple. Cross polytopes are not simple starting in dimension three. The four-dimensional polytope $P_{3 \times 3}$ defined in Example~\ref{example:basic} is simple since every vertex in its graph (see Figure~\ref{figure:graphP33} on page~\pageref{figure:graphP33}) has exactly four neighbors. 
\end{example}

A $d$-dimensional polytope or polyhedron $P$ is \defn{simplicial} if every facet  of $P$ is a $(d-1)$-simplex. Cross polytopes are simplicial polytopes. It follows from Part~(\ref{item:dualgraph}) of Lemma~\ref{lemma:polar} that the polar of a simple polytope is a simplicial polytope. In fact, the notions of simple and simplicial are polar to each other in the following way: the polytope $P$ is simplicial if and only if the polytope $P^\Delta$ is simple. For example, the $d$-dimensional cross polytope is the polar of the $d$-cube. Since cubes are simple polytopes, cross polytopes are simplicial. The polar of a simplex is a simplex. Among polytopes of dimension three and higher, the simplices are the only polytopes which are at the same time simple and simplicial (see Exercise 0.1 in~\cite{Ziegler:Lectures}).

Since the facets of simplicial polytopes are simplices, this in turn implies that all proper faces of a simplicial polytope are simplices. This is nice because then we can forget the geometry of a simplicial $d$-polytope and look only at the combinatorics of the simplicial complex formed by its faces. This simplicial complex is a topological $(d-1)$-sphere. (For more about simplicial complexes see, e.g.,~\cite{Hatcher:AlgebraicTopology} or~\cite{Miller:CombinatorialCommutativeAlgebra}.)

In Chapter~\ref{chapter:hirsch}, we will want to consider only the graphs of simple polytopes. By Remark~\ref{remark:polaritysimplesimplicial}, for simplicial polytopes, we are interested in the polar graph $G^\Delta(P)$ of a polytope $P$, defined in Section~\ref{section:graphsIntro}.

\section{Polytopes and optimization}\label{section:polytopesandLP}

Linear programming problems are the first class of problems discussed in mathematical optimization. In this section, we describe the connections of linear and integer programming to polytopes and polyhedra. We also give an overview of how to solve a linear program.

In linear programming, one is given a system of linear equalities and inequalities, and the goal is to maximize or minimize a given linear functional. We first describe linear programs in \defn{standard form}. Fix a real-valued $r \times c$ matrix $A$, a vector $b \in \R^r$, and a linear functional $\xi : \R^c \rightarrow \R$. In its standard form, a linear program is given by $A$, $b$, $\xi$ is the following problem:
\begin{equation}
\text{Maximize } \xi(x), \text{ subject to } Ax = b \text{ and } x \geq 0.
\end{equation}
The linear functional $\xi$ is called the \defn{objective function} or the \defn{cost function}. (In linear programming, it is no different to minimize or maximize: indeed, minimizing $\xi$ is the same as maximizing $-\xi$.) The equations $Ax=b$ and inequalities $x \geq 0$ are the \defn{constraints}. The coordinates of $x=(x_1,\ldots,x_c)$ are called the \defn{decision variables}. Suppose the $r \times c$ matrix $A$ has rank $r$, with $r \leq c$, and let $d = c-r$. Then, the equality $Ax = b$ defines a $d$-dimensional affine subspace whose intersection with the linear inequalities $x\geq0$ gives the \defn{feasibility polyhedron}
\[
P = \{x \in \R^c \mid Ax = b \text{ and } x \geq 0\}.
\]
Note that the resulting polyhedron is a partition polyhedron. If the feasibility polyhedron $P$ is bounded, then it is called the \defn{feasibility polytope}. If the polyhedron $P$ is non-empty, then the linear program is called \defn{feasible}. A vector $x$ belonging to the feasibility polyhedron $P$ is called a \defn{feasible solution}. A feasible solution $x$ that maximizes the linear functional $\xi$ is called an \defn{optimal solution}. Optimal solutions are typically denoted by $x^*$. Typically, it is not enough to simply say that a linear program is feasible, and simply conclude that there \emph{is} an optimal solution \emph{somewhere}. Instead, we must actually find it! We typically want to know the actual \emph{coordinates} of an optimal solution $x^*$, and not just the maximal value $\xi(x^*) \in \R$ alone.
\begin{remark}\label{remark:unboundedLP}
If the feasibility polyhedron $P$ is feasible and unbounded, then, depending on the objective function $\xi$, we may run into the ``danger'' that there are no optimal solutions! (What could go ``wrong''? The value of $\xi(x)$ may be arbitrarily large for feasible vectors $x$ in $P$.) See Figure~\ref{figure:unboundedLP} for an example. In this case, one desires the coordinates of a feasible solution $x$ and a direction vector $y \in \R^c$ such that
\begin{enumerate}
\item vectors of the form $x + \lambda y$ belong to the polyhedron $P$ for all real $\lambda \geq 0$, and
\item the value of $\xi(x + \lambda y)$ goes to infinity as $\lambda \rightarrow \infty$.
\end{enumerate}
\begin{figure}[hbt]
\begin{center}
\includegraphics[scale=1]{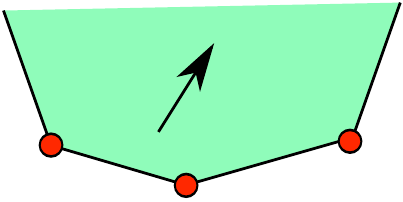}
\caption{An unbounded linear program: the arrow shows the direction in which $\xi$ increases.}
\label{figure:unboundedLP}
\end{center}
\end{figure}
When this occurs, we say the linear program is \defn{unbounded} with respect to $\xi$. When there is an optimal solution, we say the linear program is \defn{bounded} with respect to $\xi$, even if the polyhedron $P$ is unbounded. See Figure~\ref{figure:boundedLPunboundedP}.
\begin{figure}[hbt]
\begin{center}
\includegraphics[scale=1]{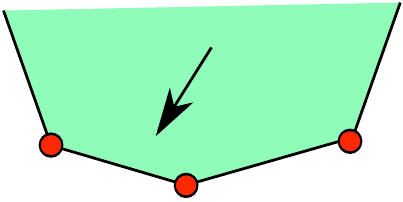}
\caption{A bounded linear program with respect to $\xi$. The feasibility polyhedron $P$ is unbounded.}
\label{figure:boundedLPunboundedP}
\end{center}
\end{figure}
\end{remark}

The feasibility polyhedron $P$ is convex and the level sets of the objective function $\xi$ are hyperplanes. It follows that an optimal solution of a linear program, if one exists, is found among the vertices of $P$. In fact, since the level sets of the linear functional $\xi$ are hyperplanes, the set of optimal solutions is a face of $P$, which is clear by the definition of a face. In particular, an optimal solution is found among the vertices of the feasibility polyhedron $P$ of a bounded linear program. If the objective function $\xi$ is sufficiently generic, and if the linear program is bounded with respect to $\xi$, then the linear program has a unique optimal solution $x^*$, and the solution $x^*$ is a vertex of $P$.

Since optimal solutions of a linear program are found among the vertices of its feasibility polyhedron, and since the number of vertices in a polyhedron is finite, this leads to a natural first algorithm to solve a linear program. First, compute all of the vertices of the feasibility polyhedron. Then output a vertex whose $\xi$-value is largest. Unfortunately, this algorithm is not very practical. As the dimension of the feasibility polyhedron grows, there are simply way too many vertices to compute. In fact, there is even a more fundamental flaw. How do you even find \emph{one} vertex of the feasibility polyhedron? Even in two dimensions, it is not obvious how to take a given set of linear inequalities and even find a solution. Consider the following example, which we will use as a running example.
\begin{example}[A sample application of linear programming]\label{example:sampleapplicationofLP}
Maxwell is opening a new restaurant. He needs to decide how much should be spent at the grocery store each month and how many hours per month to schedule employees to maximize the restaurant's profit. Let $x_1$ denote the amount to spend at the grocery store each month and let $x_2$ denote the number of labor hours per month. Suppose that the restaurant's profit is determined by the function
\begin{equation*}
\xi(x_1,x_2) = 19 x_1 - 7x_2.
\end{equation*}
Maxwell wants to know which pair $x=(x_1,x_2)$ maximizes the value $\xi(x)$ of the objective linear functional $\xi$, but there are some restrictions. Clearly, $x_1 \geq 0$ and $x_2 \geq 0$. There are more constraints Maxwell must obey. Suppose, for example, that labor laws, union rules, and other factors further restrict the choice of $x=(x_1,x_2)$ to:
\begin{align*}
5x_1 - 6x_2 &\leq 70\\
3x_1 - 2x_2 &\geq -30\\
16x_1 - 7x_2 &\leq 424\\
19x_1 + 4x_2 &\geq 180\\
11x_1 + 2x_2 &\leq 340\\
9x_1 - 16x_2 &\geq -325\\
6x_1 + 17x_2 &\leq 539\\
2x_1 - 21x_2 &\leq -176\\
11x_1 + 17x_2 &\leq 391\\
5x_1 + 4x_2 &\leq 210.
\end{align*}
In dimension two, it is easy enough to carefully graph the half-spaces, then compute the value of $\xi$ on each of the vertices. But in general, one cannot even ``graph'' the feasibility polyhedron $P$. How do you even find one feasible point $x \in P$?  The first natural idea is to travel along the $x_i$-axis (for some $i$) until you hit the boundary of $P$.

Even in this small example, that idea would fail. If it were not for the two inequalities $x_1, x_2 \geq 0$, how would we even know which direction to travel on the axes? (In fact, for this example, these two inequalities are redundant to the description of $P$.) A search along the axes would definitely fail in our example, since $P$ lies \emph{completely} in the relative interior of an orthant: the feasibility polyhedron does not even intersect the set $\{x=(x_1,x_2) \in \R^2 \mid x_1x_2 = 0\}$. Even assuming that $P$ is in the relative interior of an orthant, in a $c$-dimensional ambient space, there are $2^c$ orthants! In addition, the polyhedron $P$, if it is non-empty, may be ``far away from the origin,'' and if you try to do a search along a path that is piece-wise linear, how can you know how far to travel along a direction before turning in a new direction? (In fact, how do you know whether the current direction of travel in your path moves you towards the feasibility polyhedron $P$, or away from it? Even worse, what if you could \emph{never} find $P$ because the feasibility polyhedron in empty? How would you even be able to detect this case?)
\end{example}

To solve a linear program, we first describe how to convert \emph{any} linear program into one whose constraints are of the form $Ax = b, x \geq 0$. (See~\cite{Chvatal:Linear-Programing} or~\cite{Matousek:LinearProgramming}.)
\begin{enumerate}

\item {\bf Non-negativity.} If any decision variable $x_i$ does not have the constraint $x_i \geq 0$, then we do a variable substitution. We will replace $x_i$ by two new non-negative decision variables $x_i' \geq 0$ and $x_i'' \geq 0$. Replace every occurrence of $x_i$ by $x_i' - x_i''$. (The new decision variables $x_i'$ and $x_i''$ are called \defn{auxiliary variables}. The modified linear program no longer mentions the old decision variable $x_i$.)

\item {\bf Linear equations.} Any linear equality constraints will simply be part of the matrix equation $Ax = b$, so these should not be modified (except for any variable substitutions from the previous step).

\item {\bf Linear inequalities.} Turn each linear inequality into a linear equation by adding an auxiliary variable called a \defn{slack variable}. The linear inequality $a_1x_1 + \cdots + a_cx_c \leq b_0$ becomes $a_1x_1 + \cdots + a_cx_c + z = b_0$, with $z \geq 0$. The linear inequality $a_1x_1 + \cdots + a_cx_c \geq b_0$ becomes $a_1x_1 + \cdots + a_cx_c - z = b_0$, with $z \geq 0$.
\end{enumerate}
Putting this all together proves the following fact, which says that \emph{any} polyhedron can be written as a partition polyhedron.
\begin{lemma}\label{lemma:standardform}
Let $P \subset \R^c$ be any polyhedron. Then there is a polyhedron 
\begin{equation*}
\widetilde{P} = \{x \in \R^{\widetilde{c}} \mid Ax = b, x \geq 0 \}
\end{equation*}
 and a map $\pi : \R^{\widetilde{c}} \rightarrow \R^c$ of the form
\begin{equation*}
x=(x_1,\ldots,x_{\widetilde{p}},x'_1,\ldots,x'_q,x''_1,\ldots,x''_q) \,{\buildrel\pi\over\mapsto}\, (x_1,\ldots,x_p,x'_1-x''_1,\ldots,x'_q-x''_q)
\end{equation*}
with $\widetilde{p}\geq p$, such that the restriction $\pi|_{\widetilde{P}}$ is a bijection from $\widetilde{P}$ to $P$. The polyhedra $P$ and $\widetilde{P}$ are combinatorially equivalent.
\end{lemma}

\begin{example}
Let us convert the linear program of Maxwell's restaurant from Example~\ref{example:sampleapplicationofLP} to standard form. Both decision variables are already non-negative, so there is nothing to do in the first step above. (In the notation of Lemma~\ref{lemma:standardform}, $q=0$.)

We convert the ten non-trivial inequality constraints. The result is the new system $Ax = b$ and $x \geq 0$, where
\begin{equation*}
A =
\left[
\begin{array}{cccccccccccc}
5& - 6 &1&0&0&0&0&0&0&0&0&0\\
3& -2 &0&-1&0&0&0&0&0&0&0&0\\
16& -7 &0&0&1&0&0&0&0&0&0&0\\
19& 4 &0&0&0&-1&0&0&0&0&0&0\\
11& 2 &0&0&0&0&1&0&0&0&0&0\\
9& -16 &0&0&0&0&0&-1&0&0&0&0\\
6& 17 &0&0&0&0&0&0&1&0&0&0\\
2& -21 &0&0&0&0&0&0&0&1&0&0\\
11& 17 &0&0&0&0&0&0&0&0&1&0\\
5& 4 &0&0&0&0&0&0&0&0&0&1
\end{array}
\right]
\quad\text{ and }\quad
b = 
\left[
\begin{array}{c}
70\\
-30\\
424\\
180\\
340\\
-325\\
539\\
-176\\
391\\
210
\end{array}
\right]
.
\end{equation*}
The linear program has been converted. Now, we want to maximize the objective function $\xi(x) = \xi(x_1,\ldots,x_{12}) = 19x_1 - 7x_2$ over the feasibility polyhedron 
\begin{equation*}
\widetilde{P} = \{x \in \R^{12} \mid Ax = b, x \geq 0\}.
\end{equation*}
The feasibility polyhedron $\widetilde{P}$ is a partition polyhedron.
\end{example}

Let $A$ be a real-valued matrix of size $r \times c$ and let $b$ be a vector in $\R^r$. We now describe how to find a feasible solution $x$ to a linear program in the standard form 
\begin{equation}\label{equation:originalsystem}
Ax = b, x \geq 0,
\end{equation}
if one exists. (See~\cite{Matousek:LinearProgramming} for more details.) This method will show that feasibility of one linear program is reduced to optimality of another linear program. (See~\cite{Boyd:ConvexOptimization} and~\cite{Renegar:A-Mathematical-View} for a detailed explanation.)

To begin, we assume, without loss of generality, that the vector $b$ is in $\R_{\geq 0}^r$. Indeed, if any coordinate of the vector $b \in \R^r$ is negative, we multiply it (and the corresponding row of the matrix $A$) by $-1$. From this set of $r$ linear equality constraints, we construct a new linear program. Add $r$ non-negative auxiliary decision variables, say, $z_1,\ldots,z_r \geq 0$. For each $i=1,\ldots,r$, modify the $i$th constraint from 
\[
a_{i,1}x_{i,1} + \cdots + a_{i,c}x_{i,c} = b_i \geq 0
\]
to 
\[
a_{i,1}x_{i,1} + \cdots + a_{i,c}x_{i,c} + z_i = b_i.
\]
In terms of matrices, the new linear system has the constraint matrix $A' = [A \mid I]$, where $I$ is the $r \times r$ identity matrix. There is a very easy initial feasible point for this modified system, namely $(x_1,\ldots,x_c,z_1,\ldots,z_r) = (0,\ldots,0,b_1,\ldots,b_r) \in \R_{\geq 0}^{c+r}$. For this system, we use any algorithm for linear programming to minimize the objective function $\xi(x,z) = z_1 + \cdots + z_r$ subject to the above constraints. If the solution gives a point $(x,z) \in \R_{\geq 0}^{c+r}$ with $\xi(x,z) = 0$, then the coordinate-erasing projection of $(x,z)$ to $x \in \R^c$ is an initial feasible solution of the original linear program with the constraints \eqref{equation:originalsystem}. If the optimal solution $(x,z)$ has strictly positive $\xi$ value $\xi(x,z) > 0$, then the original problem \eqref{equation:originalsystem} has no feasible solution.

In linear programming, one assumes that the decision variables are always real-valued quantities. This restriction is too strong in general. After all, one cannot hire half of an employee! In many settings, it is more natural to consider the situation where the decision variables must be integers. An \defn{integer program} is a linear program with the additional constraint that the vector $x$ of the decision variables has all integer coordinates. An integer program is considered solved if, given a linear functional $\xi : \R^c \rightarrow \R$ and a polyhedron $P$, one knows which \defn{integral lattice point} $x$ in $P \cap \Z^c$ has the largest value of $\xi$. We do not discuss integer programming in any detail in this dissertation. (For more on integer programming see, e.g.,~\cite{Schrijver:LinearProgramming}.)

\section{Transportation polytopes}\label{section:TPintro}

Transportation polytopes are well-known objects in operations research, mathematical programming, and statistics. In statistics, transportation polytopes are known as \defn{contingency tables}. Surveys on the research in transportation polytopes and the transportation problem are found in the book by Yemelichev, Kovalev, and Kratsov (see~\cite{Yemelichev:Polytopes}), in Vlach's survey (see~\cite{Vlach:SolutionsPlanarTransportation}), in Klee and Witzgall's article (see~\cite{Klee:FacesTransportation}), and in the recent survey of De Loera and Onn (see~\cite{De-Loera:TransportationPolytopesSurvey}).

\subsection{Classical transportation polytopes}

We begin by introducing the most well-known subfamily of transportation polytopes. Fix two integers $p, q \in \Z_{> 0}$. The \defn{classical transportation polytope} $P$ of size $p \times q$ defined by the vectors $u \in \R^p$ and $v \in \R^q$ is the polytope defined in the $pq$ variables $x_{i,j} \in \R_{\geq 0}$ ($i \in [p], j \in [q]$) satisfying the $p+q$ equations
\begin{equation}\label{equation:classicalsums}
\sum_{j=1}^q x_{i,j} = u_i\ (i \in [p])
\quad\text{and}\quad
\sum_{i=1}^p x_{i,j} = v_j\ (j \in [q])
.
\end{equation}
Since $P$ is defined by the $p+q$ linear equations in \eqref{equation:classicalsums} and the $pq$ linear inequalities $x_{i,j} \geq 0$, it is a polyhedron. Since the coordinates $x_{i,j}$ of $P$ are non-negative, the summation conditions \eqref{equation:classicalsums} imply that $P$ is bounded, so classical transportation polytopes are polytopes. (In fact, $0 \leq x_{i,j} \leq \min\{u_i,v_j\}$ for all $i \in [p],j \in [q]$.) After re-indexing the variables $x_{i,j}$ ($i \in [p], j \in [q]$) as $x_1,x_2,\ldots,x_{pq}$, the equations \eqref{equation:classicalsums} and the inequalities $x_{i,j} \geq 0$ can be rewritten in the form
\begin{equation*}
P = \{x \in \R^{pq} \mid Ax = b, x \geq 0\}
\end{equation*}
with an appropriate $0$-$1$ matrix $A$ of size $(p+q) \times pq$ and a vector $b \in \R^{p+q}$. Thus, every classical transportation polytope is presented in the form \eqref{equation:partitionpolyhedron} and is, thus, a partition polyhedron. The matrix $A$ does not have full row rank. Indeed, the sum of the rows corresponding to the $u$-sum equations is the same as the sum of the rows for the $v$-sum equations. Since this is the only linear dependence among the rows of the matrix $A$, the rank of $A$ is $p+q-1$. Therefore the dimension $d$ of the affine hull of $P$ is $pq-(p+q-1) = (p-1)(q-1)$. Therefore,
\begin{corollary}\label{corollary:dimension}
Every non-empty $p \times q$ classical transportation polytope has dimension $d=pq-p-q+1$ and ambient dimension $c=pq$.
\end{corollary}
\begin{example}\label{example:classicalreindex}
Let us reconsider the polytope $P_{3 \times 3}$ from Example~\ref{example:basic}. Every point $x \in P_{3 \times 3}$ satisfies the equation $x_7+x_8+x_9=1$, so let us add in this redundant equation. We give an equivalent definition to our earlier $P_{3 \times 3}$, defining it now as $P_{3 \times 3} = \{ x \in \R^9 \mid A_{3 \times 3}x = b, x \geq 0 \}$, where
\begin{equation}
A_{3 \times 3} = \left[
\begin{array}{ccccccccc}
1&0&0&1&0&0&1&0&0 \\
0&1&0&0&1&0&0&1&0 \\
0&0&1&0&0&1&0&0&1 \\
1&1&1&0&0&0&0&0&0 \\
0&0&0&1&1&1&0&0&0 \\
0&0&0&0&0&0&1&1&1
\end{array}
\right]
\quad
\text{\rm and }
\quad
b = \left[
\begin{array}{c}
2 \\ 7 \\ 2 \\ 5 \\ 5 \\ 1
\end{array}
\right ]
.
\end{equation}
Up to permutation of rows and columns, the matrix $A_{3 \times 3}$ is the unique constraint matrix for $3 \times 3$ classical transportation polytopes. It is a $6 \times 9$ matrix of rank five. Thus, $P_{3 \times 3}$ is a four-dimensional polytope described in a nine-dimensional ambient space. By identifying the variables~$x_1, \ldots, x_9$ (respectively) with the variables $x_{1,1}, x_{1,2}, x_{1,3}, x_{2,1}, \ldots, x_{3,3}$ (respectively), the polyhedron $P_{3 \times 3}$ is a $3 \times 3$ classical transportation polytope defined by the vectors $v = (2, 7, 2)^T$ and $u = (5,5,1)^T$.

\begin{figure}[htb]
\begin{equation*}
\begin{tabular}{|c|c|c|}
\hline
$x_{1,1}$ & $x_{1,2}$ & $x_{1,3}$ \\ \hline
$x_{2,1}$ & $x_{2,2}$ & $x_{2,3}$ \\ \hline
$x_{3,1}$ & $x_{3,2}$ & $x_{3,3}$ \\ \hline
\end{tabular}
=
\begin{tabular}{|c|c|c|}
\hline
\ $2$\ &\ $2$\ &\ $1$\ \\ \hline
\ $0$\ &\ $5$\ &\ $0$\ \\ \hline
\ $0$\ &\ $0$\ &\ $1$\ \\ \hline
\end{tabular}
\end{equation*}
\caption{The vertex $z_1$ after reindexing.}\label{figure:classicalreindex}
\end{figure}

The notation $x_{i,j}$ is suggestive. We think of a point $x = (x_{i,j})_{i \in [p], j \in [q]} \in P \subseteq \R^{p \times q}$ as a $p \times q$ table. For example, the vertex $z_1 =(2,2,1,0,5,0,0,0,1)$ defined in Example~\ref{example:basicvertices}, under the identification, is shown in Figure~\ref{figure:classicalreindex}. In terms of tables, the equations in~\eqref{equation:classicalsums} are conditions on the row sums and column sums of tables that correspond to feasible points in $P$.
\end{example}
The matrix $A$ is called the \defn{defining matrix} (or the \defn{constraint matrix}) of $p \times q$ classical transportation polytopes. The vectors $u$ and $v$ are called \defn{marginals}. For $P$ to be non-empty, the vectors $u$ and $v$ should be non-negative. (The case when a coordinate $u_i$ or $v_j$ is zero is uninteresting, so we usually assume that $u \in \R_{>0}^p$ and $v \in \R_{>0}^q$.) These polytopes are called transportation polytopes because of the following scenario: consider a model of transporting goods with $p$ supply locations (with the $i$th location supplying a quantity of $u_i$), and $q$ demand locations (with the $j$th location demanding a quantity of $v_j$). The feasible points $x = (x_{i,j})_{i \in [p], j \in [q]}$ in a $p \times q$ transportation polytope $P$ model the scenario where a quantity of $x_{i,j}$ of goods is transported from the $i$th supply location to the $j$th demand location.

\begin{definition}\label{definition:supp}
If $P$ is a non-empty $p \times q$ classical transportation polytope and $x = (x_{i,j})_{i \in [p], j \in [q]}$ is in $P$, then $x_{i,j} \geq 0$ for all $i \in [p]$ and $j \in [q]$. The pairs $(i,j) \in [p] \times [q]$ where $x_{i,j}$ is \emph{strictly} positive are called \defn{support entries}. For a point $x \in P$, we define the \defn{support set} $\supp(x)$ to be $\{(i,j) \in [p]\times [q] \mid x_{i,j} > 0\}$.
\end{definition}

A necessary and sufficient condition for a classical transportation polytope to be non-empty is the sum of the supply margins equal the sum of the demand margins:
\begin{lemma}\label{lemma:littleeddielemma}
Let $P$ be the $p \times q$ classical transportation polytope defined by the marginals $u \in \R_{\geq 0}^p$ and $v \in \R_{\geq 0}^q$. The polytope $P$ is non-empty if and only if 
\begin{equation}\label{equation:classicalnonempty}
\sum_{i \in [p]} u_i = \sum_{j \in [q]} v_j.
\end{equation}
\end{lemma}
The proof of this lemma uses the well-known \defn{northwest corner rule algorithm} (see survey~\cite{Queyranne:MultiIndexTransportation} or Exercise 17 in Chapter 6 of~\cite{Yemelichev:Polytopes}).
\begin{proof}
To show necessity, suppose $\sum_{i \in [p]} u_i \not= \sum_{j \in [q]} v_j.$ By substituting \eqref{equation:classicalsums}, the left and right sides of the equation \eqref{equation:classicalnonempty} are not equal. Thus, the linear system is inconsistent and there is no solution $x$ satisfying \eqref{equation:classicalsums}.

For the converse, we construct a point $x \in P$ using the northwest corner rule algorithm: let $x_{p,q} = \min \{u_p, v_q\}$. If the minimum is obtained at $u_p$, set $x_{p,j} = 0$ for all $j \not= q$ and replace $v_q$ with $v_q - u_p$. The rest of the point $x = (x_{i,j})$ is obtained recursively as a point in a $(p-1) \times q$ classical transportation polytope. Similarly, if the minimum is obtained at $v_q$, set $x_{i,q} = 0$ for all $i \not= p$ and replace $u_p$ with $u_p - v_q$. The rest of the point $x$ is obtained as a point in a $p \times (q-1)$ transportation polytope. If the minimum was obtained at $u_p= v_q$, then the rest of the point $x$ is obtained as a point in a $(p-1)\times (q-1)$ transportation polytope.
\end{proof}

\begin{definition}
A $p \times q$ classical transportation polytope $P$ is \defn{generic} if
\begin{equation}\label{equation:nondegenerateclassicalequation}
\sum_{i \in Y} u_i \not= \sum_{j \in Z} v_j.
\end{equation}
for every non-empty proper subset $Y \subsetneq [p]$ and non-empty proper subset $Z \subsetneq [q]$. (Of course, due to \eqref{equation:classicalnonempty}, we must disallow the case where $Y = [p]$ and $Z = [q]$.)
\end{definition}
\begin{remark}
Very soon we will introduce the notion of non-degenerate transportation polytopes. We will see in Lemma~\ref{lemma:genericnondegenerate} that the notions of genericity and non-degeneracy coincide for classical transportation polytopes. (But the condition defined above, which we need now, has no generalization to most multi-way transportation polytopes, which we introduce in the next section.)
\end{remark}

If $P$ is a generic $p \times q$ classical transportation polytope $P$, then in the northwest corner rule algorithm used in the proof of Lemma~\ref{lemma:littleeddielemma}, the minimum is never attained simultaneously at $u_p$ and $v_q$. This proves:
\begin{corollary}\label{corollary:nondegeneratesupportsize}
Let $P \not= \emptyset$ be a generic $p \times q$ classical transportation polytope. Then, there is a point $x \in P$ with $|\supp(x)| = p+q-1$.
\end{corollary}

Let $P$ be a $p \times q$ classical transportation polytope. To every point $x \in P$, we define a bipartite graph $B(x)$, called the \defn{support graph} of $x$. The graph $B(x)$ is the following subgraph of the complete bipartite graph $K_{p,q}$: 
\begin{itemize}
\item {\bf Vertices of $B(x)$.} The vertices of the graph $B(x)$ are the vertices of the complete bipartite graph $K_{p,q}$. That is, the graph $B(x)$ has $p$ vertices of the first kind and $q$ vertices of the second kind.
\item {\bf Edges of $B(x)$.} An edge $(i,j)$ connecting vertex $i$ of the first kind to vertex $j$ of the second kind exists if and only if $x_{i,j}$ is \emph{strictly} positive. That is to say, the set of edges is in one-to-one correspondence with the support set $\supp(x)$ defined in Definiton~\ref{definition:supp}. The value of $x_{i,j} > 0$ is called the \defn{weight} or the \defn{flow} of the edge $(i,j)$.
\end{itemize}
\begin{example}
Let us consider the point $z_1 \in P_{3 \times 3}$ from Example~\ref{example:classicalreindex} under the reindexing. Here, $\supp(z_1)= \{(1,1),(1,2),(1,3),(2,2),(3,3)\}$. Figure~\ref{figure:Bv1} depicts the graph $B(z_1)$.
\begin{figure}[hbt]
\begin{center}
\includegraphics[scale=0.7]{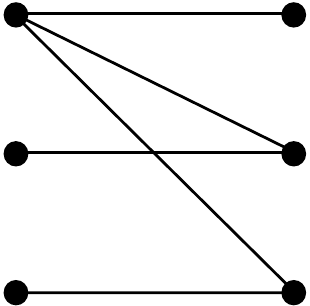}
\caption{The support graph $B(z_1)$ of the vertex $z_1 \in P_{3 \times 3}$. The nodes of $B(z_1)$ on the left correspond to the $p=3$ supplies. The nodes on the right correspond to the $q=3$ demands.}
\label{figure:Bv1}
\end{center}
\end{figure}

\end{example}
The graph properties of $B(x)$ provide a useful combinatorial characterization of the vertices of classical transportation polytopes:
\begin{lemma}\label{lemma:vertexBx}
Let $P$ be a $p \times q$ classical transportation polytope defined by the marginals $u \in \R^p_{>0}$ and $v \in \R^q_{>0}$, and let $x \in P$. Then the graph $B(x)$ is spanning. The feasible point $x$ is a vertex of $P$ if and only if $B(x)$ is a spanning forest. Moreover, if $P$ is generic, then $x$ is a vertex of $P$ if and only if $B(x)$ is a spanning tree.
\end{lemma}
\begin{proof}
Let $x \in P$. The marginals are strictly positive. Since $x$ is feasible, for each $i \in [p]$, there is a $j \in [q]$ such that $x_{i,j} > 0$. Similarly, for each $j \in [q]$, there is an $i \in [p]$ so that $x_{i,j} > 0$. Thus, each node in $B(x)$ is incident to an edge, so the graph $B(x)$ must be spanning.

Suppose $x$ is a vertex of $P$. We argue that the graph $B(x)$ cannot contain a cycle by contradiction. Suppose $B(x)$ has a cycle. Let $\lambda > 0$ be the minimum weight $x_{i,j}$ among all edges in the cycle. Since $B(x)$ is bipartite, the cycle has an even number of edges. Decompose the cycle as the disjoint union of two edge sets $E_+$ and $E_-$ so that every other edge along the cycle is in the set $E_+$ and every other edge is in $E_-$. Let $y = \frac{\lambda}{2}\sum_{(i,j) \in E_+} e_{i,j} - \frac{\lambda}{2}\sum_{(i,j) \in E_-} e_{i,j}$. (Here, $e_{i,j}$ is the basis unit vector in the direction of the variable $x_{i,j}$.) The vector $y$ is in the kernel of the defining matrix $A$ of the transportation polytope $P$. Therefore, both $x+y$ and $x-y$ belong to $P$. By Lemma~\ref{lemma:vertexequivalence}, the point $x \in P$ is not a vertex. Contradiction.

For the converse, suppose $x \in P$ and suppose that the graph $B(x)$ is a spanning forest. Any non-zero vector $y \in \R^{pq}$ that can be added to $x$ and stay in the polytope $P$ must be in the kernel of the defining matrix. The support of the vector $y$, thought of in terms of the bipartite graph $B(y)$ induces a cycle. Thus, it cannot be that both $x+y$ and $x-y$ belong to $P$. By Lemma~\ref{lemma:vertexequivalence}, $x$ is a vertex of $P$.

Now suppose that the transportation polytope $P$ is generic. Suppose, for a contradiction, that $B(x)$ is not a tree. Consider one of the connected components of $B(x)$, say, the subgraph induced by the nodes $Y \subseteq [p]$ and $Z \subseteq [q]$. Since $B(x)$ is not a tree, at least one of $Y$ or $Z$ is a proper subset. Then, $\sum_{i \in Y} u_i = \sum_{j \in Z} v_j$, which contradicts \eqref{equation:nondegenerateclassicalequation}. Therefore, the graph $B(x)$ is a tree if $x$ is a vertex of a generic classical transportation polytope.
\end{proof}
As an immediate corollary, we get:
\begin{corollary}\label{corollary:classicalTPsizeofsupport}
Let $P$ be a generic $p \times q$ classical transportation polytope. Let $x$ be a point in the transportation polytope $P$. Then $x$ is a vertex of $P$ if and only if $|\supp(x)| = p + q - 1$.
\end{corollary}

We now define a notion equivalent to genericity that will be needed in the next section:
\begin{definition}
A transportation polytope is \defn{non-degenerate} if is simple and it is of 
maximal possible dimension.
\end{definition}
The condition on maximality of dimension will be important in the next section where we introduce multi-way transportation polytopes. From our corollary, we can prove that the notions of genericity and non-degeneracy are equivalent for classical transportation polytopes:
\begin{lemma}\label{lemma:genericnondegenerate}
Let $P$ be a non-empty $p \times q$ classical transportation polytope. Then $P$ is generic if and only if $P$ is non-degenerate.
\end{lemma}
\begin{proof}
The dimension is always maximal by Corollary~\ref{corollary:dimension}. The statement follows from the equivalence in Corollary~\ref{corollary:classicalTPsizeofsupport}.
\end{proof}

The support graph gives the following characterization of edges of classical transportation polytopes. (See Lemma 4.1 in Chapter 6 of~\cite{Yemelichev:Polytopes}.)
\begin{proposition}
Let $x$ and $x'$ be distinct vertices of a classical transportation polytope $P$. Then the vertices $x$ and $x'$ are adjacent if and only if the graph $B(x) \cup B(x')$ contains a unique cycle.
\end{proposition}
This can be seen since the bases corresponding to the vertices $x$ and $x'$ differ in the addition and the removal of one element (see~\cite{Matousek:LinearProgramming} or~\cite{Schrijver:LinearProgramming}). For an elementary proof:
\begin{proof}
Let $E$ consist of all edges in the union of $B(x)$ and $B(x')$, and define $F=\overline{E}$ to be the complement of $E$ in $[p] \times [q]$. Clearly, the polytope $P \cap \{x \in \R^{p \times q} \mid x_{i,j} = 0 \text{ for all } (i,j) \in F \}$ is one-dimensional.
\end{proof}

We now introduce the \emph{Birkhoff polytope}, introduced by Birkhoff in~\cite{Birkhoff:TresObservaciones}. We will discuss new properties of Birkhoff polytopes in Section~\ref{section:regulartriangulationsBirkhoff}.
\begin{definition}\label{definition:BirkhoffPolytope}
The \defn{$p$th Birkhoff polytope}, denoted by $B_p$, is the $p \times p$ classical transportation polytope with margins $u=v=(1,1,\ldots,1)^T$.
\end{definition}
The Birkhoff polytope is also called the \defn{assignment polytope} or the \defn{polytope of doubly stochastic matrices}. It is the perfect matching polytope of the complete bipartite graph $K_{p,p}$. The following theorem states that the vertices of the Birkhoff polytope are the permutation matrices, and therefore that any doubly stochastic matrix may be represented as a convex combination of permutation matrices.
\begin{theorem}[Birkhoff-von Neumann Theorem]\label{theorem:birkhoffvonneumann}
The $p!$ vertices of the $p$th Birkhoff polytope $B_p$ are the permutation matrices of size $p \times p$.
\end{theorem}
This theorem was stated in the 1946 paper~\cite{Birkhoff:TresObservaciones} by Birkhoff and proved independently by von Neumann in 1953 (see~\cite{Neumann:A-certain-zero-sum}). Equivalent results were shown earlier in the 1894 thesis~\cite{Steinitz:Uber-die-Konstruction} of Steinitz, and the theorem also follows from the 1916 papers~\cite{Konig:Grafok-es-alkalmazasuk} and~\cite{Konig:Uber-Graphen} by K{\H{o}}nig. (For a more complete discussion on the history of the Birkhoff-von Neumann Theorem, see the preface to~\cite{Lovasz:Matching}.) The vertices of a Birkhoff polytope are examples of \defn{semi-magic squares}:
\begin{definition}
Let $p,\sigma \in \Z_{>0}$. A $p \times p$ \defn{semi-magic square} of order $\sigma$ is a $p \times p$ table $x=(x_{i,j})_{i\in[p],j\in[p]}$ of numbers in $\Z_{\geq 0}$ such that
\begin{equation*}\label{equation:semimagicsquare}
\sum_{j=1}^p x_{i,j} = \sigma\ (i \in [p])
\quad\text{and}\quad
\sum_{i=1}^p x_{i,j} = \sigma\ (j \in [p])
.
\end{equation*}
That is to say, a semi-magic square is an integral lattice point in a transportation polytope where every row and column sum is the same, namely $\sigma$. The number $\sigma$ is called the \defn{magic number}.
\begin{figure}[hbt]
\begin{center}
\includegraphics[scale=.5]{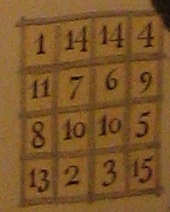}
\caption{A $4 \times 4$ magic square of $\sigma = 33$.}
\label{figure:magic-square-example}
\end{center}
\end{figure}

A $p \times p$ \defn{magic square} of order $\sigma$ is a semi-magic square that also satisfies the two additional equations 
\begin{equation*}\label{equation:magicsquare}
\sum_{j=1}^p x_{i,i} = \sigma
\quad\text{and}\quad
\sum_{i=1}^p x_{i,p-i+1} = \sigma
.
\end{equation*}
In other words, magic squares have the additional condition that the two diagonals also sum to the magic number $\sigma$. See Figure~\ref{figure:magic-square-example} for an example.
\end{definition}

\begin{definition}[Generalized Birkhoff polytopes]\label{definition:generalizedbirkhoffclassical}
We can generalize the definition of the Birkhoff polytope to rectangular arrays. The \defn{generalized Birkhoff $p \times q$ polytope} is the $p \times q$ classical transportation polytope with $u_1= \cdots = u_p = q$ and $v_1 = \cdots = v_q = p$. (This polytope is also known as the \defn{central transportation polytope} of size $p \times q$.)
\end{definition}

\subsection{Multi-way transportation polytopes}

Classical transportation polytopes are also called \defn{$2$-way transportation polytopes} since the coordinates $x_{i,j}$ have two indices. We can consider generalizations of $2$-way transportation polytopes by considering coordinates in three or more indices (e.g., $x_{i,j,k}$ or $x_{i,j,k,l}$, etc.). We introduce two natural generalizations of $2$-way transportation polytopes to $3$-way transportation polytopes, whose feasible points are $p \times q \times s$ tables of non-negative reals satisfying certain sum conditions:
\begin{itemize}
\item First, consider the $3$-way transportation polytope of size $p \times q \times s$ defined by \defn{$1$-marginals}: Let $u=(u_1,\dots,u_p) \in \R^p$,
$v=(y_1,\dots,y_q) \in \R^q$, and $w=(w_1,\dots,w_s) \in \R^s$ be three vectors. Let $P$ be the polyhedron defined by the following $p+q+s$ equations in the $pqs$ variables $x_{i,j,k} \in \R_{\geq 0}$ ($i\in[p], j\in[q], k\in[s]$):
\begin{equation}\label{equation:1marginals}
\sum_{j,k} x_{i,j,k} = u_{i}, \forall i \qquad 
\sum_{i,k} x_{i,j,k} = v_{j}, \forall j \qquad 
\sum_{i,j} x_{i,j,k} = w_{k}, \forall k.
\end{equation}
Observe that a necessary and sufficient condition for the polytope $P$ to be non-empty is that
\[
\sum_{i=1}^p u_p = \sum_{j=1}^q v_j = \sum_{k=1}^s w_k,
\label{equation:axialsizigies}
\]
and consequently the polytope $P$ is defined by only $p+q+s-2$ independent equations. (In the book~\cite{Yemelichev:Polytopes}, $3$-way transportation polytopes defined by $1$-marginals are known as \defn{$3$-way axial transportation polytopes}.)

\item Similarly, a $3$-way transportation polytope of size $p \times q \times s$ can be defined by specifying three real-valued matrices $U$, $V$, and $W$ (respectively) of sizes $q \times s$, $p \times s$, and $p \times q$ (respectively). These three matrices specify the line-sums resulting from fixing two of the indices of entries and adding over the remaining index. That is to say, the polyhedron $P$ is defined by the following $pq+ps+qs$ equations, called the \defn{$2$-marginals}, in the $pqs$ variables $x_{i,j,k} \in \R_{\geq 0}$ satisfying:
\[
\sum_{i} x_{i,j,k} = U_{j,k}, \forall j,k\qquad 
\sum_{j} x_{i,j,k} = V_{i,k}, \forall i,k\qquad 
\sum_{k} x_{i,j,k} = W_{i,j}, \forall i,j.
\label{equation:2marginals}
\]
One can see that in fact only $pq + ps + qs - p - q - s + 1$ of the defining equations are linearly independent for feasible systems. (In~\cite{Yemelichev:Polytopes}, the $3$-way transportation polytopes defined by $2$-marginals are called \defn{$3$-way planar transportation polytopes}.)
\end{itemize}
\begin{remark}
Of course, one can easily generalize these concepts to $\omega$-way tables for any integer $\omega$ and $\mu$-marginals for any $1 \leq \mu < \omega$. In~\cite{Yemelichev:Polytopes}, $\omega$-way transportation polytopes defined by $1$-marginals are called \defn{axial transportation polytopes} while $\omega$-way transportation polytopes defined by $(\omega-1)$-marginals are called \defn{planar transportation polytopes}. Clearly, $\omega$-way transportation polytopes are partition polyhedra.
\end{remark}

Observe that the $3$-way transportation polytopes of size $p \times q \times s$ defined by $1$-marginals generalize the classical transportation polytope of size $p \times q$, when $s=1$ and $w_1=\sum u_i=\sum v_j$. A less trivial rewriting of the classical $p \times 2$ transportation polytope as a $3$-way transportation polytope of size $p \times 2 \times 2$ defined by $2$-marginals is given in Theorem~\ref{theorem:22n}.

Some of the proofs will require our $3$-way transportation polytopes to be non-degenerate. We use the same definition as we did for classical transportation polytopes:
\begin{definition}
A multi-way transportation polytope is \defn{non-degenerate} if is simple and it is of maximal possible dimension.
\end{definition}
The maximum possible dimension for $p \times q \times s$ transportation polytopes defined by $1$-marginals is $pqs-p-q-s+2$. The maximum possible dimension for $p \times q \times s$ transportation polytopes defined by $2$-marginals is $(p-1)(q-1)(s-1)$. Graphs of non-degenerate transportation polytopes are of particular interest because they have the largest possible number of vertices and largest possible diameter among the graphs of all transportation polytopes of given type and parameters (e.g., $p$q$q$, and $s$). Indeed, if $P$ is a degenerate transportation polytope, by carefully perturbing the marginals that define $P$ we can get a non-degenerate polytope $P'$. (A careful explanation of how to do the perturbation in the case of classical transportation polytopes is presented as Lemma 4.6 in Chapter 6 of~\cite{Yemelichev:Polytopes} on page 281.) The perturbed marginals are obtained by taking a feasible point $x$ in $P$, perturbing the entries in the table and using the recomputed sums as the new marginals for $P'$. The graph of $P$ can be obtained from that of $P'$ by contracting certain edges, which cannot increase either the diameter nor the number of vertices.

\begin{definition}\label{definition:generalizedbirkhoffmultiway}
In Definition~\ref{definition:generalizedbirkhoffclassical}, we presented a generalization of Birkhoff polytopes to $p \times q$ rectangular arrays. The Birkhoff polytope has the following generalizations in the $3$-way setting:
\begin{enumerate}
\item
The \defn{generalized Birkhoff $3$-way axial polytope} is the axial $3$-way transportation $p \times q \times s$ polytope whose $1$-marginals are given by the vectors $u = (qs, \ldots, qs) \in \R^p$, $v = (ps, \ldots, ps) \in \R^q$, and $w = (pq, \ldots, pq) \in \R^s$.
\item
The \defn{generalized Birkhoff $3$-way planar polytope} is the planar $3$-way transportation $p \times q \times s$ polytope whose $2$-marginals are given by the $q \times s$ matrix $U_{j,k}=p$, the $p \times s$ matrix $V_{i,k}=q$, and the $p \times q$ matrix $W_{i,j}=s$.
\end{enumerate}
\end{definition}

\subsection{The transportation problem and related problems}

The transportation problem was formulated by Hitchcock in 1941 (see~\cite{Hitchcock:The-Distribution-of-a-Product}). A similar problem was studied by Monge in 1781 (see~\cite{Monge:Memoire-sur-la-theorie} and pages 227--228 in the book~\cite{Berge:The-Theory-of-Graphs} by Berge). Kantorovich studied Monge's formulation of the problem in 1942 (see~\cite{Kantorovich:On-the-translocation-of-masses}). A basic summary of the problem and some algorithms is found in~\cite{Reeb:Transportation-Problem}. In its original form, the standard \defn{transportation problem} is defined in the following way: given a $p \times q$ classical transportation polytope $P$ defined by the marginals $u \in \R^p$ and $v \in \R^q$ and a $p \times q$ matrix $\xi = (\xi_{i,j})_{i\in [p], j \in [q]}$ with real entries, solve the linear program
\begin{equation*}
\text{Minimize } \sum_{i\in[p]} \sum_{j\in[q]} \xi_{i,j}x_{i,j} \text{ subject to } x = (x_{i,j}) \in P.
\end{equation*}
The transportation problem is a special case of the minimum cost flow problem (see, e.g.,~\cite{Ahuja:Network} or~\cite{Baldoni-Silva:Counting-integer}).

We describe some variants of the transportation problem. In the usual transportation problem, the only variables are $x_{i,j}$, which is to say that the model only allows transportation of goods between sources and demands. In many applications, this assumption is too strong. One may desire more flexibility, such as allowing goods to be transported between supplies and between demands. There might also be points through which goods can be shipped from a supply to a demand. These more general problems are called \defn{transshipment problems}. Any transshipment problem can be easily converted into an equivalent transportation problem (see~\cite{Ahuja:Network}).

Another variant of the transportation problem is the assignment problem. An \defn{assignment problem} is an integer programming problem on the Birkhoff polytope. (Since the constraint matrix for classical transportation polytopes is totally unimodular, i.e., every non-singular submatrix has determinant $\pm 1$, the integer programming problem is equivalent to the linear programming relaxation (see, e.g.,~\cite{Giles:TDIandZPolyhedra}), but this is not true for general integer programs.) It is the integer version of the transportation problem. To be clear, the problem is: given a $p \times p$ classical transportation polytope $P$ defined by the marginals $u = (1,\ldots,1) \in \R^p$ and $v = (1,\ldots, 1) \in \R^p$ and a $p \times p$ matrix $\xi = (\xi_{i,j})_{i\in [p], j \in [p]}$ with real entries, solve the integer program
\begin{equation*}
\text{Minimize } \sum_{i\in[p]} \sum_{j\in[p]} \xi_{i,j}x_{i,j} \text{ subject to } x = (x_{i,j}) \in P \cap \Z^{p \times p}.
\end{equation*}
This problem is also known as the \defn{linear assignment problem}. See the recent book~\cite{Burkard:Assignment-Problems} by Burkard et al. on the assignment problem. Since $u=v=(1,1,\ldots,1) \in \R^p$, it follows that $P \cap \Z^{p \times p} = P \cap \{0,1\}^{p \times p}$, and therefore feasible solutions of this integer program are $0$-$1$ matrices. In fact, by Theorem~\ref{theorem:birkhoffvonneumann}, they are permutation matrices. The problem is called the assignment problem since it models the situation where $p$ people need to be assigned to $p$ jobs. In the model, each person is assigned exactly one job, and each job is performed by exactly one person. The $i$th person is assigned to the $j$th job exactly when $x_{i,j}=1$. Otherwise, $x_{i,j}=0$. By Lemma~\ref{lemma:vertexBx}, the problem can be converted into one of finding a maximum weight matching in a weighted bipartite graph.

\subsection{A survey of transportation polytopes and the transportation problem}

Chapter~\ref{chapter:transportation} presents our new results on  transportation polytopes. In this section, we survey the state of the art in transportation polytopes and the transportation problem. The survey~\cite{Vlach:SolutionsPlanarTransportation} by \given{Milan }Vlach, the 1984 monograph by Yemelichev, Kovalev, and Kratsov (see~\cite{Yemelichev:Polytopes}), and the paper~\cite{Klee:FacesTransportation} by Klee and Witzgall summarized the status of transportation polytopes up to the 1980s. See also the recent survey of De Loera and Onn (see~\cite{De-Loera:TransportationPolytopesSurvey}) on recent advances in transportation polytopes.

\subsubsection*{Applications of transportation polytopes to statistics}

Transportation polytopes appear in statistics, where they are called \defn{contingency tables}. See the survey by Diaconis and Gangolli (see~\cite{Diaconis:Rectangular}). A major practical application for transportation polytopes is in the study of statistical data security. The table entry security problem has been studied by Chowdhury et al. (see~\cite{Chowdhury:Disclosure-detection}), Duncan (see~\cite{Duncan:Disclosure} and~\cite{Duncan:DisclosureMultiple}) and Irving (see~\cite{Irving:DataSecurity}), among others. Heuristic algorithms that study table entry security have been studied by Fienberg (see~\cite{Fienberg:Frechet-and-Bonferroni}) and Buzzigoli and Giusti (see~\cite{Buzzigoli:LowerUpper}). In~\cite{Cox:ContingencyBounds}, Cox shows that these heuristic algorithms are not exact and presents an alternative algorithm.

Table entry security problems are related to the study of magic squares and lattice points in transportation polytopes. In~\cite{Cox:PropertiesStatistical}, Cox characterizes conditions needed on multi-way tables to guarantee that continuous bounds on integer-valued entries exist and are integral. The statistical data security problem is also related to the study of bounds on a particular entry (see, e.g.,~\cite{Dobra:Bounds-for-cell} by Dobra and Fienberg).

In~\cite{DeLoera:ComplexityStatistical}, De Loera and Onn give a complete description of the complexity of existence, counting, and entry-security in multi-way tables. Mehta and Patel (see~\cite{Mehta:NetworkAlgorithmFisher}) describe an exact test of significance for the independence of rows and columns of $p \times q$ contingency tables.

\subsubsection*{Transportation polytopes and optimization}

Before discussing the transportation problem, we discuss other questions in optimization theory related to the geometry and combinatorics of transportation polytopes. In~\cite{Pak:FourQuestions}, Pak analyzes the simplex method for the $p$th Birkhoff polytope $B_p$. Pak defines a certain linear functional $\xi$ on $\R^{p \times p}$ so that the $p$th Birkhoff polytope $B_p$ has a $\xi$-monotone path of length $C \cdot n!$, for a universal constant $C > 0$. However, Pak (see~\cite{Pak:FourQuestions}) also shows that the expected average running time of the simplex method on the Birkhoff polytope with respect to $\xi$ is in $O(p \log p)$.

Karp (see~\cite{Karp:Reducibility-Among}) proved that solving assignment programs with an arbitrary cost function on generalized $3$-way $p \times p \times p$ planar Birkhoff polytopes is NP-hard. (For more on NP-hardness and other complexity classes, see~\cite{Cormen:Algorithms}.)

De~Loera, Hemmecke, Onn, and Weismantel (see~\cite{DeLoera:NFoldInteger}) prove there is a polynomial time algorithm that, given $s$ and fixing $p$ and $q$, solves integer programming problems of $3$-way transportation polytopes of size $p \times q \times s$ defined by $2$-marginals, over any integer objective.

In~\cite{De-Loera:Convex-Integer} De~Loera, Hemmecke, Onn, Rothblum, and Weismantel present a polynomial oracle-time algorithm to solve convex integer maximization over $3$-way planar transportation polytopes, if two of the margin sizes remain fixed. More recently (in~\cite{Hemmecke:A-polynomial-oracle-time}) Hemmecke, Onn, and Weismantel prove a similar result for convex integer minimization.

In~\cite{Cuturi:Permanents-Transportation}, Cuturi analyzes the Monge-Kantorovich distance between the vectors $u$ and $v$ in terms of the transportation polytope with margins $u$ and $v$.

\subsubsection*{Solving the transportation and assignment problems}

Many researchers have studied algorithms to solve the transportation problem (see, e.g.,~\cite{Koopmans:ModelTransportation}, \cite{Proll:TransportationVariants}, \cite{Sreenivas:Probabilistic-Transportation}, \cite{Srinivasan:Accelerated-algorithms}, and~\cite{Totshek:An-Investigation-of-Real-Time}). See the recent book~\cite{Burkard:Assignment-Problems} for more on the assignment problem. In 1955, Kuhn (see~\cite{Kuhn:The-Hungarian-method}) introduced the Hungarian method, a combinatorial algorithm for solving the assignment problem, named in honor of Hungarian mathematicians K\H{o}nig and Egerv\'ary. Later, Kuhn (see~\cite{Kuhn:Variants-of-the-Hungarian}) developed a geometric model for variants of the algorithm. In 1957, Munkres (see~\cite{Munkres:AlgorithmsAssignmentTransportation}) proved that the algorithm is strongly polynomial. Edmonds proved (see~\cite{Edmonds:Paths-Trees}) Kuhn's original algorithm was in $O(n^4)$. Edmonds and Karp (see~\cite{Edmonds:improvementsAlgorithm}) and independently Tomizawa (see~\cite{Tomizawa:n3}) modified the algorithm to $O(n^3)$. (In~\cite{Balinski:A-primal-method}, Balinski and Gomory present an algorithm which is dual to the Hungarian algorithm.) An extension to the Munkres Hungarian algorithm is presented in ~\cite{Fran:An-Extension-of-the-Munkres}. Variants of the Hungarian algorithm are studied in~\cite{Klafszky:Variants-of-the-Hungarian}. In~\cite{Gill:The-k-assignment-polytope}, Gill and Linusson study a multi-way analogue of the assignment polytope. In~\cite{Queyranne:MultiIndexTransportation}, Queyranne and Spieksma describe formulations and applications for multi-index transportation problems. In~\cite{Ford:Solving-the-Transportation}, Ford and Fulkerson describe a simplified computing procedure for the transportation problem based on Kuhn's combinatorial algorithm for the assignment problem and a labeling process for solving maximal flow problems in networks. See the books~\cite{Edmonds:Matching} and~\cite{Lovasz:Matching} for more details on the matching problem. Hoffman (see~\cite{Hoffman:What-the-transportation}) summarizes significant advances in research on the transportation problem. 
In~\cite{Srinivasan:Benefit-cost-analysis}, Srinivasan and Thompson present a computer code for the transportation problem that is more efficient than the primal-dual method. In~\cite{Harris:A-Code-for-the-Transportation}, \given{Britton }Harris describes a variant of the code by Srinivasan and Thompson which greatly reduces the computation time for long and narrow transportation problems. A summary of computational results on various algorithms for the assignment problem is given in~\cite{Florian:An-experimental-evaluation}. A similar study for transportation problems is given in~\cite{Lee:University-of-California}. M\"uller-Merbach (see~\cite{Muller-Merbach:An-improved-starting} and~\cite{Muller-Merbach:VerschiedeneNaherungsverfahren}) describes an improved starting algorithm for the Ford-Fulkerson method.

In~\cite{Bertsekas:The-auction-algorithm}, \given{Dimitri }Bertsekas and \given{David }Castanon solve the transportation problem by converting it into an assignment problem and using a generalization of Bertsekas' auction algorithm. Each iteration of the auction algorithm of Bertsekas never decreases the linear functional being maximized. \given{Dimitri }Bertsekas and \given{David }Castanon (see~\cite{Bertsekas:The-auction-algorithm}) show that this modified version of the auction algorithm is very efficient for certain types of transportation problems.

Since the transportation problem is a linear programming problem, one can solve the problem using the simplex method. Dantzig (see~\cite{Dantzig:SimplexTransportation}) studied the behavior of the simplex method on transportation problems. In \cite{Arsham:A-simplex-type}, Arsham and Kahn present a simplex-type alternative to the stepping stone method of Charnes (see~\cite{Charnes:The-stepping-stone-method}) to solve the transportation problem. In~\cite{Papamanthou:Computational-experience}, Papamanthou et al. present an experimental computational study to compare the classical primal simplex algorithm and the exterior point algorithms for the transportation problem. In~\cite{Ji:A-dual-matrix-approach}, the authors describe what the so-called dual matrix approach to solving the transportation problem.

In~\cite{Nunkaew:Multiobjective-Programming}, \given{Wuttinan }Nunkaew and \given{Busaba }Phruksaphanrat solve a particular two-objective version of the transportation problem using lexicographic goal programming. For more on multi-objective programming, see the book~\cite{Ehrgott:Multicriteria-Optimization} by Ehrgott.

\subsubsection*{Solving special cases of the transportation problem}

Special cases of the transportation problem can sometimes be solved faster than the general transportation problem. For example, in~\cite{Barnes:TransporationUpper}, Barnes and Hoffman describe two special families of transportation problems that can be solved using a greedy algorithm. Glicksman et al. (see~\cite{Glicksman:Coding-the-transportation}) considers the special case for which the number of demands is many times greater than the number of supplies. The transportation problem with exclusionary side constraints has been studied by Goossens and Spieksma (see~\cite{Goossens:The-Transportation-Problem}) and Sun (see~\cite{Sun:The-transportation-problem}).

Signature algorithms solve certain classes of transportation problems in a number of steps bounded by the diameter of the dual polyhedron. Using signature algorithms, Balinski and Rispoli (see~\cite{Balinski:SignatureTransportation}) prove the Monotone Hirsch Conjecture holds for a certain class of classical transportation polytopes called signature polytopes.

\subsubsection*{Solving generalizations and variants of the transportation problem}

Many generalizations and variants of the standard transportation problem have been studied. Appa (see~\cite{Appa:Transportation}) surveys the solutions to 81 practical variants of the transportation problem by considering negative costs, and variants obtained by inequality constraints on the margins. In this section, we survey many of the generalizations and variants of the transportation problem.

In~\cite{Bammi:A-generalized-indices-transportation}, Bammi formulates a generalized-indices transportation problem and presents an algorithm for its solution. In~\cite{Baiou:StableAllocationProblem}, Ba\"iou and Balinski study the stable allocation problem, which generalizes $0$-$1$ stable matching problems to real-valued quantities.

In~\cite{Finkelshtein:An-iterative-method}, Finkelshtein modifies the standard $p \times q$ transportation problem by adding a decision variable $y \geq 0$, a $p \times q$ matrix $Z$, a constant $z$, and the additional constraint $\sum_{i\in[p],j\in[q]} Z_{i,j}x_{i,j} - y = z$. Finkelshtein (see~\cite{Finkelshtein:An-iterative-method}) presents an iterative method for its solution. In the note~\cite{Shkurba:On-the-solution-of-the-transportation}, Shkurba presents an exact method to solve this modified transportation problem. Lourie (see~\cite{Lourie:Topology-and-computation}) analyzes the stepping-stone method for solving a generalized transportation problem.

In general, the quadratic assignment problem is NP-hard (see~\cite{Sahni:P-complete-approximation}). In~\cite{Onn:ConvexCombinatorial}, Onn and Rothblum show that certain instances of the positive definite quadratic assignment problem are tractable by converting these problems into convex combinatorial optimization problems.

The \defn{transportation paradox} is the phenomenon in the transportation problem that the largest total transportation cost may not occur at the highest total quantities shipped. The paradox was discovered by Charnes and Klingman (see~\cite{Charnes:The-more-for-less-paradox}), and independently by Szwarc (see~\cite{Szwarc:The-Transportation-Paradox}). A sufficient condition for the paradox to occur was studied by Adlakha and Kowalski (see~\cite{Adlakha:A-quick-sufficient}). In~\cite{Storoy:TransportationParadox}, Stor{{\o}}y extends the work of Deineko, Klinz, and Woeginger (see~\cite{Deineko:Which-matrices}), and which present instances of the transportation problem that are immune to the transportation paradox. Liu (see~\cite{Liu:The-total-cost}) investigates the paradox when the demand and supply quantities are varying. Arsham (see~\cite{Arsham:Postoptimality-analyses}) studies the paradox via post-optimality analysis methods.

In~\cite{Lin:Solving-the-Transportation}, Lin introduces a genetic algorithm to solve transportation problem with fuzzy objective functions. Lin and Tsai (see~\cite{Lin:A-TWO-STAGE-GENETIC}) investigate solving the transportation problem with fuzzy demands and fuzzy supplies using a two-stage genetic algorithm. Li, Ida, and Gen (see~\cite{Li:Improved-genetic}) present an improved genetic algorithm for solving the fuzzy multiobjective solid transportation problem. When the transportation problem is associated with additional fixed cost for establishing the facilities or fulfilling the demand of customers, then it is called fixed charge transportation problem. Jo, Li, and Gen (see~\cite{Jo:Nonlinear-fixed}) apply the spanning tree-based genetic algorithm approach for solving the non-linear fixed charge transportation problem. The fixed-charge problem is a non-linear programming problem of practical interest in business and industry. One of its variations, the fixed-charge transportation problem, where fixed cost is incurred for every route that is used in the solution, along with the variable cost that is proportional to the amount shipped is studied by Kowalski and Lev (see~\cite{Kowalski:On-step-fixed-charge}).

In~\cite{Currin:Transportation-Problems}, Currin studies the transportation problem with inadmissible routes. In~\cite{Kuno:A-decomposition-algorithm}, Kuno and Utsunomiya address a method for solving two classes of production-transportation problems with concave production cost. In~\cite{Imam:Solving-Transportation}, Imam et al. describe a method of solving the transportation problem based on the object-oriented programming model.
In~\cite{Hartwick:A-Generalization-of-the-Transportation}, \given{John }Hartwick generalizes the Hitchcock-Koopmans analysis to take account of variable supplies and demand for a product at diverse geographically-separated locations.

\subsubsection*{Combinatorics of transportation polytopes}

The geometric combinatorial structure of transportation polytopes is a very active area of study. In~\cite{Klee:FacesTransportation}, Klee and Witzgall examine the combinatorial structure (and in particular, the number of vertices) of transportation polytopes. Klee and Witzgall conjectured (see~\cite{Klee:FacesTransportation}) and Bolker proved (see~\cite{Bolker:Transportation-polytopes}) that when $p$ and $q$ are relatively prime, the maximum possible number of vertices among $p \times q$ classical transportation polytopes is achieved by the generalized $p \times q$ Birkhoff polytope. In~\cite{Pak:Number}, Pak presents an efficient algorithm for computing the $f$-vector of the generalized Birkhoff polytope of size $p \times q$ when $q=p+1$. Hartfiel (see~\cite{Hartfiel:Full-patterns}) and Dahl (see~\cite{Dahl:Transportation-matrices}) describe the supports of certain feasible points in classical transportation polytopes.

The diameters of classical transportation polytopes have been studied extensively. Balinski and Rispoli (see~\cite{Balinski:SignatureTransportation}) explain why Kravtsov's supposed proof (see~\cite{Kravtsov:A-proof-of-the-maximal}) of the Hirsch Conjecture for transportation polytopes is incomplete. In~\cite{Cryan:RandomTransportation}, Cryan et al. analyze a natural random walk on the graph of the transportation polytopes. (In~\cite{Pak:FourQuestions}, Pak proves that nearest neighbor random walk does not mix fast on all $0$-$1$ polytopes, but that it does mix fast on Birkhoff polytopes.) In~\cite{Balinski:DualTransportation}, Balinski proves that the Hirsch Conjecture holds and is tight for dual transportation polyhedra. In~\cite{Balinksi:Faces-of-dual}, Balinski and Russakoff compute $f$-vectors of dual transportation polyhedra by analyzing partitions of $p+q-1$. McKeown and Rubin (see~\cite{McKeown:Adjacent-vertices}) and Oviedo (see~\cite{Oviedo:Adjacent-Extreme}) analyze the adjacencies of vertices in transportation polytopes.

Dyer and Frieze's (see~\cite{Dyer:RandomWalks}) polynomial diameter bound for totally unimodular polytopes based on random walks applies to classical transportation polytopes. Yemelichev, Kovalev, and Kravtsov (see Theorem 4.6 in Chapter 6 of~\cite{Yemelichev:Polytopes}) and Stougie (see~\cite{Stougie:PolynomialBound}) present improved polynomial bounds. This was improved to a quadratic bound by van~den~Heuvel and Stougie in~\cite{vandenHeuvel:QuadraticBound}. The first linear upper bound on the diameter classical transportation polytopes was proved in 2006 by \given{Graham }Brightwell, \given{Jan }van den Heuvel, and \given{Leen }Stougie (see~\cite{Brightwell:LinearTransportation}), although this was recently improved by Hurkens (see~\cite{Hurkens:Diameter4p}). 

In~\cite{Balinski:DualTransportation}, \given{Michel }Balinski proved that the Hirsch Conjecture holds for the bounded polytopes resulting from the intersection of a dual transportation polyhedron with a certain hyperplane. In joint work with \given{Jes\'us }De~Loera, \given{Shmuel }Onn, and \given{Francisco }Santos (see~\cite{DeLoera:GraphsTP} and Section~\ref{section:axialdiam} in this thesis), we prove a quadratic bound on the diameter of $3$-way axial transportation polytopes (see Theorem~\ref{theorem:main}).

In~\cite{Vlach:SolutionsPlanarTransportation}, Vlach surveys conditions for the non-emptiness of the $3$-way planar transportation polytope. Schell (see~\cite{Schell:Distribution-of-a-product}), Haley (see~\cite{Haley:The-multi-index-problem} and~\cite{Haley:Note-on-the-Letter}), Moravek and Vlach (see~\cite{Moravek:On-the-necessary-conditions} and~\cite{Moravek:On-Necessary-ConditionsClass}), and Smith (see~\cite{Smith:Further-necessary},~\cite{Smith:A-Procedure-for-Determining}, and~\cite{Smith:On-the-Moravek-and-Vlach}) prove necessary but not sufficient conditions on the margins for a $3$-way transportation polytope defined by $2$-marginals to be non-empty. The note~\cite{Emelichev:Multi-index-planar} presents criterion for a polytope to belong to the class of multi-index planar transportation polytopes with a maximum number of vertices. Kravtsov et al. (see~\cite{Kravtsov:On-some-properties},~\cite{Kravtsov:Asymptotics-of-multi-index}, and~\cite{Kravtsov:PolyhedralCombinatoricsMultiTP}) investigates combinatorial properties of multi-way transportation polytopes.

The $3$-way transportation polytopes are very interesting because of the following universality theorem of \given{Jes\'us }De~Loera and \given{Shmuel }Onn in~\cite{DeLoera:Universality}.
\begin{theorem}\label{theorem:universality}
Let $P$ be a rational convex polytope. Then, there is a $3$-way planar transportation polytope $\widetilde{P}$ isomorphic to $P$. Moreover, there is a $3$-way axial transportation polytope $\widetilde{P}$ which has a face $F$ isomorphic to $P$.
\end{theorem}
Isomorphic here means that, in particular, the polytope $\widetilde{P}$ or its face $F$ have the same face poset as $P$. (The polytope $\widetilde{P}$ is presented in the form shown in Lemma~\ref{lemma:standardform}.) The result in~\cite{DeLoera:Universality} also says that, given the polytope $P$, there is a polynomial time algorithm to construct $\widetilde{P}$ and the isomorphism mentioned. (See also~\cite{DeLoera:MarkovBases} and~\cite{DeLoera:UniversalitySlim}.)

In~\cite{Bulut:An-axial-four-index}, Bulut and Bulut study the axial $4$-way transportation polytopes and study algebraic characterizations of them. In~\cite{Bulut:Construction-and-algebraic}, the authors describe a network flow problem as a planar transportation polytope problem and solve the problem in terms of eigenvectors of certain matrices.

Bolker (see~\cite{Bolker:Simplical-Geometry}) defines an analogue of the support graph $B(x)$ for $\omega$-way transportation polytopes and studies its homological properties. (See~\cite{Hatcher:AlgebraicTopology} for an in-depth treatment of homology.) Lenz (see~\cite{Lenz:Toric-Ideal}) proves a degree bound on a certain reduced Gr\"obner basis for classical transportation polytopes.

In~\cite{DeLoera:EffectiveLatticePointCounting}, De Loera et al. report on, among other things, the computational results of counting lattice points in multi-index transportation polytopes using the software {\tt LattE} (see~\cite{latte}). In~\cite{Haase:Grobner-Basis}, Haase and Paffenholz prove that the toric ideals of every $3 \times 3$ transportation polytope $T$ is quadratically generated, if $T$ is not a multiple of the third Birkhoff polytope $B_3$.

\subsubsection*{Birkhoff polytopes}

The computation of volumes and triangulations of the Birkhoff polytope is related to the problem of generating a random doubly stochastic matrix (see~\cite{Chan:VolumePolytope}). Here, we survey past results on the triangulations, volumes, and lattice point enumeration of Birkhoff polytopes. In Section~\ref{section:regulartriangulationsBirkhoff}, we present our new results on the existence of non-regular triangulations of Birkhoff polytopes.

The explicit volume of the $p$th Birkhoff polytope $B_p$ is known (see~\cite{Pixton:The-Volumes-of-Birkhoff}) up to $p = 10$. In~\cite{Pak:FourQuestions}, Pak proves that the volume of the Birkhoff polytope is equal to the volume of a different polytope. \given{E. Rodney }Canfield and \given{Brendan }McKay (see~\cite{Canfield:VolumeBirkhoff}) present an asymptotic formula for the volume of the $p$th Birkhoff polytope $B_p$.

In~\cite{De-Loera:A-generating-function}, De~Loera, Liu, and Yoshida present a generating function for the number of semi-magic squares. Using this generating function, De~Loera et al. (see~\cite{De-Loera:A-generating-function}) present formulas for the coefficients of the Ehrhart polynomial of the $p$th Birkhoff polytope $B_p$, and a combinatorial formula for the volume of the $p$th Birkhoff polytope $B_p$ for all $p$. Barvinok (see~\cite{Barvinok:AsymptoticEstimatesNumberTransportation}) presents an asymptotic upper and lower bounds for the volumes of $p \times q$ classical transportation polytopes and the number of $p \times q$ semi-magic rectangles. In~\cite{Carlitz:EnumerationSymmetricArrays}, Carlitz describes lattice points of dilations of the Birkhoff polytope using exponential generating functions.

Counting magic squares and lattice points in (dilations of) Birkhoff polytopes is useful for computing their volumes. This problem has been studied by Ahmed (see~\cite{Ahmed:Algebraic-Combinatorics} and~\cite{Ahmed:Polytopes-of-Magic}), Beck and Pixton (see~\cite{Beck:EhrhartBirkhoff}), Beck, Cohen, Cuomo, and Gribelyuk (see~\cite{Beck:NumMagic}), Chan and Robbins (see~\cite{Chan:VolumePolytope}), Diaconis and Gamburd (see~\cite{Diaconis:Random-Matrices}), Halleck (see~\cite{Halleck:MagicSquaresDiophantine}), Hemmecke (see~\cite{Hemmecke:On-the-computation-of-Hilbert}), and Stanley (see~\cite{Stanley:LinearHomDioMagicLabel} and~\cite{Stanley:Magic-labelings}), among others.

In~\cite{Brualdi:DStochasticI}, Brualdi and Gibson use the permanent function to determine geometric properties of the Birkhoff polytope. In~\cite{Brualdi:DStochasticII}, Brualdi and Gibson study the graph of the Birkhoff polytope. In~\cite{Brualdi:DStochasticIII}, Brualdi and Gibson investigate the affine and combinatorial properties of the Birkhoff polytope. In~\cite{Brualdi:DStochasticIV}, Brualdi and Gibson investigate the extreme points, faces and their dimensions of the convex polytope of doubly stochastic matrices which are invariant under a fixed row and column permutation. In~\cite{Escolano:Birkhoff-polytopes}, Escolano et al. establish a link between Birkhoff polytopes and heat kernels on graphs.

Additional results are known for subpolytopes of the Birkhoff polytope. One can consider permutation polytopes, obtained as the convex hull of some vertices of a Birkhoff polytope. (Note that this notion of permutation polytope is distinct from the permutation polytopes of Billera and Sarangarajan in~\cite{Billera:CombinatoricsPermutationPolytopes}.) In~\cite{Onn:Geometry-Complexity}, Onn analyzes the geometry, complexity and combinatorics of permutation polytopes. In~\cite{Baumeister:On-permutation-polytopes}, Baumeister et al. study the faces and combinatorial types that appear in small permutation polytopes. Brualdi (see~\cite{Brualdi:Convex-polytopes}) investigates the faces of the convex polytope of doubly stochastic matrices which are invariant under a fixed row and column permutation. The \defn{$p$th tridiagonal Birkhoff polytope} is the convex hull of the vertices of the Birkhoff polytope whose support entries are in $\{(i,j) \in [p] \times [p] \mid |i-j| \leq 1\}$. In~\cite{da-Fonseca:Fibonacci-numbers}, da Fonseca et al. count the number of vertices of tridiagonal Birkhoff polytopes. In~\cite{Costa:The-number-of-faces}, Costa et al. present a formula that counts the number of faces tridiagonal Birkhoff polytopes.

Costa et al. (see~\cite{Costa:The-diameter-of-the-acyclic}) define a \defn{$p$th acyclic Birkhoff polytope} to be any polytope that is the convex hull of the set of matrices whose support corresponds to (some subset of) the edges (including loops) of a fixed tree graph. In~\cite{Costa:Face-counting}, Costa et al. count the faces of acyclic Birkhoff polytopes. In~\cite{Costa:The-diameter-of-the-acyclic}, Costa et al. prove an upper bound on the diameter of acyclic Birkhoff polytopes, which generalizes the diameter result of Dahl in~\cite{Dahl:Tridiagonal-doubly}.

Let the \defn{$p$th even Birkhoff polytope} be the convex hull of the $\frac12p!$ permutation matrices corresponding to even permutations. In~\cite{Cunningham:On-the-even-permutation}, Cunningham and Wang confirm a conjecture of Brualdi and Liu (see~\cite{Brualdi:The-polytope-of-even}) that the $p$th even Birkhoff polytope cannot be described as the solution set of polynomially many linear inequalities. In~\cite{Hood:Some-facets}, Hood and Perkinson describe some of the facets of the even Birkhoff polytope and prove the conjecture of Brualdi and Liu (see~\cite{Brualdi:The-polytope-of-even}) that the number of facets of the $p$th even Birkhoff polytope is not polynomial in $p$. In~\cite{Below:On-a-theorem-of-L.-Mirsky}, von Below shows that the condition of Mirsky given in~\cite{Mirsky:Even-doubly} is not sufficient for determining membership of a point in an even Birkhoff polytope. Cunningham and Wang (see~\cite{Cunningham:On-the-even-permutation}) also investigate the membership problem for the even Birkhoff polytope. In~\cite{Below:Even-and-odd-diagonals}, von~Below and R\'enier describe even and odd diagonals in even Birkhoff polytopes.

In~\cite{Cho:Convex-polytopes}, Cho and Nam introduce a signed analogue of the Birkhoff polytope.


\section[Enumeration of partition polytopes]{Enumeration of partition polytopes with fixed constraint matrix}\label{section:enumerationofpartition}

In this section, we discuss the necessary mathematical background to justify the proof of the following theorem:
\begin{theorem}
Given a fixed $r \times c$ real-valued matrix $A$, there is a finite algorithm that outputs a list of simple polytopes $\{P_1,\ldots,P_z\}$ so that \emph{any} simple polytope of the form
\begin{equation*}
P = \{x \in \R^c \mid Ax = b, x \geq 0\}
\end{equation*}
is combinatorially equivalent to one of the $P_i$'s.
\end{theorem}
We will describe how to implement this procedure. To make everything precise, we will illustrate everything using a running example. Our running example for the matrix $A$ will be the constraint matrix $A = A_{2 \times 3}$ of $2 \times 3$ classical transportation polytopes. In particular, using this methodology, we computed a \emph{complete} catalogue of non-degenerate transportation polytopes with small margin sizes:
\begin{theorem}\label{theorem:genericcase}
The following tables (see Appendix~\ref{appendix:catalogTP}) give a complete catalogue of transportation polytopes of small sizes:
\begin{itemize}
\item The only possible numbers $f_0$ of vertices of non-degenerate 
$2 \times 3$, 
$2 \times 4$, 
$2 \times 5$,
$3 \times 3$, and 
$3 \times 4$ 
classical transportation polytopes are those given in Tables~\ref{table:classical-2-3},
\ref{table:classical-2-4}, 
\ref{table:classical-2-5}, 
\ref{table:classical-3-3}, and 
\ref{table:classical-3-4}, respectively.

\item The only possible numbers $f_0$ of vertices of non-degenerate 
$2 \times 2 \times 2$ and 
$2 \times 2 \times 3$ 
axial transportation polytopes are those given in Tables~\ref{table:axial-2-2-2} and~\ref{table:axial-2-2-3}, respectively.

Every non-degenerate $2 \times 2 \times 4$ axial transportation polytope has between $32$ and $504$ vertices.  Every non-degenerate $2 \times 3 \times 3$ axial transportation polytopes has between $81$ and $1056$ vertices.  The number of vertices of non-degenerate $3 \times 3 \times 3$ axial transportation polytopes is at least $729$.

\item The only possible numbers $f_0$ of vertices of non-degenerate 
$2 \times 2 \times 2$ and
$2 \times 3 \times 3$
planar transportation polytopes are those given in Tables~\ref{table:planar-2-2-3} and~\ref{table:planar-2-3-3}, respectively.

Every non-degenerate $2 \times 3 \times 4$ planar transportation polytope has between $7$ and $480$ vertices.
\end{itemize}
\end{theorem}
The catalogue (see Appendix~\ref{appendix:catalogTP}) was obtained through an exhaustive enumeration of combinatorial types of transportation polytopes whose foundation is the theory of secondary polytopes and parametric linear programming. In some cases when the full enumeration was impossible we can at least obtain lower and upper bounds for the number of vertices that these polytopes can have.

We begin by recalling Lemma~\ref{lemma:partitionvertices}, which said that $x \in \R^c$ is a vertex of the partition polyhedron $P = \{x \in \R^c \mid Ax = b, x \geq 0 \}$ if and only if $x$ is a basic feasible solution corresponding to the basis $\mathcal{A}$ of $A$. Moreover, if the polyhedron $P$ is assumed to be \emph{simple} then the basic feasible solution must be \emph{strictly} positive on the entries corresponding to the basis $\mathcal{A}$. Geometrically, a basis $\mathcal{A} \subseteq A$ produces a vertex of a {simple} polyhedron $P$ if and only if $b$ lies in the \emph{interior} of the cone generated by $\mathcal{A}$. In conclusion, the vertices of a simple partition polyhedron $P$ are in bijection with the bases that define cones in $\R^r$ containing the vector $b \in \R^r$ within their interior.

We discuss the enumeration of partition polyhedra (as all transportation polytopes are) by discussing what happens to the combinatorics of $P_b = \{x \in \R^c \mid Ax=b, x \geq 0 \}$ as the vector $b \in \R^r$ changes while the $r \times c$ matrix $A$ remains fixed. (For the case of transportation polytopes, the matrix $A$ describes the type and size of the transportation polytope, and $b$ is the vector given by the margins.) This study, for general matrices, is known as \defn{parametric linear programming} (see Chapter 1 Section 2 and Chapter 9 Section 5 of~\cite{DeLoera:Triangulations}). Since the vector $b$ will vary, we use the notation $P_b$ to describe the partition polyhedron given by $b \in \R^r$ for our fixed matrix $A$.

What happens when we let the vector $b$ vary? If the polyhedron $P_b$ is simple and the change in $b$ is small, the facets of $P_b$ move but the combinatorial type of $P_b$ does not change. Only when a basic solution changes from being feasible to not feasible, or vice versa, the combinatorics of $P_b$ (that is, the face lattice and, in particular, the graph of $P_b$) can change (see~\cite{Blind:On-puzzles-and-polytope} and~\cite{Kalai:A-simple-way-to-tell}).

Put differently: Let $\Sigma_A$ denote the set of all cones generated by bases of $A$. Let $\partial \Sigma_A$ denote the union of the boundaries of all elements of $\Sigma_A$. The connected components of $\cone(A) \setminus \partial \Sigma_A$ are open convex polyhedral cones called the \defn{chambers} of $A$. (To be precise, the connected components of $\cone(A) \setminus \partial \Sigma_A$ are the relative interiors of polyhedral cones.) We call the \defn{chamber associated to} a given vector $b \in \R^r$ the intersection of the interiors of simple cones that contain the vector $b$ in their interior. Every sufficiently generic vector $b \in \cone(A)$ is in a chamber (as opposed to lying on $\partial \Sigma_A$). Two vectors $b$ and $b'$ in the same chamber define simple polytopes $P_{b}$ and $P_{b'}$ that are equivalent up to combinatorial type. The collection of all the chambers is the \defn{chamber complex} (or \defn{chamber system}) associated with $A$.  Putting all this together we conclude:
\begin{proposition}\label{proposition:everychamber}
Fix a matrix $A$ of size $r \times c$ of full row rank $r$. To represent every possible combinatorial type of simple polytope of the form $P_b = \{x \in \R^c \mid Ax=b, x \geq 0 \}$, for a fixed $A$ and over varying vector $b \in \R^r$, it is enough to choose one $b \in \R^r$ from each chamber of the chamber complex of $A$.
\end{proposition}

We illustrate Proposition~\ref{proposition:everychamber} through an example. The discussion here gives a full description of the software {\tt transportgen} (see~\cite{transportgen}).
\begin{example}
\label{example:triprism}
Consider the $5 \times 6$ defining matrix $A$ of all classical $2 \times 3$ transportation polytopes. That is:
\[
A=\left [
\begin {array}{cccccc} 
1&1&1&0&0&0\\
0&0&0&1&1&1\\
1&0&0&1&0&0\\
0&1&0&0&1&0\\
0&0&1&0&0&1
\end {array}\right ].
\]
Up to permutation of coordinates, the system $\{x \in \R^6 \mid Ax=b,\ x \geq 0\}$ defines all classical $2 \times 3$ transportation polytopes with marginals given by the vector $b \in \R^5$.

The matrix $A$ does not have full row rank, so instead we remove the last row. Let $A_{2 \times 3}$ be the $4 \times 6$ matrix
\[
A_{2 \times 3}=\left [
\begin {array}{cccccc} 
1&1&1&0&0&0\\
0&0&0&1&1&1\\
1&0&0&1&0&0\\
0&1&0&0&1&0
\end {array}\right ].
\]
The columns of the matrix $A_{2 \times 3}$ span a four-dimensional cone in $\R^5$. It will be relevant later that if we slice this cone by an affine hyperplane (such as $\sum_i x_i =1$) we obtain the three-dimensional triangular prism shown in Figure~\ref{figure:chambers23}, but embedded in $\R^4$.
 
The chamber complex can be obtained by slicing the prism with the six planes containing a vertex of the prism and the edge ``opposite'' to it.  The resulting chamber complex is hard to visualize or draw, even in this small case, but we will see later how to recover the structure of the chamber complex for this example using {Gale transforms}. In particular, as we will see, this decomposes the triangular prism into $18$ chambers.
\end{example}

It is very easy to ``sample'' inside the chamber complex and find chambers of different numbers of bases, i.e., transportation polytopes with different numbers of vertices. One can simply throw random positive values to the cell entries of a $3$-way $p \times q \times s$ table and then compute the $1$-marginals or $2$-marginals associated to it. But with this method it is not obvious how to guarantee that one has obtained \emph{all} the possible chambers. We want to ensure that we have visited \emph{every} chamber. To do this, we use the following approach based on Gale transforms and regular triangulations. (For a detailed treatment on triangulations, see~\cite{DeLoera:Triangulations}.)

Let $A$ be a vector configuration of $c$ vectors in $\R^r$. We usually think of the column vectors of a matrix $A$ of size $r \times c$. A vector configuration $B$ of $c$ vectors in $\R^g$ is called a \defn{Gale transform} of the vector configuration $A$ if the row space of the matrix with columns given by $B$ is the orthogonal complement in $\R^c$ of the row space of the matrix with columns given by $A$. (Here, we assume that the matrix $A$ is of full row rank $r$ so that $g=c-r$.) Gale transforms are essential tools in the study of convex polytopes because the combinatorial properties
of $A$ and $B$ are intimately related (see Chapter 6 in~\cite{Ziegler:Lectures} for details). Proofs of the following statements can be found in~\cite{Billera:SecondaryPolyhedra} and~\cite{Gelfand:Discriminants}. (See also Chapters 4 and 5 of~\cite{DeLoera:Triangulations},~\cite{Lee:Regular-triangulations}, and~\cite{Thomas:Lectures-in-Geometric}.)
\begin{lemma}\label{lemma:secondary}
Let $A$ be a vector configuration and $\widehat{A}$ be a Gale transform of $A$.
\begin{enumerate}

\item\label{item:regulartriangulationscorrespondence} 
The chambers of $A$ are in bijection with the regular triangulations of the Gale transform $\widehat{A}$ of $A$: From a chamber in $A$, one can recover a regular triangulation of $\widehat{A}$ via complementation, namely for a basis $\mathcal{A}$ of vectors in $A$ the elements of $\widehat{A}$ not belonging to $\mathcal{A}$ form a basis for $\widehat{A}$.  The collection of those bases gives a triangulation of $\widehat{A}$.

\item
There exists a polyhedron, the \defn{secondary polyhedron}, whose vertices are in bijection with the regular triangulations of the Gale transform $\widehat{A}$.

\item
The face lattice of the chamber complex of the vector configuration $A$ is anti-isomorphic to the face lattice of the secondary polyhedron of the Gale transform $\widehat{A}$ of $A$. The latter is, in turn, isomorphic to the refinement poset of all regular subdivisions of $\widehat{A}$.

\item
If the cone generated by $A$ is pointed (for example, if all of its entries are non-negative as it is the case for transportation polytopes), then its Gale transform $\widehat{A}$ is a totally cyclic vector configuration. (In other words, the cone generated by $\widehat{A}$ is all of $\R^g$.)

\end{enumerate}
\end{lemma}

To summarize, this means:
\begin{corollary}
To enumerate all simple partition polyhedra $P$ of the form $P = \{x \in \R^c \mid Ax = b, x \geq 0\}$ with fixed constraint matrix $A$ and vary vector $b \in \R^r$ is the same as enumerating all the chambers of the chamber complex of $A$, which is the same as enumerating all of the regular triangulations of $\widehat{A}$.
\end{corollary}

\begin{example}
We take again the $4 \times 6$ matrix $A_{2 \times 3}$ of rank four defined in Example~\ref{example:triprism}. A Gale transform $B_{2 \times 3} = \widehat{A}_{2 \times 3}$ of $A_{2 \times 3}$ consists of the columns of the $2 \times 6$ matrix
\[
B_{2 \times 3}
=
\widehat{A}_{2 \times 3}
=
\left[
\begin{array}{cccccc}
1&-1&0&-1&1&0\\
1&0&-1&-1&0&1
\end{array}
\right].
\]

In Figure~\ref{figure:chambers23} we represent the Gale diagram $\widehat{A}_{2 \times 3}$ and its $18$ regular triangulations, each one providing a combinatorial type of non-degenerate $2 \times 3$ classical transportation polytope, although repeated combinatorial types occur. (A representative vector in each of the $18$ chambers is chosen in the presentation of the $18$ non-degenerate $2 \times 3$ classical transportation polytopes presented in~\cite{TPDB}.) The chamber adjacency, which corresponds to bistellar flips, is indicated by dotted edges.
\begin{figure}[hbt]
\begin{center}
\includegraphics[scale=.25]{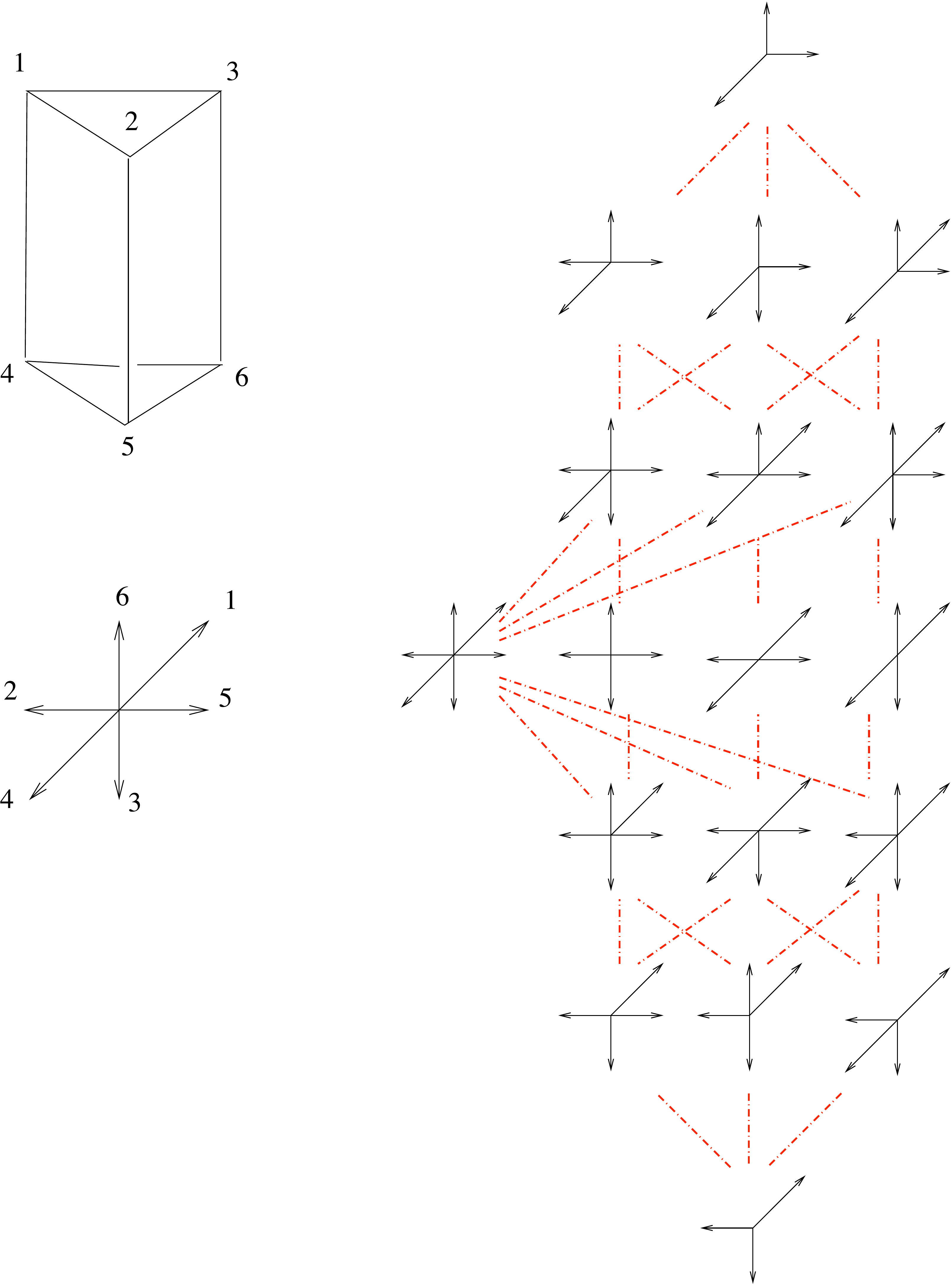}
\caption{The $18$ regular triangulations of $\widehat{A}_{2 \times 3}$, which correspond to the $18$ chambers, each defining a non-degenerate $2 \times 3$ transportation polytope.}
\label{figure:chambers23}
\end{center}
\end{figure}
\end{example}

Thus, generating all the combinatorial types of non-degenerate transportation polytopes is the same as listing the distinct regular triangulations of the Gale transform of the defining matrix $A$. In the case of transportation polytopes, the matrix $A$ depends only on the margin type ($1$-marginals, $2$-marginals, etc.) and the margin sizes ($p$, $q$, $s$, etc.).

Now, it is well-known that the regular triangulations of a vector
configuration can all be generated by applying bistellar flips to a seed regular triangulation (see, e.g.,~\cite{Billera:SecondaryPolyhedra},~\cite{DeLoera:Triangulations},~\cite{Ziegler:Lectures}). Bistellar flips are combinatorial operations that transform one triangulation into another and regularity of triangulations can be determined by checking feasibility of a certain linear program. In our case, the linear program is the very one that defines the polytope $P_b$. An example of this linear programming feasibility problem is shown in Example 5.2.10 in~\cite{DeLoera:Triangulations}.

\begin{example}[Example~\ref{example:triprism} continued]
\label{example:triprism2b}
Consider the only triangulation of $\widehat{A}_{2 \times 3}$ with six cones (the leftmost one in the middle row in Figure~\ref{figure:chambers23}). The necessary and sufficient conditions in the non-negative vector $a=(a_1,a_2,\ldots,a_6)$ in order to produce this triangulation are that each $a_i$ be smaller than the sum of the two adjacent to it. That is,
\begin{eqnarray*}
a_1 < a_5 + a_6, & a_2 < a_4 + a_6, & a_3 < a_4 + a_5, \cr
a_4 < a_2 + a_3, & a_5 < a_1 + a_3, & a_6 < a_1 + a_2. 
\end{eqnarray*}
Thus, these conditions on the marginals characterize the $2 \times 3$ transportation polytopes that are hexagons. 
\end{example}

This method is implemented in the {\tt C++} program {\tt transportgen} (see~\cite{transportgen}). This program calls {\tt TOPCOM} (see~\cite{topcom}), a package for triangulations that computes, among other things, the list of all regular triangulations of a given point configuration $B$. (In our case, $B$ is the Gale transform $\widehat{A}$ of $A$.) The program calls {\tt polymake} (see~\cite{polymake}) for the Gale transform. The input to this program is a matrix $A$ of size $r \times c$. The program outputs one vector $b \in \R^r$ per chamber in the chamber complex of $A$. The result is a list of transportation polytopes, one per chamber, given in the {\tt polymake} file format (see~\cite{polymake}) representing all possible combinatorial types of simple transportation polytopes defined by $A$.

\subsection{Lower and upper bounds via integer programming}

Even for seemingly small cases, such as $3 \times 3 \times 3$ transportation polytopes defined by $1$-marginals, listing all chambers (and thus all combinatorial types of transportation polytopes) is practically impossible. In these cases we can follow a different approach to at least obtain upper and lower bounds for the number of vertices of transportation polytopes. By the discussion above, this is the same as finding bounds for the number of simplices in triangulations of the Gale transform $\widehat{A}$. Here we follow the method proposed in~\cite{DeLoera:AllTriangulations}, based on the universal polytope. This universal polytope, introduced by Billera, Filliman, and Sturmfels in~\cite{Billera:SecondaryComplexity}, has \emph{all} triangulations (regular or not) of a given vector configuration $B$ in $\R^g$ as vertices, and projects to the secondary polytope. The universal polytope has much higher dimension than the secondary polytope. In fact, its ambient dimension is the number of possible bases of the configuration $A$, thus no more than $\binom{|A|}{g+1}$. It has the advantage that the number of simplices in different triangulations is given by the values of a certain linear functional $\psi$.

More precisely, we think of the chambers inside $\cone(A)$ as the vertices of the following high-dimensional $0$-$1$ polytope: Assume $A$ is a vector configuration of $c$ vectors in $\R^r$.  Let $N$ be the number of $r$-dimensional simple cones in $A$.  We define $U_A$ as the convex hull in $\R^N$ of the set of incidence $0$-$1$ vectors of all chambers of $A$. For a chamber $T$ the \defn{incidence vector} $v_T$ has coordinates $\,(v_T)_\mathcal{A} = 1$ if the basis $\mathcal{A} \in T$ and $\,(v_T)_\mathcal{A} = 0$ if $\mathcal{A}$ is not a basis of $T$.  The polytope $U_A$ is the \defn{universal polytope} defined in general by Billera, Filliman, and Sturmfels in~\cite{Billera:SecondaryComplexity}, although there it is defined in terms of the triangulations of the Gale transform of $A$.

In~\cite{DeLoera:AllTriangulations}, it was shown that the vertices of the universal polytope $U_A$ of $A$ are exactly the integral points inside a polyhedron that has a simple inequality description in terms of the oriented matroid of $A$. (See~\cite{Bokowski:Computational-Oriented},~\cite{DeLoera:AllTriangulations}, or~\cite{Ziegler:Lectures} for more on oriented matroids.) The concrete integer programming problems in question were solved using the {\tt CPLEX} Linear Solver${}^\textrm{TM}$.  The program to generate the linear constraints is a small {\tt C++} program available on the web (see~\cite{UniversalGenerator}).

\begin{example}
\label{example:triprism3}
Let $A$ be the $4 \times 6$ matrix $A_{2 \times 3}$ from Example~\ref{example:triprism}. Let $S$ denote the set of all bases that can be defined in $A$. Then $N = |S| = \binom{c}{g} = \binom{6}{2} = 15$. So the universal polytope $U_A$ is defined in $\R^{15}$, where each coordinate is indexed by a $2$-subset $\mathcal{A}$ of $\{1,\ldots,6\}$. Thus $U_A$ is the convex hull in $\R^N$ of the incidence vectors $v_T$ corresponding to the $18$ chambers of $A$. By Lemma~\ref{lemma:secondary}, this is equivalent to the convex hull of the incidence vectors $v_T$ of the $18$ triangulations of $\widehat{A}$.  For example, the triangulation $T = \{ \{1,2\},\{1,3\}, \{2,4\}, \{3,4\} \}$ in Figure~\ref{figure:chambers23} gives the incidence vector $v_T = e_{\{1,2\}} + e_{\{1,3\}} + e_{\{2,4\}} + e_{\{3,4\}}$ (where $e_\gamma$ is the basis unit vector in the direction $\gamma$) as one of the vectors of the convex hull.  

The convex hull of these $18$ incidence vectors is a $6$-dimensional $0$-$1$ polytope $U_A$ in $\R^{15}$. That the dimension is (at most) six follows from the following considerations:
\begin{itemize}
\item Since the pairs $\{1,4\}$, $\{2,5\}$ and $\{3,6\}$ are not full-dimensional and thus never appear as a simplex in any triangulation $T$ of $\widehat{A}$, the polytope $U_A$ is contained in the subspace $x_{\{1,4\}} = x_{\{2,5\}} = x_{\{3,6\}} = 0$.  
\item Since the vector $1$ has $2$ and $6$ on one side and $5$ and $3$ on the other, in every triangulation the sum  $x_{\{1,2\}} + x_{\{1,6\}}$ equals the sum $x_{\{1,3\}} + x_{\{1,5\}}$ (and it equals zero or one depending on whether the triangulation uses the vector $1$ or not). This implies the first of the following equalities, the rest being the analogous statement for the other five vectors:
\begin{eqnarray*}
x_{\{1,2\}} + x_{\{1,6\}} - x_{\{1,3\}} - x_{\{1,5\}} = &0,\\
x_{\{2,3\}} + x_{\{2,4\}} - x_{\{1,2\}} - x_{\{2,6\}} = &0,\\
x_{\{1,3\}} + x_{\{3,5\}} - x_{\{2,3\}} - x_{\{3,4\}} = &0,\\
x_{\{3,4\}} + x_{\{4,5\}} - x_{\{2,4\}} - x_{\{4,6\}} = &0,\\
x_{\{1,5\}} + x_{\{5,6\}} - x_{\{3,5\}} - x_{\{4,5\}} = &0,\\
x_{\{2,6\}} + x_{\{4,6\}} - x_{\{1,6\}} - x_{\{5,6\}} = &0.
\end{eqnarray*}
Observe that one of these equations is redundant, since the sum of the left-hand sides is already zero.

\item Since every triangulation needs to cover the angle between, for example, vectors $1$ and $6$, and this angle is covered only by the cones $16$, $12$ and $56$ (see Figure~\ref{figure:chambers23}), we have that
\[x_{\{1,6\}} + x_{\{1,2\}} + x_{\{5,6\}} =  1.\]
\end{itemize}
The results in~\cite{DeLoera:AllTriangulations} say that the polytope $U_A$ is the convex hull of the non-negative integer points in $\R^{15}$ satisfying this list of equations.

We now denote by $\psi : \R^N \rightarrow \R$ the cost vector defining the linear functional
\[\psi(x) = (1,1,\ldots,1) \cdot x = \sum_{\mathcal{A} \in S} x_\mathcal{A}.\]
Then the values $\psi(x)$ of the linear functional $\psi$ on $U_A \cap \{0,1\}^N$ are the only possible values for the number of vertices of non-degenerate polytopes of the form $P_b = \{x \in \R^c \mid Ax=b, x \geq 0 \}$.  In particular, the solutions to the linear programming relaxations: ``minimize (respectively maximize) $\psi(x)$ subject to $x \in U_B$'' give lower (respectively upper) bounds to the possible values for the number $f_0$ of vertices of non-degenerate polytopes of the form $P_b = \{x \in \R^c \mid Ax=b, x \geq 0 \}$.  In the running example, $3 \leq \psi(x) \leq 6$ whenever $x \in U_A \cap \{0,1\}^N$.  From Table~\ref{table:classical-2-3}, we observe that the number of vertices of a non-degenerate $2 \times 3$ transportation polytope equals $3$, $4$, $5$, or $6$.
\end{example}


\chapter[The Hirsch Conjecture Today]{The Hirsch Conjecture Today}\label{chapter:hirsch}

\setcounter{theorem}{0} 

The Hirsch Conjecture was posed in 1957 in a letter from \given{Warren }Hirsch to \given{George }Dantzig (see page 168 of~\cite{Dantzig:Linear}). Besides its implications in linear programming, which motivated the conjecture in the first place, it is one of the most fundamental open questions in combinatorial geometry. It states that the graph of a $d$-dimensional polytope with $n$ facets cannot have diameter greater than $n-d$. That is to say, we can go from any vertex of the polytope to any other vertex using at most $n-d$ edges.
\begin{conjecture}[Hirsch Conjecture]\label{conjecture:hirsch}
Let $n > d \geq 2$. Let $P$ be a $d$-dimensional polytope with $n$ facets, and let $G(P)$ be its graph. Then $\diam(G(P)) \leq n - d$.
\end{conjecture}
\begin{remark}
See Appendix~\ref{appendix:CounterHirsch} for a very recent update on the status of the conjecture. Santos (see~\cite{Santos:CounterHirsch}) announced the construction of a counter-example to the Hirsch Conjecture as this dissertation was being submitted.
\end{remark}

Despite being one of the oldest and most basic problems in polyhedral combinatorics, what we know is quite scarce. Most notably, no polynomial upper bound is known for the diameters of polytopes. That is to say, the following Polynomial Diameter Conjecture is open:
\begin{conjecture}[Polynomial Diameter Conjecture]\label{conjecture:hirsch-poly}
Is there a polynomial function $f(n,d)$ such that for any polytope (or polyhedron)
$P$ of dimension $d$ with $n$ facets, $\diam(G(P)) \leq f(n,d)$?.
\end{conjecture}
In contrast, very few polytopes are known where the bound $n-d$ is attained. This chapter surveys the state of the art on the Hirsch Conjecture as of today. An earlier survey on the geometry of the Hirsch Conjecture was written by \given{Victor }Klee and \given{Peter }Kleinschmidt (see~\cite{Klee:The-d-step-conjecture}) in 1987. (On a related note, the surveys by Megiddo (see~\cite{Megiddo:On-the-complexity-of-linear}) and Todd (see~\cite{Todd:The-many-facets}) describe the complexity of linear programming algorithms.) This chapter is based on a new survey written jointly with \given{Francisco }Santos (see~\cite{KimSantos:HirschSurvey}). 

Any polytope or polyhedron $P$, given by its facet-description, can be perturbed to a simple one $P'$ by a generic and small change in the coefficients of its defining inequalities. This will make non-simple vertices ``explode'' and become clusters of new vertices, all of which will be simple. This process cannot decrease the diameter of the graph, since we can recover the graph of $P$ from that of $P'$ by collapsing certain edges.
\begin{proposition}
The diameter of $P'$ is an upper bound on the diameter of $P$.
\end{proposition}
Hence, to study the Hirsch Conjecture, one only needs to consider simple polytopes:
\begin{lemma}
\label{lemma:simple}
The diameter of any polytope $P$ is bounded above by the diameter of some simple polytope $P'$ with the same dimension and number of facets.
\end{lemma}
Graphs of simple polytopes are better behaved than graphs of arbitrary polytopes. Their main property in the context of the Hirsch Conjecture is that if $u$ and $v$ are vertices joined by an edge in a simple polytope then there is a single facet containing $u$ and not $v$, and a single facet containing $v$ and not $u$. That is, at each step along the graph of $P$ we enter a single facet and leave another one. 

\begin{definition}
Let $H(n,d)$ denote the maximum diameter of graphs of $d$-polytopes with $n$ facets. Let $H_u(n,d)$ denote the maximum diameter of graphs of $d$-polyhedra with $n$ facets.
\end{definition}
\begin{remark}\label{remark:equivalentHnd}
The Hirsch Conjecture is, then, the assertion that $H(n,d) \leq n-d$ for all $n > d \geq 2$. Since all polytopes are polyhedra, clearly $H(n,d) \leq H_u(n,d)$ for all pairs $(n,d)$ with $n > d \geq 2$. By Lemma~\ref{lemma:simple}, the quantity $H(n,d)$ is, equivalently, the maximum diameter of graphs among \emph{simple} $d$-dimensional polytopes with $n$ facets. By Remark~\ref{remark:polaritysimplesimplicial}, this is also the maximum diameter among the \emph{dual graphs} of \emph{simplicial} $d$-polytopes with $n$ vertices.
\end{remark}

In the rest of this chapter, we survey the status of the Hirsch Conjecture today. In Section~\ref{section:HirschAndLP}, we discuss the connection between the Hirsch Conjecture and linear programming.

A polytope is called \defn{Hirsch-sharp} if it meets the Hirsch Conjecture with equality. In other words, a $d$-dimensional polytope with $n$ facets is Hirsch-sharp if its diameter is exactly $n-d$. Section~\ref{section:basic-results} discusses Hirsch-sharp polytopes. In Section~\ref{section:Q4}, we describe a special Hirsch-sharp polytope first discovered by \given{Victor }Klee and \given{David }Walkup in their seminal 1967 paper~\cite{Klee:d-step}. In Section~\ref{section:wedging}, we present useful operations which preserve Hirsch-sharpness. Section~\ref{section:many-hirsch-sharp} surveys the work of \given{Kerstin }Fritzsche, \given{Fred }Holt, and \given{Victor }Klee (see~\cite{Fritzsche99morepolytopes},~\cite{Holt:Many-polytopes}, and~\cite{Holt:Hsharpd7}) which prove that Hirsch-sharp $d$-polytopes with $n$ facets exist whenever $n > d \geq 7$.

Section~\ref{section:positive-results} surveys results in support of the Hirsch Conjecture, or at least the weaker Polynomial Diameter Conjecture (see Conjecture~\ref{conjecture:hirsch-poly}).  Section~\ref{section:equivalences} surveys results of Klee and Walkup (see~\cite{Klee:d-step}), proving the equivalence of the Hirsch Conjecture to two other natural conjectures, the $d$-step and the Non-revisiting Conjectures. Section~\ref{section:small-dimension} surveys the cases of small $d$ or $n-d$, which are some cases where the Hirsch Conjecture is known to hold. Section~\ref{section:upper-bounds} presents general upper bounds on the diameters of all polytopes. Section~\ref{section:continuous} discusses the recent work of Deza, Terlaky, and Zinchenko (see~\cite{Deza:Central-path},~\cite{Deza:The-continuous-d-step}, and~\cite{Deza:Curvature}) on a continuous analogue of the Hirsch Conjecture arising in the context of the central path method for linear programming and convex optimization. Section~\ref{section:special} surveys known diameter bounds for special classes of polytopes, including transportation polytopes.

We close this chapter with Section~\ref{section:negative}, which surveys evidence against the Hirsch Conjecture (or the Polynomial Diameter Conjecture should the Hirsch Conjecture turn out to be false). This section proves that three natural variants of the Hirsch Conjecture are false. Section~\ref{section:unbounded-monotone} presents the work of Klee and Walkup in~\cite{Klee:d-step} which shows the Unbounded Hirsch Conjecture is false. We also show Todd's result (see~\cite{Todd:The-monotonic-bounded}) that the Monotone Hirsch Conjecture is false. Section~\ref{section:combinatorial} summarizes the work of Mani and Walkup in~\cite{Mani:A-3-sphere-counterexample}, where they prove that the Topological Hirsch Conjecture is false.

\section{The Hirsch Conjecture and linear programming}\label{section:HirschAndLP}

The original motivation for the Hirsch Conjecture comes from its relation to the simplex algorithm for linear programming. The surveys by Megiddo (see~\cite{Megiddo:On-the-complexity-of-linear}) and Todd (see~\cite{Todd:The-many-facets}) describe the complexity of linear programming algorithms. In 1979, Khachiyan (see~\cite{Khachiyan}) proved that linear programming problems can be solved in polynomial time on the input size of the polyhedron via the \emph{ellipsoid method} for linear programming. In 1984, Karmarkar (see~\cite{Karmarkar}) devised the \emph{interior point method} for linear programming. (See~\cite{Renegar:A-Mathematical-View} for a complete treatment of interior point methods.)

Although the latter is more applicable (it is easier to implement and has better complexity) than the former, to this day the most commonly used method for linear programming is the \emph{simplex method}, devised by \given{George }Dantzig in 1947. \given{Jack }Dongarra and \given{Francis }Sullivan regard the simplex method among the top ten algorithms of the twentieth century (see~\cite{Dongarra:Top-Ten-Algorithms-of-the-Century}). In geometric terms, the simplex method first finds an arbitrary vertex in the feasibility polyhedron. Then, it uses local rules to move from vertex to adjacent vertex in such a way that the value $\xi(x)$ of the given linear functional $\xi$ increases at every step. When there is no such \emph{pivot step} that can increase the functional, convexity implies that we have achieved the maximum possible value of it.  (For a practical explanation of how to implement the simplex method, see the book~\cite{Chvatal:Linear-Programing}.)

Bounding the diameter of graphs of polytopes has received a lot of attention because of its connection to the performance of the simplex method for linear programming. Clearly, a lower bound for the performance of the simplex method under \emph{any} pivot rule is the diameter of the polyhedron $P$. The converse is not true, since knowing that $P$ has a small graph diameter does not in principle tell us how to go from one vertex to another in a small number of steps. In particular, many of the results on diameters of polyhedra presented later (e.g., Theorems~\ref{theorem:quasipolynomial}~and~\ref{theorem:01hirsch}) do not have direct implications for the efficiency of the simplex method: The proofs construct a short path pairs of vertices only \emph{after} specifying the coordinates of both vertices!
\begin{remark}
More relevant in the context of the simplex method is the proof that there are ``randomized'' pivot rules that get to the optimum vertex in \emph{subexponential} time. The exact bound is $e^{K\sqrt{d\log n}}$, where $K$ is a fixed constant (see~\cite{Kalai:A-subexponential-randomized} and~\cite{Matousek:A-subexponential-bound}). See Theorem~\ref{theorem:randomizedpivot}. Also interesting are results by Spielman and Teng (see~\cite{Spielman:WhySimplexUsually}) and improvements by Vershynin (see~\cite{Vershynin:BeyondHirsch}) saying that ``random polytopes'' have polynomial diameter. More precisely: any polytope can be perturbed to have a diameter that is polynomial in the values of $n$, $d$, and the inverse of the perturbation parameter (see Theorem~\ref{theorem:vershynin}). This result seems to explain why the simplex method works well in practice. 
\end{remark}

In fact, the complexity of the simplex method depends on a local rule (known as a \emph{pivot rule}) chosen to move from vertex to vertex. (In the survey~\cite{Terlaky:PivotSurvey}, Terlaky and Zhang study the finiteness of the simplex algorithm under various pivot rules.) The \emph{a priori} ``best'' pivot rule is ``move to the neighbor where the functional increases most,'' but \given{Victor }Klee and \given{George }Minty (see~\cite{Klee-Minty}) showed in 1972 that this can lead to paths of exponential length, even in polytopes that are combinatorially equivalent to cubes. The same worst-case exponential behavior has been proved for nearly every deterministic rule devised so far, although not for all of them. However, there are subexponential, but not yet polynomial, randomized pivot algorithms (see Theorem~\ref{theorem:randomizedpivot}). Still, the simplex algorithm is highly efficient \emph{in practice} on most linear optimization problems. (See~\cite{Bixby:Solving-Real-World} for more on practical performance in solving linear programs.)

There is another reason why investigating the complexity of the simplex method is important, even if we already know polynomial time algorithms. The algorithms of Khachiyan and Karmarkar are polynomial in the bit length of the input. (For a discussion of bit length, see Section 3.2 of~\cite{Schrijver:LinearProgramming}.) However, it would be interesting to know whether a polynomial algorithm for linear programming in the real number machine model of Blum, Cucker, Shub, and Smale (see~\cite{BCSS}) exists. That is to say, is there an algorithm that uses a polynomial number of arithmetic operations on the coefficients of the linear program, rather than on their bits; or, better yet, a \emph{strongly polynomial algorithm}, i.e., one that is polynomial both in the arithmetic sense and the bit sense? These two related problems were included by \given{Steve }Smale (see~\cite{Smale}) in his list of mathematical problems for the next century. A polynomial pivot rule for the simplex method would solve them in the affirmative.

In this context, the much weaker statement asserting the existence of a polynomial upper bound in $n$ and $d$ is relevant (see the Polynomial Diameter Conjecture, stated as Conjecture~\ref{conjecture:hirsch-poly} on page~\pageref{conjecture:hirsch-poly}), if the linear one turns out to be false (see, e.g., the discussion in~\cite{kalaiblog}). The best bound for \emph{all} polytopes is a quasi-polynomial bound by Kalai and Kleitman in~\cite{Kalai:Quasi-polynomial}. We present their result later as Theorem~\ref{theorem:quasipolynomial}.

\begin{remark}
Many of the algorithms for linear programming discussed here make sense (with some modifications) in the larger context of convex optimization. See, e.g.,~\cite{Boyd:ConvexOptimization},~\cite{Renegar:A-Mathematical-View},~\cite{Rockafellar:Variational}, and~\cite{Wets:Optimization}.
\end{remark}

\section{Hirsch-sharp polytopes}\label{section:basic-results}

Recall that a $d$-polytope (or polyhedron) with $n$ facets is called \defn{Hirsch-sharp} if it meets the Hirsch Conjecture with equality; that is, if its diameter is exactly $n-d$. Here we show several ways to construct Hirsch-sharp polytopes. Our current knowledge is that Hirsch-sharp $d$-polytopes with $n$ facets: 
\begin{itemize}
\item exist if one of the following three conditions holds: $n\leq 2d$, $n\leq 3d-3$, or $d\geq 7$.
\item do not exist for $d\le3$ if $n>2d$, or for $(n,d)\in\{(10,4), (11,4), (12,4)\}$.
\item are unknown, but may exist in all other cases: that is, if $d\in\{4,5,6\}$ and $n>3d-3$, except for the three pairs $(n,4)$ mentioned above.
\end{itemize}
However, all Hirsch-sharp polytopes with $n>2d$ that are known are obtained from one particular Hirsch-sharp polytope, by simple geometric operations of wedging, truncating and gluing, as we show in Section~\ref{section:many-hirsch-sharp}. That is to say: to some extent, we only know one non-trivial Hirsch-sharp polytope, the four-dimensional polytope discovered by \given{Victor }Klee and \given{David }Walkup in their seminal 1967 paper~\cite{Klee:d-step}. In Section~\ref{section:Q4} we describe this particular Hirsch-sharp polytope. Our description gives integer coordinates for it much smaller than the original ones.

\subsection{Easy constructions}\label{section:examples}

Let us start with the most basic Hirsch-sharp polytopes.

\begin{enumerate}

\item 
{\bf Cubes.} The $d$-dimensional cube $\Box_d$ has $2d$ facets. The distance between vertices is the Hamming distance (i.e., number of different coordinates); thus the diameter of $\Box_d$ is $d$. 

\item 
{\bf Products.}
One reason that the Hirsch Conjecture is natural is that it is ``invariant'' under taking products. If $P$ and $Q$ meet the Hirsch Conjecture with equality, then their Cartesian product $P\times Q$ does the same. Indeed, both the dimension and number of facets of $P\times Q$ are the sum of those of $P$ and $Q$. The diameter of the product is the sum of the diameters: if we want to go from vertex $(u_1,v_1)$ to vertex $(u_2, v_2)$ we can do so by going from $(u_1,v_1)$ to $(u_2,v_1)$ along $P\times \{v_1\}$ and then to $(u_2,v_2)$ along $\{u_2\}\times Q$; and there is no better way.

\item 
{\bf Products of simplices.}
In particular, any product of simplices of any dimension satisfies the Hirsch Conjecture with equality. Let $P = \Delta_{i_1} \times \cdots \times \Delta_{i_k}$ be the product of $k$ simplices.  The dimension of $P$ is $\sum_{j=1}^k i_j$,  $P$ has $\sum_{j=1}^k (i_j+1)$ facets, and its diameter is $k$. 

\begin{corollary}
For every $d< n \leq 2d$ there are simple $d$-polytopes with $n$ facets and diameter $n-d$.
\end{corollary}

\begin{proof}
Let $k=n-d\leq d$ and let $i_1,\dots,i_k$ be any partition of $d$ into $k$ positive integers (that is, $i_1 + \cdots +i_k=d$). Let  $P=\Delta_{i_1} \times \cdots \times \Delta_{i_k}$.
\end{proof}

\item {\bf Polyhedra with $n\leq 2d$.}
There is another way to construct Hirsch-sharp polytopes when $d < n \leq 2d$. Let $k = n-d$ and $u$ be the origin in $\R^d$. Let $v=(1,\dots,1,0,\dots,0)$ be the point whose first $k$ coordinates are $1$ and whose remaining $d-k$ coordinates are $0$.

Consider the polytope $P$ defined by the following $d+k$ inequalities:
\[
x_i \geq 0,\quad \forall i; \qquad \psi_j(x)\geq 0, \quad j=1,\dots, k,
\]
where the $\psi_i$ are affine linear functionals that vanish at $v$ and are positive at $u$. No matter what choice we make for the $\psi_j$'s, as long as they are sufficiently generic to make $P$ simple, $P$ will have diameter (at least) $k$; to go from $v$ to $u$ we need to enter the $k$ facets $x_j=0$, $j=1,\dots, k$, and each step gets you into at most one of them. 

In principle, $P$ may be an unbounded polyhedron; but if one of the $\psi_j$'s is, say, $k- \sum x_i $, then it will be bounded.

\end{enumerate}

So, Hirsch-sharp polytopes with $n \leq 2d$ are easy to find. We consider them ``trivial'' Hirsch-sharp polytopes. In the following sections we show examples of ``non-trivial'' ones. But before that let us mention that unbounded Hirsch-sharp polyhedra with any number of facets are also easy to obtain:

\begin{proposition}
For every $n \geq d$ there are simple unbounded $d$-polyhedra with $n$ facets and diameter $n-d$.
\end{proposition}

\begin{proof}
The proof is by induction on $n$, the base case $n=d$ being the orthant $\R_{\geq 0}^d =\{x \in \R^d \mid x_i\geq 0, \forall i\}$. Our inductive hypothesis is not only that we have constructed a $d$-polyhedron $P$ with $n-1$ facets and diameter $n-d-1$; also, that vertices $u$ and $v$ at distance $n-d-1$ exist in it with $v$ incident to some unbounded ray $l$. Let $H$ be a supporting hyperplane of $l$, and tilt it slightly at a point $v'$ in the interior of $l$ to obtain a new hyperplane $H'$. (See Figure~\ref{figure:unbounded-hirsch-sharp}.) Then, the polyhedron $P'$ obtained cutting $P$ with the tilted hyperplane has $n$ facets and diameter $n-d$; $v$ is the only vertex adjacent to $v$ in the graph, so we need at least $1 + (n-d-1)$ steps to go from $v'$ to $u$.
\end{proof}
\begin{figure}[hbt]
\begin{center}
\includegraphics[scale=.61]{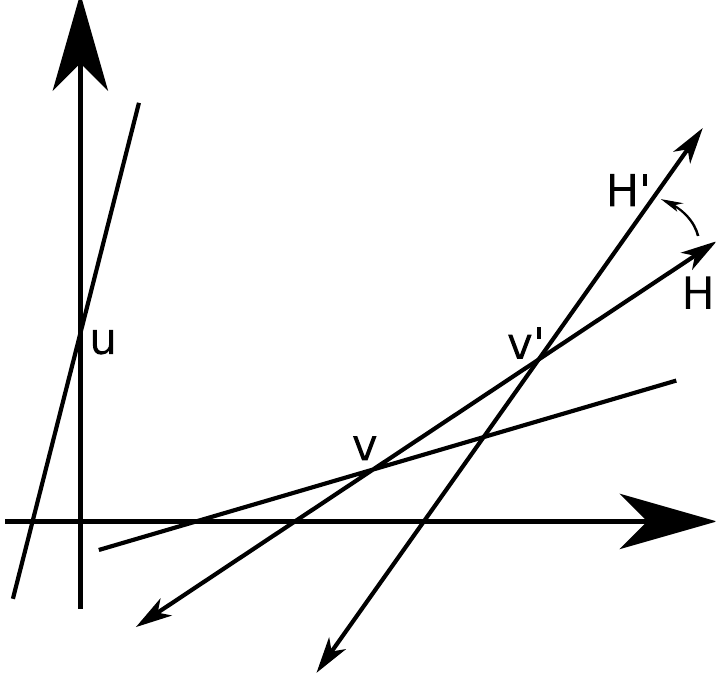}
\end{center}
\caption{Tilting the hyperplane $H$: an example in dimension two.}\label{figure:unbounded-hirsch-sharp}
\end{figure}

\subsection{The Klee-Walkup polytope $Q_4$} \label{section:Q4}

In their seminal 1967 paper~\cite{Klee:d-step} on the Hirsch Conjecture and related issues, \given{Victor }Klee and \given{David }Walkup describe a $4$-dimensional polytope called $Q_4$ with nine facets whose diameter is five. Innocent as this might look, this first ``non-trivial'' Hirsch-sharp polytope is at the basis of the construction of every remaining Hirsch-sharp polytope known to date (see Section~\ref{section:many-hirsch-sharp}). It is also instrumental in disproving the unbounded and monotone versions of the Hirsch Conjecture, which we will discuss in Section~\ref{section:unbounded-monotone}. Moreover, its existence is something of an anomaly: Altshuler, Bokowski, and Steinberg (see~\cite{Altshuler:The-classification-of-simplicial}) list all $1296$ combinatorial types of simplicial spheres with nine vertices. Of them, $1142$ are polytopal (that is, are the boundary complex of some polytope) and the polar of $Q_4$ is the only one that is Hirsch-sharp.

We prefer to describe the polytope in the polar view (that is, we will describe $Q_4^\Delta$), where it is simplicial instead of simple and the diameter is considered for the polar graph $G^\Delta(Q_4^\Delta)$, which is we know by Part~(\ref{item:dualgraph}) of Lemma~\ref{lemma:polar} is isomorphic to the graph $G(Q_4)$ of $Q_4$. So, let $Q_4^\Delta$ be the convex hull of the following nine points in $\R^4$. The coordinates here are much smaller than the original ones in~\cite{Klee:d-step}:
\[
\begin{tabular}{lll}
$a:= (-3,3,1,2)$,&&$e:= (3,3,-1,2)$,\\
$b:=(3,-3,1,2),$&&$f:=(-3,-3,-1,2),$\\
$c:=(2,-1,1,3)$,&&$g:=(-1,-2,-1,3)$,\\
$d:=(-2,1,1,3)$,&&$h:=(1,2,-1,3)$,\\
&$w:=(0,0,0,-2)$.
\end{tabular}
\]
The polar (simple) polytope $Q_4 = Q_4^{\Delta\Delta}$ of $Q_4^\Delta$ (see Part~(\ref{item:polarinvolution}) of Lemma~\ref{lemma:polar}) is obtained by converting each vertex $v$ of the simplicial polytope $Q_4^\Delta$ into an
inequality $v \cdot x\leq 1$. For example, the inequality corresponding to vertex $a$ below is $-3x_1+3 x_2+ x_3+ 2x_4\leq 1$. 

The key property of the polytope $Q_4^\Delta$ is:
\begin{theorem}[Klee-Walkup~\cite{Klee:d-step}]\label{theorem:klee-walkup}
Any path in the simplicial polytope $Q_4^\Delta$ from the tetrahedron $abcd$ to the tetrahedron $efgh$ needs at least five steps.
\end{theorem}

To prove this, you may simply  input these coordinates into any software able to compute the (polar) graph of a polytope. Our suggestion for this is {\tt polymake} (see~\cite{polymake}). However, understanding this polytope can be the key to the construction of counter-examples to the Hirsch Conjecture, so it is worth presenting the following hybrid computer-human proof. 

\begin{proof}
When pivoting from the tetrahedron $abcd$ to the tetrahedron $efgh$ we already need four steps to introduce, one by one, the four vertices $e$, $f$, $g$ and $h$. Hence, all paths that use the extra vertex $w$ have necessarily length five of bigger. This means we can concentrate on the subcomplex $K$ of $\partial Q_4^\Delta$ consisting of tetrahedra that do not use the vertex $w$. This subcomplex is called the \defn{anti-star} of $w$ in $\partial Q_4^\Delta$. We claim that this subcomplex consists of the $15$ tetrahedra in Figure~\ref{figure:polarK}.
\begin{figure}[hbt]
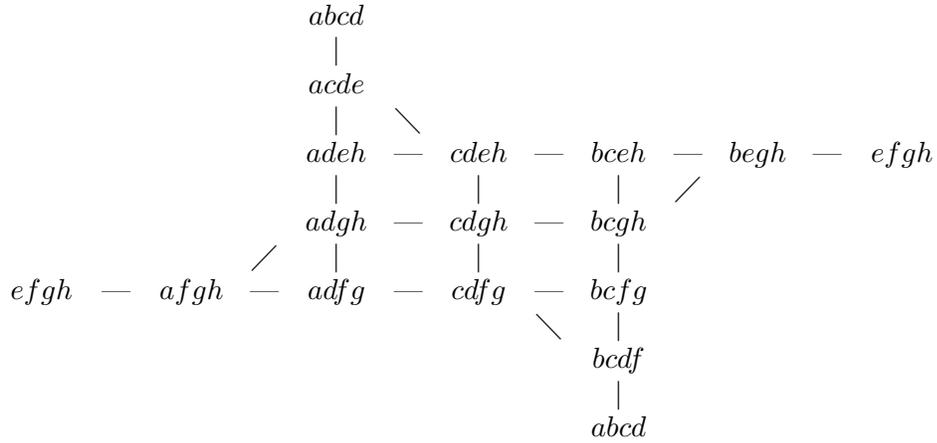

\[
\begin{matrix}
 & & abcd && && &&&\\
 && | & & && &&  &\\
 & & acde && && &&&\\
  && | & \diagdown & &&&& &\\
  && adeh &\text{---} & cdeh &\text{---} & bceh &\text{---} & begh&\text{---}\quad  efgh \\
  && | && | && | & \diagup & &\\
  && adgh &\text{---} & cdgh &\text{---} & bcgh  && &\\
   & \diagup & | && | && | &&&\\
 efgh\quad\text{---}\quad afgh & \text{---} & adfg &\text{---} & cdfg &\text{---} & bcfg &&&\\
   && && & \diagdown & |  &&&\\
  && && & & bcdf & &&\\
   && && & & |  &&&\\
     && && & & abcd &  &&\\
 \end{matrix}
\]
\caption{The polar graph of the subcomplex $K$}\label{figure:polarK}
\end{figure}

Figure~\ref{figure:polarK} shows adjacencies among tetrahedra; that is, it is the polar graph of $K$. Observe that $abcd$ and $efgh$ are repeated in the figure, to better reflect symmetry. From the picture it is easy to conclude our statement: there is no tetrahedron that can be reached in two steps from both $abcd$ and $efgh$, so the diameter is at least five.
\end{proof}

From the picture we can also see which tetrahedra of $\partial Q_4^\Delta$ use the vertex $w$: there is one for each triangle that appears only once in the list. For example, since the tetrahedron $abcd$ is adjacent only to $acde$ and $bcdf$, the triangles $abc$ and $bcd$ are joined to the vertex $w$. The boundary of the anti-star of the vertex $w$, that is, the \defn{link} of $w$ in $Q_4^\Delta$ turns out to be, combinatorially, the triangulation of the boundary of a cube displayed in Figure~\ref{figure:klee-walkup-cube}.
\begin{figure}[htb]
\begin{center}
\includegraphics[width=1.5in]{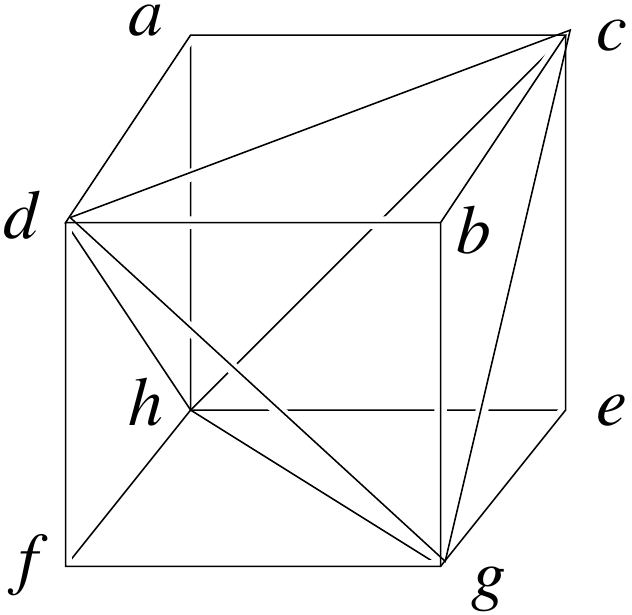}
\caption{The link of $w$ in $Q_4^\Delta$ is combinatorially a triangulation of the boundary of a cube.}
\label{figure:klee-walkup-cube}
\end{center}
\end{figure}
The anti-star $K$ of $w$ in the boundary $\partial Q_4^\Delta$ of $Q_4^\Delta$ is a topological triangulation of the interior of the cube. But we need to deform the cube a bit to realize this triangulation geometrically. This is shown in Figure~\ref{figure:klee-walkup}: the quadrilaterals $abcd$ and $efgh$ are displayed separately as lying in two different horizontal planes (so that the two relevant tetrahedra $abcd$ and $efgh$ degenerate to flat quadrilaterals), and the central part of the figure shows the intersection of the complex $K$ with their bisecting plane. Tetrahedra with three points on one plane and one in the other appear as triangles and tetrahedra with two points on either side appear as quadrilaterals. The tetrahedra $abcd$ and $efgh$ do not show up in the figure, since they do not intersect the intermediate plane. For the interested reader, this picture is an example of a \defn{mixed subdivision} of the Minkowski sum of two polygons. The fact that triangulations of polytopes with their vertices lying in two parallel hyperplanes can be pictured as mixed subdivisions is the \emph{polyhedral Cayley trick} (see, e.g.,~\cite{Santos:CayleyTrick} or Chapter 9 of~\cite{DeLoera:Triangulations}).

\begin{figure}[htb]
\begin{center}
\includegraphics[width=4in]{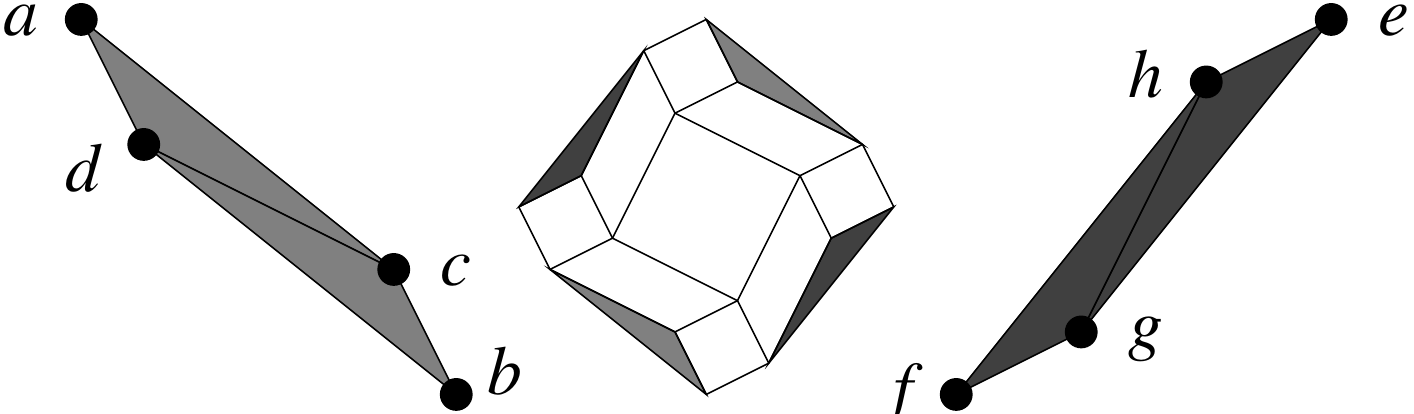}
\caption{The Klee-Walkup complex as a mixed subdivision. The shadowed triangles represent tetrahedra adjacent to $abcd$ and $efgh$}
\label{figure:klee-walkup}
\end{center}
\end{figure}

\subsection{Wedge and one-point suspension}\label{section:wedging}

Here we describe a very basic, yet extremely fruitful, operation that one can do to a polytope. In the simple version it is called \defn{wedging} and its simplicial counterpart (i.e., in the polar setting) is the \defn{one-point suspension}. We will prove that the Hirsch Conjecture is invariant under wedges and truncations. (To simplify our exposition, in this section we assume that all our polytopes are full-dimensional, i.e., the ambient dimension $c$ is always equal to the dimension $d$.)

\subsubsection*{Wedging}

Let $F$ be a facet of a polytope $P$, and let $f(x)\leq b$ be the inequality corresponding to $F$. The wedge of $P$ over $F$ is the polytope 
\[
\wed_F(P) := P \times [0,\infty) \cap \{ (x,t)\in \R^d\times \R : f(x) + t \leq b \}.
\]
Put differently, $\wed_F(P)$ is formed by intersecting the half-cylinder $C:=P \times [0,\infty)$ with a closed half-space $J$ in $\R^{d+1}$ such that:
\begin{itemize}
\item the intersection $J \cap C$ is bounded and has nonempty interior, and
\item the boundary hyperplane $H := \partial J$ is such that $H \cap C = F$.
\end{itemize}

\begin{lemma}
\label{lemma:wedge}
Let $P$ be a $d$-polytope with $n$ facets. Let  $\wed_F(P)$ be its wedge on a certain facet $F$. Then, $\wed_F(P)$ has dimension $d+1$, $n+1$ facets, and
\[
\diam(\wed_F(P)) \geq \diam(P).
\]
\end{lemma}

\begin{proof}
The wedge increases both the dimension and the number of facets by one. Indeed, the polytope $\wed_F(P)$ has a vertical facet projecting to each facet of the polytope $P$ other than the facet $F$, plus the two facets that cut the cylinder $P \times \R$, and whose intersection projects to $F$. See Figure~\ref{figure:wedge} for an example.

\begin{figure}[hbt]
\begin{center}
\includegraphics[scale=0.60]{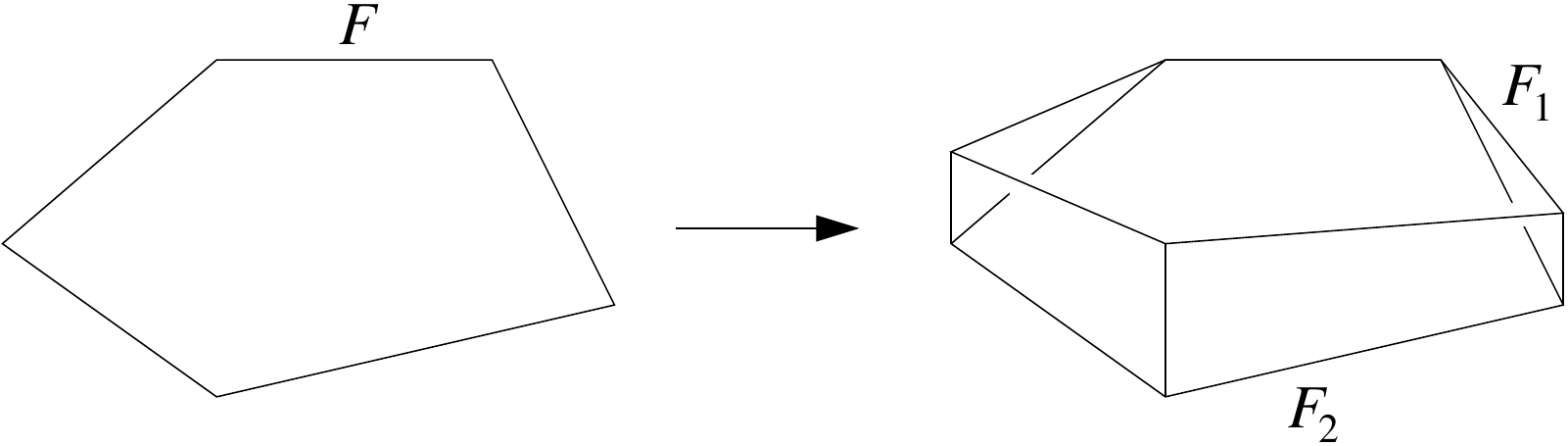}
\caption{A $5$-gon and a wedge on its top facet}\label{figure:wedge}
\end{center}
\end{figure}

For the diameter, since every edge of the polytope $\wed_F(P)$ projects either to an edge of $P$ or to a vertex of $P$, the diameter of $\wed_F(P)$ is at least that of $P$. 
\end{proof}

In particular, if the polytope $P$ is Hirsch-sharp, then the polytope $\wed_F(P)$ is either Hirsch-sharp or a counter-example to the Hirsch Conjecture. The properties that $P$ would need for the latter to be the case will be made explicit in Remark~\ref{remark:non-hirsch-wedge}.

\subsubsection*{One-point suspension}

It will be useful to consider the same operation in the polar setting, where it is known as the \defn{one-point suspension}. We refer the reader to Section 4.2 of~\cite{DeLoera:Triangulations} for an expanded overview. Let $w$ be a vertex of the polytope $P$. The one-point suspension of $P\subset \R^d$ at the vertex $w$ is the polytope
\[
\ops_w(P) := \conv\big( (P \times \{0\}) \cup (\{w\} \times \{-1,+1\})\big)\subset\R^{d+1}.
\]
That is, $\ops_w(P)$ is formed by taking the convex hull of $P$ (in an ambient space of one higher dimension) with a ``raised'' and ``lowered'' copy of the vertex $w$. See Figure~\ref{figure:ops} for an example.

\begin{figure}[hbt]
\begin{center}
\includegraphics[scale=0.6]{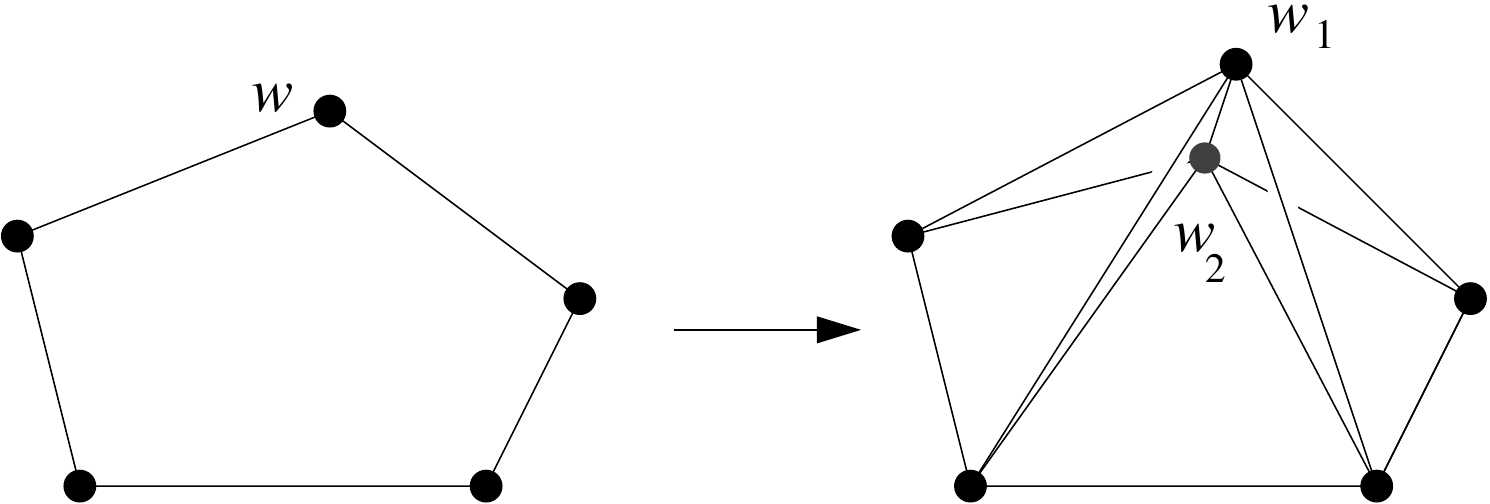}
\caption{The simplicial version of Figure~\ref{figure:wedge}: a $5$-gon and a one-point suspension on its topmost vertex}\label{figure:ops}
\end{center}
\end{figure}

Recasting Lemma~\ref{lemma:wedge} in the polar setting gives:
\begin{lemma}\label{lemma:ops}
Let $P$ be a $d$-polytope with $n$ vertices. Let $\ops_w(P)$ be its one-point suspension on a certain vertex $w$. Then $\ops_w(P)$ is a $(d+1)$-dimensional polytope with $n+1$ vertices, and the diameter of the polar graph $G^\Delta(\ops_w(P))$ of $\ops_w(P)$ is at least the diameter of the polar graph of $P$.
\end{lemma}

Since the wedge of a simple polytope produces a simple polytope, the one-point suspension of a simplicial polytope is a simplicial polytope. In fact, the one-point suspension can be described for abstract simplicial complexes: Let $L$ be a simplicial complex and $w$ a vertex of it. Recall that the anti-star $\ast_L(w)$ of $w$ is the subcomplex consisting of simplices not using $w$ and the link $\lk_L(w)$ of $w$ is the subcomplex of simplices not using $w$ but joined to $w$. If $L$ is a piecewise-linear $k$-sphere, then $\ast_L(w)$  and $\lk_L(w)$ are a $k$-ball and a $(k-1)$-sphere, respectively. The one-point suspension of $L$ at $w$ is the following complex:
\begin{equation}\label{equation:opswL}
\ops_w(L):=(\ast_L(w) * w_1)\cup (\ast_L(w) * w_2) \cup (\lk_L(w)* \overline{w_1w_2}).
\end{equation}
Here $*$ denotes the \defn{join} operation: $L*K$ has as simplices all joins of one simplex of $L$ and one of $K$. In Figure~\ref{figure:ops} the three parts of the formula are the three triangles using $w_1$ but not $w_2$, the three using $w_2$ but not $w_1$, and the two using both, respectively.

In the next section we will make use of an iterated one-point suspension. That is, in $\ops_w(P)$ we take the one-point suspension over one of the new vertices $w_1$ and $w_2$, then again in one of the new vertices created, and so on. We leave it to the reader to check that, at the level of simplicial complexes, the one-point suspension iterated $k$ times on the simplicial complex $L$ produces the simplicial complex $\ops_w(L)^{(k)}$ below, where $\Delta_k$ is a $k$-simplex with vertices $w_1,\dots,w_{k+1}$ and $\partial \Delta_k$ is its boundary:
\[
\ops_w(L)^{(k)}:= (\ast_L(w) * \partial \Delta_k) \cup (\lk_L(w)*{\Delta_k}).
\]
Observe that this generalizes the formula for $\ops_w(L)$ presented above in~\eqref{equation:opswL}.

\subsection{Many Hirsch-sharp polytopes}\label{section:many-hirsch-sharp}

In Section~\ref{section:Q4} we saw the first example of a non-trivial Hirsch-sharp polytope, the Klee-Walkup $4$-polytope $Q_4$ with $9$ facets. In this section we see constructions of other Hirsch-sharp polytopes, which together prove the following:

\begin{theorem}[\cite{Fritzsche99morepolytopes,Holt:Hsharpd7,Holt:Many-polytopes}]
\label{theorem:hirsch-sharp}
Hirsch-sharp $d$-polytopes with $n$ facets exist in at least the following cases:
\begin{enumerate}
\item\label{item:HSn2d} $n \leq 2d$;
\item\label{item:HSn3d3} $n \leq 3d-3$; and
\item\label{item:HSd7} $d \geq 7$.
\end{enumerate}
\end{theorem}

The ``trivial'' case $n \leq 2d$ was shown in Section~\ref{section:examples}. The case $n\leq 3d-3$ was first proved in 1998 in~\cite{Holt:Many-polytopes}, together with the case $d\geq 14$. The latter was then improved to $d\geq 8$ in~\cite{Fritzsche99morepolytopes}, and to $d\geq 7$ in~\cite{Holt:Hsharpd7}. We also know that Hirsch-sharp polytopes of dimensions two and three exist only when $n\leq 2d$ (see Section~\ref{section:small-dimension}). But the existence of Hirsch-sharp polytopes with many facets in dimensions four to six remains open.
\begin{openproblem}
Are there Hirsch-sharp $4$-polytopes with $13$ or more facets? Are there Hirsch-sharp $5$-polytopes with $13$ or more facets? Are there Hirsch-sharp $6$-polytopes with $16$ or more facets?
\end{openproblem}
We do know that they do not exist in dimension four with $10$, $11$ or $12$ facets, since Goodey (see~\cite{Goodey}), Schuchert (see~\cite{Schuchert}), and Bremner et al. (see~\cite{Bremner:MoreBounds}) proved that $H(10,4)=5$, $H(11,4)=6$ and $H(12,4)=7$. (Recall that $H(n,d)$ is the maximum diameter among $d$-polytopes with $n$ facets.)

Part~(\ref{item:HSn3d3}) of Theorem~\ref{theorem:hirsch-sharp} follows from the iterated application of the next lemma to the Klee-Walkup polytope $Q_4$.

\begin{lemma}[Holt-Klee~\cite{Holt:Many-polytopes}]
\label{lemma:3d-3}
If there are Hirsch-sharp $d$-polytopes with $n > 2d$ facets, then there are also Hirsch-sharp $(d+1)$-polytopes with $n+1$, $n+2$, and $n+3$ facets.
\end{lemma}

\begin{proof}
Let $u$ and $v$ be vertices at distance $n-d$ in a simple $d$-polytope with $n$ facets. Let $F$ be a facet not containing any of them, which exists since $n>2d$. When we wedge on $F$ we get two edges $u_1u_2$ and $v_1v_2$ with the properties that the distance from any $u_i$ to any $v_i$ is again (at least) $d$. We can then truncate one or both of $u_1$ and $v_1$ to obtain one or two more facets in a polytope that is still Hirsch-sharp. See Figure~\ref{figure:sharp-3d-3}.
\end{proof}

\begin{figure}[hbt]
\begin{center}
\includegraphics[scale=0.60]{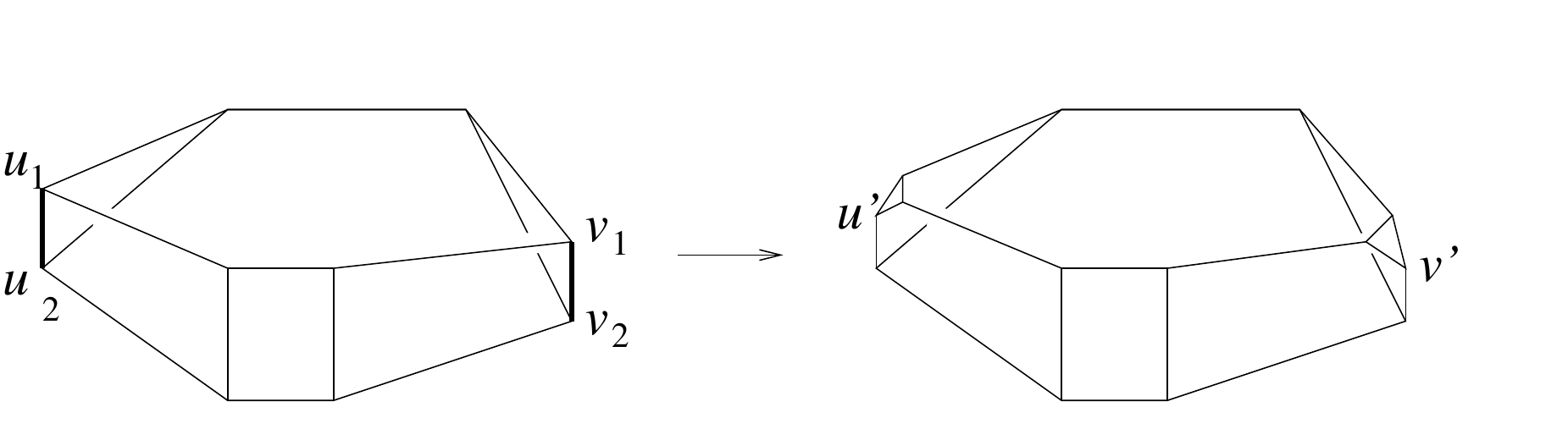}
\caption{After wedging in a Hirsch-sharp polytope we can truncate twice}\label{figure:sharp-3d-3}
\end{center}
\end{figure}

The next construction is more natural in the simplicial framework. So, as a warm-up, we include (see Figure~\ref{figure:sharp-3d-3-ops}) the simplicial version of Figure~\ref{figure:sharp-3d-3}. We already said that the polar of wedging is one-point suspension. The polar of truncation of a vertex is the \defn{stellar subdivision} of a facet by adding to our polytope a new vertex very close to that facet.
\begin{remark}
The stellar subdivision is also known as \defn{stacking}. See, e.g., page 621 in~\cite{MillerReinerSturmfels:GeomCombParkCity}.
\end{remark}

\begin{figure}[hbt]
\begin{center}
\includegraphics[scale=0.8]{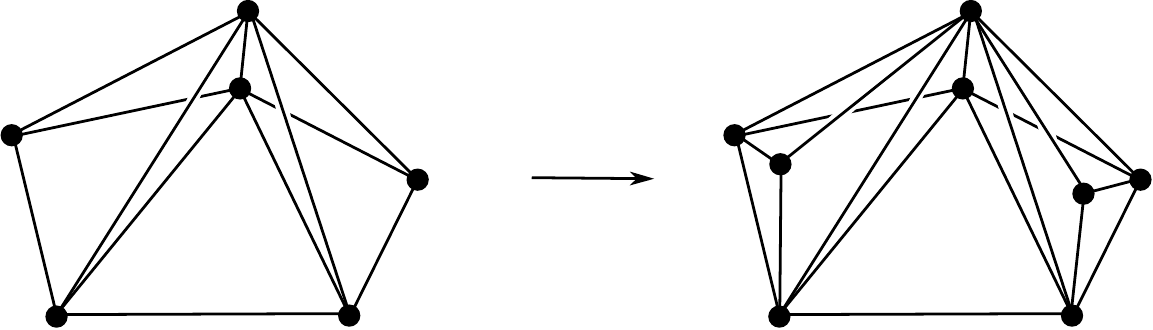}
\caption{The simplicial version of Figure~\ref{figure:sharp-3d-3}. Wedging becomes one-point suspension and truncation is stellar subdivision}\label{figure:sharp-3d-3-ops}
\end{center}
\end{figure}

The key property in the proof of Lemma~\ref{lemma:3d-3} is that the wedge and one-point suspension operations do not only preserve Hirsch-sharpness; they also increase the number of vertices or facets (respectively) that are at Hirsch distance from one another. This suggests looking at what happens when we iterate the process.
\begin{lemma}[Holt-Klee~\cite{Holt:Many-polytopes}]
\label{lemma:ops-iterated}
Let $P$ be a simplicial $d$-polytope with more than $2d$ vertices. Let $A$ and $B$ be two facets of it at distance $\delta$ in the polar graph and let $w$ be a vertex not contained in any of them. Let $P^{(k)}$ be the $k$th one-point suspension of $P$ on the vertex $w$. 

Then, $P^{(k)}$ has two $(k+1)$-tuples of facets $\{A_1,\dots,A_{k+1}\}$ and $\{B_1,\dots,B_{k+1}\}$ with every $A_i$ at distance from $\delta$ every $B_i$. All the facets in each tuple are adjacent to one another.
\end{lemma}

\begin{proof}
We use the following formula, from Section~\ref{section:wedging}, for the iterated one-point suspension of the simplicial complex $L=\partial P$:
\[
\ops_w(L)^{(k)}:= (\ast_{L}(w) * \partial \Delta_k) \cup (\lk_{L}(w)* {\Delta_k}).
\]
Here $\Delta_k$ is a $k$-simplex. The two groups of facets in the statement are $A * \partial \Delta_k$ and  $B * \partial \Delta_k$. The details are left to the interested reader.
\end{proof}

This is the basis for the following result of Fritzsche and Holt in~\cite{Fritzsche99morepolytopes}. From this, we can get the same for every $d \geq 8$ via wedging. The result has been improved to dimension $d=7$ by Holt (see~\cite{Holt:Hsharpd7}) with a generalization of these same arguments, but we skip this part since it is a bit more technical. That proves part 3 of Theorem~\ref{theorem:hirsch-sharp}.

\begin{corollary}[Fritzsche-Holt~\cite{Fritzsche99morepolytopes}]\label{corollary:glue}
There are Hirsch-sharp $8$-polytopes with any number of facets.
\end{corollary}

\begin{proof}
We look at a new operation on polytopes. We call it gluing and it is simply a combinatorial/geometric version of the \emph{connected sum} of topological manifolds (see Chapter 3 of~\cite{Hatcher:AlgebraicTopology}). Let $P_1$ and $P_2$ be two simplicial $d$-polytopes and let $F_1$ and $F_2$ be respective facets. The manifolds are $\partial P_1$ and $\partial P_2$ (two $(d-1)$-spheres); from them we remove  the interiors of $F_1$ and $F_2$ after which we glue their boundaries. See Figure~\ref{figure:gluing}, where the operation is performed on two facets of the same polytope. On the top part we glue the polytopes ``as they come,'' which does not preserve convexity. But if projective transformations are made on $P_1$ and $P_2$ that send points that are close to  $F_1$ and $F_2$ to infinity, then the gluing  preserves convexity, so it yields a polytope that we denote $P_1 \# P_2$. This is shown on the bottom part of Figure~\ref{figure:gluing}.

\begin{figure}[hbt]
\begin{center}
\includegraphics[scale=0.60]{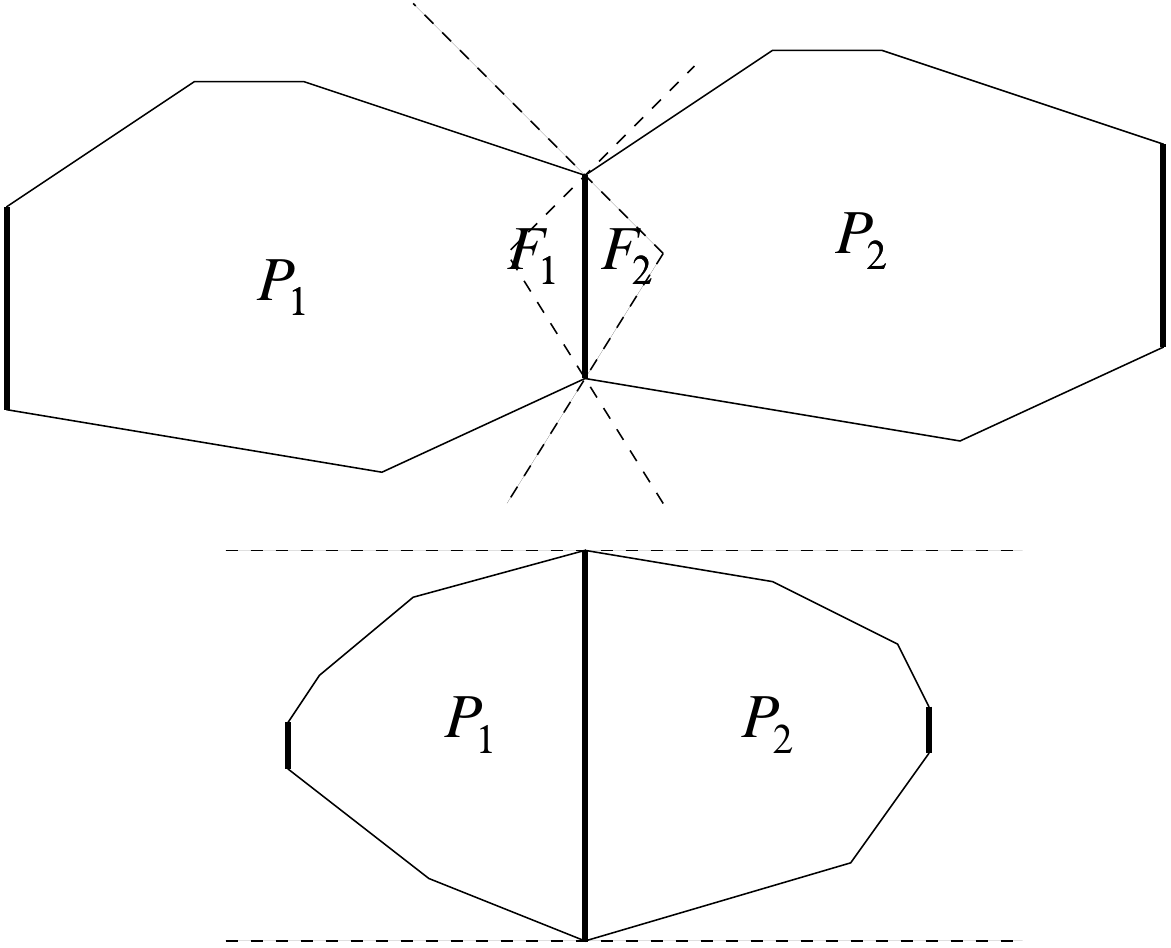}
\caption{Gluing two simplicial polytopes along one facet. In the version of the bottom, a projective transformation is done to $P_1$ and $P_2$ before gluing, to guarantee convexity of the outcome}\label{figure:gluing}
\end{center}
\end{figure}

Gluing \emph{almost} adds the diameters of the two original polytopes. Suppose that the facets $F_1$ and $F_2$ are at distances $\delta_1$ and $\delta_2$ to certain facets $F'_1$ and $F'_2$ of $P_1$ and $P_2$. Then, to go from $F'_1$ to $F'_2$ in $P_1 \# P_2$ we need at least $(\delta_1 -1) + 1 + (\delta_2-1) = \delta_1 + \delta_2 - 1$ steps.

But we can do better if we combine gluing with the iterated one-point suspension.
Consider the simplicial Klee-Walkup $4$-polytope $Q_4^\Delta$ described in Section~\ref{section:Q4} and let $A$ and $B$ two facets of it at distance five. Let $P'$ be the fourth one-point suspension of it on the vertex $w$ not contained in $A\cup B$. Observe that the polytope $P'$ has 13 vertices and dimension eight. By the lemma, $P'$ has two groups of five facets, $\{A_1,\dots,A_{5}\}$ and $\{B_1,\dots,B_{5}\}$, with every $A_i$ at Hirsch distance from every $B_i$ and all the facets in each group adjacent to one another. 

We now glue several copies of $P'$ to one another, a $B_i$ from each copy glued to an $A_i$ of the next one.  Each gluing adds five vertices and, in principle, four to the diameter. But Lemma~\ref{lemma:ops-iterated} implies the following nice property for $P'$: half of the eight facets adjacent to each $A_i$ are at distance four to half of the facets adjacent to each $B_i$. Using the language of \given{Kerstin }Fritzsche, \given{Fred }Holt, and \given{Victor }Klee (in~\cite{Fritzsche99morepolytopes},~\cite{Holt:Many-polytopes}, and~\cite{Holt:Hsharpd7}), we call those facets the \defn{slow neighbors} of each $A_i$ or $B_i$, and call the others \defn{fast}. Since half of the total neighbors are slow, we can make all gluings so that every fast neighbor is glued to a slow one and vice-versa. This increases the diameter by one at every gluing, and the result is Hirsch-sharp.

The above construction yields Hirsch-sharp polytopes of dimension eight with $n = 13 + 5k$ vertices, for every $k \geq 0$. We can get the intermediate values of $n$ too, via truncation. By Lemma~\ref{lemma:3d-3}, every time we do a one-point suspension on a Hirsch-sharp simplicial polytope we can increase the number of facets by one or two via a stellar subdivision at each end. Since the polytope $P'$ we are gluing is a four-fold one-point suspension, and since there are two ends that remain unglued (the $A$-face of the first copy and the $B$-face of the last) we can do up to eight stellar subdivisions to it and still preserve Hirsch-sharpness.
\end{proof}

\begin{remark}
\given{Fred }Holt, \given{Victor }Klee, and \given{Kerstin }Fritzsche prove these results via simple polytopes, rather than simplicial. The analogue of ``gluing along a simplicial facet'' is called ``blending at a simple vertex'' in~\cite{Fritzsche99morepolytopes},~\cite{Holt:Many-polytopes}, and~\cite{Holt:Hsharpd7}. The trick in the proof is called ``fast-slow blending.''
\end{remark}

\section{Positive results}\label{section:positive-results}

In this section, we give evidence for the Hirsch Conjecture, or at least its polynomial version, by presenting special cases for which it holds, three upper bounds for diameters of polytopes (one after perturbation), and its equivalence to two other seemingly natural conjectures. We also summarize some results on a continuous analogue of the conjecture.

\subsection{The $d$-step and Non-revisiting Conjectures}\label{section:equivalences}

The equivalence of the Hirsch Conjecture to other two conjectures that we discuss in this section shows that the conjectured bound of $n-d$ is a natural bound to consider.

We start with the $d$-step Conjecture. This is simply the Hirsch Conjecture restricted to polytopes whose number of facets is twice their dimension:
\begin{conjecture}[The $d$-step Conjecture]\label{conjecture:d-step}
For every $d \geq 2$, $H(2d,d) \leq d$.
\end{conjecture}

At first, it is seemingly just a special case of the Hirsch Conjecture. Its equivalence to the full Hirsch Conjecture follows from:
\begin{theorem}[Klee-Walkup~\cite{Klee:d-step}]
\label{theorem:dstep-hirsch}
Let $k$ be fixed. Then,
\begin{equation*}
\max_{d \geq 2} H(k+d,d) = H(2k,k).
\end{equation*}
In particular, the Hirsch Conjecture (Conjecture~\ref{conjecture:hirsch}) holds if and only if the $d$-step Conjecture (Conjecture~\ref{conjecture:d-step}) holds.
\end{theorem}
\begin{proof}
We are going to show that:
\begin{equation}\label{equation:inequalities2}
\cdots \leq H(2k+2,k+2) \leq H(2k+1,k+1) \leq H(2k,k)
\end{equation}
and
\begin{equation}\label{equation:inequalities}
\cdots \leq H(2k-2,k-2) \leq H(2k-1,k-1) \leq H(2k,k) \leq H(2k+1,k+1) \leq \cdots
\end{equation}
In \eqref{equation:inequalities2}, we are looking at polytopes with $n \leq 2d$. In \eqref{equation:inequalities}, $n$ and $d$ are arbitrary.

To prove \eqref{equation:inequalities2}, let $P$ be a polytope with $n < 2d$ and let $u$ and $v$ be vertices of it.
Then $u$ and $v$ have some facet $F$ in common, since each vertex is incident to at least $d$ facets. The facet $F$ has dimension $d-1$, and each facet of it is the intersection of $F$ with another facet of $P$. Hence, $F$ has at most $n-1$ facets.
Since every path on $F$ is a path on $P$, we get \eqref{equation:inequalities2}.

For \eqref{equation:inequalities}, we utilize the wedge operation. As we said in Lemma~\ref{lemma:wedge} this operation applied to $P$ increases the dimension and number of facets by one, and it gives a polytope with the same or bigger diameter.
\end{proof}

One can interpret the Hirsch Conjecture as saying that if one wishes to go from vertex $u$ to vertex $v$ of a polytope $P$, one does not expect to have to enter and leave the same facet several times.
This suggests the following conjecture:
\begin{conjecture}[The Non-revisiting Conjecture]\label{conjecture:nonrevisiting}
Let $P$ be a simple polytope. Let $u$ and $v$ be two arbitrary vertices of $P$. Then, there is a path from $u$  to $v$ which at every step enters a different facet of $P$.
\end{conjecture}
The Non-revisiting Conjecture asserts that for every two vertices $u$ and $v$ of a polytope $P$, there is a path in the graph of $P$ that never revisits a facet that it has previously abandoned. Paths with the conjectured property are called \defn{non-revisiting paths}. These paths are also called \defn{$W_v$-paths} and Conjecture~\ref{conjecture:nonrevisiting} is also known as the $W_v$ Conjecture. The length of non-revisiting paths is bounded by $n-d$: at each step we must enter a different facet, and the $d$ facets that our initial vertex lies in cannot be among them. Klee and Walkup (see~\cite{Klee:d-step}) proved the Non-revisiting Conjecture is equivalent to the Hirsch Conjecture. We outline the proof, beginning with the following lemma:
\begin{lemma}
\label{lemma:nonrevisiting}
Let $n$ denote the number of facets of a polytope $P$ and let $d$ denote its dimension. The following properties are equivalent:
\begin{enumerate}
\item Every simple polytope $P$ has the non-revisiting property.
\item Every simple polytope $P$ with $n=2d$ has the non-revisiting property.
\end{enumerate}
\end{lemma}
\begin{proof}
One direction is obvious. For the other, let $P$ be a polytope with $n\ne 2d$ and suppose it does not have the non-revisiting property. That is, there are vertices $u$ and $v$ such that every path from $u$ to $v$ revisits some facet that it previously abandons.
We will construct another polytope $P'$ without the non-revisiting property and with:
\begin{itemize}
\item One less facet and dimension than $P$ if $n < 2d$, and
\item One more facet and dimension than $P$ if $n > 2d$.
\end{itemize}

\begin{figure}[hbt]
  \begin{center}
    \includegraphics[scale=0.83]{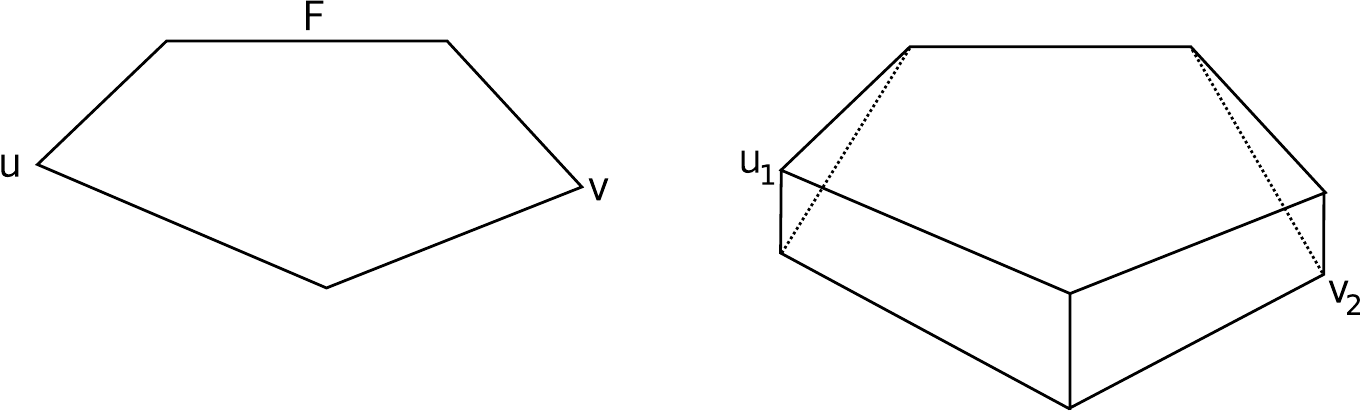}
    \caption{The pentagon $P$ and the wedge $P'=\wed_F(P)$ over its facet $F$: the upper pentagonal facet of $P'$ is $F_1$ and the lower pentagonal facet is $F_2$.}
       \label{figure:wedge-when-n-larger-2d}
  \end{center}
\end{figure}

In the first case, $u$ and $v$ lie in a common facet $F$ and we simply let $P'=F$. 

In the second case, let $F$ be a facet not containing $u$ nor $v$ and let $P'=\wed_F(P)$ be the wedge over $F$. Let $F_1$ and $F_2$ be the two facets of $P'$ whose intersection projects to $F$ (see Figure~\ref{figure:wedge-when-n-larger-2d}). Let $u_1$ and $v_2$ be the vertices of $P'$ that project to $u$ and $v$ and lie, respectively, on $F_1$ and $F_2$. Now, consider a path from $u_1$ to $v_2$ on $P'$ and project it to a path from $u$ to $v$ on $P$:
\begin{itemize}
\item If the path on $P$ revisits a facet (call it $G$) other than $F$, then the path on $P'$ revisits the facet that projects to the facet $G$.

\item If the path on $P$ revisits $F$, then the path on $P'$ revisits either $F_1$ or $F_2$.
\end{itemize}
In either case, the polytope $P'$ does not have the non-revisiting property.
\end{proof}

\begin{remark}
\label{remark:non-hirsch-wedge}
In the proof of Lemma~\ref{lemma:wedge} we noted that, applied to a Hirsch-sharp polytope $P$ the wedge operator $W_F$ produced either another Hirsch-sharp polytope or a counter-example to the Hirsch Conjecture. The last proof shows that the latter can happen only if $P$ does not have the non-revisiting property.
\end{remark}

\begin{theorem}[Klee-Walkup~\cite{Klee:d-step}]
\label{theorem:dstep-nonrevisiting}
The Non-revisiting Conjecture (Conjecture~\ref{conjecture:nonrevisiting}) holds if and only if the $d$-step Conjecture (Conjecture~\ref{conjecture:d-step}) holds.
\end{theorem}
\begin{proof}
Clearly, if $P$ has the non-revisiting property, then it satisfies the $d$-step Conjecture.
To prove the converse, we assume without loss of generality that $P$ is a simple $d$-polytope with $n$ facets, where $n=2d$. We also assume, by induction, that all polytopes with number of facets minus dimension smaller than $d$ have the non-revisiting property.

Let $u$ and $v$ be two vertices of $P$. We argue by induction on the number of common facets of $u$ and $v$. In the base case where $u$ and $v$ do not share any facet, any path of length $n-d=d$ is non-revisiting. For the inductive step, assume that $u$ and $v$ are vertices of a facet $F$ of $P$:
\begin{itemize}
\item If the facet $F$ has less than $n-1$ facets itself, then $F$ has the non-revisiting property by induction on ``number of facets minus dimension,'' and we are done.

\item If $F$ has $n-1$ facets, since it has dimension $d-1$ there is a facet $F'$ of $F$ not containing $u$ nor $v$. Let $P'=\wed_{F'}(F)$. Let $u_1$ and $v_2$ be vertices of $P'$ projecting to vertices $u$ and $v$ of $P$ such that $F_1$ contains $u_1$ and $F_2$ contains $v_2$. As in the proof of Lemma~\ref{lemma:nonrevisiting}, $F_1$ and $F_2$ denote the non-vertical facets of the wedge $P'$. Then $P'$ has dimension $d$ and $2d$ facets, but $u'$ and $v'$ have one less facet in common than $u$ and $v$ had. By the inductive hypothesis, there is a non-revisiting path from $u'$ to $v'$ in the polytope $P'$. When this path is projected to $F$, it retains the non-revisiting property, and it is also a non-revisiting path on the original polytope $P$.
\end{itemize}
Thus, a $d$-dimensional polytope with $2d$ facets satisfying the $d$-step Conjecture is a polytope that has the non-revisiting property.
\end{proof}

\subsection{Small dimension or few facets}\label{section:small-dimension}

Recall that $H(n,d)$ denotes the maximum diameter among the graphs of $d$-dimensional polytopes with $n$ facets. An exact formula for $H(n,d)$ when $d = 3$ was first proved by \given{Victor }Klee in~\cite{Klee:PathsII}.
\begin{theorem}[Klee~\cite{Klee:PathsII}]\label{theorem:hirschford3}
$H(n,3) = \lfloor \frac{2n}{3} \rfloor - 1$.
\end{theorem}
\begin{proof}
We first show that $H(n,3) \leq \lfloor \frac{2n}{3} \rfloor - 1$.
Let $P$ be a simple $3$-polytope with $n$ facets. By Euler's formula, $P$ has $2n-4$ vertices. Let $v$ and $v'$ be two distinct vertices of $P$. By Balinski's Theorem (see Theorem~\ref{theorem:balinski}, which \given{Michel }Balinski proved in~\cite{Balinski:On-the-graph-structure}), the graphs of $3$-polytopes are $3$-connected. Thus, there are three disjoint paths that go from $v$ to $v'$. These paths use at most $2n-6$ intermediate vertices. Therefore, the shortest of these paths goes through at most $\lfloor \frac{2n}{3}\rfloor-2$ vertices, so it has at most $\lfloor \frac{2n}{3} \rfloor - 1$ edges.

\begin{figure}[hbt]
  \begin{center}
    \includegraphics[scale=0.70]{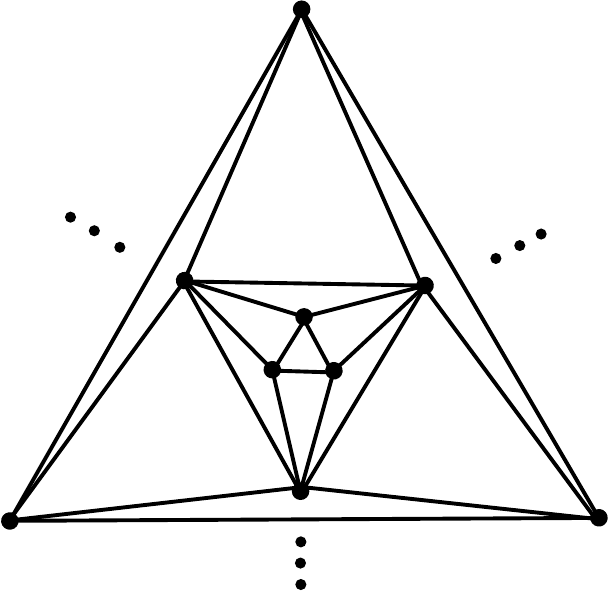}
    \caption{A family of simplicial $3$-dimensional polytopes with $3+3k$ vertices and diameter $2k+1$. Two steps are needed to cross each layer of six triangles.}
       \label{figure:d3lowerbound}
  \end{center}
\end{figure}

To prove $H(n,3) \geq \lfloor \frac{2n}{3} \rfloor - 1$, we first consider the case where $n$ is a multiple of three. We construct a family of simplicial $3$-polytopes with $3+3k$ vertices 
with diameter $2k+1$ for every $k \geq 0$. The diagram of Figure~\ref{figure:d3lowerbound} shows the case $k=2$. Starting from the triangle depicted in the center, one needs a minimum of three steps to reach a facet outside of the next equilateral triangle, and two more for each extra layer of three vertices and six triangles.
\begin{figure}[hbt]
  \begin{center}
    \includegraphics[scale=0.60]{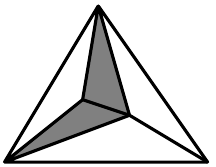}
    \caption{Modification needed to the inner triangle of Figure~\ref{figure:d3lowerbound} for $5+3k$ vertices}
       \label{figure:d3lowerboundmod2}
  \end{center}
\end{figure}

The case of $4+3k$ (where the bound does not increase) is obtained by any stellar subdivision of the picture in the previous case. For $5+3k$ we subdivide the inner triangle of Figure~\ref{figure:d3lowerbound} with two more vertices, as depicted in Figure~\ref{figure:d3lowerboundmod2}, which makes one more step necessary to exit that triangle from any of the two shaded triangles.
\end{proof}

The Hirsch Conjecture is also known to hold when $P$ does not have too many facets. By Theorem~\ref{theorem:dstep-hirsch}, if we fix $k$, the Hirsch Conjecture for $d$-polytopes with $d+k$ facets follows from the conjecture for $k$-polytopes with $2k$ facets. Hence, the case $n - d \leq 3$ follows from Theorem~\ref{theorem:hirschford3}. The cases when $n - d \leq 5$ were proved by \given{Victor }Klee and \given{David }Walkup in~\cite{Klee:d-step}. The case $n - d = 6$ was recently proved by \given{David }Bremner and \given{Lars }Schewe in~\cite{Bremner:DiameterFewFacets}.
\begin{theorem}[\cite{Klee:d-step},~\cite{Bremner:DiameterFewFacets}]
$H(n,d) \leq n-d$ when $n-d \leq 6$.
\end{theorem}

It is worth noting that this and Theorem~\ref{theorem:hirschford3} exhaust all the pairs $(n,d)$ where the Hirsch Conjecture is known to hold. In particular, the Hirsch Conjecture is open for $4$-polytopes with $13$ facets.

In summary, the pairs $(n,d)$ for which the Hirsch Conjecture is known to be true is:
\begin{enumerate}
\item $d\leq 3$, and $n$ arbitrary;

\item $n\leq d+6$ and $d$ arbitrary. This follows from Theorem~\ref{theorem:dstep-hirsch} by showing $H(8,4)=4$, $H(10,5)=5$ and $H(12,6)=6$. The first two were proved by \given{Victor }Klee and \given{David }Walkup in~\cite{Klee:d-step}. The third was proved by \given{David }Bremner and \given{Lars }Schewe in~\cite{Bremner:DiameterFewFacets}.

\item $H(11,4)=6$ was proved by \given{Peter }Schuchert (see~\cite{Schuchert}). Following up on~\cite{Bremner:DiameterFewFacets}, Bremner, Deza, Hua, and Schewe (see~\cite{Bremner:MoreBounds}) proved $H(12,4)=H(12,5)=H(13,6)=7$.

\end{enumerate}

Combining this with the information on Hirsch-sharp polytopes we can give a ``plot'' of the function $H(n,d) - (n-d)$ in Table~\ref{table:hirsch-tight-table}. The horizontal coordinate is $n-2d$, so that the column marked ``0'' corresponds to the polytopes relevant to the $d$-step Conjecture.
\begin{table}[htb]
\begin{center}
\begin{tabular}{|c|cccccccccc|}
\hline
$n -2d$ & -1&0&1&2&3&4&5&6&7& $\cdots$\\
\hline
$\underline{\quad d\quad}$ &&&&&&&&&&\\
2  &=&$=$ & $<$ & $<$ & $<$ & $<$ & $<$ & $<$ & $<$  & $\cdots $ \\
3  &=&$=$& $<$ & $<$ & $<$  & $<$ & $<$ & $<$ & $<$ & $\cdots $ \\
4  &=&$=$& $=$ & $<$ & $<$ & $<$ & ? & ? & ? & $\cdots $   \\
5  &=&$=$ &$=$&$=$&?&?&?&?&?&$\cdots $\\
6  &=&$=$ &$=$&$\ge$&$\ge$&?&?&?&?&$\cdots $\\
7  &=& $\ge$ & $\ge$ & $\ge$ & $\ge$ & $\ge$ & $\ge$  & $\ge$ & $\ge$ & $\cdots $   \\
8  &$\ge$& $\ge$ & $\ge$ & $\ge$ & $\ge$ & $\ge$ & $\ge$  & $\ge$ & $\ge$ & $\cdots $   \\
$\vdots$   &$\vdots$&$\vdots$&$\vdots$&$\vdots$&$\vdots$&$\vdots$&\ \ $\vdots$\ \ &$\vdots$\ \ &$\vdots$& $\ddots$\\
\hline
\end{tabular}
\caption{$H(n,d)$ versus $n-d$, the state of the art}
\end{center}
\end{table}\label{table:hirsch-tight-table}

\subsection{General upper bounds on diameters}\label{section:upper-bounds}

The best upper bound known for the diameter of arbitrary polytopes is given in~\cite{Kalai:Quasi-polynomial} by Kalai and Kleitman.  This subexponential bound holds even for unbounded polyhedra. The proof is so short (it is only half a page!) that we reproduce it here:
\begin{theorem}[Kalai-Kleitman~\cite{Kalai:Quasi-polynomial}]
\label{theorem:quasipolynomial}
Every polyhedron of dimension $d$ and with $n$ facets has diameter bounded above by $n^{\log_2(d)+2}= n^2 d^{\log_2n}$.
\end{theorem}

\begin{proof}
Let $P$ be a $d$-dimensional polyhedron with $n$ facets, and let
$v$ and $u$ be two vertices of $P$. Let $k_v$ (respectively $k_u$) be the maximal positive number such that the union of all vertices in all paths in $G(P)$ starting from $v$ (respectively $u$) of length at most $k_v$ (respectively $k_u$) are incident to at most $\frac{n}{2}$ facets. Clearly, there is a facet $F$ of $P$ so that we can reach $F$ by a path of length $k_v+1$ from $v$ and a path of length $k_u+1$ from $u$.

Recall that $H_u(n,d)$ denotes the maximum diameter among all $d$-dimensional polyhedra with $n$ facets. We claim that $k_v\leq H_u(\lfloor \frac{n}{2} \rfloor,d)$. (Of course, this implies that the same bound holds for $k_u$.)

To prove this, let $Q$ be the polyhedron defined by taking only the inequalities of $P$ corresponding to facets that can be reached from $v$ by a path of length at most $k_v$. By construction, all vertices of $P$ at distance at most $k_v$ from $v$ are also vertices in $Q$, and vice-versa. In particular, if $w$ is a vertex of $P$ whose distance from $v$ is $k_v$ then its distance from $v$ in $Q$ is also $k_v$. Since $Q$ has at most $n/2$ facets, we get $k_v\leq H_u(\lfloor \frac{n}{2} \rfloor,d)$.

The claim implies the following recursive formula for $H_u$:
\begin{eqnarray*}
H_u(n,d)
&\leq & 2 H_u\left(\left\lfloor \frac{n}{2} \right\rfloor,d\right)  + H_u(n,d-1) + 2,
\end{eqnarray*}
which we can rewrite as
\[
\frac{H_u(n,d)+1}{n}
\leq  \frac{H_u\left(\left\lfloor \frac{n}{2} \right\rfloor,d\right)+1}{n/2}  + \frac{H_u(n,d-1)+1}{n}.
\]

This suggests calling $h(k,d):=(H(2^k,d) - 1)/ 2^k$ and applying the recursion with $n=2^k$, to get:
\[
h(k,d)\leq h(k-1,d) + h(k,d-1).
\]
This implies $h(k,d)\leq \binom{k+d}{d}$, or
\[
H_u(2^k,d) \leq 2^k\binom{k + d}{d}.
\]
From this the statement follows if we assume $n\leq 2^d$ (that is, $k\leq d$). For $n\geq 2^d$ we use Larman's bound 
$H_u(n,d)\leq n 2^d \leq n^2$, proved below.
\end{proof}

The diameters of polytopes of a fixed dimension $d$ admit a linear bound. An upper bound of $n3^{d-3}$ was found by Barnette in 1967 (see~\cite{Barnette}). Larman (see~\cite{Larman}) then improved the bound to $n 2^{d-3}$. The proof presented here is taken from a very recent paper (see~\cite{Eisenbrand:Diameter-of-Polyhedra}) by Eisenbrand, H\"ahnle, Razborov, and Rothvo\ss{}. As in the previous result, we prove it for all polyhedra, not just bounded polytopes.

\begin{theorem}[Larman~\cite{Larman}]
\label{theorem:linear-in-fixed-d}
For every $n>d\geq 3$, the maximal diameter $H_u(n,d)$ of $d$-dimensional polyhedra with $n$ facets is no more than $n 2^{d-3}$.
\end{theorem}
\begin{proof}
The proof is by induction on $d$. The base case $d=3$ was Theorem~\ref{theorem:hirschford3}.

Let $u$ be an initial vertex of our polytope $P$, of dimension $d>3$. For each other vertex $v\in \vert(P)$ we consider its distance $\dist(u_1,v)$, and use it to construct a sequence of facets $F_1,\dots,F_k$ of $P$ as follows:

\begin{itemize}
\item Let $F_1$ be a facet that reaches ``farthest from $u$'' among those containing $u$. That is, 
let $\delta_1$ be the maximum distance to $u$ of a vertex  sharing a facet with $u$, and let $F_1$ be that facet. 

\item Let $\delta_2$ be the maximum distance to $u$ of a vertex  sharing a facet with some vertex at distance $\delta_1+1$ from $u$, and let $F_2$ be that facet. 

\item Similarly, while there are vertices at distance $\delta_i + 1$ from $u$,  let $\delta_{i+1}$ be the maximum distance to $u$ of a vertex  sharing a facet with some vertex  at distance $\delta_i+1$ from $u$, and let $F_{i+1}$ be that facet.
\end{itemize}

We now stratify the vertices of $P$ according to the distances $\delta_1,\delta_2,\dots,\delta_k$ so obtained. 
Observe that $\delta_k$ is the diameter of $P$. By convention, we let $\delta_0=-1$:
\[
V_i:=\{v\in \vert(P) : \dist(u,v) \in (\delta_{i-1}, \delta_{i}]\}.
\]
We call a facet $F$ of $P$ \defn{active in $V_i$} if it contains a vertex of $V_i$. The crucial property that our stratification has is that no facet of $P$ is active in more than two $V_i$'s. Indeed, each facet is active only in $V_i$'s with consecutive values of $i$, but a facet intersecting $V_i$, $V_{i+1}$ and $V_{i+2}$ would contradict the choice of the facet $F_{i+1}$. In particular, if $n_i$ denotes the number of facets active in $V_i$ we have 
\[
\sum_{i=1}^k n_i \leq 2n.
\]

Since each $F_i$ has vertices with distances to $u$ ranging from at least $\delta_{i-1}+1$ to $\delta_i$, 
we have that $\diam(F_i)\geq \delta_{i} - \delta_{i-1} -1$. Even more,
let $Q_i$, $i=1,\dots,k$ be the polyhedron obtained by removing from the facet-definition of $F_i$ the equations of facets of $P$ that are not active in $V_i$ (which may exist since $F_i$ may have vertices in $V_{i-1}$). By an argument similar to the one used for the polyhedron $Q$ of the previous proof, $Q_i$ has still diameter at least $\delta_{i} - \delta_{i-1} -1$. But, by inductive hypothesis, we also have that the diameter of $Q_i$ is at most  $2^{d-4} (n_i-1)$, since it has dimension $d-1$ and at most $n_i-1$ facets.
Putting all this together we get the following bound for the diameter  $\delta_k$ of $P$:

\begin{eqnarray*}
\delta_k &=& \sum_{i=1}^k (\delta_i - \delta_{i-1} -1)  +(k-1) \\
&< & \sum_{i=1}^k 2^{d-4} (n_i - 1) + k \\
&= & 2^{d-4} \sum_i n_i  - k (2^{d-4} -1)
\leq 2^{d-3} n.
\end{eqnarray*}

\end{proof}

\begin{remark}
\label{remark:limitsofabstraction}
It has been pointed out repeatedly that the proofs of Theorems~\ref{theorem:quasipolynomial} and~\ref{theorem:linear-in-fixed-d} use only very limited properties of graphs of polytopes. For example, \given{Victor }Klee and \given{Peter }Kleinschmidt (see Section 7.7 in~\cite{Klee:The-d-step-conjecture}) show that Theorem~\ref{theorem:linear-in-fixed-d} holds for the ridge-graphs of all pure simplicial complexes, and more general objects. Recently, Eisenbrand et al.~\cite{Eisenbrand:Diameter-of-Polyhedra} 
have shown the following similar result:
\end{remark}

\begin{theorem}[Eisenbrand et al.~\cite{Eisenbrand:Diameter-of-Polyhedra}]
\label{theorem:eisenbrand}
Let $G$ be a graph whose vertices are certain  subsets of size $d$ of the $n$-element set $\{1,\dots,n\}$. Assume that for every pair of vertices $u$ and $v$ in $G$ there is a path from $u$ to $v$ using only vertices that contain $u\cap v$.  Then:
\begin{enumerate}
\item $\diam(G)\leq n^{1+\log d}$.
\item $\diam(G)\leq n 2^{d-1}$.
\end{enumerate}
\end{theorem}
\begin{remark}
Kalai remarks that the first bound $n^{1+\log d}$ holds for unbounded polyhedra.
\end{remark}
The novelty in~\cite{Eisenbrand:Diameter-of-Polyhedra} is that there are graphs with the hypotheses of Theorem~\ref{theorem:eisenbrand} and with $\diam(G)\geq c n^{3/2}$, for a certain constant $c$. It is not clear whether this is support against the Hirsch Conjecture or it simply indicates that the arguments in the proofs of Theorems~\ref{theorem:quasipolynomial} and~\ref{theorem:linear-in-fixed-d} do not take advantage of  additional properties that graphs of polytopes have and which may prevent their diameters from growing. For example, in Eisenbrand et al.'s setting for $d=1$, any connected graph with vertex set $\{1,\dots,n\}$ is valid. (Slightly different formulations for graphs of abstract polytopes are studied in, e.g.,~\cite{Adler:Existence-of-A-avoiding},~\cite{Adler:Maximum-diameter}, and~\cite{Lawrence-Jr.:Abstract-polytopes}.)

Since the Hirsch Conjecture is strongly motivated by the simplex algorithm of linear programming, it is natural to ask questions about the number of iterations needed under particular pivot rules. Kalai, and independently, Matou\v{s}ek, Sharir and Welzl (see~\cite{Kalai:A-subexponential-randomized} and~\cite{Matousek:A-subexponential-bound}, respectively) proved the existence of randomized pivot rules for the simplex method with subexponential running time.
\begin{theorem}[Kalai~\cite{Kalai:A-subexponential-randomized}; Matou\v{s}ek, Sharir, and Welzl~\cite{Matousek:A-subexponential-bound}]
\label{theorem:randomizedpivot}
There exist randomized algorithms where the expected number of arithmetic operations needed in the worst case by a linear programming problem with $d$ variables and $n$ inequalities is at most $\exp(K\sqrt{d \log n})$, where $K$ is a fixed constant.
\end{theorem}

Also, though polynomial bounds on the diameter of all polytopes are unknown, there are results saying that random polytopes have polynomial diameter. One way to formalize this is by considering perturbations of the facet-defining inequalities of the given polytope. Specifically, let $P$ be the feasibility polyhedron
\[
P = \{x \in \R^d \mid \langle a_i, x\rangle \leq b, (i=1,\ldots, n)\}
\]
of a certain linear program. Then $P$ has dimension $d$ and (at most) $n$ facets. If we replace the vectors $a_i \in \R^d$ and $b \in \R^n$ with independent Gaussian random vectors with means $\mu_i=a_i$ and $\mu=b$ (respectively), and standard deviations $\sigma(\max_i \|(\mu_i,\mu)\|)$ we say that we have \defn{perturbed randomly within a parameter $\sigma$}.
In~\cite{Spielman:WhySimplexUsually}, Spielman and Teng proved that the expected diameter of a linear program that is perturbed within a parameter $\sigma$ is polynomial in $d$, $n$, and $\sigma^{-1}$. 
In~\cite{Vershynin:BeyondHirsch}, \given{Roman }Vershynin improved the bound to be polylogarithmic in the number $n$ of inequality constraints.
\begin{theorem}[Vershynin~\cite{Vershynin:BeyondHirsch}]
\label{theorem:vershynin}
If the original linear program is perturbed randomly within a parameter $\sigma$, the expected diameter is $O(d^9 \sigma^{-4} \log^4 d \, \log^2 n)$.
\end{theorem}

It is worth noticing that this result is not only structural. Its proof  shows that the simplex method can find a path of that length in the perturbed polyhedron.

\subsection{A continuous Hirsch Conjecture}\label{section:continuous}

Here we summarize some recent work of Deza, Terlaky, and Zinchenko (see~\cite{Deza:Central-path},~\cite{Deza:The-continuous-d-step}, and~\cite{Deza:Curvature}) in which they propose interesting continuous analogues of the Hirsch Conjecture and the $d$-step Conjecture. The analogy comes from looking at the central path method for linear and convex optimization. As in the simplex method, the idea is to move from a feasible point to another feasible point on which the given objective linear functional is improved. In contrast to the simplex method, where the path travels from vertex to neighboring vertex along the graph of the feasibility polyhedron $P$, this method follows a certain curve through the strict  interior of the polytope.

Inspired by work of Dedieu, Malajovich, and Shub (see~\cite{Dedieu}), Deza et al. consider the total curvature of a central path on a polytope to be an analogue of the graph-diameter, and show lower bounds that can be interpreted as ``continuous Hirsch-sharpness,'' as well as a result analogue to the equivalence between the Hirsch Conjecture and the $d$-step Conjecture.
More precisely, to each linear program,
\[
\text{ Minimize } \xi(x), \text{ subject to } Bx = a \text{ and } x \geq 0,
\]
the method associates a \emph{(primal) central path}  $\gamma_\xi : [0, \beta) \rightarrow \R^d$ which is an analytic curve through the interior of the feasible region and such that $\gamma_\xi(0)$ is an optimal solution of the problem. The central path is well-defined and unique even if the program has more than one optimal solution, but its definition is implicit, so that there is no direct way of computing $\gamma_\xi(0)$. To get to $\gamma_\xi(0)$, one starts at any feasible solution and tries to follow a curve that approaches more and more the central path, using for it certain barrier functions. (Barrier functions play a role similar to the choice of pivot rule in the simplex method. The standard barrier function is the logarithmic function $f(x) = -\sum_{i=1}^n \ln(A_i x - b_i)$.)  For further description of the method, refer to~\cite{Boyd:ConvexOptimization} and~\cite{Renegar:A-Mathematical-View}.

In practice, one does not follow the curve exactly. Rather, one does Newton-like iteration steps trying not to get too far from the actual curve. How much can one improve in a single step is related to the curvature of the central path, if the path is rather straight one can do long steps without deviating too far from it. Otherwise, one needs to use shorter steps. Thus, the \defn{total curvature} $\lambda_\xi(P)$ of the central path, defined in the usual differential-geometric way, is an important parameter which can be considered a continuous analogue of the diameter of the polytope $P$, or at least of the maximum distance from 
any vertex to a vertex maximizing the functional $\xi$.

Theorem~\ref{theorem:arrangement-curvature} below yields an upper bound that is in $O(n^d)$ for $\Lambda(n,d)$, but it had been conjectured that $\lambda_\xi(P)$ is bounded by a constant for each dimension $d$ and that it grows at most linearly with varying $d$. Deza et al.~have disproved both statements: in~\cite{Deza:Central-path}, they constructed polytopes for which $\lambda_\xi(P)$ grows exponentially with $d$. More strongly, in~\cite{Deza:Curvature} they construct a family of polytopes that show that the total curvature $\lambda_\xi$ cannot be bounded only in terms of $d$. To state this result, let $\Lambda(n,d)$ denote the largest total curvature of the central path over all polytopes $P$ defined by $n$ inequalities in dimension $d$ and over all linear objectives $\xi$. They show:

\begin{theorem}[\cite{Deza:Curvature}]\label{theorem:continuous-lower-bound}
For every fixed dimension $d\geq 2$, $\liminf_{n \rightarrow \infty} \frac{\Lambda(n,d)}{n} \geq \pi$.
\end{theorem}

Deza et al.~consider this result a continuous analogue of the existence of Hirsch-sharp polytopes. Motivated by this they pose the following conjecture:

\begin{conjecture}[Continuous Hirsch Conjecture]
$\Lambda(n,d) \in O(n)$. That is, there is a constant $K$ such that $\Lambda(n,d)\leq Kn$ for all $n$ and $d$.
\end{conjecture}

Theorem~\ref{theorem:continuous-lower-bound} says that if the Continuous Hirsch Conjecture is true, then it is (asymptotically) tight.
Deza et al.~also conjecture a continuous version of the $d$-step Conjecture, and show it to be equivalent to the Continuous Hirsch Conjecture, thus providing an analogue of Theorem~\ref{theorem:dstep-hirsch}:

\begin{conjecture}[Continuous $d$-step Conjecture]
The function $\Lambda(2d,d)$ grows linearly in its input. That is to say, $\Lambda(2d,d)$ is $O(d)$.
\end{conjecture}

\begin{theorem}[\cite{Deza:The-continuous-d-step}]
The Continuous Hirsch Conjecture is equivalent to the Continuous $d$-step Conjecture. That is, if $\Lambda(d,2d) \in O(d)$ for all $d$, then $\Lambda(n,d)\in O(n)$ for all $d$ and $n$.
\end{theorem}

The curvature of the central path had also been studied by 
Dedieu, Malajovich, and Shub in~\cite{Dedieu}. But instead of looking at a single polytope, they consider the average total curvature of the central paths of all bounded cells in a simple arrangement. 

An arrangement $\mathcal{A}$ of $n$ hyperplanes in dimension $d$ is called \defn{simple} if every $n$ hyperplanes intersect at a unique point. (Any sufficiently generic collection of $n \geq d+1$ hyperplanes in $d$-dimensional space is simple.) It is easy to show that any simple arrangement of $n$ hyperplanes in $\R^d$ has exactly $\sigma=\binom{n-1}{d}$ bounded full-dimensional cells. For a simple arrangement $\mathcal{A}$ with bounded cells $P_1,\dots,P_\sigma$
and a given objective function $\xi$, Dedieu et al. define:
\[
\lambda_\xi(\mathcal{A})= \frac{1}{\sigma} \sum_{i=1}^{\sigma} \lambda_c(P_i).
\]
They prove:
\begin{theorem}[\cite{Dedieu}]
\label{theorem:arrangement-curvature}
$\lambda_\xi(\mathcal{A}) \leq 2\pi d$, for every simple arrangement.
\end{theorem}

That is, even if individual cells can give total curvature linear in $n$ (and perhaps worse) by Theorem~\ref{theorem:continuous-lower-bound}, the average over all cells of a given arrangement is bounded by a function of $d$ alone.
By analogy, Deza et al. (see~\cite{Deza:Curvature}) consider the average diameter of the graphs of all bounded cells in a simple arrangement $\mathcal{A}$. Denote it $\delta(\mathcal{A})$ and let $H_\mathcal{A}(n,d)$ be the maximum of $\delta(\mathcal{A})$ over all  simple arrangements defined by $n$ hyperplanes in dimension $d$. They conjecture that:

\begin{conjecture}[Hyperplane Diameter Conjecture]
$H_\mathcal{A}(n,d) \leq d$.
\end{conjecture}

They notice that this inequality, modulo a linear factor, would follow from the Hirsch Conjecture:

\begin{proposition}
If the Hirsch Conjecture holds, then $H_\mathcal{A} (n,d) \leq d + \frac{2d}{n-1}$.
\end{proposition}
\begin{proof}
Let $\{P_i \mid i \in I\}$ be the collection of bounded cells of $\mathcal{A}$. Let $n_i\leq n$ denote the number of facets of the bounded cell $P_i$. If the Hirsch Conjecture holds, we have $\delta(P_i) \leq n_i - d$, which implies $\delta(\mathcal{A}) \leq \frac{\sum_{i=1}^I (n_i - d)}{I} = \frac{\sum_{i=1}^I n_i}{I} - n$. Since a facet belongs to at most two cells, $\sum_{i=1}^I n_i$ is less than twice the number of bounded facets of $\mathcal{A}$. Now since a bounded facet contained in a hyperplane $H$ of the arrangement $\mathcal{A}$ corresponds to a bounded cell of the simple arrangement $\mathcal{A} \cap H$ of one dimension less, we have $\sum_{i=1}^I n_i \leq 2n \binom{n-2}{d-1}$. Thus we obtain $\delta(\mathcal{A}) \leq \frac{2n \binom{n-2}{d-1}}{\binom{n-1}{d}} - d = \frac{2nd}{n-1} - d$, which implies the bound.
\end{proof}

\subsection{Special classes of polytopes}\label{section:special}

Where the conjectured upper bound cannot be proved, it is interesting to study upper bounds for special families of polytopes. Many of the polytopes appearing in combinatorial optimization belong to the class of~\emph{network flow polytopes}, which include~\emph{transportation polytopes}. In this section we mention that polynomial bounds are known for these and for some generalizations and related classes of polytopes. In addition, polytopes whose vertex coordinates are all zeroes and ones satisfy the Hirsch Conjecture.

\subsubsection*{Integer vertices}

A polytope $P$ is called a \defn{$0$-$1$ polytope} if every coordinate of every vertex of $P$ is either $0$ or $1$. That is to say, $P$ is the convex hull of some of the vertices of a cube. In~\cite{Naddef:Hirsch01}, \given{Denis }Naddef proved that the Hirsch Conjecture holds for $0$-$1$ polytopes.
\begin{theorem}[Naddef~\cite{Naddef:Hirsch01}] 
\label{theorem:01hirsch}
Let $P$ be a $d$-polytope with $n$ facets such that every vertex of $P$ is in $\{0,1\}^d$.  Then $\diam(P) \leq n - d$.
\end{theorem}
\begin{proof}
We assume that $P$ is full-dimensional. This is no loss of generality since, if the dimension of $P$ is strictly less than $d$, then $P$ can be isomorphically projected to a face of the cube $[0,1]^d$.

Let $u$ and $v$ be two vertices of $P$. By symmetry, we may assume that $u = (0,\ldots,0)$. If there is an $i$ such that $v_i=0$, then $u$ and $v$ are both on the face of the cube corresponding to $\{x \in \R^d \mid x_i = 0\}$, and the statement follows by induction.  Therefore, we assume that $v = (1,\ldots,1)$.
Now, pick any neighboring vertex $v'$ of $v$. There is an $i$ such that $v'_i=0$.  Then, $u$ and $v'$ are vertices of a lower-dimensional $0$-$1$ polytope and we have performed one pivot from $v$ to $v'$. The result follows by induction on $d$.
\end{proof}

In~\cite{Onn:ConvexCombinatorial}, Onn and Rothblum prove the following strengthening of Theorem~\ref{theorem:01hirsch}:
\begin{theorem}[Onn and Rothblum~\cite{Onn:ConvexCombinatorial}]
Let $P$ be a $d$-polytope with $n$ facets such that every vertex of $P$ is in $\{0,1\}^d$. Then, under any linear functional $\xi$, there is a non-decreasing path from any vertex $v$ to any vertex $v^*$ maximizing $\xi$, of length at most $d$ using each edge-direction at most once.
\end{theorem}

As a generalization of Theorem~\ref{theorem:01hirsch}, in~\cite{Kleinschmidt:Diameter} \given{Peter }Kleinschmidt and \given{Shmuel }Onn prove the following bound on the diameter of lattice polytopes in $[0,k]^d$. A polytope is called a \defn{lattice polytope} if every coordinate of every vertex is integral.

\begin{theorem}[Kleinschmidt and Onn~\cite{Kleinschmidt:Diameter}]
The diameter of a lattice polytope contained in $[0,k]^d$ cannot exceed $kd$.
\end{theorem}
Using this theorem, one can obtain a bound for the diameter of all rational $d$-dimensional polytopes with $n$ facets. However, the bound obtained is not polynomial in $n$ and $d$, and it does not improve on the bound in~\cite{Kalai:Quasi-polynomial} of Kalai and Kleitman (see~\cite{Onn:KOboundNotPolynomial}).

\subsubsection*{Transportation Polytopes}

Assuming the Hirsch Conjecture is true, every $p \times q$ transportation polytope has diameter at most $p+q-1$.
\given{Graham }Brightwell, \given{Jan }van den Heuvel, and \given{Leen }Stougie (see~\cite{Brightwell:LinearTransportation}) prove the diameter to be at most eight times that. This has been improved by \given{Cor }Hurkens (see~\cite{Hurkens:Diameter4p}) to:
\begin{theorem}[Hurkens~\cite{Hurkens:Diameter4p}]
The diameter of any $p \times q$ transportation polytope is at most $3(p+q-1)$.
\end{theorem}

The dual transportation polyhedron given by the $p \times q$ matrix $C = (C_{ij})$ is the unbounded polyhedron $D_{p,q}(C) = \{(u_1,\dots,u_p,v_1,\ldots,v_q) \in \R^{p+q} \mid u_i + v_j \leq C_{ij} \text{ for all } i,j\}$.
Duality here is in the sense of linear programming. In~\cite{Balinski:DualTransportation}, \given{Michel }Balinski proved that the Hirsch Conjecture holds for the bounded polytopes resulting from the intersection of a dual transportation polyhedron $D_{p,q}(C)$ with a certain hyperplane.
\begin{theorem}[Balinski~\cite{Balinski:DualTransportation}] 
Let $C$ be an $p \times q$ matrix. The diameter of the polytope 
\[D_{p,q}(C) \cap \{(u_1,\dots,u_p,v_1,\ldots,v_q) \in \R^{p+q} \mid u_1 = 0\}\]
is at most $(p-1)(q-1)$. This bound is the best possible.
\end{theorem}

In joint work with \given{Jes\'us }De~Loera, \given{Shmuel }Onn, and \given{Francisco }Santos (see~\cite{DeLoera:GraphsTP}), we prove a quadratic bound (in $p+q+s$) on the diameter of $3$-way axial $p \times q \times s$ transportation polytopes. By the universality theorem of De~Loera and Onn in~\cite{DeLoera:Universality} (see also Theorem~\ref{theorem:universality}), a generalization of this bound to \emph{faces} of these polytopes would prove the polynomial Hirsch Conjecture.
\begin{theorem}[\cite{DeLoera:GraphsTP}]\label{theorem:originalaxial}
The diameter of every $3$-way axial $p \times q \times s$ transportation polytope is at most $2(p+q+s-3)^2$.
\end{theorem}
In Chapter~\ref{chapter:transportation}, we examine the diameters of $3$-way transportation polytopes, where we will present and prove Theorem~\ref{theorem:originalaxial}. See Section~\ref{section:axialdiam}.

\subsubsection*{Network flow polytopes}

Network flow polytopes are another generalization of transportation polytopes. Let $G=(V,E)$ be a graph with $v=|V|$ nodes and $e=|E|$ directed arcs. For each arc $(i,j) \in E$, fix a capacity lower bound $l_{ij}$ and upper bound $u_{ij}$. The network flow polytope $P$ determined by the directed graph $G$ with capacity bounds $\{(l_{ij},u_{ij}) \mid (i,j) \in E\}$ and demand function $c : V \rightarrow \R$ is a polytope in the $e$ variables $x_{ij}$ ($(i,j) \in E$) with the $2e$ inequalities
\[ l_{ij} \leq x_{ij} \leq u_{ij} \text{, for each } (i,j) \in E\]
and the $v$ equations
\[\sum_{\{j \in V \mid (i,j) \in E\}} x_{ij} - \sum_{\{j \in V \mid (j,i) \in E\}} x_{ji} = c(i) \text{, for each } i \in V.\]
Note that the $p \times q$ transportation polytope defined by vectors $a \in \R^p$ and $b \in \R^q$ can be obtained as the network flow polytope associated to the complete bipartite graph $G$ on $p$ and $q$ nodes with all edges directed in the same direction, by taking $l_{ij}=0$, $u_{ij}=\infty$ and 
\[
c_i = -a_i \text{ for } i =1,\ldots,p,\qquad
c_{i+p} = b_{i} \text{ for } i=1,\ldots,q.
\]

Every sufficiently generic set of parameters produces a simple $(e-v+1)$-dimensional polytope $P$ with at most $2e$ facets. The results in~\cite{Cunningham:Theoretical-properties},~\cite{Goldfarb:Polynomial-simplex}, and~\cite{Orlin:PolytimeNetworkSimplex} prove the following upper bound on the diameter of network flow polytopes.
\begin{theorem}
The diameter of the network flow polytope on a directed graph $G=(V,E)$ is $O(ev \log v)$.
\end{theorem}
We note that $ev \log v$ is $O(n^2 \log n)$, where $n$ is the number of facets of $P$.

A matrix is totally unimodular if all of its subdeterminants are $0$, $1$, or $-1$. \given{Martin }Dyer and \given{Alan }Frieze (see~\cite{Dyer:RandomWalks}) gave a polynomial upper bound on the diameter of polyhedra whose defining matrix is totally unimodular, a case that includes all transportation and network flow polytopes.
\begin{theorem}[\cite{Dyer:RandomWalks}]
Let $A$ be a totally unimodular $n \times d$ matrix and let $b \in \R^n$.  The diameter of the polyhedron $P = \{x \in \R^d \mid Ax \leq b \}$ is in $O(d^{16}n^3(\log(dn))^3)$.
\end{theorem}
In the statement, the polyhedron $P$ can be assumed to be a $d$-polytope with $n$ facets. 

\section{Negative results}\label{section:negative}

To contrast the previous section, we present counter-examples to three natural variants of the Hirsch Conjecture. 

\subsection{The Unbounded and Monotone Hirsch Conjectures are false}\label{section:unbounded-monotone}

In the Hirsch Conjecture as we have stated it, we only consider bounded polytopes. However, in the context of linear programming the feasible region may well not be bounded, so the conjecture is equally relevant for \emph{unbounded} polyhedra. In fact, that is how Hirsch originally posed the question (see page 168 of~\cite{Dantzig:Linear}).

Moreover, for the simplex method in linear programming one follows \emph{monotone paths}: starting at an initial vertex $x$ of the feasibility polyhedron $P$, one does pivot steps (that is, one moves along edges) always increasing the value of the linear functional $\xi$ to be maximized, until one arrives at a vertex $x^*$ where no pivot step gives a greater value to $\xi$. Convexity then implies that the vertex $x^*$ is the global maximum for $\xi$ in the feasible region. This raises the question whether a \emph{monotone} version of the Hirsch Conjecture holds. Both variants of the Hirsch Conjecture fail. But before going any further, let us state both conjectures carefully:
\begin{itemize}

\item The Unbounded Hirsch Conjecture: {\sl Is the diameter of every (bounded or not) $d$-polyhedron $P$ with $n$ facets at most $n-d$?}

\item The Monotone Hirsch Conjecture: {\sl Is there, for every $d$-polytope $P$ with $n$ facets and every linear functional $\xi$,  a $\xi$-monotone path with at most $n-d$ edges from any vertex $v$ to a vertex $v^*$ where $\xi$ is maximized?} \emph{Monotonically} means that we require the value of $\xi$ to increase at every step.
\end{itemize}

Here we show that the unbounded and monotone versions of the Hirsch Conjecture fail. Both proofs are based on the Hirsch-sharp polytope $Q_4$ described in Section~\ref{section:Q4}. In fact, we want to emphasize that knowing the mere existence of such a polytope is enough. We are not going to use any property of $Q_4$ other than the fact that it is Hirsch-sharp, simple, and has $n>2d$. Simplicity is not a real restriction since it can always be obtained without decreasing the diameter (see Lemma~\ref{lemma:simple}). The inequality $n>2d$, however, is essential as the constructions below would not work with, for example, a $d$-cube.

\begin{theorem}[Klee-Walkup~\cite{Klee:d-step}]
\label{theorem:unbounded}
There is a simple unbounded polyhedron $\tilde Q_4$ with eight facets and dimension four and whose graph has diameter five.
\end{theorem}

\begin{proof}
Let $Q_4$ be the simple Klee-Walkup polytope with nine facets,  and let $u$ and $v$ be vertices of $Q_4$ at distance five from one another. By simplicity, the vertices $u$ and $v$ lie in (at most) eight facets in total and there is (at least) one facet $F$ not containing $u$ nor $v$. 

Let $\tilde Q_4$ be the unbounded polyhedron obtained by a projective transformation that sends this ninth facet to infinity. The graph of $\tilde Q_4$ contains both $u$ and $v$, and is a subgraph of that of $\tilde Q_4$, hence its diameter is still at least five. See Figure~\ref{figure:unbounded} for a schematic rendition of this idea.
\end{proof}

\begin{figure}[htb]
\begin{center}
\includegraphics[width=4in]{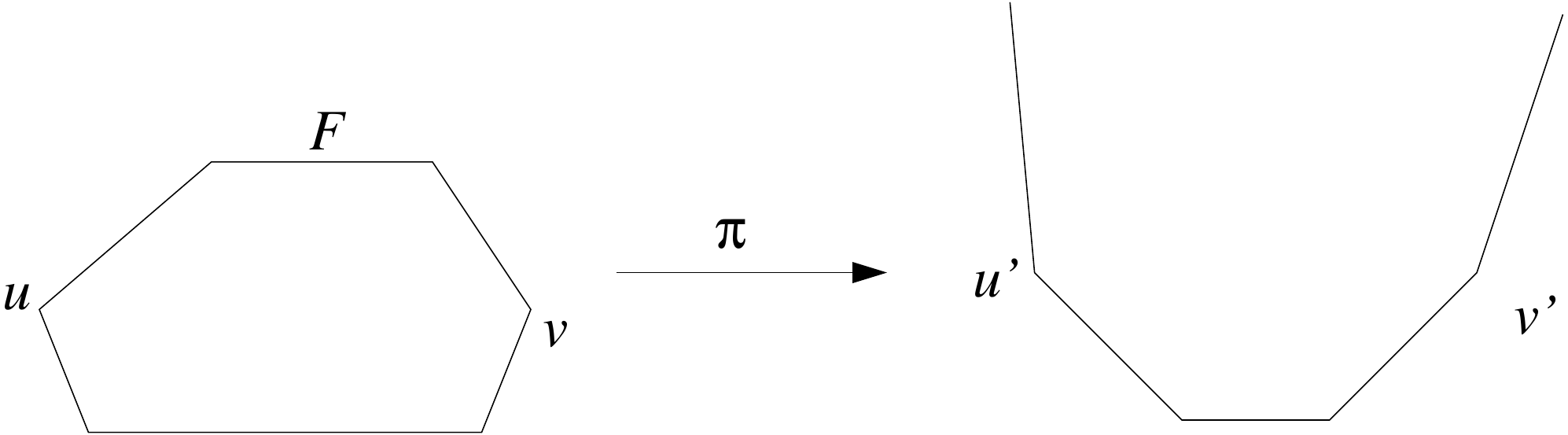}
\caption{Disproving the Unbounded Hirsch Conjecture}
\label{figure:unbounded}
\end{center}

\end{figure}

\begin{remark}
It is interesting to observe that the ``converse'' of the above proof also works: from any non-Hirsch unbounded polyhedron $\tilde Q$ with eight facets and dimension four, one can build a bounded polytope with nine facets and diameter still five, as follows:

Let $u$ and $v$ be vertices of $\tilde Q$ at distance five from one another.
Construct the polytope $Q$ by cutting $\tilde Q$ with a hyperplane that leaves all the vertices of $\tilde Q$ on the same side. This adds a new facet and changes the graph, by adding new vertices and edges on that facet. But $u$ and $v$ will still be at distance five: to go from $u$ to $v$ either we do not use the new facet $F$ that we created (that is, we stay in the graph of $\tilde Q_4$) or we use a pivot to enter the facet $F$ and at least another four to enter the four facets containing $v$: since the Hirsch Conjecture holds for 3-dimensional polyhedra, $u$ and $v$ cannot lie in a common facet of $\tilde Q$.
\end{remark}

We now turn to the Monotone Hirsch Conjecture:
\begin{theorem}[Todd~\cite{Todd:The-monotonic-bounded}]
\label{theorem:monotone}
There is a simple bounded $4$-polytope $P$ with eight facets, two vertices $u$ and $v$ of $P$, and a linear functional $\xi$ such that:
\begin{enumerate}
\item $v$ is the only maximal vertex for $\xi$.
\item Any edge-path from $u$ to $v$ and monotone with respect to $\xi$ has length at least five.
\end{enumerate}
\end{theorem}

\begin{proof}
Let $Q_4$ be the Klee-Walkup polytope. 
Let $F$ be the same ``ninth facet'' as in the previous proof, one that is not incident to the two vertices $u$ and $v$ that are at distance five from each other. Let $H_2$ be the supporting hyperplane containing $F$ and let $H_1$ be any supporting hyperplane at the vertex $v$. Finally, let $H_0$ be a hyperplane containing the (codimension two)
intersection of $H_1$ and $H_2$ and which lies ``slightly beyond $H_1$,'' as in Figure~\ref{figure:monotone}. (Of course, if $H_1$ and $H_2$ happen to be parallel, then $H_0$ is taken to be parallel to them and close to $H_1$.)
 The exact condition we need on $H_0$ is that it does not intersect $Q_4$ and the small, wedge-shaped region between $H_0$ and $H_1$ does not contain the intersection of any 4-tuple of facet-defining hyperplanes of $Q_4$. 

\begin{figure}[htb]
\begin{center}
\includegraphics[width=0.9\textwidth]{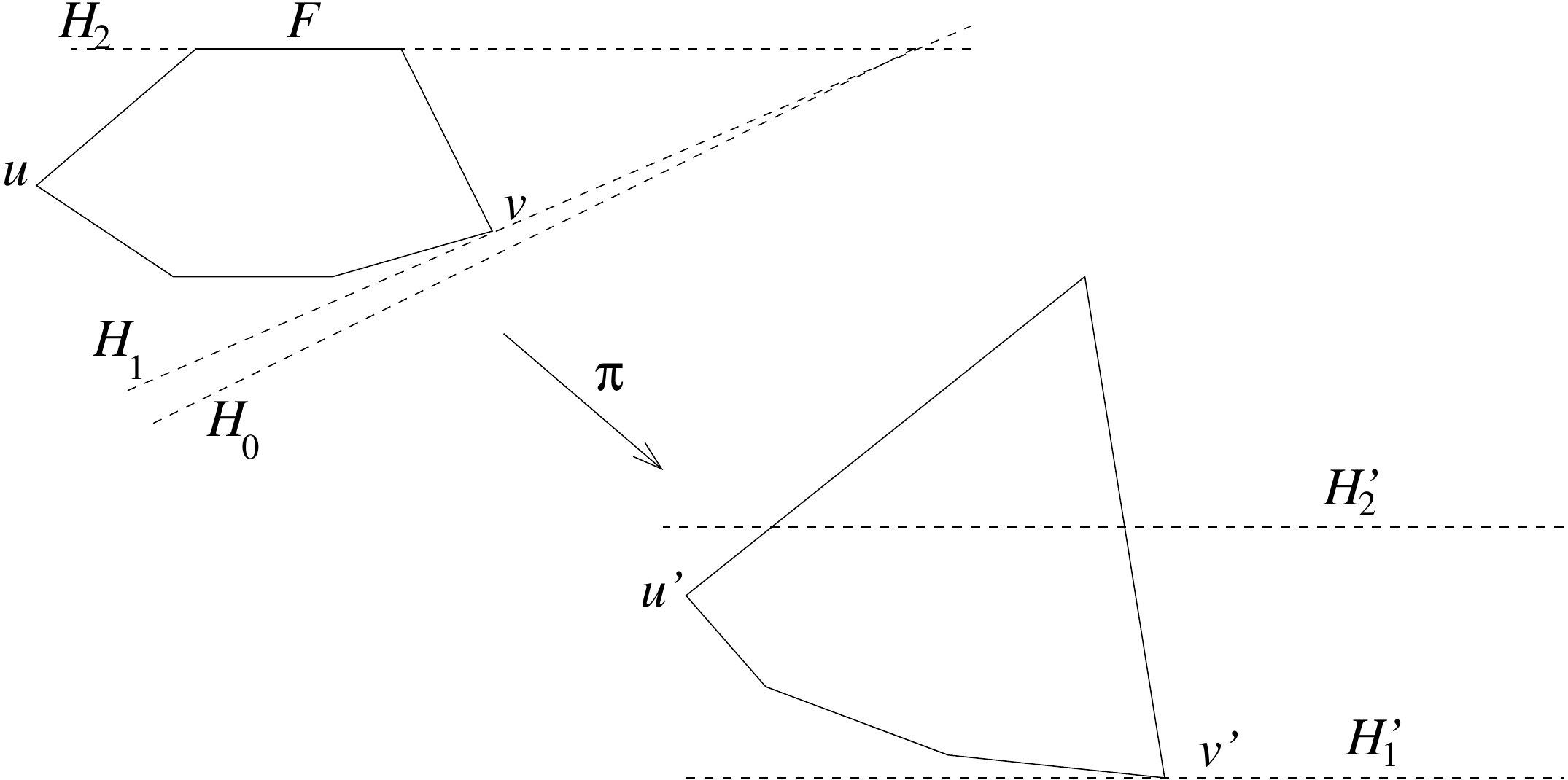}
\caption{Disproving the Monotone Hirsch Conjecture}
\label{figure:monotone}
\end{center}
\end{figure}

We now make a projective transformation $\pi$ that sends $H_0$ to be the hyperplane at infinity. In the polytope $Q'_4=\pi(Q_4)$ we ``remove'' the facet $F'=\pi(F)$ that is not incident to the two vertices $u'=\pi(u)$ and $v'=\pi(v)$. That is, we consider the polytope $Q''_4$ obtained from
$Q'_4$ by forgetting the inequality that creates the facet $F'$ (see Figure~\ref{figure:monotone} again).  Then $Q''_4$ will have new vertices not present in $Q'_4$, but it also has the following properties:

\begin{enumerate}
\item $Q''_4$ is bounded. Here we are using the fact that the wedge between $H_0$ and $H_1$ contains no intersection of facet-defining hyperplanes: this implies that no facet of $Q''_4$ can go ``past infinity.''

\item It has eight facets: four incident to $u'$ and four incident to $v'$.

\item The functional $\xi$ that is maximized at $v'$ and constant on its supporting hyperplane $H'_1=\pi(H_1)$ is also constant on $H'_2=\pi(H_2)$, and $u'$ lies on the same side of $H'_1$ as $v'$.
\end{enumerate}

In particular, no $\xi$-monotone path from $u'$ to $v'$ crosses $H'_1$,
which means it is also a path from $u'$ to $v'$ in the polytope $Q'_4$, combinatorially isomorphic to  $Q_4$.
\end{proof}
The diameter of Todd's polytope is four, so it does not give a counter-example to the Hirsch Conjecture.

In both the constructions of Theorems~\ref{theorem:unbounded} and~\ref{theorem:monotone} one can glue several copies of the initial block $Q_4$ to one another. The basic idea is (the polar of) the same one used in Corollary~\ref{corollary:glue}. We skip details, but in both cases we increase the number of facets by four and the diameter by five, per $Q_4$ glued, obtaining:
\begin{theorem}[Klee-Walkup, Todd]
The Unbounded and Monotone Hirsch Conjectures are false:
\begin{enumerate}
\item There are unbounded $4$-polyhedra with $4+4k$ facets and diameter $5k$, for every $k\geq 1$.

\item There are bounded $4$-polyhedra with $5+4k$ facets and vertices $u$ and $v$ of them with the property that any monotone path from $u$ to $v$ with respect to a certain linear functional $\xi$ maximized at $v$ has length at least $5k$.
\end{enumerate}
\end{theorem}

This leaves us with the following questions:
\begin{openproblem}
Improve these constructions so as to get the ratio of ``diameter versus facets'' bigger than $5/4$. Can the ratio be as big as two? That is: Is there a $d$-dimensional polytope $P$ with $n$ facets having the property that, for every linear functional $\xi$ on $P$, there is a $\xi$-monotone path with at least $2(n-d)$ from any vertex $v$ to a vertex $v^*$ where $\xi$ is maximized? Is there an unbounded $d$-dimensional polyhedron $P$ with $n$ facets whose diameter is $2(n-d)$ or greater?
\end{openproblem}
A ratio bigger than two for the Unbounded Hirsch Conjecture would probably yield counter-examples to the bounded Hirsch Conjecture.

\given{G\"unter }Ziegler (see page 87 of~\cite{Ziegler:Lectures}) poses the following conjecture stronger than the Monotone Hirsch Conjecture: 
\begin{conjecture}[\emph{Strict} Monotone Hirsch Conjecture]
For every linear functional $\xi$ on a $d$-polytope with $n$ facets there is a $\xi$-monotone path of length at most $n-d$.
\end{conjecture}
Put differently, in the Monotone Hirsch Conjecture we add the requirement that not only $v$ but also $v^*$ has a supporting hyperplane where $\xi$ is constant.

Finally, even though the Unbounded and Monotone Hirsch Conjectures are false in general, it would be interesting to know if there are any non-trivial families of polytopes (or spheres) where it is true:
\begin{openproblem}
Are there any non-trivial (infinite) families of polyhedra for which the Unbounded Hirsch Conjecture holds?
\end{openproblem}
\begin{openproblem}
Are there any non-trivial (infinite) families of polyhedra for which the Monotone Hirsch Conjecture holds?
\end{openproblem}

\subsection{The Topological Hirsch Conjecture is false}\label{section:combinatorial}

We discuss a topological generalization of the Hirsch Conjecture. Since (the boundary of) every simplicial $d$-polytope is a topological triangulation of the $(d-1)$-dimensional sphere, we can ask whether the simplicial version of the Hirsch Conjecture, the one where we walk from simplex to simplex rather than from vertex to vertex, holds for arbitrary triangulations of spheres. As in the previous section, to be precise:
\begin{itemize}
\item The Topological Hirsch Conjecture: {\sl If $T$ is topological triangulation of the $(d-1)$-sphere with $n$ vertices and $G^\Delta(T)$ is its polar graph, is $\diam(G^\Delta(T))$ at most $n-d$?}
\end{itemize}
Here, $\diam(G^\Delta(T))$ denotes the diameter of the polar graph, as in the case of simplicial polytopes. This generalizes the Hirsch Conjecture since the proper faces of a simplicial $d$-polytope form a triangulation of the $(d-1)$-sphere. But the converse is not true: starting in $d=3$ and with $8$ vertices there are \emph{non-polytopal} triangulations of $d$-spheres, that are not combinatorially isomorphic to the boundary of any polytope.

The first counter-example to this statement was a simplicial non-Hirsch $27$-sphere $Z$ with $56$ vertices found by \given{David }Walkup in 1979 (see~\cite{Walkup}), but simpler counter-examples were soon constructed by Walkup and Mani in~\cite{Mani:A-3-sphere-counterexample}. (The analogous problem for polyhedral maps on general surfaces is studied in~\cite{Pulapaka:Nonrevisiting-Paths}.) Both constructions are based on the equivalence of the Hirsch Conjecture to the Non-revisiting Conjecture (see Theorem~\ref{theorem:dstep-nonrevisiting}). The proof of the equivalence is purely combinatorial, so it holds true for topological spheres. Mani and Walkup's counter-examples are $3$-spheres $C$ and $D$ (respectively) with $20$ and $16$ vertices (respectively). We only describe $D$ here, which is obtained from $C$ by contracting four edges and deleting degenerate tetrahedra. Wedging on $D$ eight times produces a non-Hirsch $11$-sphere $E$ with $24$ vertices.
\begin{theorem}[Mani-Walkup~\cite{Mani:A-3-sphere-counterexample}]
\label{theorem:mani-walkup}
There is a triangulated $3$-sphere $D$ with $16$ vertices and without the non-revisiting property. Wedging on it eight times produces a non-Hirsch $11$-sphere $E$ with $24$ vertices. The dual diameter of $E$ is more than $12$.
\end{theorem}

The part of the Mani-Walkup $3$-sphere $D$ that implies failure of the non-revisiting property involves only $12$ of the $16$ vertices. More precisely, Mani and Walkup show the following:

\begin{lemma}
\label{lemma:mani}
Let $K$ be the three-dimensional simplicial complex on the vertices $a$, $b$, $c$, $d$, $m$, $n$, $o$, $p$, $q$, $r$, $s$, and $t$ consisting of the following $26$ tetrahedra:
\[
\begin{tabular}{ccc}
& $abcd$ &\\
$abcr$ & & $acdr$\\
$abdt$ & & $bcdt$\\
$abmr$ & & $bcnr$\\
$cdor$ & & $dapr$\\
$abmt$ & & $bcnt$\\
$cdot$ & & $dapt$\\
\end{tabular}
\qquad \qquad
\begin{tabular}{ccc}
& $mnop$ &\\
$mnoq$ & & $mopq$\\
$nops$ & & $mnps$\\
$anoq$ & & $bopq$\\
$cpmq$ & & $dmrq$\\
$anos$ & & $bops$\\
$cpms$ & & $dmrs$\\
\end{tabular}
\]
Then:
\begin{enumerate}
\item The complex $K$ can be embedded in a $3$-sphere.
\item No triangulation of the $3$-sphere containing $K$ as a subcomplex has the non-revisiting property.
\item There is a triangulation of the $3$-sphere with $16$ vertices and containing $K$ as a subcomplex.
\end{enumerate}
\end{lemma}

We are not going to prove part 3 of the lemma. The construction is somehow complicated and, moreover, in a sense that part is irrelevant. Indeed, once we know that the complex $K$ can be embedded in a $3$-sphere we can rely on Whitehead's Completion Lemma (see~\cite{Whitehead}) to conclude that the complex $K$ can be completed to a triangulation of the whole $3$-sphere. The only drawback of this approach is that we cannot control \emph{a priori} the number of extra vertices needed in the completion, but that will only affect the number of vertices of the final $3$-sphere (and the number of wedges needed to get a non-Hirsch sphere from it).

\begin{proof}[Proof of parts 1 and 2 of Lemma~\ref{lemma:mani}]
The proof follows from the following description of the simplicial complex $K$: it consists of two triangulations of  bipyramids over the octagons $ambncodp$ and $aobpcmdn$, glued along the eight vertices of the octagons. See Figure~\ref{figure:mani-walkup-pyr-o}.
\begin{figure}[hbt]
\begin{center}
\includegraphics[scale=.40]{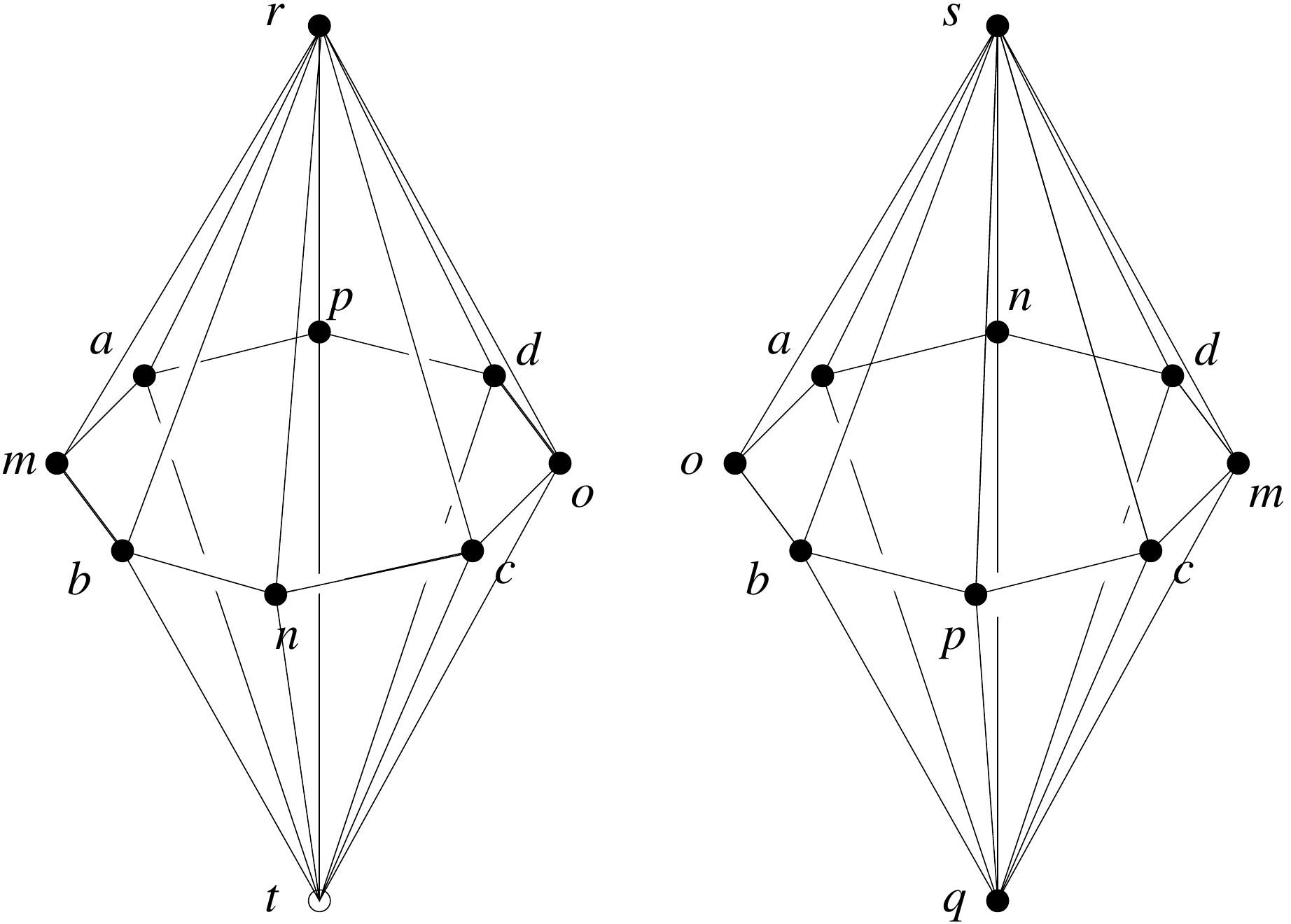}
\caption{Two octagonal bipyramids in the triangulation of $S$ depicting the outside-most $16$ tetrahedra}
\label{figure:mani-walkup-pyr-o}
\end{center}
\end{figure}

Once this is shown, part 1 is easy. Since we are living in a topological world, we can ``pinch'' the equatorial vertices of one of the octagons and there is no obstruction to glue them to their counterparts in the other octagon. One key property is that we are not gluing any of the edges: the order of vertices in the two octagons is not the same, and it is designed so that no edge appears in both. 

For part 2, let us see the construction in more detail. It starts with a core tetrahedron inside each bipyramid, namely $abcd$ and $mnop$. See Figure~\ref{figure:mani-walkup-tetra}.

\begin{figure}[hbt]
\begin{center}
\includegraphics[scale=.40]{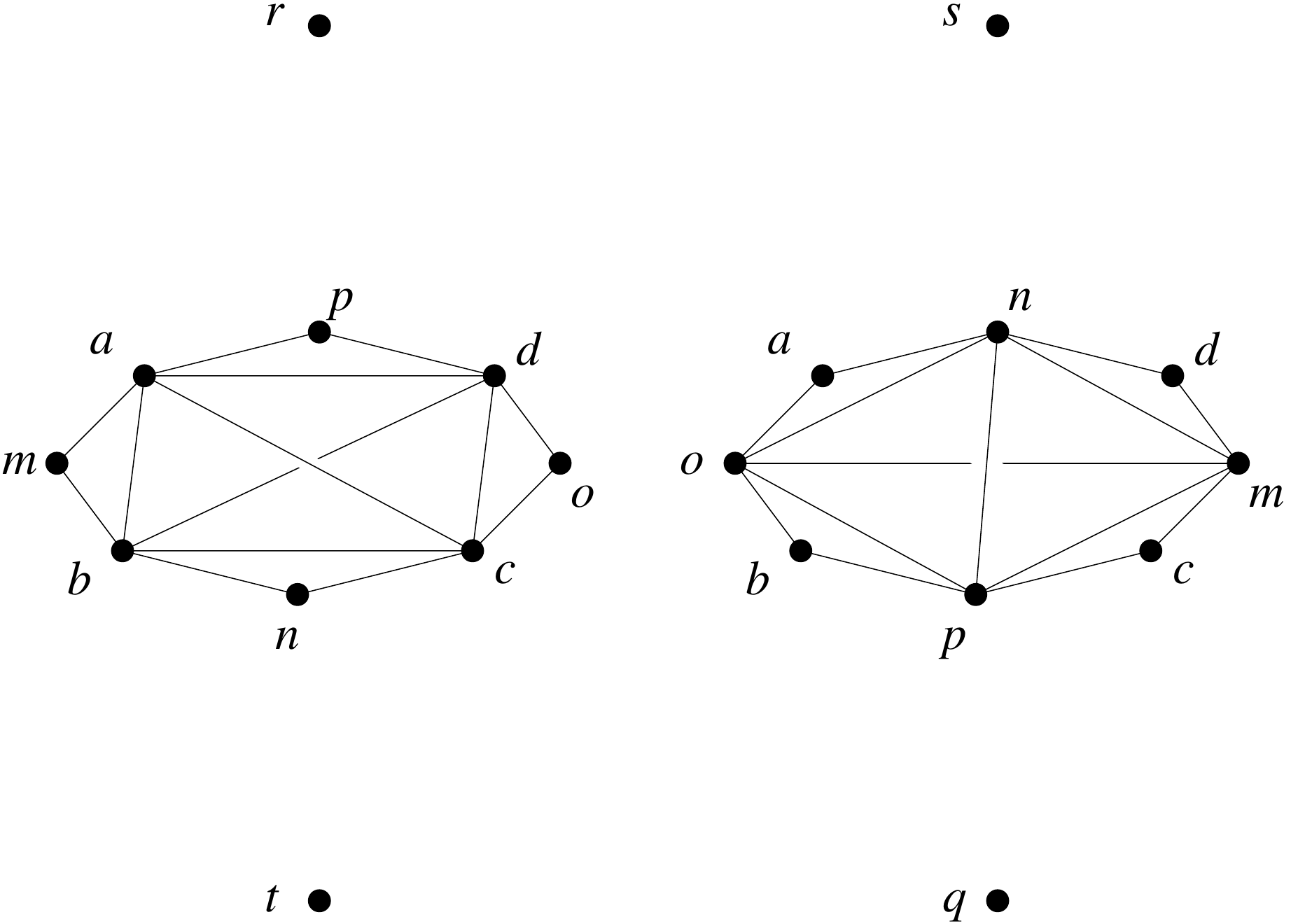}
\caption{The simplices $abcd$ and $mnop$}
\label{figure:mani-walkup-tetra}
\end{center}
\end{figure}

Each of these tetrahedra is surrounded by two tetrahedra joined to each apex of the corresponding bipyramid, as shown in Figure~\ref{figure:mani-walkup-pyr-t}. These are the tetrahedra in the second and third line of the statement and together with the initial ones they triangulate two octahedra.
Finally, these octahedra are each surrounded by eight more tetrahedra each: those obtained joining the four triangles left uncovered in each octagon (see Figure~\ref{figure:mani-walkup-pyr-t} again) to the two apices of their bipyramid.
\begin{figure}[hbt]
\begin{center}
\includegraphics[scale=.40]{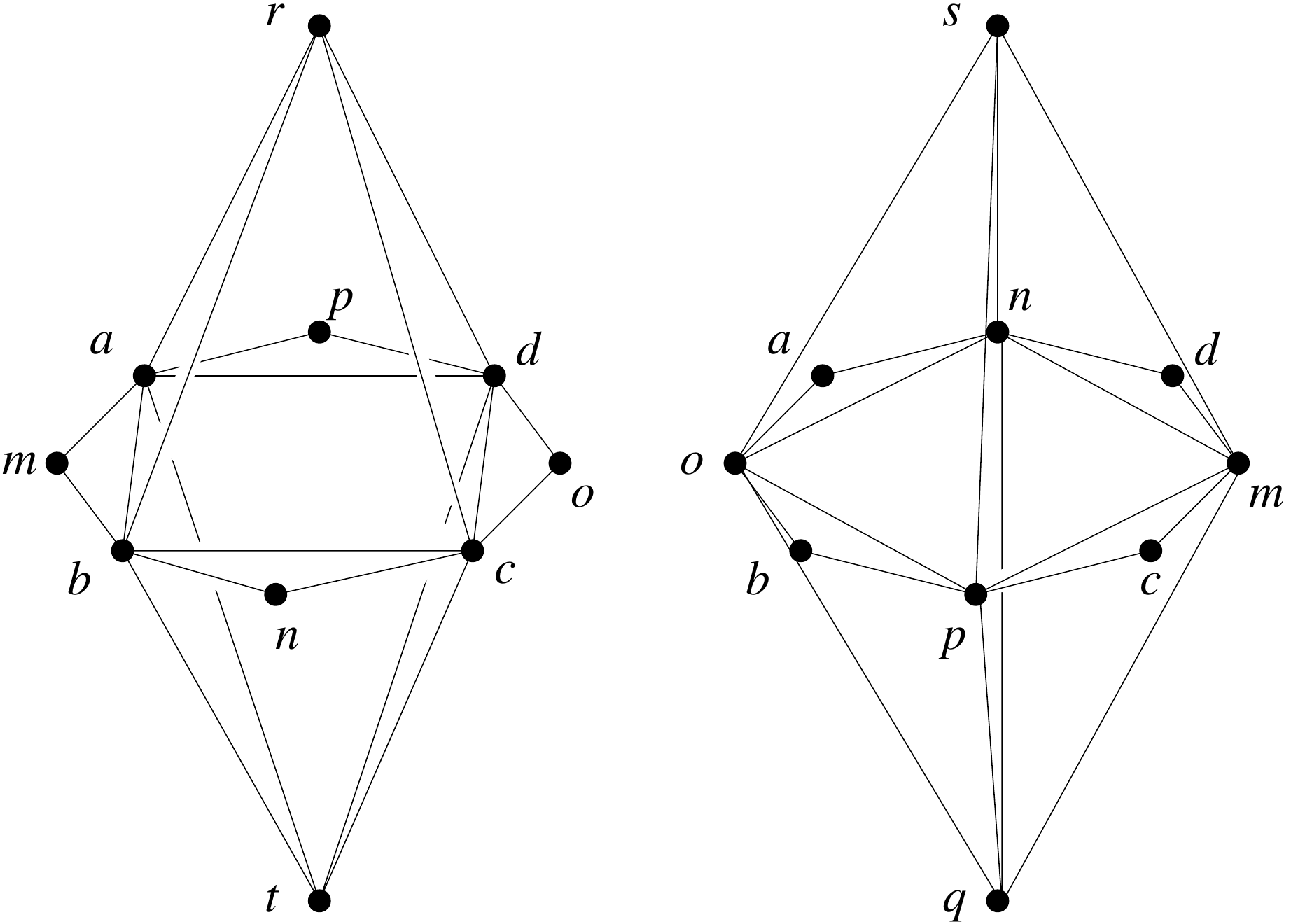}
\caption{Eight additional simplices in the Mani-Walkup triangulation}
\label{figure:mani-walkup-pyr-t}
\end{center}
\end{figure}

From this description it is easy to prove part 2 of the lemma, as follows: Every path from the tetrahedron $abcd$ to the tetrahedron $mnop$ must leave the bipyramid on the left of Figure~\ref{figure:mani-walkup-pyr-o}, and it will do so through one of the sixteen boundary triangles. These triangles are the joins of the eight edges of the octagon to the two apices. In particular, our path will at this point have abandoned three of the vertices of $abcd$ and entered one of $mnop$. For the non-revisiting property to hold, the abandoned ones should not be used again, and the entered one should not be abandoned, since it is a vertex of our target tetrahedron. But then it is impossible for us to enter the second bipyramid: we should do so via another triangle that joins an octagon edge to an apex, and non-revisiting implies that this edge should use the same vertex form $abcd$ and the same vertex from $mnop$. This is impossible since the two octagons have no edge in common.

Let us explain this in a concrete example. By symmetry, there is no loss of generality in assuming that we exit from the left bipyramid via the triangle $amr$. Since we cannot abandon $m$, we must enter the second bipyramid via one of the boundary triangles using $m$, namely one of $mcs$, $mcq$, $mds$ or $mdq$. This violates the non-revisiting property, since $c$ and $d$ had already been abandoned.
\end{proof}

Unfortunately, this triangulated $3$-sphere does not give a counter-example to the Hirsch Conjecture. It would give a counter-example if it were \emph{polytopal} (that is, if 
it were combinatorially isomorphic to the boundary complex of a four-dimensional polytope). However, \given{Amos }Altshuler (see~\cite{Altshuler:NotPolytopal}) has shown that (for the explicit completion of the subcomplex $K$ given in~\cite{Mani:A-3-sphere-counterexample}) this is not the case. We summarize the proof here.
\begin{theorem}[Altshuler~\cite{Altshuler:NotPolytopal}]
The topological triangulated spheres $C$, $D$, and $E$ are not polytopal.
\end{theorem}
\begin{proof}
Suppose for a contradiction that the $3$-sphere $C$ with $20$ vertices had a polytopal realization $P$. Define $Q$ to be the convex hull of all vertices of $P$ except for the vertex $t$ (defined in~\cite{Mani:A-3-sphere-counterexample}). The polytope $Q$ has a triangulation using no additional vertices, so the star $\star_C(t)$ of $t$ can be replaced by a simplicial complex $C'$ which only uses vertices of $C$ and whose boundary complex coincides with that of $\star_C(t)$. Altshuler proves that the triangle $abd$ is in the link $\lk_C(t)$ of $t$. Then, one of $abdc$, $abdi$, $abdj$, $abdk$, $abdm$, $abdn$, $abdo$, $abdp$, or $abds$ must be a tetrahedron in $C'$. But, since the edges $ac$, $di$, $dj$, $aj$, $ak$, $bl$, $dm$, $an$, $ao$, $bp$,  and $as$ appear in the anti-star $\ast_C(t)$, none of them can appear in the complex $C'$, which means that none of the ten tetrahedra could exist. Contradiction.

The same proof works on the $3$-sphere $D$ with $16$ vertices. The polyhedron $Q$ is again defined by taking the convex hull without the vertex $t$. The $11$-sphere $E$ is not polytopal since it is built up from $D$ by repeated applications of the one-point suspension.
\end{proof}
It remains an open question if \emph{any} completion of $K$ to the $3$-sphere is polytopal.
\begin{openproblem}
Find a completion of the complex $K$ to a $3$-sphere that is polytopal.
\end{openproblem}
Even more strongly, it is probably the case that $K$ cannot be embedded in $\R^3$ with linear tetrahedra, a necessary condition for polytopality by the Schlegel construction (see~\cite{Ziegler:Lectures}).
\begin{openproblem}
Embed the complex $K$ in $\R^3$ using linear tetrahedra.
\end{openproblem}
As in the monotone and bounded cases, several copies of the construction can be glued to one another. Doing so provides triangulations of the $11$-sphere with $12 +12k$ vertices and diameter at least $13k$, for any $k$. It remains to know how far this false conjecture is from being true. Examples whose diameter is $2(n-d)$ or more would be very significant. Current ones achieve $\frac{13}{12}(n-d)$.
\begin{openproblem}
Is there a triangulation of a $(d-1)$-dimensional sphere using $n$ vertices whose dual graph has diameter $2(n-d)$ or greater?
\end{openproblem}
Even though the Topological Hirsch Conjecture is false in general, it would be interesting to know if there are any non-trivial families of spheres where it is true.
\begin{openproblem}
Are there any non-trivial families of triangulated spheres for which the Topological Hirsch Conjecture holds?
\end{openproblem}


\chapter[Geometric Combinatorics of Transportation Polytopes]{Geometric Combinatorics of Transportation Polytopes}\label{chapter:transportation}

\setcounter{theorem}{0} 

This chapter discusses new results on the combinatorial properties of $2$-way and $3$-way transportation polytopes. In Section~\ref{section:enumerationofpartition}, we discussed our process of classifying all transportation polytopes of a certain size. Based on the data we collected (see Appendix~\ref{appendix:catalogTP}), we discovered and proved the following results:
\begin{theorem}
\label{theorem:gcd}
The number of vertices of a non-degenerate $p \times q$ classical
transportation polytope is divisible by the greatest common divisor $\gcd(p,q)$ of $p$ and $q$.
\end{theorem}
\begin{theorem}
\label{theorem:22n}
The $p \times 2 \times 2$ planar transportation polytopes are in $1$-$1$ correspondence with the $p \times 2$ classical transportation polytopes, with corresponding pairs being linearly isomorphic.
\end{theorem}
Theorem~\ref{theorem:gcd} is presented in Section~\ref{section:gcd} and Theorem~\ref{theorem:22n} is proved in Section~\ref{section:identicalclassesTP}. Note that Theorem~\ref{theorem:22n} is best possible in the sense that for $p,q \geq 3$ there are many more types of planar $p \times q \times 2$ transportation polytopes than types of $p \times q$ transportation polytopes. (See the complete list of $3 \times 3$ classical and $2 \times 3 \times 3$ planar transportation polytopes in Tables~\ref{table:classical-3-3} and~\ref{table:planar-2-3-3}, respectively.)

In Section~\ref{section:regulartriangulationsBirkhoff}, we discuss the existence of non-regular triangulations for the vertices of Birkhoff polytopes $B_n$, for $n \geq 4$. We prove that the polytope $B_4$ has non-regular triangulations. Together with previously-known results, this proves that the polytope $B_n$ only has regular triangulations if and only if $n \leq 3$.

We present new bounds for the diameters of transportation polytopes and subpolytopes of them. In Section~\ref{section:manyHirschsharpTPs}, we construct infinitely-many Hirsch-sharp transportation polytopes. In Section~\ref{section:diameter-p2}, we prove the Hirsch Conjecture for $p \times 2$ classical transportation polytopes. In Section~\ref{section:axialdiam}, we prove the following diameter bound for $3$-way transportation polytopes defined by $1$-marginals:
\begin{theorem}
\label{theorem:main}
The graph of every $3$-way transportation polytope of size $p \times q \times s$ defined by $1$-marginals has diameter at most $2(p+q+s-2)^2$.
\end{theorem}
This bound was first given in~\cite{DeLoera:GraphsTP}, proved in collaboration with \given{Jes\'us }De~Loera, \given{Shmuel }Onn, and \given{Francisco }Santos (see Theorem~\ref{theorem:originalaxial}). If the Hirsch Conjecture (Conjecture~\ref{conjecture:hirsch}) is true, then the diameters of $3$-way transportation polytopes of size $p \times q \times s$ defined by $1$-marginals are at most $p+q+s-2$.
\begin{lemma}
Assume the Hirsch Conjecture. Then the diameter of every $3$-way transportation polytope of size $p \times q \times s$ defined by $1$-marginals is at most $p+q+s-2$.
\end{lemma}
\begin{proof}
Let $P$ be a $3$-way $p \times q \times s$ axial transportation polytope defined by $1$-marginals. The dimension of $P$ is $pqs-(p+q+s-2)$. Since facets can only be obtained by the inequalities of the form $x_{i,j,k} \geq 0$, the number $n$ of facets of $P$ is bounded above by $pqs$. By the Hirsch Conjecture is true, the diameter of $P$ is no more than $n-d$, which is no more than $pqs-(pqs-(p+q+s-2))$.
\end{proof}
\begin{remark}
Note that the ``converse'' does not prove the Hirsch Conjecture is true for $3$-way axial transportation polytopes defined by $1$-marginals. That is to say, a result that the diameter of these polytopes is no more than $p+q+s-2$ does not prove the Hirsch Conjecture is true if the $p \times q \times s$ polytope has strictly less than $pqs$ facets.
\end{remark}

A similar result for the graph of a $p \times q$ classical transportation polytope was given by \given{Graham }Brightwell et al. (see~\cite{Brightwell:LinearTransportation}), who proved an upper bound of $8(p+q-2)$ for the diameter. More recently, \given{Cor }Hurkens (see~\cite{Hurkens:Diameter4p}) has obtained a bound of $3(p+q-1)$, a factor of three away from the predicted value of the Hirsch Conjecture.

As we observed,~\cite{Brightwell:LinearTransportation} provided the first linear bound for the diameter of the graphs of $2$-way transportation polytopes. Theorem~\ref{theorem:main} provides a quadratic bound for $3$-way transportation polytopes defined by $1$-marginals and, moreover, a sublinear one if we assume that the three parameters $p$, $q$, and $s$ are approximately the same. (Observe that the number of facets of a $3$-way transportation polytope is bounded above by the product $pqs$ of its size parameters).

Bounding the diameter of $3$-way transportation polytopes is particularly interesting because of the following results proved by \given{Jes\'us }De~Loera and \given{Shmuel }Onn in~\cite{DeLoera:UniversalitySlim}:
\begin{enumerate}

\item Any rational convex polytope can be rewritten as a face $F$ of a $3$-way transportation polytope of size $p \times q \times s$ defined by $1$-marginals.  The sizes $p,q,s$, the $1$-marginals $u,v,w$, and the entries $x_{i,j,k}$ that are prescribed to be zero in the face $F$ can be computed in polynomial time on the size of the input.

\item More dramatically, any rational polytope is isomorphically representable as a \emph{planar} $3$-way transportation polytope. 
\end{enumerate}

That is to say, a version of Theorem~\ref{theorem:main} for the $3$-way transportation polytopes by $2$-marginals, or a version for $3$-way transportation polytopes defined by $1$-marginals that allows one to prescribe some variables to be zero, would provide a polynomial upper bound on the diameter of the graph of \emph{every} convex rational polytope. Another consequence of these results is that the method of Section~\ref{section:enumerationofpartition} for enumerating all combinatorial types of planar $3$-way transportation polytopes, yields, in particular, an enumeration of \emph{all} types of rational convex polytopes.

We close with Section~\ref{section:networkflow}, which proves new diameter bounds for network flow polytopes, which are subpolytopes of transportation polytopes.

Let us finally mention that our systematic listing of non-degenerate transportation polytopes of small sizes provides the solution to at least four open problems and conjectures about transportation polytopes stated in the monograph~\cite{Yemelichev:Polytopes}:

\begin{enumerate}

\item In~\cite{Klee:FacesTransportation}, \given{Victor }Klee and \given{Christoph }Witzgall prove that the largest possible number of vertices in classical transportation polytopes of size $p \times q$ is achieved by the generalized Birkhoff $2$-way polytope (see Definition~\ref{definition:generalizedbirkhoffclassical} on page~\pageref{definition:generalizedbirkhoffclassical}). Problem~32 on page~400 of~\cite{Yemelichev:Polytopes} conjectured that the same holds in general.

But in Example~\ref{example:notBirkhoff} we provide an explicit counter-example of this for planar $3$-way transportation polytopes. (See the definition of the generalized Birkhoff $3$-way planar polytope in Definition~\ref{definition:generalizedbirkhoffmultiway} on page~\pageref{definition:generalizedbirkhoffmultiway}.)

\item Question 36 on page~396 of~\cite{Yemelichev:Polytopes} asked: \emph{Is it true that every integer of the form $(p-1)(q-1)(s-1)+t$, where $1 \leq t \leq pq+ps+qs-p-q-s$, and only these integers, can equal the number of facets of a non-degenerate planar $3$-way transportation polytope of order $p \times q \times s$ defined by $2$-marginals, where $p,q,s \geq 2$?}

For the case $p=q=2$ and $s=3$, the conjecture asks if every integer from $3$ to $11$, and only these integers, equal the number of facets of non-degenerate $2 \times 2 \times 3$ planar transportation polytopes defined by $2$-marginals. Table~\ref{table:planar-2-2-3} on page~\pageref{table:planar-2-2-3} answers the question negatively: the number $n$ of facets can be $3$, $4$, $5$, or $6$ (and these do indeed occur), but $7$ through $11$ are in fact missing.

\item Similarly, Conjecture 33 on page 400 of~\cite{Yemelichev:Polytopes} asked: \emph{Is it true that every integer from $1$ to $pq+ps+qs-p-q-s+1$, and only these numbers, are realized as the diameter of a planar $3$-way transportation polytope defined by $2$-marginals of order $p \times q \times s$?}

The same case $p=q=2$, and $s=3$ shows that this is false. The transportation polytopes obtained are polygons with up to six sides, hence of diameter at most three, instead of $10$.

\item Open problem~37 on page~396 of~\cite{Yemelichev:Polytopes} asks \emph{whether the number $f_0$ of vertices of every $p \times q \times s$ non-degenerate $3$-way planar transportation polytope defined by $2$-marginals satisfies} $(p-1)(q-1)(s-1)+1 < f_0 < 2(p-1)(q-1)(s-1)$.

We show the answer is no even in the case of non-degenerate $2 \times 2 \times 4$ planar transportation polytopes.
\end{enumerate}

In addition to the four solved problems above, Theorems
\ref{theorem:genericcase} and~\ref{theorem:gcd} are initial steps on the solution of Problem~25 in page~399 of~\cite{Yemelichev:Polytopes}. It asks to find the complete distribution of possible number of vertices for transportation polytopes.

\begin{example} \label{example:notBirkhoff} 
Here is an application of this method, which gives a counter-example to open problem 37 of~\cite{Yemelichev:Polytopes}. The $2$-marginals $U$, $V$, and $W$ below define a $3 \times 3 \times 3$ planar transportation polytope which has more vertices ($270$ vertices) than the generalized Birkhoff $3$-way planar polytope, with only $66$ vertices:
\begin{align*}
U&=\left[
\begin{array}{ccc}
      164424  &    324745  &    127239 \\
      262784  &    601074  &   9369116 \\ 
      149654  &   7618489  &   1736281    
\end{array}
\right],
\end{align*}
\begin{align*}
V&=\left[
\begin{array}{ccc}
163445   &    49395 &     403568  \\
1151824   &   767866 &    8313284 \\
1609500   &  6331023 &    1563901 
\end{array}
\right],
\end{align*}
\begin{align*}
W&=\left[
\begin{array}{ccc}
184032  &    123585 &     269245 \\ 
886393  &   6722333 &     935582 \\ 
1854344 &    302366 &    9075926 
\end{array}
\right].
\end{align*}
\end{example}

Based on the data collected from the enumeration process, we
conjecture the following to be true:

\begin{conjecture}
  The graph of every non-degenerate $p \times q$ transportation
  polytope has a Hamiltonian cycle if $pq > 4$.
\end{conjecture}

\begin{conjecture}
If $P$ is a non-degenerate $3$-way transportation polytope of size $p \times q \times s$ ($p, q, s \geq 3$) defined by $1$-marginals, then the diameter of its graph $G(P)$ is equal to $n - d$, where $d = pqs - p - q - s + 2$ is the dimension and $n \leq pqs$ is the number of facets of $P$.
\end{conjecture}
Among $p \times q$ classical transportation polytopes ($p,q \leq 5$), there are non-degenerate $d$-polytopes where the diameter of the graph $G(P)$ is \emph{strictly} less than $n - d$, where $n \leq pq$ is the number of facets. (See~\cite{TPDB}.)

\section[The number of vertices of classical transportation polytopes]{\texorpdfstring{The number of vertices of $p \times q$ classical transportation polytopes}{The number of vertices of classical transportation polytopes}}\label{section:gcd}

In this section, we prove Theorem~\ref{theorem:gcd}, which says that the number of vertices of a non-degenerate $p \times q$ classical transportation polytope is divisible by $\gcd(p,q)$. Recall that $\Delta_d$ denotes the $d$-dimensional simplex, which has $d+1$ vertices. For ease of notation, we will also denote the $d$-simplex with $d+1$ vertices by $D_{d+1}$. In particular, by this shift of subscript, we have $D_{d+1} = \Delta_d$ for all $d \geq 0$. The first observation, already hinted in Example~\ref{example:triprism}, is that the vector configuration $A_{p,q}$ associated to these transportation polytopes is (a cone over) the set of vertices of the product $D_{p,q} = D_p \times D_q$ of two simplices $D_{p}$ and $D_{q}$ of dimensions $p-1$ and $q-1$, respectively. So, we are interested in the cardinalities of chambers in the product of two simplices. Here and in what follows we call the \defn{cardinality} of a chamber $c$ of $A_{p,q}$ the number of bases of $A_{p,q}$ that contain the chamber $c$. We denote it by $|c|$. The proof of Theorem~\ref{theorem:gcd} consists of the following two steps, which are established respectively in the two lemmas below:
\begin{figure}[hbt]
  \begin{center}
    \includegraphics[scale=0.78]{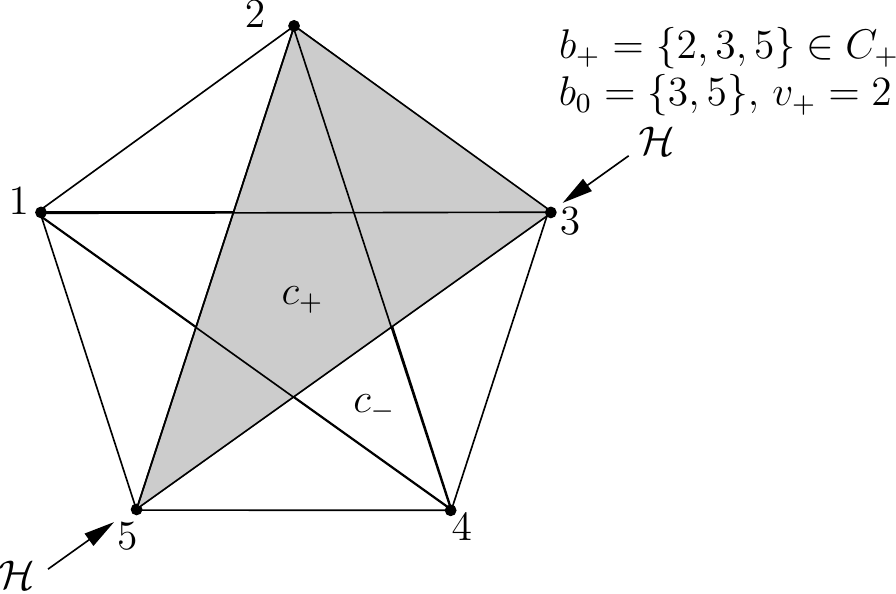}
    \caption{A cross-section of the chamber complex of some cone with two adjacent chambers} \label{figure:cutlemma}
  \end{center}
\end{figure}
\begin{itemize}
\item There is a ``seed'' chamber in $D_{p,q}$ whose cardinality is indeed a multiple of $\gcd(p,q)$.
\item The difference in the cardinalities of any two adjacent chambers of $D_p\times D_q$ is a multiple of $\gcd(p,q)$.
\end{itemize}
Since the chamber complex is a connected polyhedral complex (where two adjacent chambers are divided by a hyperplane supported on the vector configuration) the two lemmas settle the proof.

Let us define the \defn{lexicographic chamber} of $D_{p,q}$
recursively as the (unique) chamber incident to the lexicographic
chamber of $D_{p,q-1}$. The recursion starts with $D_{p,1}$,
which is an $(p-1)$-simplex and contains a unique chamber.  Observe
that the definition of the lexicographic chamber is not symmetric in
$p$ and $q$.  For example, the lexicographic chamber of the triangular
prism $D_{3,2}$ is the one incident to a basis of the prism, and
has cardinality $3$.  The lexicographic chamber of $D_{2,3}$ is
incident to one of the edges parallel to the axis of the prism, and
has cardinality four.

\begin{lemma}
\label{lemma:lex-chamber}
The lexicographic chamber of $D_{p,q}$ is contained in exactly
$p^{q-1}$ simplices.
\end{lemma}
\begin{proof}
The cardinality of the lexicographic chamber of $D_{p,q}$
equals the cardinality of the lexicographic chamber of
$D_{p,q-1}$ times the number of vertices of $D_{p,q}$ not
lying in its facet $D_{p,q-1}$. The latter equals $p$.
\end{proof}

When moving from a chamber $c_-$ to an adjacent one $c_+$ we ``cross"
a certain hyperplane $\mathcal{H}$ spanned by all except one of the
elements of any basis containing $c_+$ but not containing $c_-$.  Let
us denote by $C_+$ and $C_-$ the subsets of $A_{p,q}$ lying in the
sides of $\mathcal{H}$ containing $c_+$ and $c_-$ respectively.
(Remember that, in our case, $A_{p,q}$ equals the set of vertices of
$D_{p,q}$.) Observe also that the common boundary $c_0 \subset
\mathcal{H}$ of $c_+$ and $c_-$ is a chamber in the vector configuration
$A_{p,q}\cap \mathcal H$.

\begin{lemma}
\label{lemma:chambertochamber}
Let $A_{p,q}$ is the set of vertices of $D_{p,q}$. Then,
\begin{enumerate}
\item $|c_+|- |c_-| = |c_0| (|C_+|-|C_-|)$.
\item $|C_+|-|C_-|$ is a multiple of $\gcd(p,q)$.
\end{enumerate}
\end{lemma}

\begin{proof} A basis $b_+$ contains $c_+$ but not $c_-$ if and only if $b_+$
is of the form $b_0\cup \{v_+\}$, where $b_0$ is a basis of $B\cap
\mathcal H$ containing $c_0$ and $v_+$ is an element of $C_+$.  This
and the analogous property for $c_-$ proves the first part.

For the second part, we restate a few facts in the terminology of oriented matroids. This makes the proof easier to write. (For, details see~\cite{Bjorner:OrientedMatroids} or~\cite{Bokowski:Computational-Oriented}.)

\begin{itemize}
\item In oriented matroid terminology a pair $(C_+,C_-)$ consisting of
  the subconfigurations on one and the other side of a hyperplane
  $\mathcal H$ spanned by a subset of $A_{p,q}$ is called a
  \defn{cocircuit} of $A_{p,q}$. That is, part 2 is a statement about
  the cocircuits in the oriented matroid $\mathcal M_{p,q}$ associated
  to the vertices of the product of two simplices.

\item The oriented matroid $\mathcal M_{p,q}$ coincides with the one
  associated to the complete directed bipartite graph $K_{p,q}$
  (i.e., the complete bipartite graph with all of its edges oriented
  from one part to the other). Thus, part 2 is a statement about the
  cocircuits in the oriented matroid of the directed bipartite graph 
  $K_{p,q}$.

\item The cocircuits of a directed graph $G=(V,E)$ are all read off
  from \defn{cuts} in the graph. By this we mean that the vertex set
  $V$ is decomposed into two parts $(V_+,V_-)$. The cocircuit
  $(C_+,C_-)$ associated to the cut $(V_+,V_-)$ has $C_+$ consisting
  of all the edges directed from $V_+$ to $V_-$ and $C_-$ consisting
  of all the edges directed from $V_-$ to $V_+$.

\end{itemize}

Using the dictionary between the directed graph $K_{p,q}$ and the
product of simplices we can finish the proof.  Let $(V_+, V_-)$ be a
cut in the complete directed bipartite graph $K_{p,q}$. Since our
graph is bipartite, we have $V_+$ and $V_-$ naturally decomposed as
$V_+^{(p)}\cup V_+^{(q)}$ and $V_-^{(p)}\cup V_-^{(q)}$, respectively.
The sizes of $C_+$ and $C_-$ are then:
\[
|C_+| = |V_+^{(p)}| \cdot  |V_-^{(q)}|
\qquad \hbox{and}\qquad
|C_-| = |V_-^{(p)}| \cdot  |V_+^{(q)}|.
\]
Now, using that $|V_+^{(p)}| + |V_-^{(p)}| =p$ and $|V_+^{(q)}| + |V_-^{(q)}| =q$ we get:
\[
|C_+| - |C_-| = 
|V_+^{(p)}| \cdot (q - |V_+^{(q)}|) - 
|V_+^{(q)}| \cdot (p - |V_+^{(p)}|) = 
|V_+^{(p)}| \cdot q - |V_+^{(q)}| \cdot p,
\] 
which is clearly a multiple of $\gcd(p,q)$.
\end{proof}

\section{Identical classes of transportation polytopes}\label{section:identicalclassesTP}

Here, we present the proof of Theorem~\ref{theorem:22n}, which says the $p \times 2 \times 2$ planar transportation polytopes are linearly isomorphic to the $p \times 2$ classical transportation polytopes. (The theorem is obtained as joint work with De Loera, Onn, and Santos.) First, we prove the following lemma:

\begin{lemma}\label{lemma:plan22n}
The planar $2 \times p \times q$ transportation polytopes are exactly the $p \times q$ transportation polytopes with bounded entries.
\end{lemma}
\begin{proof} 
Every planar $2 \times p \times q$ transportation polytope
\[P\ =\ \left\{(x_{i,j,k})\in\R_{\geq 0}^{2\times p\times q}:
\sum_k x_{i,j,k}=U_{i,j},\ \sum_j x_{i,j,k}=V_{i,k},\ x_{1,j,k}+x_{2,j,k}=W_{j,k}\right\}\]
is linearly isomorphic to a $p \times q$ transportation polytope with bounded entries,
\[Q\ =\ \left\{(x_{1,j,k})\in\R_{\geq 0}^{1 \times p\times q}:
\sum_k x_{1,j,k}=U_{1,j},\ \sum_j x_{1,j,k}=V_{1,k},\ x_{1,j,k}\leq W_{j,k}\right\}\ ,\]
via the projection $\R^{2 \times p \times q} \rightarrow \R^{p \times q}$ taking $(x_{i,j,k}) \mapsto (x_{1,j,k})$, which maps $P$ bijectively onto $Q$. Conversely, every $p \times q$ transportation polytope $Q$ with bounded entries is linearly isomorphic to a planar $2 \times p \times q$ transportation polytope $P$ by defining 
$U_{2,j}:=(\sum_k W_{j,k})-U_{1,j}$ for $j=1,\dots,p$ 
and 
$V_{2,k}:=(\sum_j W_{j,k})-V_{1,k}$ for $k=1,\dots,q$.
\end{proof}

\begin{proof}[Proof of Theorem~\ref{theorem:22n}]
Consider any planar $2\times 2\times p$ transportation polytope
\[
P\ =\ \left\{(x_{i,j,k}) \in\R_{\geq 0}^{2\times 2\times p} : \sum_k x_{i,j,k}=U_{i,j},\
x_{i,1,k}+x_{i,2,k}=V_{i,k},\ x_{1,j,k}+x_{2,j,k}=W_{j,k} \right\}
\]
defined by $1$-marginals. The equations of the last two types imply that for each $k$ we can express all the $x_{i,j,k}$ in
terms of $x_{1,1,k}$ as follows:
\begin{eqnarray*}
x_{1,2,k} &=& V_{1,k} - x_{1,1,k}, \cr
x_{2,1,k} &=& W_{1,k} - x_{1,1,k}, \\
x_{2,2,k} &=& x_{1,1,k} + W_{2,k} - V_{1,k} = x_{1,1,k} + V_{2,k} - W_{1,k}.
\end{eqnarray*}
In particular, $P$ is linearly isomorphic to its projection
\[
Q \ =\ \left\{(x_{1,1,k})  \in \R_{\geq 0}^{2\times 2\times p}:  \alpha_k \leq x_{1,1,k}\leq \beta_k, \sum_k x_{1,1,k}=U_{1,1} \right\}\ ,
\]
where $\alpha_k= \max\{0, W_{1,k} - V_{2,k}\}=\max\{0, V_{1,k} -
W_{2,k}\}$ and $\beta_k=\min \{V_{1,k},W_{1,k}\}$. Now, by applying a
translation to $Q$, there is no loss of generality in assuming that
$\alpha_k=0$ for all $k$. Then $Q$ is a $1$-way transportation polytope
with bounded entries, isomorphic (by Lemma~\ref{lemma:plan22n}) to a
$p \times 2$ classical transportation polytope.

Conversely, any $p \times 2$ classical transportation polytope of the form
\[
Q\ =\ \left\{(x_{j,k})\in\R_{\geq 0}^{2\times p}: \sum_k x_{j,k}=u_{j},\
x_{1,k} + x_{2,k}=v_{k}\right\}
\]
defined by marginals $u = (u_1,u_2) \in \R_{\geq 0}^2$ and $v = (v_1,\ldots,v_k,\ldots,v_p) \in \R_{\geq 0}^p$ is linearly isomorphic to the following planar $2\times 2\times p$
transportation polytope:
\[
P\ =\ \left\{(x_{i,j,k}) \in\R_{\geq 0}^{2 \times 2 \times p}: 
\begin{tabular}{l}
$\sum_k x_{1,j,k}=\sum_k x_{2,3-j,k}=u_{j}$, \cr
$\sum_i x_{i,1,k}= \sum_i x_{i,2,k}= 
\sum_j x_{1,j,k}= \sum_j x_{2,j,k}=v_{k}$
\end{tabular}
\right\}\ .
\]
The equations relating the solutions of $Q$ to those of $P$ are
$x_{j,k}=x_{1,j,k}=x_{2,3-j,k}$.
\end{proof}
The above result is the best possible since the list of $3 \times 3$ classical transportation problems presented in Table~\ref{table:classical-3-3} on page~\pageref{table:classical-3-3} is not the same as the list of $2 \times 3 \times 3$ planar transportation polytopes presented in Table~\ref{table:planar-2-3-3} on page~\pageref{table:planar-2-3-3}.

\section{Regular triangulations and Birkhoff polytopes}\label{section:regulartriangulationsBirkhoff}

The computation of triangulations of the $p$th Birkhoff polytope is related to the problem of generating a $p \times p$ doubly stochastic matrix uniformly at random (see, e.g.,~\cite{Chan:VolumePolytope}, or~\cite{Cappellini:Random-bistochastic} where Cappellini et al. use a probability measure on $B_p$ to estimate the volume of the Birkhoff polytope). In this section, we study the existence of non-regular triangulations for Birkhoff polytopes. (Recall that the $p$th Birkhoff polytope $B_p$, introduced in Definition~\ref{definition:BirkhoffPolytope} on page~\pageref{definition:BirkhoffPolytope}, is the convex hull of the $p!$ permutation matrices of size $p \times p$.) In 1996, De Loera (see~\cite{De-Loera:Nonregular-triangulations}) proved that the product $\Delta_3 \times \Delta_3$ of two tetrahedra has non-regular triangulations. Using this, one can show that the Birkhoff polytope $B_p$ has non-regular triangulations whenever $p \geq 8$. De Loera, Rambau, and Santos (see~\cite{DeLoera:Triangulations}) recently used the results in~\cite{De-Loera:Nonregular-triangulations} to show that the Birkhoff polytopes $B_6$ and $B_7$ have non-regular triangulations.

Here, we prove that non-regular triangulations of the Birkhoff polytope $B_4$ exist. The polytope $B_2$ is a line segment and the polytope $B_3$ is a $4$-polytope with $6$ vertices, which implies that it has only two triangulations, both of which are regular (see, e.g., Section 2.4 of~\cite{DeLoera:Triangulations}). Thus, by computing ranks, we know that the polytope $B_p$ only has regular triangulations when $p \leq 3$. Combined with the results in this section, this tells us that the Birkhoff polytope $B_p$ has non-regular triangulations if and only if $p \geq 4$.

We first define triangulations, following the book~\cite{DeLoera:Triangulations}. Since we only consider distinct points in convex position (i.e., no point is the convex combination of the remaining points), we will not need to give the definitions in their most technical forms here.
\begin{definition}
Let $A$ be a finite collection of distinct points in $\R^c$ in convex position. Let $P = \conv(A)$ and denote by $d$ the dimension of the polytope $P$. Then, a collection $\mathcal{T}$ of simplices $\sigma_1,\ldots,\sigma_S$ is a \defn{triangulation} of $A$ if:
\begin{enumerate}
\item The collection $\mathcal{T}$ is a \defn{simplicial complex}: that is, if the simplex $\sigma$ is in $\mathcal{T}$ and $\tau$ is a face of $\sigma$, then $\tau$ is in $\mathcal{T}$.
\item The vertex set of each $d$-simplex $\sigma \in \mathcal{T}$ is a subset of $A$.
\item The union of all simplices in $\mathcal{T}$ is $P$, namely: $\bigcup_{i=1}^S \sigma_i = \conv(A)$.
\item For any two simplices $\sigma_i$ and $\sigma_j$ in $\mathcal{T}$, their intersection $\tau = \sigma_i \cap \sigma_j$ belongs to $\mathcal{T}$ and is a face of both $\sigma_i$ and $\sigma_j$
\end{enumerate}
\end{definition}
In this section, we study triangulations that are regular. Let $w : A \rightarrow \R$ be a function that assigns a real number $w(a)$ to each point $a$ in $A$. The function $w$ is called a \defn{lifting}. The \defn{lower envelope} of the lifting $w$ of $A$ are the facets of the polytope $\conv\{(a,w(a)) \mid a \in A\}$ that are visible from below. (The facets of the polytope $\conv\{(a,w(a)) \mid a \in A\}$ that are ``visible from below'' are those facets whose outward normal vector $\xi = (\xi_1,\ldots,\xi_{c+1})$ has the property that the last coordinate $\xi_{c+1} \leq 0$.
\begin{definition}
Let $A$ be a finite collection of distinct points in $\R^c$ in convex position. A triangulation $\mathcal{T}$ of $A$ is called \defn{regular} if it can be obtained by projecting the lower envelope of a lifting $w$ of $A$ to $\R^{c+1}$.
\end{definition}

\begin{theorem}
The fourth Birkhoff polytope $B_4$ has a non-regular triangulation.
\end{theorem}
To prove this theorem, we consider the face $F$ of the Birkhoff polytope $B_4$ given by the additional equations $x_{1,1} = x_{4,4} = 0$. The $14$ vertices of this $7$-dimensional face $F$ of $B_4$ are given in Figure~\ref{figure:B4facevertices}.
\begin{figure}[hbt]
\begin{equation*}
X_a = \left[\begin{array}{cccc}0&1&0&0\\0&0&0&1\\0&0&1&0\\1&0&0&0\end{array}\right] \qquad
X_b = \left[\begin{array}{cccc}0&1&0&0\\0&0&1&0\\0&0&0&1\\1&0&0&0\end{array}\right] \qquad
X_c = \left[\begin{array}{cccc}0&1&0&0\\0&0&0&1\\1&0&0&0\\0&0&1&0\end{array}\right]
\end{equation*}
\begin{equation*}
X_d = \left[\begin{array}{cccc}0&1&0&0\\1&0&0&0\\0&0&0&1\\0&0&1&0\end{array}\right] \qquad
X_e = \left[\begin{array}{cccc}0&0&1&0\\0&0&0&1\\0&1&0&0\\1&0&0&0\end{array}\right] \qquad
X_f = \left[\begin{array}{cccc}0&0&1&0\\0&0&0&1\\1&0&0&0\\0&1&0&0\end{array}\right]
\end{equation*}
\begin{equation*}
X_g = \left[\begin{array}{cccc}0&0&1&0\\0&1&0&0\\0&0&0&1\\1&0&0&0\end{array}\right] \qquad
X_h = \left[\begin{array}{cccc}0&0&1&0\\1&0&0&0\\0&0&0&1\\0&1&0&0\end{array}\right] \qquad
X_i = \left[\begin{array}{cccc}0&0&0&1\\0&0&1&0\\0&1&0&0\\1&0&0&0\end{array}\right]
\end{equation*}
\begin{equation*}
X_j = \left[\begin{array}{cccc}0&0&0&1\\0&0&1&0\\1&0&0&0\\0&1&0&0\end{array}\right] \qquad
X_{k} = \left[\begin{array}{cccc}0&0&0&1\\0&1&0&0\\1&0&0&0\\0&0&1&0\end{array}\right] \qquad
X_{l} = \left[\begin{array}{cccc}0&0&0&1\\0&1&0&0\\0&0&1&0\\1&0&0&0\end{array}\right]
\end{equation*}
\begin{equation*}
X_{m} = \left[\begin{array}{cccc}0&0&0&1\\1&0&0&0\\0&1&0&0\\0&0&1&0\end{array}\right] \qquad
X_{n} = \left[\begin{array}{cccc}0&0&0&1\\1&0&0&0\\0&0&1&0\\0&1&0&0\end{array}\right].
\end{equation*}
\caption{The $14$ vertices of the face $F$.}\label{figure:B4facevertices}
\end{figure}

After homogenization (see, e.g.,~\cite{DeLoera:Triangulations} or~\cite{Ziegler:Lectures}), we obtain $\mathcal{A}$, the vector configuration given by the columns of the matrix shown in Figure~\ref{figure:vectorconfigA-B4}.
\begin{figure}[hbt]
{\singlespacing
\begin{equation}
\left[
\begin{array}{cccccccccccccc}
1&1&1&1&1&1&1&1&1&1&1&1&1&1\\
0&0&0&0&0&0&0&0&0&0&0&0&0&0\\
1&1&1&1&0&0&0&0&0&0&0&0&0&0\\
0&0&0&0&1&1&1&1&0&0&0&0&0&0\\
0&0&0&0&0&0&0&0&1&1&1&1&1&1\\
0&0&0&1&0&0&0&1&0&0&0&0&1&1\\
0&0&0&0&0&0&1&0&0&0&1&1&0&0\\
0&1&0&0&0&0&0&0&1&1&0&0&0&0\\
1&0&1&0&1&1&0&0&0&0&0&0&0&0\\
0&0&1&0&0&1&0&0&0&1&1&0&0&0\\
0&0&0&0&1&0&0&0&1&0&0&0&1&0\\
1&0&0&0&0&0&0&0&0&0&0&1&0&1\\
0&1&0&1&0&0&1&1&0&0&0&0&0&0\\
1&1&0&0&1&0&1&0&1&0&0&1&0&0\\
0&0&0&0&0&1&0&1&0&1&0&0&0&1\\
0&0&1&1&0&0&0&0&0&0&1&0&1&0\\
0&0&0&0&0&0&0&0&0&0&0&0&0&0\\
\end{array}
\right].
\end{equation}}
\caption{The vector configuration $\mathcal{A}$, as the columns of a matrix.}\label{figure:vectorconfigA-B4}
\end{figure}

Let $\mathcal{T}$ be the triangulation of the vector configuration $\mathcal{A}$ (respectively, point configuration $\{X_a,\ldots,X_n\}$) whose highest-dimensional cones (respectively, simplices) are  displayed in Figure~\ref{figure:B4triangulationsimplices}.
\begin{figure}[hbt]
\fbox{\begin{minipage}{5.0in}
\doublespacing
\begin{equation*}
\begin{array}{llll}
\sigma_{1} : abcdefgi,&
\sigma_{2} : abdefghi,&
\sigma_{3} : acdefgik,&
\sigma_{4} : adefghik,\\
\sigma_{5} : aefghikl,&
\sigma_{6} : adeghikl,&
\sigma_{7} : afghijkl,&
\sigma_{8} : acdefikm,\\
\sigma_{9} : adefhikm,&
\sigma_{10} : acdfijkm,&
\sigma_{11} : adfhijkm,&
\sigma_{12} : aefhiklm,\\
\sigma_{13} : deghiklm,&
\sigma_{14} : adehiklm,&
\sigma_{15} : acdfjkmn,&
\sigma_{16} : adfhjkmn,\\
\sigma_{17} : aefhilmn,&
\sigma_{18} : afhiklmn,&
\sigma_{19} : adhiklmn,&
\sigma_{20} : afhijkmn,\\
\sigma_{21} : adhijkmn,&
\sigma_{22} : afhijkln,&
\sigma_{23} : abcdfijk,&
\sigma_{24} : abcdfgik,\\
\sigma_{25} : abfghijk,&
\sigma_{26} : abghijkl,&
\sigma_{27} : abdfhijk,&
\sigma_{28} : abdfghik,\\
\sigma_{29} : abdghikl,&
\sigma_{30} : abhijkln,&
\sigma_{31} : abdhikln,&
\sigma_{32} : abdhijkn.
\end{array}
\end{equation*}
\vskip0.3in
\end{minipage}}
\caption{The highest-dimensional cones $\sigma_1,\ldots, \sigma_{32}$ in the triangulation $\mathcal{T}$}\label{figure:B4triangulationsimplices}
\end{figure}
In Figure~\ref{figure:B4triangulationsimplices}, each $7$-dimensional simplex (or cone over the simplex, for the vector configuration) is defined by the indices of its vertices. For instance, the $7$-dimensional simplex $\sigma_1$ is the convex hull of the eight points $X_a$, $X_b$, $X_c$, $X_d$, $X_e$, $X_f$, $X_g$, and $X_i$. In Appendix~\ref{appendix:triangB4}, we prove that $\mathcal{T}$ is a triangulation of $\mathcal{A}$. (See page~\pageref{appendix:triangB4}.)

We will prove that the triangulation $\mathcal{T}$ of $\mathcal{A}$ is non-regular by considering a Gale transform of $\mathcal{A}$. A Gale transform $\mathcal{B}$ of $\mathcal{A}$ is given by the columns of the matrix shown in Figure~\ref{figure:B4gale}.
\begin{figure}[hbt]
\begin{equation}\label{equation:B4GaleTransformMatrix}
\left[
\begin{array}{cccccccccccccc}
 0& 0& 1&-1& 0&-1& 0& 1& 0& 0& 0& 0& 0& 0\\
 0& 0& 0& 0& 1&-1& 0& 0&-1& 1& 0& 0& 0& 0\\
 0& 1&-1& 0& 1& 0&-1& 0&-1& 0& 1& 0& 0& 0\\
-1& 1& 0& 0& 1& 0&-1& 0&-1& 0& 0& 1& 0& 0\\
 0& 1& 0&-1& 0& 0& 0& 0&-1& 0& 0& 0& 1& 0\\
-1& 1& 1&-1& 1&-1& 0& 0&-1& 0& 0& 0& 0& 1
\end{array}
\right].
\end{equation}
\caption{The Gale transform $\mathcal{B}$.}\label{figure:B4gale}
\end{figure}

We use the same indices for the vectors in $\mathcal{B}$ that we used for $\mathcal{A}$. For example, the first column in~\eqref{equation:B4GaleTransformMatrix} defines the vector $Y_a = (0,0,0,-1,0,-1) \in \R^6$, the second column defines the vector $Y_b = (0,0,1,1,1,1) \in \R^6$, the third column defines the vector $Y_c = (1,0,-1,0,0,1) \in \R^6$, and so on. We define $32$ full-dimensional simple cones $C_1, \ldots, C_{32}$ in $\R^6$. Each cone is defined by ``complementation of indices'' of the corresponding simplex. For example since $\sigma_1$ is the convex hull the $X_z$ with $z$ in $\{a,b,c,d,e,f,g,i\}$, the cone $C_1$ is generated by the vectors $Y_z$, with $z$ in $\{a,b,\ldots,n\} \setminus \{a,b,c,d,e,f,g,i\} = \{h,j,k,l,m,n\}$. (That is to say, the cone $C_1$ is just the positive orthant $\{y \in \R^6 \mid y_1,\ldots,y_6 \geq 0\}$.) Similarly, since the simplex $\sigma_2$ is given by the indices $abdefghi$, the cone $C_2$ is the set of all positive linear combinations of $Y_c$, $Y_j$, $Y_k$, $Y_l$, $Y_m$, and $Y_n$.

It is known (see, e.g., Chapter 5 of~\cite{DeLoera:Triangulations},~\cite{Lee:Regular-triangulations}, or~\cite{Thomas:Lectures-in-Geometric}, and part (\ref{item:regulartriangulationscorrespondence}) of Lemma~\ref{lemma:secondary} on page~\pageref{lemma:secondary}) that the triangulation $\mathcal{T}$ of $\mathcal{A}$ is regular if and only if the intersection
\begin{equation*}
D = \bigcap_{i=1}^{32} \operatorname{relint}(C_i)
\end{equation*}
is non-empty. It is easy to check the following facts:
\begin{itemize}
\item All points $y$ in the cone $C_{1}$ satisfy the inequality $y_3 \geq 0$.
\item All points $y$ in the cone $C_{10}$ satisfy the inequality $-y_2 -y_5 + y_6 \geq 0$.
\item All points $y$ in the cone $C_{13}$ satisfy the inequality $-y_4 + y_5 \geq 0$.
\item All points $y$ in the cone $C_{32}$ satisfy the inequality $y_2 - y_3 + y_4 - y_6 \geq 0$.
\end{itemize}
In fact, each of the four inequalities above is a facet-defining inequality of their respective cone. Now, if there is a point $y \in \R^6$ that belongs to $D$, then $y$ must satsify:
\begin{align*}
      y_3                &> 0,\\
-y_2           -y_5 +y_6 &> 0,\\
          -y_4 +y_5      &> 0,\\
y_2  -y_3 +y_4      -y_6 &> 0.
\end{align*}
But the sum of the left-hand sides is the same as the sum of the right-hand sides, so this is an inconsistent system. In particular, $D$ is empty, which proves the triangulation $\mathcal{T}$ is non-regular. (In fact, with further analysis one can show that the intersection of all $32$ cones is the single ray $\{ y \in \R^6 \mid \lambda (0, 0, 0, 0, 0, 1), \lambda \geq 0\}$.) We can complete this triangulation $\mathcal{T}$ of $\mathcal{A}$ to a triangulation of $B_4$ by pulling or placing (see~\cite{DeLoera:Triangulations}). Therefore the fourth Birkhoff polytope $B_4$ has non-regular triangulations. As an immediate corollary, we get:
\begin{theorem}
The Birkhoff polytope $B_p$ has non-regular triangulations for $p \geq 4$, and all triangulations of $B_p$ are regular for $p < 4$.
\end{theorem}
\begin{remark}
Non-regularity of a triangulation can also be checked using Theorem~8.3 in~\cite{Sturmfels:GrobnerPolytopes}. The theorem, which is proved using linear program duality, says that the regular triangulations of $\mathcal{A}$ are the initial complexes of the toric ideal $I_\mathcal{A}$. (To show that a triangulation is non-regular, one can show that the triangulation $\Delta_\prec(I_A)$ is not $\mathcal{T}$. See~\cite{Sturmfels:GrobnerPolytopes} for details.)
\end{remark}

\section{Many Hirsch-sharp transportation polytopes}\label{section:manyHirschsharpTPs}

We already have seen the results of~\cite{Fritzsche99morepolytopes},~\cite{Holt:Many-polytopes}, and~\cite{Holt:Hsharpd7}, summarized in Theorem~\ref{theorem:hirsch-sharp}, which says that there are ``many Hirsch-sharp polytopes.'' That is to say, there are Hirsch-sharp $d$-polytopes with $n > d$ facets whenever $d \geq 7$. Does this remain true if we only look at classical and $3$-way axial transportation polytopes? In this section, we prove that there are many Hirsch-sharp transportation polytopes. Recall that a polytope is Hirsch-sharp if it meets the Hirsch Conjecture with equality.

First, we prove a ``permuted version'' of Lemma~\ref{lemma:littleeddielemma}. It says that the northwest corner rule algorithm can be done out of order by specifying permutations of the rows and columns. See Exercise 17 in Chapter 6 of~\cite{Yemelichev:Polytopes}. To make the proof easier, we prove it for non-degenerate transportation polytopes.
\begin{lemma}\label{lemma:orderNWC}
Let the marginals $u \in \R_{\geq 0}^p$ and $v \in \R_{\geq 0}^q$ define a non-degenerate $p \times q$ classical transportation polytope $P \not= \emptyset$. Let $\sigma$ be a permutation of $[p]$ and $\tau$ be a permutation of $[q]$. Then the polytope $P$ has a vertex $x$ with $x_{\sigma(p),\tau(q)} > 0$.
\end{lemma}
\begin{proof}
We construct a point $x \in P$ using the following modified form of the well-known northwest corner rule algorithm. Let $\sigma$ be any permutation on $[p]$ and let $\tau$ be any permutation of $[q]$.

Let $x_{\sigma(p),\tau(q)} = \min \{u_{\sigma(p)}, v_{\tau(q)}\}$. If the minimum is obtained at $u_{\sigma(p)}$, set $x_{\sigma(p),j} = 0$ for all $j \not= \tau(q)$ and replace $v_{\tau(q)}$ with $v_{\tau(q)} - u_{\sigma(p)}$. The rest of the point $x = (x_{i,j})$ is obtained recursively as a point in a $(p-1) \times q$ classical transportation polytope. Similarly, if the minimum is obtained at $v_{\tau(q)}$, set $x_{i,\tau(q)} = 0$ for all $i \not= \sigma(p)$ and replace $u_{\sigma(p)}$ with $u_{\sigma(p)} - v_{\tau(q)}$. The rest of the point $x$ is obtained as a point in a $p \times (q-1)$ transportation polytope.
\end{proof}
Note that this generalizes Lemma~\ref{lemma:littleeddielemma}, which is the special case when $\sigma$ and $\tau$ are both their respective identity permutations.
\begin{example}
\begin{figure}[hbt]
\begin{tabular}{|c|c|c|c|c|}
\hline
$x_{1,1}$ & $x_{1,2}$ & $x_{1,3}$ & $x_{1,4}$ & $x_{1,5}$ \\ \hline
$x_{2,1}$ & $x_{2,2}$ & $x_{2,3}$ & $x_{2,4}$ & $x_{2,5}$ \\ \hline
$x_{3,1}$ & $x_{3,2}$ & $x_{3,3}$ & $x_{3,4}$ & $x_{3,5}$ \\ \hline
\end{tabular}
\caption{Layout of the coordinates for the generalized Birkhoff polytope of size $3 \times 5$.}\label{figure:exampleTP35}
\end{figure}
As an example, let $P$ be the $3 \times 5$ generalized Birkhoff polytope. That is to say, we are looking at the polytope whose variables are $x_{i,j}$ arranged according to Figure~\ref{figure:exampleTP35} satisfying the summation conditions that each row must sum to $5$ and each column must add up to $3$.
Let $\sigma$ be the permutation on $[3]= \{1,2,3\}$
\begin{equation}\label{equation:sigma-permutation}
\sigma = \left(
\begin{array}{ccccc}
1&2&3\\
2&3&1
\end{array}
\right),
\end{equation}
and let $\tau$ be the permutation on $[5]$
\begin{equation}\label{equation:tau-permutation}
\tau = \left(
\begin{array}{ccccc}
1&2&3&4&5\\
5&1&2&3&4
\end{array}
\right).
\end{equation}
To fill in the table shown in Figure~\ref{figure:exampleTP35} according to the proof of Lemma~\ref{lemma:orderNWC}, we start with the $(\sigma(p),\tau(q))=(\sigma(3),\tau(5))=(1,4)$ entry of the table. This entry is highlighted in boldface in Figure~\ref{figure:exampleTP35bold}.
\begin{figure}[hbt]
\begin{tabular}{|c|c|c|c|c|}
\hline
$x_{1,1}$ & $x_{1,2}$ & $x_{1,3}$ & {\bf $x_{1,4}$} & $x_{1,5}$ \\ \hline
$x_{2,1}$ & $x_{2,2}$ & $x_{2,3}$ & $x_{2,4}$ & $x_{2,5}$ \\ \hline
$x_{3,1}$ & $x_{3,2}$ & $x_{3,3}$ & $x_{3,4}$ & $x_{3,5}$ \\ \hline
\end{tabular}
\caption{The first entry to fill in is $x_{1,4}$.}\label{figure:exampleTP35bold}
\end{figure}
What is the largest entry that can fit here? Since the rows must sum to $5$ and the columns must sum to $3$, the largest possible value here is $3$. Then, since the minimum of $\min\{u_{\sigma(p),\tau(q)}\}$ was obtained at $v_{\tau(q)}$, the rest of our point is found as a point in a $3 \times 4$ transportation polytope. The next entry we examine is the $(\sigma(p),\tau(q-1))=(\sigma(3),\tau(4))=(1,3)$ entry (which we can see has to be the minimum of $5-3=2$ and $3$) of our $3 \times 5$ table. When the whole algorithm completes, we obtain the vertex shown in Figure~\ref{figure:exampleTP35permutedvertex}.
\begin{figure}[hbt]
\begin{tabular}{|c|c|c|c|c|}
\hline
$0$ & $0$ & $2$ & $3$ & $0$ \\ \hline
$2$ & $0$ & $0$ & $0$ & $3$ \\ \hline
$1$ & $3$ & $1$ & $0$ & $0$ \\ \hline
\end{tabular}
\caption{The vertex obtained by using $\sigma$ and $\tau$.}\label{figure:exampleTP35permutedvertex}
\end{figure}
Compare this vertex to the one obtained by using the identity permutations. See Figure~\ref{figure:exampleTP35nonpermutedvertex}.
\begin{figure}[hbt]
\begin{tabular}{|c|c|c|c|c|}
\hline
$3$ & $2$ & $0$ & $0$ & $0$ \\ \hline
$0$ & $1$ & $3$ & $1$ & $0$ \\ \hline
$0$ & $0$ & $0$ & $2$ & $3$ \\ \hline
\end{tabular}
\caption{The vertex obtained by the normal northwest corner rule algorithm.}\label{figure:exampleTP35nonpermutedvertex}
\end{figure}
\end{example}

\begin{theorem}\label{theorem:strongTPHirschSharp}
Let $p$ and $q$ be relatively prime with $\min\{p,q\} \geq 3$ and $q > 2p$. Then the $p \times q$ generalized classical Birkhoff polytope $P$ is Hirsch-sharp.
\end{theorem}
\begin{proof}
Let $P$ be the $p \times q$ generalized classical Birkhoff polytope. Then the dimension of $P$ is $(p-1)(q-1)$ and the number of facets is bounded above by $pq$. By Theorem~3.1 in Chapter~6 of~\cite{Yemelichev:Polytopes}, the polytope $P$ has exactly $pq$ facets. Thus, the conjectured Hirsch bound for the diameter of $P$ is $p+q-1$.

We now construct two vertices $v$ and $v'$ of $P$ whose distance is at least $p+q-1$. Let $v$ be the vertex obtained from Lemma~\ref{lemma:orderNWC} using the identity permutations for $\sigma$ and $\tau$. Let $\sigma'$ be the identity permutation on $[p]$. If $q$ is odd, let $\tau'$ be the permutation of $[q]$ that sends $1,\ldots,q$ (respectively) to $\lceil\frac{q}{2}\rceil,\ldots,q,1,\ldots,\lfloor\frac{q}{2}\rfloor$ (respectively). If $q$ is even, let $\tau'$ be the permutation that sends $1,\ldots,q$ (respectively) to $\frac{q}{2}+1,\ldots,q,1,\ldots,\frac{q}{2}$ (respectively). Let $v'$ be the vertex obtained from Lemma~\ref{lemma:orderNWC} using the permutations $\sigma'$ and $\tau'$.

Since $p$ and $q$ are coprime, the transportation polytope $P$ is non-degenerate. Now we note that the supports of $v$ and $v'$ are disjoint (which is easy to see since $q > 2p$). Thus, the distance between the vertices $v$ and $v'$ is at least $|\supp(v)|=|\supp(v')|$, which (by Corollary~\ref{corollary:nondegeneratesupportsize}) is $p+q-1$ since $P$ is non-degenerate.
\end{proof}
\begin{remark}
The condition $q > 2p$ can be dropped in Theorem~\ref{theorem:strongTPHirschSharp}, but the proof becomes more difficult. (The definition for the permutations $\sigma$ and $\tau$ need to be given with more caution.) The case of what occurs when $p$ and $q$ are close to each other is clear by example: see Figures~\ref{figure:exampleTP35permutedvertex} and~\ref{figure:exampleTP35nonpermutedvertex} for the $3 \times 5$ case. The two vertices shown in these figures have disjoint support. Thus, in this example, we see that the diameter of the generalized $3 \times 5$ Birkhoff polytope is at least $7$. So, we have constructed a Hirsch-sharp $8$-polytope with $15$ facets.
\end{remark}

\section{\texorpdfstring{The diameter of $p \times 2$ transportation polytopes}{The diameter of p by 2 transportation polytopes}}\label{section:diameter-p2}

In this section, we show that the Hirsch Conjecture holds for $p \times 2$ transportation polytopes.
\begin{theorem}\label{theorem:thm_2n}
Let $P \not= \emptyset$ be a classical transportation polytope of size $p \times 2$ with $n \leq 2p$ facets. Then, the dimension of $P$ is $d=p-2$ and the diameter of $P$ is at most $n - d$.
\end{theorem}

To prove this theorem, we note that the coordinate-erasing projection of $P$ to the coordinates $x_{1,1}, x_{2,1}, \ldots, x_{p,1}$ of the first column shows that $P$ is the intersection of a hyperplane with a rectangular prism. (In particular, if the intervals are all equal and one has a cube, then the Minkowski sum of two consecutive hypersimplices $D(p,i)$ and $D(p,i+1)$ can be realized as a transportation polytope of size $p \times 2$.) After an affine transformation, the polytope $P$ is the intersection of a hyperplane and a cube. (The transformation takes the cube $[0,u_1] \times \cdots \times [0,u_p]$ to the cube $[0,1]^p$. That is to say, the $i$th coordinate $y_i$ in the cube $[0,1]^p$ is $x_{i,1}/u_i$.) Thus, to prove Theorem~\ref{theorem:thm_2n}, we will prove that the Hirsch Conjecture holds for polytopes obtained as the intersection of a cube and a hyperplane.

Fix a dimension $d \in \N$.  Let $H = \{x \in \R^d \mid a_1x_1 + \cdots + a_dx_d = b\}$ be the hyperplane determined by the non-zero normal vector $a = (a_1, \ldots, a_d)$ and the constant $b \in \R$.  Let $\Box_d$ denote $d$-dimensional cube with $0$-$1$ vertices. Then, let $P$ denote the polytope obtained as their intersection $P = \Box_d \cap H$.

If the dimension of the polytope $P$ is less than $d-1$, then $P$ is a face of $\Box_d$. In that case, $P$ itself is a cube of lower dimension, so we assume that the polytope $P$ is of dimension $d-1$. We may also assume that the polytope $P$ is not a facet of the $d$-cube, so that $H$ intersects the relative interior of $\Box_d$.

Without assuming any genericity, a simple dimension argument shows that the vertices of the polytope $P$ are either on the relative interior of an edge of the cube $\Box_d$ or are vertices of the cube. We assume that $H$ is sufficiently generic.  Then, no vertex of the cube $\Box_d$ will be a vertex of $P$. For each vertex $v$ of $P$, we define its \defn{side signature} $\sigma(v)$ to be a string of length $d$ consisting of the characters $*$, $0$, and $1$ by the following rule:
\begin{equation}
\sigma(v)_i = \left\{
\begin{array}{ll}
0 & \hbox{if } v_i = 0,\\
1 & \hbox{if } v_i = 1,\\
* & \hbox{if } 0 < v_i < 1.\\
\end{array}
\right.
\end{equation}
By genericity, it cannot be the case that there are two vertices of $P$ with the same side signature.  Indeed, if there were two distinct vertices $v$ and $w$ with the same side signature, then $P$ will contain the entire edge of the cube containing them both, and $v$ and $w$ will not be vertices.

Let $H_{i,0}$ denote the hyperplane $\{x \in \R^d \mid x_i = 0\}$ and let $H_{i,1}$ denote the hyperplane $\{x \in \R^d \mid x_i = 1\}$. If there is an $i \in [d]$ such that the hyperplane $H$ does not intersect $H_{i,0}$ nor $H_{i,1}$, then we can project $P$ to a lower-dimensional face of $I_d$. Thus, for each $i \in [d]$, we can assume that $H$ intersects at least one of $H_{i,0}$ or $H_{i,1}$.

Given two vertices $v=(v_1,\ldots,v_d)$ and $w=(w_1,\ldots,w_d)$ of $P = I_d \cap H$, we define the \defn{Hamming distance} between them based on their side signatures: 
\begin{equation}
\hamm(v,w) = \sum_{i=1}^d \hamm( \sigma(v)_i, \sigma(w)_i ),
\end{equation}
where
\begin{equation*}
\hamm(0,1) = \hamm(1,0) = \hamm(1,*) = \hamm(*,1) = \hamm(0,*) = \hamm(*,0) = 1
\end{equation*}
and
\begin{equation*}
\hamm(0,0)=\hamm(1,1)=\hamm(*,*) = 0.
\end{equation*}

\begin{lemma}
Let $P$ defined as above using a sufficiently-generic hyperplane $H$.  Let $v$ and $w$ be two vertices of $P$.  Let $f(P)$ denote the number of facets of $P$.

If $v$ and $w$ have the $*$ in the same coordinate and $P$ does not intersect either of the two facets in that direction, i.e., there is an $i$ such that $\sigma(v)_i=*=\sigma(w)_i$ and $P \cap H_{i,1} = \emptyset = P \cap H_{i,0}$, then $f(P) \geq (d-1) + \hamm(v,w)-1$.

Otherwise, $f(P) \geq (d-1) + \hamm(v,w)$.
\end{lemma}
\begin{proof}
Without loss of generality, we can assume that the side signature $\sigma(v)$ of the vertex $v$ is $(*,0,0,\ldots,0)$ and that the side signature $\sigma(w)$ of the vertex $w$ is either of the form $(0,*,0,0,\ldots,0,1,1,\ldots,1)$ with $k \geq 0$ trailing ones or of the form $(*,0,0,\ldots,0,1,1,\ldots,1)$ with $k \geq 1$ trailing ones, after applying a suitable rotation to the cube.

In the first case, $\hamm(v,w)=k+1$ and we have at least $d$ ``$0$-facets'' and $k$ ``$1$-facets.''

In the second case, we have $d-1$ ``$0$-facets'', $k$ ``$1$-facets'' and (unless there is an $i$ such that $\sigma(v)_i=*=\sigma(w)_i$ and $P \cap H_{i,1} = \emptyset = P \cap H_{i,0}$) at least one more facet.  Thus, $f(P) \geq (d-1) + k + 1 = (d-1) + \hamm(v,w)$, unless we are in the special case, in which case $f(P) \geq (d-1) + k + 1 - 1$.
\end{proof}

\begin{lemma}
Let $P$ defined as above using a sufficiently-generic hyperplane $H$.  Let $v$ and $w$ be two vertices of $P$.  Then, there is a pivot from the vertex $v$ to a vertex $v'$ with $\hamm(v',w) = \hamm(v,w)-1$.
\end{lemma}
\begin{proof}
Without loss of generality, we can assume that $\sigma(v) = (*,0,0,\ldots,0)$ and that $\sigma(w)$ is either $(*,1,1,\ldots,1)$ or $(1,1,\ldots,1,*)$.

If the side signature $\sigma(w)$ of $w$ is $(*,1,1,\ldots,1)$, performing a pivot on the vertex $v$ in any one of the $d-1$ last coordinates reduces the Hamming distance.

Otherwise, the side signature $\sigma(w)$ of $w$ is $(1,1,\ldots,1,*)$. We now describe what can occur when pivoting from the vertex $v$ to a new vertex $v'$. We claim that at least one of the $d-1$ possible pivots on the vertex $v$ does not put a $0$ in the first coordinate of the side signature $\sigma(v')$ of the new vertex $v'$. Otherwise, the hyperplane $H$ cuts the polytope $P$ as a vertex figure: that is to say, the polytope $P$ cuts the corner $(1,0,\ldots,0)$ of the cube. See Figure~\ref{figure:cubehyperplanevertexfigure} for a picture.
\begin{figure}[hbt]
  \begin{center}
    \includegraphics[scale=0.9]{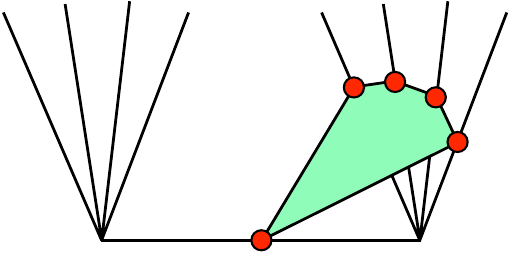}
    \caption{The vertex $v$ on the horizontal axis (the $x_1$ coordinate increases moving to the right) and its neighboring vertices on orthogonal edges of the cube.} \label{figure:cubehyperplanevertexfigure}
  \end{center}
\end{figure}

The remaining kind of pivots on $v$ that result in a new vertex $v'$ give side signatures $\sigma(v')$ of one of the following three forms:
\begin{enumerate}
\item The signature $\sigma(v')$ of the neighbor $v'$ of the vertex $v$ could be \[(1,0,0,\ldots,*,0,\ldots,0),\] which reduces the Hamming distance by one.
\item The signature $\sigma(v')$ of the neighbor $v'$ of the vertex $v$ could be \[(*,0,0,\ldots,0,1,0,0,\ldots,0),\] which reduces the Hamming distance by one.
\item Otherwise, \emph{one} remaining pivot could give the side signature \[(1,0,0,\ldots,0,*)\] for $\sigma(v')$.
\end{enumerate}
This third type of pivot does not reduce the Hamming distance. But if this is the only pivot that could give this and the first two kinds of pivots cannot be performed, then all of the remaining pivots are the kind that put $0$ in the first coordinate of the side signature $\sigma(v)$ of $v$. But this would imply that $H$ could not have intersected the hyperplane $H_{d+}$, and thus $P$ would be a $(d-1)$-cube with one vertex truncated.
\end{proof}
By applying an affine transformation to $P = \Box_d \cap H$, we obtain a $p \times 2$ classical transportation polytope, thus:
\begin{corollary}
The Hirsch Conjecture holds for $p \times 2$ classical transportation polytopes.
\end{corollary}

\begin{openproblem}
Prove that the Hirsch Conjecture holds for $p \times 3$ classical transportation polytopes.
\end{openproblem}
\begin{openproblem}
Prove that the Hirsch Conjecture holds for polytopes $P$ of the form $\Box_d \cap H_1 \cap H_2$, where $\Box_d$ is a $d$-dimensional cube and $H_1$ and $H_2$ are hyperplanes.
\end{openproblem}

\section[Diameter of $3$-way axial transportation polytopes]{\texorpdfstring{The diameter of $3$-way transportation polytopes defined by $1$-marginals}{The diameter of 3-way transportation polytopes defined by 1-marginals}}\label{section:axialdiam}

In this section, we prove Theorem~\ref{theorem:main}, which gives the first quadratic upper bound on the diameter of all $3$-way $p \times q \times s$ transportation polytopes defined by $1$-marginals. The contents of this section is joint work with De Loera, Onn, and Santos.

Here we consider a $3$-way $p \times q \times s$ transportation polytope $P$ defined by certain $1$-marginal vectors $u$, $v$, and $w$.  Recall that for bounding its diameter there is no loss of generality in assuming $P$ non-degenerate, that is, that $u$, $v$ and $w$ are sufficiently generic.  In the non-degenerate case, at every vertex $V$ of our polytope, exactly $pqs-(p+q+s-2)$ variables are zero, and exactly $p+q+s-2$ are non-zero. As in the case of classical transportation polytopes, the set of triplets $(i,j,k)$ indexing non-zero variables will be called the \defn{support} of the vertex $V$, denoted $\supp(V)$.

We say that a vertex $V$ of the axial transportation polytope $P$ is \emph{well-ordered} if the
triplets $(i,j,k)$ that form its support are totally ordered with
respect to the following \emph{coordinate-wise partial order}:

\begin{equation} \label{equation:order}
(i,j,k)\le (i',j',k') \quad\hbox{ if } \qquad i\le i', \ 
\hbox{and } j\le j', \ \hbox{and}\ k\le k'.
\end{equation}
Observe that a set of $p+q+s-2$ triplets satisfying this must contain
exactly one triplet $(i,j,k)$ with $i+j+k=z$ for each $z=3,\dots,
p+q+s$. Actually, supports of well-ordered vertices are the
\emph{monotone staircases} from $(1,1,1)$ to $(p,q,s)$ in the $p
\times q \times s$ grid.

\begin{lemma}
\label{lemma:well-ordered}
If $x$, $y$ and $z$ are generic, then the axial transportation polytope $P$ has a unique
well-ordered vertex $\hat{V}$.
\end{lemma}

\begin{proof} Existence is guaranteed by the ``northwest corner rule
algorithm'', which fills the entries of the table in the prescribed
order (see the survey~\cite{Queyranne:MultiIndexTransportation} or Exercise 17 in Chapter 6 of~\cite{Yemelichev:Polytopes}). More explicitly: let $\hat{V}_{p,q,s}= \min\{x_p,y_q,z_s\}$. Genericity implies that the three values $x_p$,
$y_q$ and $z_s$ are different. Without loss of generality we assume
that the minimum is $z_s$. Then, our choice of $\hat{V}_{p,q,s}$ makes
$\hat{V}_{i,j,s}=0$ for every other pair $(i,j)$. The rest of our
vertex $\hat{V}$ is a vertex of the $p \times q\times (s-1)$ axial
transportation polytope with margins $u'=(u_1, \dots, u_{p-1}, u_l -
w_s)$, $v'=(v_1, \dots, v_{q-1}, v_q - w_s)$, and $w'=(w_1, \dots,
w_{n-1})$.
  
Uniqueness follows from the same argument, simply noticing that the
support of a well-ordered vertex always contains the entry $(p,q,s)$,
and no other entry from one of the three planes $(p,*,*)$, $(*,q,*)$
and $(*,*,s)$. This, recursively, implies that the vertex can be
obtained by the northwest corner rule.
\end{proof}

\begin{remark}
The proof of Lemma~\ref{lemma:well-ordered} above is the $3$-way analogue of Lemma~\ref{lemma:littleeddielemma}. Another proof of Lemma~\ref{lemma:well-ordered} can be obtained using the formalism of chambers developed in the previous sections: it is obvious (and is proved in~\cite{DeLoera:AllTriangulations}) that if $c$ denotes a chamber
  of $A$, and $T$ is a triangulation of $\cone(A)$, then there is a
  unique maximal-dimension simplex in $T$ that contains $c$. Thus,
  Lemma~\ref{lemma:well-ordered} follows from the fact that monotone
  staircases in the $p \times q \times s$ grid form a triangulation of
  the vector configuration $A_{p,q,s}$ of axial $p \times q \times s$
  transportation polytopes.  The latter is well-known, once we observe
  that $A_{p,q,s}$ is the vertex set of a product of three simplices.
  The triangulation in question is called the ``staircase
  triangulation'' of it (see Section 2.3 in Chapter 6 of~\cite{DeLoera:Triangulations}).
\end{remark}

\begin{example}\label{example:axial333}
  To illustrate Lemma~\ref{lemma:well-ordered} consider the
  non-degenerate $3 \times 3 \times 3$ axial transportation polytope
  the axial transportation polytope $P$ with:
\[
\sum_{j,k} x_{1,j,k} = 112\qquad 
\sum_{j,k} x_{2,j,k} = 18\qquad 
\sum_{j,k} x_{3,j,k} = 30\qquad 
\]
\[
\sum_{i,k} x_{i,1,k} = 40\qquad 
\sum_{i,k} x_{i,2,k} = 6\qquad 
\sum_{i,k} x_{i,3,k} = 114\qquad 
\]
\[
\sum_{i,j} x_{i,j,1} = 82\qquad 
\sum_{i,j} x_{i,j,2} = 44\qquad 
\sum_{i,j} x_{i,j,3} = 34\qquad 
\]

The unique well-ordered vertex $\hat{V}$ of the axial transportation polytope $P$ has the non-zero coordinates
$x_{(1,1,1)}=40$,
$x_{(1,2,1)}=6$,
$x_{(1,3,1)}=36$,
$x_{(1,3,2)}=30$,
$x_{(2,3,2)}=14$,
$x_{(2,3,3)}=4$, and
$x_{(3,3,3)}=30$.
Note that the non-zero entries of $\hat{V}$ are totally ordered (they 
are presented above in increasing order) with respect to \eqref{equation:order}.
Figure~\ref{figure:monostair} depicts the associated monotone staircase.
\begin{figure}[hbt]
  \begin{center}
    \includegraphics[scale=0.7]{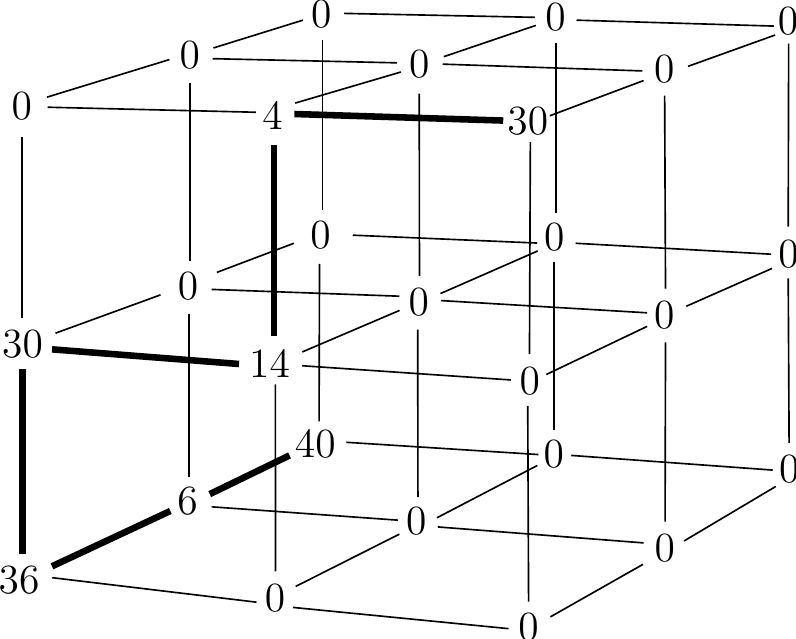}
    \caption{A well-ordered vertex and its staircase.} \label{figure:monostair}
  \end{center}
\end{figure}

\end{example}

Our bound on the diameter of the axial transportation polytope $P$ is based on an explicit path
that goes from any initial vertex $V$ of the axial transportation polytope $P$ to the unique
well-ordered vertex $\hat{V}$. To build this path we rely on the
following stratified version of the concept of well-ordered vertex.
We say that a vertex $V$ of the axial transportation polytope $P$ is \emph{well-ordered starting at level $z$}, where $z$ is an integer between 3 and $p+q+s$ if:

\begin{enumerate}
\item For each $y=z,\dots, p+q+s$, the support of $V$ contains exactly
  one triplet $(i,j,k)$ with $i+j+k = y$.
\item Those triplets are well-ordered.  (The partial
  order given in \eqref{equation:order} is a total order on these triplets.)
\item All other triplets in the support have entries which are
  index-wise smaller than or equal to those of the unique triplet
  $(i_0,j_0,k_0)$ with $i_0+j_0+k_0=z$.
\end{enumerate}

For example, the only vertex ``well-ordered starting at level 3'' is
the well-ordered vertex $\hat{V}$. Slightly less trivially, it is also
the unique vertex ``well-ordered starting at level 4.''  On the other
extreme, all vertices that contain $(p,q,s)$ as a support triplet are
well-ordered starting at $p+q+s$.  Observe that from any vertex
of the axial transportation polytope $P$ we can move, by a single pivot edge in the sense of the simplex method, to another vertex containing any prescribed entry $(i,j,k)$ to be non-zero. In particular, we can move to a vertex that
has $(p,q,s)$ in its support.  So, we can assume from the beginning
that $(p,q,s)$ is in the support of our initial vertex $V$, and will
add one to the count of edges traversed to arrive to $\hat{V}$.

\begin{lemma}\label{lemma:decrease-p}
  If $V$ is a vertex of the axial transportation polytope $P$ that is well-ordered starting at
  level $z \in\{5,\dots, p+q+s\}$, then there is a path of at most
  $2(z-4)$ edges of the axial transportation polytope $P$ that leads from $V$ to a vertex that
  is well-ordered starting at level $z-1$.
\end{lemma}

\begin{proof} Let $(i_0,j_0,k_0)$ be the unique triplet in the support of $V$
with $i_0+j_0+k_0=p$.  We first observe that there is no loss of
generality in assuming that $z=p+q+s$ (that is,
$(i_0,j_0,k_0)=(p,q,s)$). This is because the vertices of the axial transportation polytope $P$ that are well-ordered starting at level $z$ and agree with $V$ in all
the triplets with sum of indices greater than or equal to $z$ are the
vertices of a non-degenerate $i_0 \times j_0 \times k_0$ axial
transportation polytope, obtained as in the proof of
Lemma~\ref{lemma:well-ordered}.

So, from now on we assume that $V$ is well-ordered starting at level
$z=p+q+s$.  Let $S_1$ be the set of support triplets in $V$, other
than $(p,q,s)$, that have first index equal to $l$.  Similarly, let
$S_2$ and $S_3$ be the sets of support triplets that have,
respectively, second and third indices equal to $m$ and $n$.

Our goal is to modify $V$ until one of $S_1$, $S_2$ or $S_3$ becomes
empty, but always keeping the triplet $(p,q,s)$ in the support.  Once
this is done, a single pivot step can be used to obtain a vertex that
is well-ordered starting at $(p,q,s)-1$ as follows: Without loss of
generality assume that $S_1$ is empty (the cases when $S_2$ or $S_3$
are empty are treated identically). In particular, neither $(p,q-1,s)$
nor $(p,q,s-1)$ are in the support.  If $(p-1,q,s)$ is in the support
then our vertex is already well-ordered starting at $p+q+s-1$. If
not, we do the pivot step that inserts $(p-1,q,s)$. This pivot step
cannot remove $(p,q,s)$ or insert $(p,q-1,s)$ or $(p,q,s-1)$ in the
support.  (The $(p,q,s)$ coordinate is not removed from the support
since the entry remains constant in this pivot.  The $(p,q-1,s)$ and
$(p,q,s-1)$ coordinates remain zero because only non-zero entries of
$V$ and the entry $(p-1,q,s)$ change in the pivot). This pivot
produces a vertex well-ordered starting at $p+q+s-1$.  Figure
\ref{figure:lower_p} gives a picture for this case.

\begin{figure}[hbt]
  \begin{center}
    \includegraphics[scale=0.7]{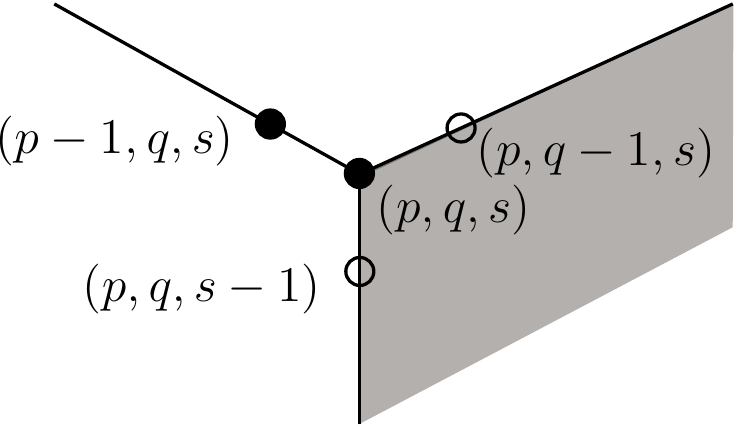}
    \caption{A well-ordered vertex starting at $p+q+s-1$.} \label{figure:lower_p}
  \end{center}
\end{figure}

Given a vertex $V$ well-ordered at $p+q+s$, we specify a sequence of
pivots in the graph of the axial transportation polytope $P$ to a vertex $V'$ such that one of
$S_1$, $S_2$ or $S_3$ is empty for $V'$.  Lemma~\ref{lemma:decrease-S}
below shows how to get such a $V'$ in a number of steps bounded above by
\[
2|S_1 \cup S_2 \cup S_3| - 3 \le 2 (z-3) - 3 = 2z - 9.
\] 
In one more step, that is, at most $2z-8$, we get to a vertex that is
well-ordered starting at $p+q+s-1$. This completes the proof of our
lemma.
\end{proof}

For Lemma~\ref{lemma:decrease-S} let us introduce the following notation:
\[
R_1:= S_1 \setminus (S_2\cup S_3),\quad
R_2:= S_2 \setminus (S_1\cup S_3),\quad
R_3:= S_3 \setminus (S_1\cup S_2),\quad
\]
\[
R_{12}:=S_1\cap S_2, \quad
R_{13}:=S_1\cap S_3, \quad
R_{23}:=S_2\cap S_3.
\]
That is, $R_i$ consists of the elements of $S_1 \cup S_2 \cup S_3$ that belong only to $S_i$,
and $R_{ij}$ of those that belong to $S_i$ and $S_j$. Observe that, by definition, no element belongs
to the three $S_i$'s, so that $S_1 \cup S_2 \cup S_3$ is the disjoint union of these six subsets.

\begin{lemma}
\label{lemma:decrease-S}
With the above notation and the conditions of the proof of the previous lemma, suppose that 
no $S_i$ is empty. Then:
\begin{enumerate}
\item If both $R_i$ and $R_{jk}$ are non-empty, with $\{i,j,k\}=\{1,2,3\}$, then there is a single pivot step that decreases $|S_1 \cup S_2 \cup S_3|$.
\item If the three $R_{ij}$'s are non-empty, then there is a sequence of two pivot steps that decreases $|S_1 \cup S_2 \cup S_3|$.
\item If the three $R_i$'s are non-empty, then there is a sequence of two pivot steps that decreases $|S_1 \cup S_2 \cup S_3|$.
\item If none of the above happens, then $S_1 \cup S_2 \cup S_3$ is contained in one of the $S_i$'s, say $S_1$. Then, there is a sequence of $|S_1|-1$ pivot steps that makes $S_1 \cup S_2 \cup S_3$ empty.
\end{enumerate}
All in all, there is a sequence of at most $2|S_1 \cup S_2 \cup S_3| - 3$ pivot steps that makes some $S_i$ empty.
\end{lemma}

\begin{proof} 
Let us first show how the conclusion is obtained.  We argue by induction on $|S_1 \cup S_2 \cup S_3|$, the base case being $|S_1 \cup S_2 \cup S_3| = 2$, which is the minimum to have no $S_i$ empty. The base case implies we are in the situation of either part (1) or part (4), and a single pivot step makes an $S_i$ empty.

If $|S_1 \cup S_2 \cup S_3| > 2$ and one of the conditions (1), (2) or (3) holds, then we do the step or the two pivot steps mentioned there and apply induction. If none of these three conditions hold then it is easy to see that (4) must hold. (Remember that we are assuming that no $S_i$ is empty, and $S_i=R_i\cup R_{ij}\cup R_{ik}$).
Part (4) guarantees we have a sequence of $|S_1|-1 \le 2|S_1 \cup S_2 \cup S_3| - 3$ pivot steps that makes an $S_i$ empty.

So, let us prove each of the four items in the lemma. Let $\alpha$ denote the entry $(p,q,s)$.

\begin{enumerate} 
\item Suppose without loss of generality that $R_{12}$ and $R_3$ are not empty.
Let  $\beta\in R_{12}$ and $\gamma\in R_3$.
  The reader may find it useful to follow our proof using Figure\ref{figure:case1} which depicts the situation.
\begin{figure}[hbt]
  \begin{center}
    \includegraphics[scale=0.7]{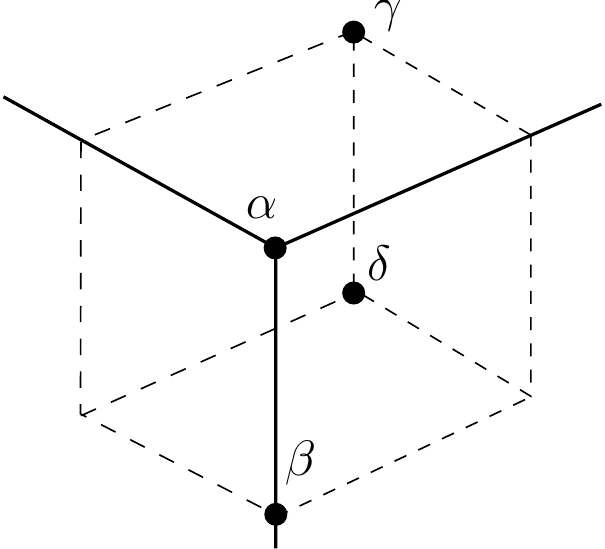}
    \caption{A layout of entries $\alpha$, $\beta$, $\gamma$ and $\delta$.} \label{figure:case1}
  \end{center}
\end{figure}
Observe
that $\alpha=(p,q,s)$ is the index-wise maximum of $\beta$ and
$\gamma$. Let $\delta$ be the index-wise minimum of them. First observe that $\delta$ is not in the support of vertex $V$. Otherwise, we could add $\pm \frac{1}{2}\min\{x_\alpha,x_\beta,x_\gamma,x_\delta\}
\left(e_\alpha+e_\delta-e_\beta-e_\gamma\right)$ to $V$ and stay in
the axial transportation polytope $P$. Hence, $V$ would be a convex combination of two other
points from the axial transportation polytope $P$ (and thus not a vertex), parallel to the
direction of $e_\alpha+e_\delta-e_\beta-e_\gamma$. (Here $e_{i,j,k}$ denotes the basis unit vector in the direction of the variable $x_{i,j,k}$.) By Lemma~\ref{lemma:vertexequivalence}, the feasible point $V$ in $P$ would not be a vertex. Therefore, the triplet $\delta$ is not in the support of the vertex $V$.

Next, consider $V' = V + \min\{x_\beta, x_\gamma\}
\left(e_\alpha+e_\delta-e_\beta-e_\gamma\right)$.  Observe that $V'$
has different support than $V$ since either $\beta$ or $\gamma$ has
been removed (not both, because $x_\beta\ne x_\gamma$ by
non-degeneracy).  Also, since $V'$ cannot have support strictly
contained in that of $V$, $\delta$ must have been added. That is, the
supports of the vertices $V$ and $V'$ differ in the deletion and
insertion of a single element, which means they are adjacent in the
graph of the polytope the axial transportation polytope $P$. As desired, when going from $V$ to $V'$
the cardinality of $S_1\cup S_2 \cup S_3$ is decreased by one.

\begin{example}[Example~\ref{example:axial333} continued]
Consider the vertex $V$ with non-zero coordinates 
$x_{(1,1,2)}=28$, 
$x_{(2,1,2)}=12$, 
$x_{(2,2,3)}=6$, 
$x_{(1,3,2)}=2$, 
$x_{(1,3,1)}=82$, 
$x_{(3,3,2)}=2$, and 
$x_{(3,3,3)}=28$.
In this example, $\alpha=(3,3,3)$, $\beta=(3,3,2)$, $\gamma=(2,2,3)$, 
and $\delta=(2,2,2)$.  After clearing $x_\beta$, we arrive at the vertex 
$V'$ with non-zero coordinates
$x'_{(1,1,2)}=28$, 
$x'_{(2,1,2)}=12$, 
$x'_{(2,2,3)}=4$, 
$x'_{(1,3,2)}=2$, 
$x'_{(1,3,1)}=82$, 
$x'_{(2,2,2)}=2$, and 
$x'_{(3,3,3)}=28$.
\end{example}

\item Suppose now that none of the $R_{ij}$'s is empty, and let $\beta\in R_{13}$, $\gamma\in R_{23}$ and $\delta' \in R_{12}$.
we apply the same pivot as in case one, which makes $\delta$, the coordinate-wise minimum of $\beta$ and $\gamma$, enter the support. This pivot does not decrease $|S_1\cup S_2 \cup S_3|$, but it leads to a situation where we have $\delta \in R_3$ and 
$\delta' \in R_{12}$. Hence, we can apply part one and decrease $|S_1\cup S_2 \cup S_3|$ with a second step.

\item[(4)] Let us prove now part (4) and leave (3), which is more complicated, for the end. Observe that if 
$S_1\cup S_2 \cup S_3 = S_1$ but $S_2$ and $S_3$ are not empty, then necessarily $R_{12}$ and $R_{13}$ are both non-empty. While this holds, we can do the same pivot steps as before with a $\beta\in R_{12}$ and a $\gamma\in R_{13}$. Each step decreases by one the cardinality of $R_{12}\cup R_{13}$, increasing the cardinality of $R_1$. The process finishes when $R_{12}$ (hence $S_2$) or $R_{13}$ (hence $S_3$) becomes empty, which happens, in the worst case, in $|S_1| -1$ steps.

\item[(3)] Finally, consider the case where the three $R_i$'s are non-empty.
 Let $\beta = (p,j_1,k_1)\in R_1$,
  $\gamma=(i_2,q,k_2)\in R_2$ and $\delta=(i_3,j_3,s)\in R_3$, and, as
  before, $\alpha=(p,q,s)$.  Figure~\ref{figure:case2} depicts the situation
  to help the reader with following our proof.
\begin{figure}[hbt]
  \begin{center}
    \includegraphics[scale=0.7]{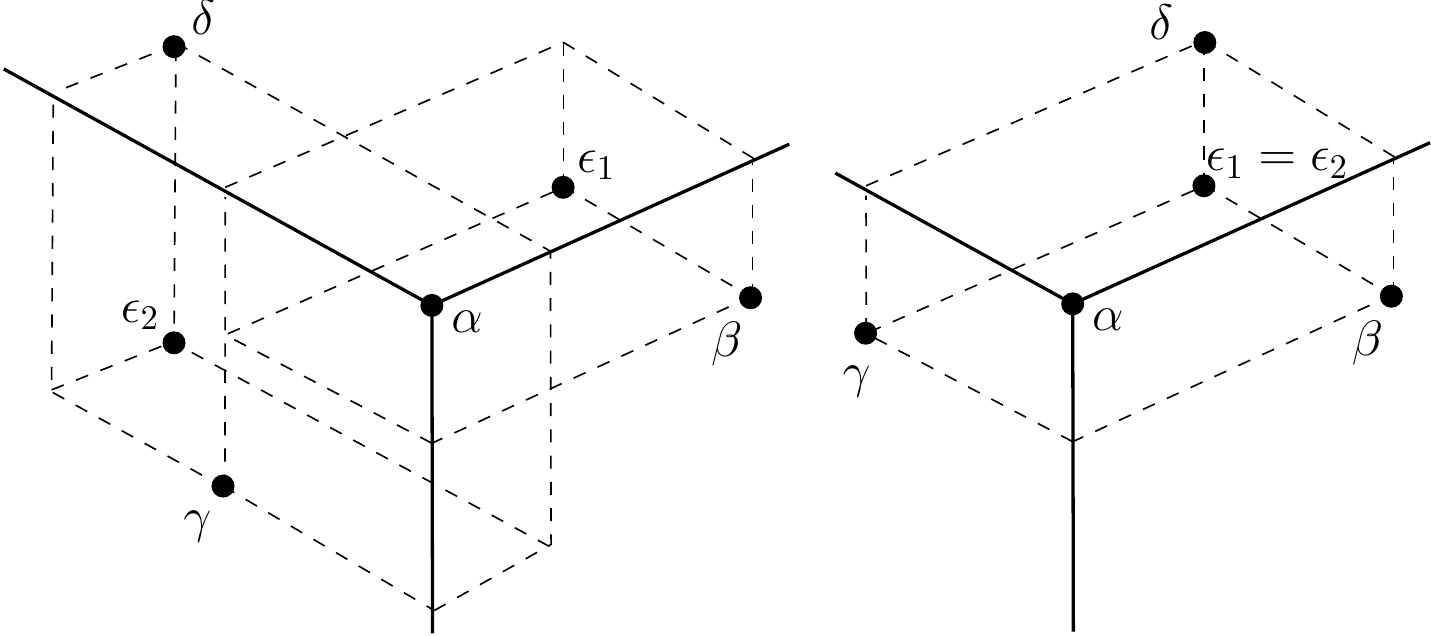}
    \caption{Possible layouts of the entries $\alpha$, $\beta$, $\gamma$, $\delta$, $\epsilon_1$ and $\epsilon_2$.}\label{figure:case2}
  \end{center}
\end{figure}
Let $\epsilon_1$ and $\epsilon_2$ be two triplets of indices with the
property that $\{\alpha,\epsilon_1,\epsilon_2\}$ and $\{\beta, \gamma,
\delta\}$ use exactly the same three first indices, the same three
second indices, and the same three third indices. For example, let us make the
choice $\epsilon_1=(i_2,j_1,k_1)$ and $\epsilon_2=(i_3,j_3,k_2)$ as in
the left part of Figure~\ref{figure:case2}. By non-degeneracy, the smallest
value among $x_\beta$, $x_\gamma$ and $x_\delta$ at $V$ is unique.  We assume, without loss of generality, that the smallest among them is
$x_\beta$. Let $W$ be the point of the axial transportation polytope $P$ obtained by changing
the following six coordinates:
\begin{align*}
x'_\alpha &= x_\alpha+ x_\beta,\\
x'_\beta &= x_\beta - x_\beta=0,\\
x'_\gamma &= x_\gamma - x_\beta,\\
x'_\delta &= x_\delta - x_\beta,\\
x'_{\epsilon_1} &= x_{\epsilon_1}+ x_\beta,\\
x'_{\epsilon_2} &= x_{\epsilon_2}+ x_\beta.
\end{align*}
It may occur that $\epsilon_1=\epsilon_2$, as shown in the right side
of Figure~\ref{figure:case2}. Then we do the same pivot except we increase
the corresponding entry $x_{\epsilon_1}= x_{\epsilon_2}$ twice as
much.

Observe that one of $\epsilon_1$ or $\epsilon_2$ may already be in the
support of $V$, but not both: Otherwise $W$ would have support
strictly contained in that of $V$, which is impossible because $V$ is
a vertex and has minimal support. If one of $\epsilon_1$ or
$\epsilon_2$ were already in the support of $V$, or if
$\epsilon_1=\epsilon_2$, then $W$ is a vertex and we take $V'=W$. As
in the first case, $V'$ is obtained from $V$ by traversing a single
edge and has one less support element in $S_1 \cup S_2 \cup S_3$ than $V$: None of $\epsilon_1$ or $\epsilon_2$
can have a common entry with $\alpha$, since none of $\beta$, $\gamma$
and $\delta$ has two common entries with $\alpha$.

However, if $\epsilon_1$ and $\epsilon_2$ are different and none of
them was in $V$, then $W$ has one-too-many elements in its support to
be a vertex, which means it is in the relative interior of an edge $E$
and $L=VW$ is not an edge.  See Figure~\ref{figure:octa}. Moreover, both $E$ and
$VW$ lie in a
two-dimensional face $F$.  
This is so because every support containing the support of a vertex defines
a face of dimension equal the excess of elements it has. In our case, $F$ is 
the face 
with support $\hbox{support}(V) \cup \hbox{support}(W) =
\hbox{support}(V) \cup \{\epsilon_1,\epsilon_2\}$.

\begin{figure}[hbt]
  \begin{center}
    \includegraphics[scale=0.7]{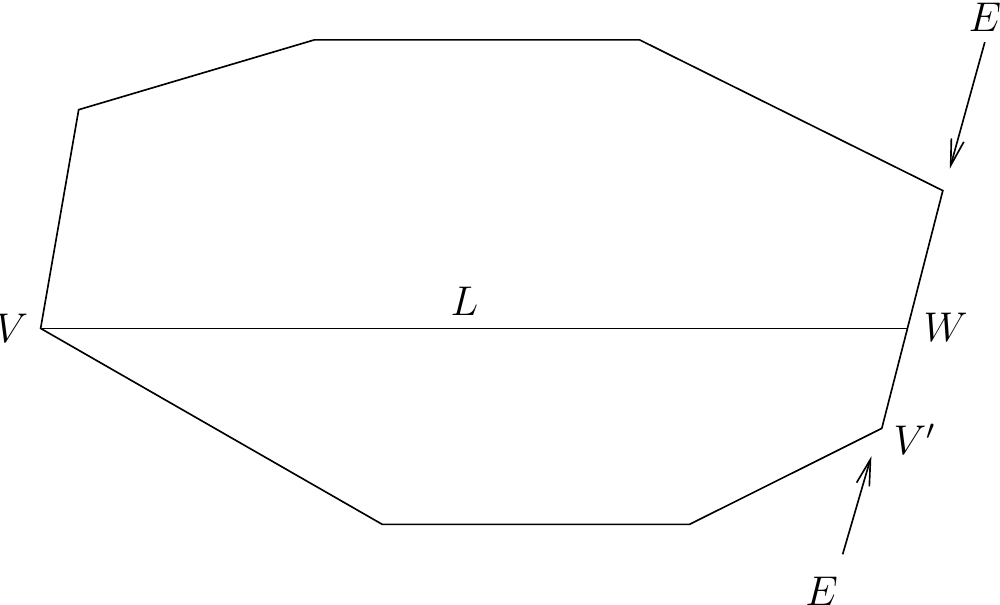}
    \caption{The octagon containing segment $VW$ arising when entries 
$\epsilon_1$ and $\epsilon_2$ are different and none of
them was in the support of $V$} \label{figure:octa}
  \end{center}
\end{figure}

We now look more closely at the structure of $F$. Each edge $H$ of $F$
is the intersection of $F$ with a facet of our transportation polytope. That is,
there is a unique variable $\eta$ that is constantly zero along $H$  but not zero as we
move on $F$ in other directions. For example, since $\epsilon_1$ and $\epsilon_2$
are zero at $V$ but not constant on $F$ (they increase along $L$),  $V$
is the common end of the edges defined by  $\epsilon_1$ and $\epsilon_2$.

Our goal is to show that there is a vertex $V'$ of $F$ at distance at most two from $V$
and incident to the edge defined by one of the variables  $\beta$, $\gamma$, and $\delta$.
At such a vertex $V'$ we will have decreased $|S_1\cup S_2\cup S_3|$ by one, as claimed.
The key remark is that there are at most two edges of $F$ not produced by one of the variables
$\alpha$, $\beta$, $\gamma$, $\delta$, $\epsilon_1$, and $\epsilon_2$:
every variable $\eta$ other than those six is constant along $L$, so it either produces an edge parallel to $L$ or no edge at all. In particular, $F$ is at most an octagon, as in Figure~\ref{figure:octa}. Now:

\begin{itemize}
\item If $F$ has  five or less edges, then every vertex of $F$ is at distance one or two from $V$.
Take as $V'$ either end of the end-point $W$ of $L$. This works because at $W$ one of $\beta$, $\gamma$
or $\delta$ is zero, by construction.

\item If $F$ has six or more edges, then the two vertices $V'$ and $V''$ of $F$ at distance two from $V$ are at distance at least two from each other. So, together they are incident to four different edges, none of which is the edge of $\epsilon_1$ or $\epsilon_2$. (Remember that the edges of $\epsilon_1$ and $\epsilon_2$ are incident to $V$).
At least one of these four edges is defined by $\beta$, $\gamma$, or $\delta$,
because there are (at most) three other possible edges: the one of $\alpha$ and two parallel to $L$.
\end{itemize}
In either case, at most two edges are needed to go from $V$ to a vertex $V'$ incident to an edge defined by one of the variables $\beta$, $\gamma$, or $\delta$.
\end{enumerate}
Thus, there is a sequence of at most $2|S_1 \cup S_2 \cup S_3| - 3$ pivot steps that makes at least one of $S_1$, $S_2$, or $S_3$ become empty.
\end{proof}

\begin{example}
To make ideas completely clear, using the same polytope the axial transportation polytope $P$ as
in Example~\ref{example:axial333}, we consider its vertex $V$ with non-zero coordinates
$v_{(1,1,3)}=25$, 
$v_{(3,1,1)}=15$, 
$v_{(3,2,1)}=6$, 
$v_{(1,3,1)}=61$, 
$v_{(1,3,2)}=26$, 
$v_{(2,3,2)}=18$ and 
$v_{(3,3,3)}=9$.
Here, $\alpha=(3,3,3)$, $\beta=(3,2,1)$, $\gamma=(2,3,2)$, $\delta=(1,1,3)$, 
$\epsilon_1=(2,2,1)$, and $\epsilon_2=(1,1,2)$.
The triplet $\beta$ is not in the support of $W$, and $W$
has non-zero coordinates
$w_{(1,1,3)}=19$, 
$w_{(1,1,2)}=6$,
$w_{(3,1,1)}=15$, 
$w_{(2,2,1)}=6$, 
$w_{(1,3,1)}=61$, 
$w_{(1,3,2)}=26$, 
$w_{(2,3,2)}=12$, and 
$w_{(3,3,3)}=15$.

The vertices of the axial transportation polytope $P$ with support contained in 
$support(V) \cup \{\epsilon_1, \epsilon_2\}$ form the $4$-gon $F=VBCD$ where
$B$ is the vertex with non-zero coordinates
\[
b_{(1,1,3)}=12, 
b_{(1,1,2)}=26, 
b_{(3,1,1)}=2, 
b_{(3,2,1)}=6, 
b_{(3,3,3)}=22, 
b_{(2,3,2)}=18,
b_{(1,3,1)}=74,
\]
$C$ is the vertex with non-zero coordinates
\[
c_{(1,1,3)}=22, 
c_{(3,1,1)}=18, 
c_{(2,2,1)}=6, 
c_{(3,3,3)}=12,
c_{(1,3,2)}=32, 
c_{(2,3,2)}=12, 
c_{(1,3,1)}=58
\]
and $D$ is the vertex with non-zero coordinates
\[
d_{(1,1,3)}=6, 
d_{(1,1,2)}=32, 
d_{(3,1,1)}=2, 
d_{(2,2,1)}=6, 
d_{(3,3,3)}=28, 
d_{(2,3,2)}=12,
d_{(1,3,1)}=74.
\]
Note $W$ is in the edge $E=CD$.  We let $V' = D$, the endpoint of $E$ 
closer to $V$.  Thus, we use one edge to go from $V$ to $V'$.
\end{example}

\begin{proof}[Proof of Theorem~\ref{theorem:main}]
Starting with any vertex ``well-ordered starting at $z=p+q+s$'' (which can be reached in a single step) we use Lemma~\ref{lemma:decrease-p} to decrease one unit by one the level at which our vertex starts to be well-ordered until we reach the unique well-ordered vertex $\hat{V}$. Thus, the number of steps needed to go from an arbitrary vertex $V$ to $\hat{V}$ is at most
\[
1 + \sum_{y=5}^z 2(y-4) =  1 + 2 \sum_{y=1}^{z-4} y
=1 + 2\binom{z-3}{2} \leq (z-3)^2.
\]
To go from one arbitrary vertex to another, twice as many steps suffice.
\end{proof}

\section{Subpolytopes of transportation polytopes}\label{section:networkflow}

In this section, we study network flow polytopes. Network flow polytopes are faces of classical transportation polytopes (see~\cite{Baldoni-Silva:Counting-integer},~\cite{Ford:A-simple-algorithm}, and~\cite{Ford-Jr.:Flows-in-Networks}). We will prove an upper bound on the diameter of two subclasses of network flow polytopes.

Let $G=(N,A)$ be a directed acyclic graph with $n$ nodes and $m$ arcs. (See Figure~\ref{figure:digraph}.) Each node $i \in N$ is assigned a number $b(i)$ which indicates its supply or demand. An arc $(i,j) \in A$ indicates a directed edge that goes from the node $i$ to the node $j$. The network flow polytopes come in many varieties. Here, we describe two specific kinds of network flow polytopes. (See~\cite{Ahuja:Network}.)
\begin{figure}[hbt]
  \begin{center}
    \includegraphics[scale=0.35]{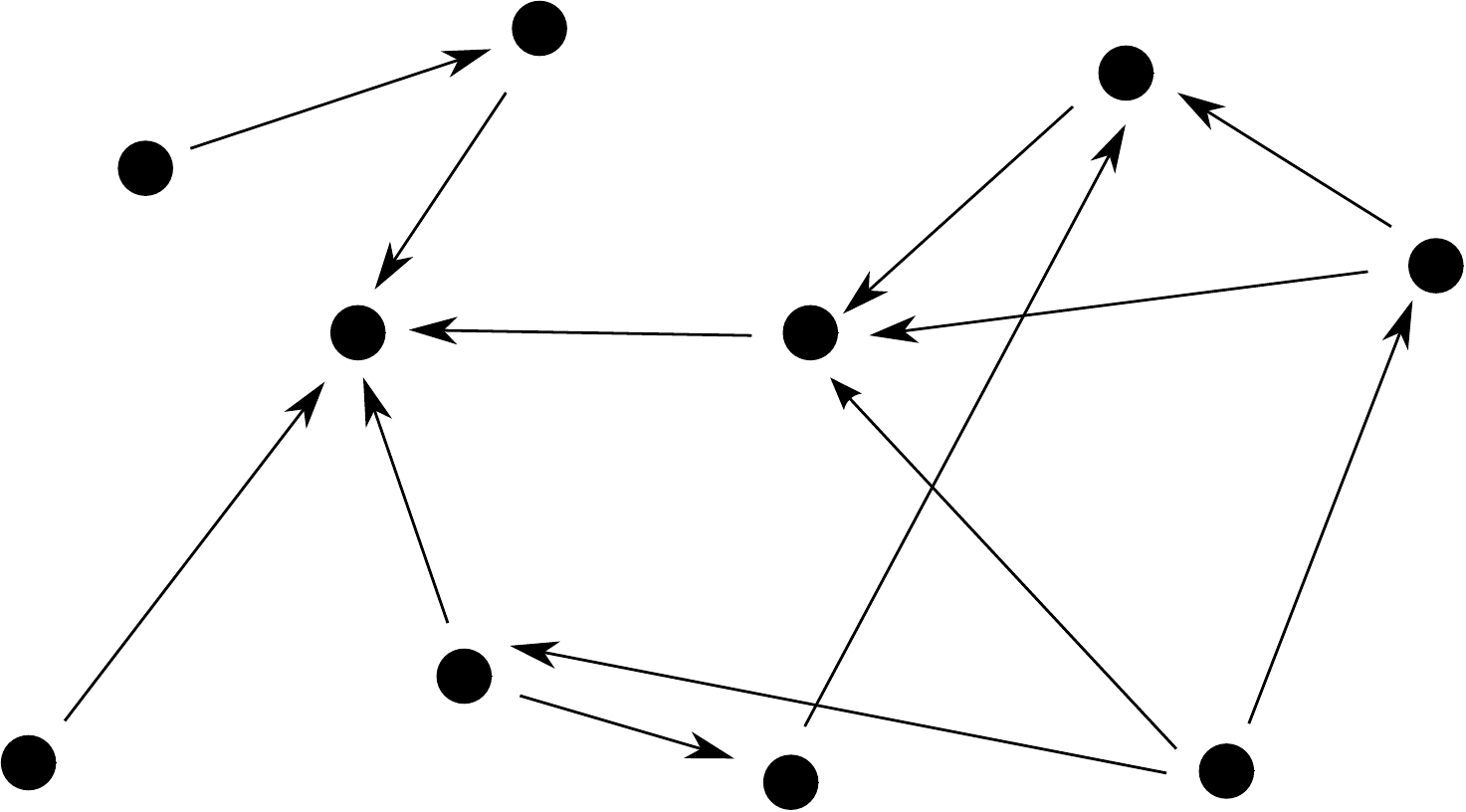}
    \caption{A directed acyclic graph} \label{figure:digraph}
  \end{center}
\end{figure}

The \defn{singly-capacitated network flow polytope} $P$ determined by $u$, $l$, $b$, and the acyclic directed graph $G$ is the polytope in the $m$ variables
\begin{equation*}
x_{i,j} \geq 0 \text{ for } (i,j) \in A
\end{equation*}
subject to the $n$ equations
\begin{equation*}
\sum_{\{j \mid (i,j) \in A\}} x_{i,j} - 
\sum_{\{j \mid (j,i) \in A\}} x_{j,i} = b(i),\ \forall\ i \in N.
\end{equation*}

\begin{remark}
We require that the directed graph $G$ is acyclic. If $G$ had a directed cycle, then we could add any positive quantity $\epsilon > 0$ to each edge in a cycle, and thus $P$ would be an unbounded polyhedron.

We can, thus, assume that the nodes are labeled from $1$ to $n$ and that directed edges $(i,j)$ from $i$ to $j$ always have the property that $i > j$.
\end{remark}

In the doubly-capacitated version, each arc $(i,j) \in A$ also has a capacity upper bound $u_{i,j}$ and a lower bound $l_{i,j}$.  The \defn{doubly-capacitated network flow polytope} $P$ determined by $G$, $u$, $l$ and $b$ is the polytope in the $m$ variables
\begin{equation*}
x_{i,j} \in [l_{i,j}, u_{i,j}]\text{ for } (i,j) \in A
\end{equation*}
subject to the $n$ equations
\begin{equation*}
\sum_{\{j \mid (i,j) \in A\}} x_{i,j} - 
\sum_{\{j \mid (j,i) \in A\}} x_{j,i} = b(i),\ \forall\ i \in N.
\end{equation*}

In~\cite{Orlin:PolytimeNetworkSimplex}, Orlin proved the following bound on the diameter of all doubly-capacitated network flow polytopes.
\begin{theorem}[Orlin~\cite{Orlin:PolytimeNetworkSimplex}]\label{theorem:orlin_bound}
The graph of every doubly-capacitated network flow polytope with $n$ nodes and $m$ arcs has diameter $O(mn \log n)$.    
\end{theorem}

\begin{definition}
Fix $n \in \N$ and $1 \leq k < n$ an integer. Let $G_{n,k}=(V,E)$ be the directed graph on the $n$ nodes labeled $V=\{1, \ldots, n\}$ with the directed edge set $E = \{(i,j) \in [n] \times [n] : 1 \leq j-i \leq k\}$.
\end{definition}
If the Hirsch Conjecture is true, then the diameter of the singly-capacitated network flow polytope for $G_{n,k}$ with any supply vector $b$ is no more than $|E|-(|V|-1)$. Here, we show that the diameter is no more than twice that value. We prove a new linear bound on the diameters of the polytopes given by the graphs $G_{n,k}$.
\begin{theorem}\label{theorem:gnk}
The diameter of the singly-capacitated network flow polytope $P$ of $G_{n,k}$ with any supply vector $b$ is at most $2|E|$.
\end{theorem}
\begin{proof}
Let $P$ be the singly-capacitated network flow polytope corresponding to the directed graph $G_{n,k}$. We assume that the polytope $P$ is non-empty. Let $x$ be any arbitrary vertex of $P$. We will construct a sequence of pivots from the vertex $x$ to the vertex $\hat{x}$ that contains all edges $(i,j)$ of the form $j+1=i$.

Suppose the vertex $x$ does not have the edge $(i+1,i)$. Consider the minimum $i$ such that the edge $(i+1,i)$ is not in the support graph of $x$. Since the support graph of $x$ is a tree, adding the edge $(i+1,i)$ in the support graph induces a directed cycle. We pivot $x$ on that edge $(i+1,i)$ which removes some other edge. (Moreover, the edge that is removed cannot be one of $(2,1),(3,2),\ldots,(i,i-1)$ since these edges do not appear in the cycle.) The resulting vertex $x'$ has the property that the edges of the form $(2,1)$, $(3,2)$, $\ldots$, $(i,i-1)$ belong in the support graph of $x'$. Continuing in this way, we reach the vertex that only has edges of the form $(j,j-1)$.

Each edge is pivoted on at most once in going from the vertex $x$ of $P$ to the unique vertex $\hat{x}$ of $P$ whose support graph corresponds to a directed path. Thus, the length of the path on the graph of $P$ is bounded above by $|E|$. The distance between two arbitrary vertices in the polytope is at most twice that distance.
\end{proof}

More generally, we define the following subclass of directed acyclic graphs.
\begin{definition}
A directed acyclic graph on the $n$ nodes $V=\{1,\ldots,n\}$ is said to have the \defn{intermediate property} if:
\begin{itemize}
\item There is an edge from $i+1$ to $i$ for all $i =1,\ldots,n-1$,
\item The node $i$ is a sink of the $G \setminus \{1,\ldots,i-1\}$ for each $i$, and
\item If there is an edge from $j$ to $i-1$ in $G$, then there is an edge from $j$ to $i$.
\end{itemize}
See Figure~\ref{figure:networkintermediate} for an example.
\end{definition}
\begin{figure}[hbt]
  \begin{center}
    \includegraphics[scale=0.77]{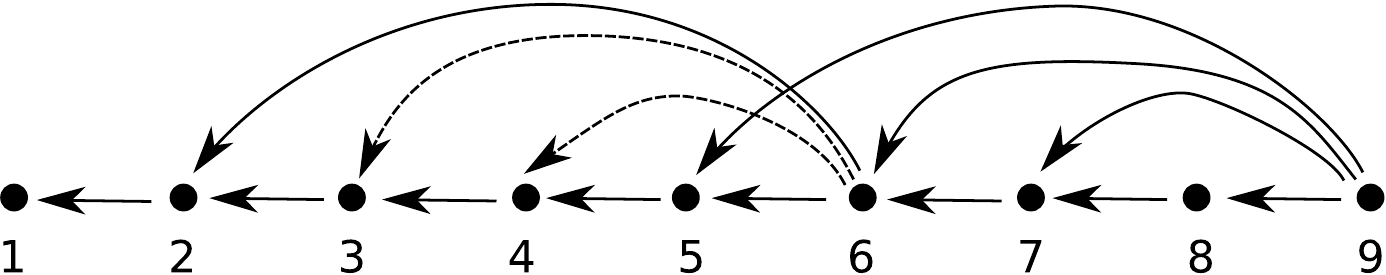}
    \caption{A directed acyclic graph that has the intermediate property: since the edge $(6,2)$ is present, the edges $(6,3)$ and  $(6,4)$, which are indicated with dotted lines, must be present.} \label{figure:networkintermediate}
  \end{center}
\end{figure}

As an extension of Theorem~\ref{theorem:gnk}, we show that:
\begin{theorem}
Let $G$ be any directed acyclic graph on $n$ nodes and $m$ edges with the intermediate property. Then, the diameter of the singly-capacitated network flow polytope $P$ on $G$ with any supply vector $b$ is no more than $2m$.
\end{theorem}
\begin{proof}
Let $x$ be any arbitrary vertex of the singly-capacitated network flow polytope $P$ corresponding to the directed graph $G$ with the intermediate property. We will construct a sequence of pivots from the vertex $x$ to the vertex $\hat{x}$ that contains all edges $(j,i)$ of the form $i+1=j$.

Suppose the vertex $x$ does not have the edge $(i+1,i)$. Consider the minimum $i$ such that the edge $(i+1,i)$ is not in the support graph of $x$. Since the support graph of $x$ is a tree, adding the edge $(i+1,i)$ in the support graph induces a directed cycle, whose existence is guaranteed by the intermediate property. We pivot $x$ on that edge $(i+1,i)$ which removes some other edge. (By the orientation, none of the edges of the form $(2,1)$, $(3,2)$, $\ldots$, $(i+1,i)$ can be removed from this pivot.) The resulting vertex $x'$ has the property that the edges of the form $(2,1)$, $(3,2)$, $\ldots$, $(i+1,i)$ belong in the support graph of $x'$. Continuing in this way, we reach the vertex that only has edges of the form $(j+1,j)$.

Each edge is pivoted on at most once in going from the vertex $x$ of $P$ to the unique vertex $\hat{x}$ of $P$ whose support graph corresponds to a directed path. Thus, the length of the path on the graph of $P$ is bounded above by $|E|$. The distance between two arbitrary vertices in the polytope is at most twice that distance.
\end{proof}


\appendix

\chapter[A Catalogue of Transportation Polytopes]{A Catalogue of Transportation Polytopes}\label{appendix:catalogTP}

\setcounter{theorem}{0} 

This appendix presents a summary of the complete catalogue of transportation polytopes of small sizes using the software {\tt{transportgen}} (see~\cite{transportgen}). Using computational tools, we give a complete catalogue of \emph{non-degenerate} $2$-way and $3$-way transportation polytopes (both axial and planar) of small sizes. This allows us to explore properties of transportation polytopes (e.g., their diameters and how close they were to the Hirsch Conjecture bound). The summary of the catalogue is stated in Theorem~\ref{theorem:genericcase}.

The catalogue was obtained via the exhaustive and systematic computer enumeration of all combinatorial types of non-degenerate transportation polytopes. The theoretical foundations of it are the notions of parametric linear programming, chamber complex, Gale diagrams and secondary polytopes (see~\cite{DeLoera:Triangulations}) are presented in Section~\ref{section:enumerationofpartition}.  A more complete catalogue of transportation polytopes is available in a database on the web (see~\cite{TPDB}).

{\singlespacing

\begin{center}
\begin{longtable}{|c|p{4.0in}|}
\caption{Possible diameters and numbers of vertices and facets for non-degenerate $2 \times 3$ classical transportation polytopes}\label{table:classical-2-3} \\
\hline \multicolumn{1}{|c|}{Property} & \multicolumn{1}{l|}{Values} \\ \hline 
\endfirsthead
\multicolumn{2}{c}%
{{\bfseries \tablename\ \thetable{} Continued}: $2 \times 3$ classical transportation polytopes} \\
\hline \multicolumn{1}{|c|}{Property} & \multicolumn{1}{l|}{Values} \\ \hline 
\endhead
\hline \multicolumn{2}{|c|}{{\bf Continued on next page}} \\ \hline
\endfoot
\endlastfoot
Diameters & 1 2 3 \\ \hline
Num. vertices & 3 4 5 6 \\ \hline
Num. facets & 3 4 5 6 \\ \hline
\end{longtable}
\end{center}

\begin{center}
\begin{longtable}{|c|p{4.0in}|}
\caption{Possible diameters and numbers of vertices and facets for non-degenerate $2 \times 4$ classical transportation polytopes}\label{table:classical-2-4} \\
\hline \multicolumn{1}{|c|}{Property} & \multicolumn{1}{l|}{Values} \\ \hline 
\endfirsthead
\multicolumn{2}{c}%
{{\bfseries \tablename\ \thetable{} Continued}: $2 \times 4$ classical transportation polytopes} \\
\hline \multicolumn{1}{|c|}{Property} & \multicolumn{1}{l|}{Values} \\ \hline 
\endhead
\hline \multicolumn{2}{|c|}{{\bf Continued on next page}} \\ \hline
\endfoot
\endlastfoot
Diameters & 1 2 3 4\\ \hline
Num. vertices & 4 6 8 10 12\\ \hline
Num. facets & 4 5 6 7 8 \\ \hline
\end{longtable}
\end{center}

\begin{center}
\begin{longtable}{|c|p{4.0in}|}
\caption{Possible diameters and numbers of vertices and facets for non-degenerate $2 \times 5$ classical transportation polytopes}\label{table:classical-2-5} \\
\hline \multicolumn{1}{|c|}{Property} & \multicolumn{1}{l|}{Values} \\ \hline 
\endfirsthead
\multicolumn{2}{c}%
{{\bfseries \tablename\ \thetable{} Continued}: $2 \times 5$ classical transportation polytopes} \\
\hline \multicolumn{1}{|c|}{Property} & \multicolumn{1}{l|}{Values} \\ \hline 
\endhead
\hline \multicolumn{2}{|c|}{{\bf Continued on next page}} \\ \hline
\endfoot
\endlastfoot
Diameters & 1 2 3 4 5\\ \hline
Num. vertices & 5 8 11 12 14 15 16 17 18 19 20 21 22 23 24 25 26 27 28 29 30 \\ \hline
Num. facets & 5 6 7 8 9 10 \\ \hline
\end{longtable}
\end{center}

\begin{center}
\begin{longtable}{|c|p{4.0in}|}
\caption{Possible diameters and numbers of vertices and facets for non-degenerate $3 \times 3$ classical transportation polytopes}\label{table:classical-3-3} \\
\hline \multicolumn{1}{|c|}{Property} & \multicolumn{1}{l|}{Values} \\ \hline 
\endfirsthead
\multicolumn{2}{c}%
{{\bfseries \tablename\ \thetable{} Continued}: $3 \times 3$ classical transportation polytopes} \\
\hline \multicolumn{1}{|c|}{Property} & \multicolumn{1}{l|}{Values} \\ \hline 
\endhead
\hline \multicolumn{2}{|c|}{{\bf Continued on next page}} \\ \hline
\endfoot
\endlastfoot
Diameters & 2 3 4\\ \hline
Num. vertices & 9 12 15 18\\ \hline
Num. facets & 6 7 8 9 \\ \hline
\end{longtable}
\end{center}

\begin{center}
\begin{longtable}{|c|p{4.0in}|}
\caption{Possible diameters and numbers of vertices and facets for non-degenerate $3 \times 4$ classical transportation polytopes}\label{table:classical-3-4} \\
\hline \multicolumn{1}{|c|}{Property} & \multicolumn{1}{l|}{Values} \\ \hline 
\endfirsthead
\multicolumn{2}{c}%
{{\bfseries \tablename\ \thetable{} Continued}: $3 \times 4$ classical transportation polytopes} \\
\hline \multicolumn{1}{|c|}{Property} & \multicolumn{1}{l|}{Values} \\ \hline 
\endhead
\hline \multicolumn{2}{|c|}{{\bf Continued on next page}} \\ \hline
\endfoot
\endlastfoot
Diameters & 2 3 4 5 6 \\ \hline
Num. vertices & 16 21 24 26 27 29 31 32 34 36 37 39 40 41 42 44 45 46 48 49 50 52 53 54 56 57 58 60 61 62 63 64 66 67 68 70 71 72 74 75 76 78 80 84 90 96\\ \hline
Num. facets & 8 9 10 11 12\\ \hline
\end{longtable}
\end{center}

\begin{center}
\begin{longtable}{|c|p{4.0in}|}
\caption{Possible diameters and numbers of vertices and facets for non-degenerate $2 \times 2 \times 2$ axial transportation polytopes}\label{table:axial-2-2-2} \\
\hline \multicolumn{1}{|c|}{Property} & \multicolumn{1}{l|}{Values} \\ \hline 
\endfirsthead
\multicolumn{2}{c}%
{{\bfseries \tablename\ \thetable{} Continued}: $2 \times 2 \times 2$ axial transportation polytopes} \\
\hline \multicolumn{1}{|c|}{Property} & \multicolumn{1}{l|}{Values} \\ \hline 
\endhead
\hline \multicolumn{2}{|c|}{{\bf Continued on next page}} \\ \hline
\endfoot
\endlastfoot
Diameters & 2 3 4\\ \hline
Num. vertices & 8 11 14\\ \hline
Num. facets & 6 7 8\\ \hline
\end{longtable}
\end{center}

\begin{center}
\begin{longtable}{|c|p{4.0in}|}
\caption{Possible diameters and numbers of vertices and facets for non-degenerate $2 \times 2 \times 3$ classical transportation polytopes}\label{table:axial-2-2-3} \\
\hline \multicolumn{1}{|c|}{Property} & \multicolumn{1}{l|}{Values} \\ \hline 
\endfirsthead
\multicolumn{2}{c}%
{{\bfseries \tablename\ \thetable{} Continued}: $2 \times 2\times 3$ axial transportation polytopes} \\
\hline \multicolumn{1}{|c|}{Property} & \multicolumn{1}{l|}{Values} \\ \hline 
\endhead
\hline \multicolumn{2}{|c|}{{\bf Continued on next page}} \\ \hline
\endfoot
\endlastfoot
Diameters & 2 3 4 5\\ \hline
Num. vertices & 18 24 30 32 36 38 40 42 44 46 48 50 52 54 56 58 60 62 64 66 68 70 72 74 76 78 80 84 86 96 108\\ \hline
Num. facets & 9 10 11 12\\ \hline
\end{longtable}
\end{center}

\begin{center}
\begin{longtable}{|c|p{4.0in}|}
\caption{Possible diameters and numbers of vertices and facets for non-degenerate $2 \times 2 \times 3$ planar transportation polytopes}\label{table:planar-2-2-3} \\
\hline \multicolumn{1}{|c|}{Property} & \multicolumn{1}{l|}{Values} \\ \hline 
\endfirsthead
\multicolumn{2}{c}%
{{\bfseries \tablename\ \thetable{} Continued}: $2 \times 2\times 3$ planar transportation polytopes} \\
\hline \multicolumn{1}{|c|}{Property} & \multicolumn{1}{l|}{Values} \\ \hline 
\endhead
\hline \multicolumn{2}{|c|}{{\bf Continued on next page}} \\ \hline
\endfoot
\endlastfoot
Diameters & 1 2 3\\ \hline
Num. vertices & 3 4 5 6\\ \hline
Num. facets & 3 4 5 6 \\ \hline
\end{longtable}
\end{center}

\begin{center}
\begin{longtable}{|c|p{4.0in}|}
\caption{Possible diameters and numbers of vertices and facets for non-degenerate $2 \times 3 \times 3$ planar transportation polytopes}\label{table:planar-2-3-3} \\
\hline \multicolumn{1}{|c|}{Property} & \multicolumn{1}{l|}{Values} \\ \hline 
\endfirsthead
\multicolumn{2}{c}%
{{\bfseries \tablename\ \thetable{} Continued}: $2 \times 3 \times 3$ planar transportation polytopes} \\
\hline \multicolumn{1}{|c|}{Property} & \multicolumn{1}{l|}{Values} \\ \hline 
\endhead
\hline \multicolumn{2}{|c|}{{\bf Continued on next page}} \\ \hline
\endfoot
\endlastfoot
Diameters & 3 4 5 6\\ \hline
Num. vertices &12 15 16 19 20 21 22 23 24 25 26 27 28 29 30 31 32 33 34
 \\ \hline
Num. facets & 7 8 9 10 11 12 \\ \hline
\end{longtable}
\end{center}
}

\chapter[A Triangulation of the Fourth Birkhoff Polytope]{A Triangulation of the Fourth Birkhoff Polytope}\label{appendix:triangB4}

\setcounter{theorem}{0} 

In Section~\ref{section:regulartriangulationsBirkhoff}, we presented a collection $\mathcal{T}$ of simplices and claimed that they were the maximal simplices of a triangulation of a certain face $F$ of the fourth Birkhoff polytope. In this appendix, we prove that the collection of simplices defined in Section~\ref{section:regulartriangulationsBirkhoff} are indeed a polyhedral subdivision. (The fact that the polyhedral subdivision $\mathcal{T}$ is a triangulation follows immediately, since all polyhedra in $\mathcal{T}$ are simplices.) To prove this, we use the following characterization of polyhedral subdivisions presented as Theorem~4.5.9 in ~\cite{DeLoera:Triangulations}.
\begin{theorem}
A non-empty set $\mathcal{T}$ of $d$-dimensional subconfigurations of a point configuration $\mathcal{A}$ in $\R^d$ is the set of maximal cells of a polyhedral subdivision of $\mathcal{A}$ if and only if it satisfies the following conditions: 
\begin{itemize}
\item For each facet $G$ of a $d$-cell $B$ in $\mathcal{T}$, either $G$ is contained in a facet of $\mathcal{A}$ or there is another $d$-cell $B$ in $\mathcal{T}$ that contains $G$ as a facet.
\item $\conv(B) \cap \conv(B') = \conv(B \cap B')$ for all $B, B'\in \mathcal{T}$.
\item For all label sets $B, B' \in \mathcal{T}$, $B \cap B'$ is a face of both $B$ and $B'$.
\end{itemize}
\end{theorem}
Since the points $X_a,\ldots,X_n$ are in convex position, we simply need to check the following property:
\begin{itemize}
\item For every facet $G$ of every maximal simplex $\sigma_i$ in $\mathcal{T}$, either:
\begin{itemize}
\item $G$ is contained in a facet of $F$, or
\item there is another maximal simplex $\sigma_j$ in $\mathcal{T}$  that contains $G$ as a facet, with $j \not= i$.
\end{itemize}
\end{itemize}
Below, we verify this property for each of the eight facets of the $32$ maximal simplices:

The facets of the simplex $\sigma_{1}$ are $bcdefgi$, $acdefgi$, $abdefgi$, $abcefgi$, $abcdfgi$, $abcdegi$, $abcdefi$, and $abcdefg$.
\begin{enumerate}
\item The facet $bcdefgi$ is contained in the facet of $B_4$ defined by $x_{3,3}=0$.
\item The facet $acdefgi$ is also a facet of the simplex $\sigma_{3}$.
\item The facet $abdefgi$ is also a facet of the simplex $\sigma_{2}$.
\item The facet $abcefgi$ is contained in the facet of $B_4$ defined by $x_{2,1}=0$.
\item The facet $abcdfgi$ is also a facet of the simplex $\sigma_{24}$.
\item The facet $abcdegi$ is contained in the facet of $B_4$ defined by $x_{4,2}=0$.
\item The facet $abcdefi$ is contained in the facet of $B_4$ defined by $x_{2,2}=0$.
\item The facet $abcdefg$ is contained in the facet of $B_4$ defined by $x_{1,4}=0$.
\end{enumerate}

The facets of the simplex $\sigma_{2}$ are $abdefgi$, $bdefghi$, $adefghi$, $abefghi$, $abdfghi$, $abdeghi$, $abdefhi$, and $abdefgh$.
\begin{enumerate}
\item The facet $abdefgi$ is also a facet of the simplex $\sigma_{1}$.
\item The facet $bdefghi$ is contained in the facet of $B_4$ defined by $x_{3,3}=0$.
\item The facet $adefghi$ is also a facet of the simplex $\sigma_{4}$.
\item The facet $abefghi$ is contained in the facet of $B_4$ defined by $x_{4,3}=0$.
\item The facet $abdfghi$ is also a facet of the simplex $\sigma_{28}$.
\item The facet $abdeghi$ is contained in the facet of $B_4$ defined by $x_{3,1}=0$.
\item The facet $abdefhi$ is contained in the facet of $B_4$ defined by $x_{2,2}=0$.
\item The facet $abdefgh$ is contained in the facet of $B_4$ defined by $x_{1,4}=0$.
\end{enumerate}

The facets of the simplex $\sigma_{3}$ are $acdefgi$, $cdefgik$, $adefgik$, $acefgik$, $acdfgik$, $acdegik$, $acdefik$, and $acdefgk$.
\begin{enumerate}
\item The facet $acdefgi$ is also a facet of the simplex $\sigma_{1}$.
\item The facet $cdefgik$ is contained in the facet of $B_4$ defined by $x_{3,3}=0$.
\item The facet $adefgik$ is also a facet of the simplex $\sigma_{4}$.
\item The facet $acefgik$ is contained in the facet of $B_4$ defined by $x_{2,1}=0$.
\item The facet $acdfgik$ is also a facet of the simplex $\sigma_{24}$.
\item The facet $acdegik$ is contained in the facet of $B_4$ defined by $x_{4,2}=0$.
\item The facet $acdefik$ is also a facet of the simplex $\sigma_{8}$.
\item The facet $acdefgk$ is contained in the facet of $B_4$ defined by $x_{2,3}=0$.
\end{enumerate}

The facets of the simplex $\sigma_{4}$ are $adefghi$, $adefgik$, $defghik$, $aefghik$, $adfghik$, $adeghik$, $adefhik$, and $adefghk$.
\begin{enumerate}
\item The facet $adefghi$ is also a facet of the simplex $\sigma_{2}$.
\item The facet $adefgik$ is also a facet of the simplex $\sigma_{3}$.
\item The facet $defghik$ is contained in the facet of $B_4$ defined by $x_{3,3}=0$.
\item The facet $aefghik$ is also a facet of the simplex $\sigma_{5}$.
\item The facet $adfghik$ is also a facet of the simplex $\sigma_{28}$.
\item The facet $adeghik$ is also a facet of the simplex $\sigma_{6}$.
\item The facet $adefhik$ is also a facet of the simplex $\sigma_{9}$.
\item The facet $adefghk$ is contained in the facet of $B_4$ defined by $x_{2,3}=0$.
\end{enumerate}

The facets of the simplex $\sigma_{5}$ are $aefghik$, $efghikl$, $afghikl$, $aeghikl$, $aefhikl$, $aefgikl$, $aefghkl$, and $aefghil$.
\begin{enumerate}
\item The facet $aefghik$ is also a facet of the simplex $\sigma_{4}$.
\item The facet $efghikl$ is contained in the facet of $B_4$ defined by $x_{1,2}=0$.
\item The facet $afghikl$ is also a facet of the simplex $\sigma_{7}$.
\item The facet $aeghikl$ is also a facet of the simplex $\sigma_{6}$.
\item The facet $aefhikl$ is also a facet of the simplex $\sigma_{12}$.
\item The facet $aefgikl$ is contained in the facet of $B_4$ defined by $x_{2,1}=0$.
\item The facet $aefghkl$ is contained in the facet of $B_4$ defined by $x_{2,3}=0$.
\item The facet $aefghil$ is contained in the facet of $B_4$ defined by $x_{4,3}=0$.
\end{enumerate}

The facets of the simplex $\sigma_{6}$ are $adeghik$, $aeghikl$, $deghikl$, $adghikl$, $adehikl$, $adegikl$, $adeghkl$, and $adeghil$.
\begin{enumerate}
\item The facet $adeghik$ is also a facet of the simplex $\sigma_{4}$.
\item The facet $aeghikl$ is also a facet of the simplex $\sigma_{5}$.
\item The facet $deghikl$ is also a facet of the simplex $\sigma_{13}$.
\item The facet $adghikl$ is also a facet of the simplex $\sigma_{29}$.
\item The facet $adehikl$ is also a facet of the simplex $\sigma_{14}$.
\item The facet $adegikl$ is contained in the facet of $B_4$ defined by $x_{4,2}=0$.
\item The facet $adeghkl$ is contained in the facet of $B_4$ defined by $x_{2,3}=0$.
\item The facet $adeghil$ is contained in the facet of $B_4$ defined by $x_{3,1}=0$.
\end{enumerate}

The facets of the simplex $\sigma_{7}$ are $afghikl$, $fghijkl$, $aghijkl$, $afhijkl$, $afgijkl$, $afghjkl$, $afghijl$, and $afghijk$.
\begin{enumerate}
\item The facet $afghikl$ is also a facet of the simplex $\sigma_{5}$.
\item The facet $fghijkl$ is contained in the facet of $B_4$ defined by $x_{1,2}=0$.
\item The facet $aghijkl$ is also a facet of the simplex $\sigma_{26}$.
\item The facet $afhijkl$ is also a facet of the simplex $\sigma_{22}$.
\item The facet $afgijkl$ is contained in the facet of $B_4$ defined by $x_{2,1}=0$.
\item The facet $afghjkl$ is contained in the facet of $B_4$ defined by $x_{3,2}=0$.
\item The facet $afghijl$ is contained in the facet of $B_4$ defined by $x_{4,3}=0$.
\item The facet $afghijk$ is also a facet of the simplex $\sigma_{25}$.
\end{enumerate}

The facets of the simplex $\sigma_{8}$ are $acdefik$, $cdefikm$, $adefikm$, $acefikm$, $acdfikm$, $acdeikm$, $acdefkm$, and $acdefim$.
\begin{enumerate}
\item The facet $acdefik$ is also a facet of the simplex $\sigma_{3}$.
\item The facet $cdefikm$ is contained in the facet of $B_4$ defined by $x_{3,3}=0$.
\item The facet $adefikm$ is also a facet of the simplex $\sigma_{9}$.
\item The facet $acefikm$ is contained in the facet of $B_4$ defined by $x_{3,4}=0$.
\item The facet $acdfikm$ is also a facet of the simplex $\sigma_{10}$.
\item The facet $acdeikm$ is contained in the facet of $B_4$ defined by $x_{4,2}=0$.
\item The facet $acdefkm$ is contained in the facet of $B_4$ defined by $x_{2,3}=0$.
\item The facet $acdefim$ is contained in the facet of $B_4$ defined by $x_{2,2}=0$.
\end{enumerate}

The facets of the simplex $\sigma_{9}$ are $adefhik$, $adefikm$, $defhikm$, $aefhikm$, $adfhikm$, $adehikm$, $adefhkm$, and $adefhim$.
\begin{enumerate}
\item The facet $adefhik$ is also a facet of the simplex $\sigma_{4}$.
\item The facet $adefikm$ is also a facet of the simplex $\sigma_{8}$.
\item The facet $defhikm$ is contained in the facet of $B_4$ defined by $x_{3,3}=0$.
\item The facet $aefhikm$ is also a facet of the simplex $\sigma_{12}$.
\item The facet $adfhikm$ is also a facet of the simplex $\sigma_{11}$.
\item The facet $adehikm$ is also a facet of the simplex $\sigma_{14}$.
\item The facet $adefhkm$ is contained in the facet of $B_4$ defined by $x_{2,3}=0$.
\item The facet $adefhim$ is contained in the facet of $B_4$ defined by $x_{2,2}=0$.
\end{enumerate}

The facets of the simplex $\sigma_{10}$ are $acdfikm$, $cdfijkm$, $adfijkm$, $acfijkm$, $acdijkm$, $acdfjkm$, $acdfijm$, and $acdfijk$.
\begin{enumerate}
\item The facet $acdfikm$ is also a facet of the simplex $\sigma_{8}$.
\item The facet $cdfijkm$ is contained in the facet of $B_4$ defined by $x_{3,3}=0$.
\item The facet $adfijkm$ is also a facet of the simplex $\sigma_{11}$.
\item The facet $acfijkm$ is contained in the facet of $B_4$ defined by $x_{3,4}=0$.
\item The facet $acdijkm$ is contained in the facet of $B_4$ defined by $x_{1,3}=0$.
\item The facet $acdfjkm$ is also a facet of the simplex $\sigma_{15}$.
\item The facet $acdfijm$ is contained in the facet of $B_4$ defined by $x_{2,2}=0$.
\item The facet $acdfijk$ is also a facet of the simplex $\sigma_{23}$.
\end{enumerate}

The facets of the simplex $\sigma_{11}$ are $adfhikm$, $adfijkm$, $dfhijkm$, $afhijkm$, $adhijkm$, $adfhjkm$, $adfhijm$, and $adfhijk$.
\begin{enumerate}
\item The facet $adfhikm$ is also a facet of the simplex $\sigma_{9}$.
\item The facet $adfijkm$ is also a facet of the simplex $\sigma_{10}$.
\item The facet $dfhijkm$ is contained in the facet of $B_4$ defined by $x_{3,3}=0$.
\item The facet $afhijkm$ is also a facet of the simplex $\sigma_{20}$.
\item The facet $adhijkm$ is also a facet of the simplex $\sigma_{21}$.
\item The facet $adfhjkm$ is also a facet of the simplex $\sigma_{16}$.
\item The facet $adfhijm$ is contained in the facet of $B_4$ defined by $x_{2,2}=0$.
\item The facet $adfhijk$ is also a facet of the simplex $\sigma_{27}$.
\end{enumerate}

The facets of the simplex $\sigma_{12}$ are $aefhikl$, $aefhikm$, $efhiklm$, $afhiklm$, $aehiklm$, $aefiklm$, $aefhklm$, and $aefhilm$.
\begin{enumerate}
\item The facet $aefhikl$ is also a facet of the simplex $\sigma_{5}$.
\item The facet $aefhikm$ is also a facet of the simplex $\sigma_{9}$.
\item The facet $efhiklm$ is contained in the facet of $B_4$ defined by $x_{1,2}=0$.
\item The facet $afhiklm$ is also a facet of the simplex $\sigma_{18}$.
\item The facet $aehiklm$ is also a facet of the simplex $\sigma_{14}$.
\item The facet $aefiklm$ is contained in the facet of $B_4$ defined by $x_{3,4}=0$.
\item The facet $aefhklm$ is contained in the facet of $B_4$ defined by $x_{2,3}=0$.
\item The facet $aefhilm$ is also a facet of the simplex $\sigma_{17}$.
\end{enumerate}

The facets of the simplex $\sigma_{13}$ are $deghikl$, $eghiklm$, $dghiklm$, $dehiklm$, $degiklm$, $deghklm$, $deghilm$, and $deghikm$.
\begin{enumerate}
\item The facet $deghikl$ is also a facet of the simplex $\sigma_{6}$.
\item The facet $eghiklm$ is contained in the facet of $B_4$ defined by $x_{1,2}=0$.
\item The facet $dghiklm$ is contained in the facet of $B_4$ defined by $x_{2,4}=0$.
\item The facet $dehiklm$ is also a facet of the simplex $\sigma_{14}$.
\item The facet $degiklm$ is contained in the facet of $B_4$ defined by $x_{4,2}=0$.
\item The facet $deghklm$ is contained in the facet of $B_4$ defined by $x_{2,3}=0$.
\item The facet $deghilm$ is contained in the facet of $B_4$ defined by $x_{3,1}=0$.
\item The facet $deghikm$ is contained in the facet of $B_4$ defined by $x_{3,3}=0$.
\end{enumerate}

The facets of the simplex $\sigma_{14}$ are $adehikl$, $adehikm$, $aehiklm$, $dehiklm$, $adhiklm$, $adeiklm$, $adehklm$, and $adehilm$.
\begin{enumerate}
\item The facet $adehikl$ is also a facet of the simplex $\sigma_{6}$.
\item The facet $adehikm$ is also a facet of the simplex $\sigma_{9}$.
\item The facet $aehiklm$ is also a facet of the simplex $\sigma_{12}$.
\item The facet $dehiklm$ is also a facet of the simplex $\sigma_{13}$.
\item The facet $adhiklm$ is also a facet of the simplex $\sigma_{19}$.
\item The facet $adeiklm$ is contained in the facet of $B_4$ defined by $x_{4,2}=0$.
\item The facet $adehklm$ is contained in the facet of $B_4$ defined by $x_{2,3}=0$.
\item The facet $adehilm$ is contained in the facet of $B_4$ defined by $x_{3,1}=0$.
\end{enumerate}

The facets of the simplex $\sigma_{15}$ are $acdfjkm$, $cdfjkmn$, $adfjkmn$, $acfjkmn$, $acdjkmn$, $acdfkmn$, $acdfjmn$, and $acdfjkn$.
\begin{enumerate}
\item The facet $acdfjkm$ is also a facet of the simplex $\sigma_{10}$.
\item The facet $cdfjkmn$ is contained in the facet of $B_4$ defined by $x_{4,1}=0$.
\item The facet $adfjkmn$ is also a facet of the simplex $\sigma_{16}$.
\item The facet $acfjkmn$ is contained in the facet of $B_4$ defined by $x_{3,4}=0$.
\item The facet $acdjkmn$ is contained in the facet of $B_4$ defined by $x_{1,3}=0$.
\item The facet $acdfkmn$ is contained in the facet of $B_4$ defined by $x_{2,3}=0$.
\item The facet $acdfjmn$ is contained in the facet of $B_4$ defined by $x_{2,2}=0$.
\item The facet $acdfjkn$ is contained in the facet of $B_4$ defined by $x_{3,2}=0$.
\end{enumerate}

The facets of the simplex $\sigma_{16}$ are $adfhjkm$, $adfjkmn$, $dfhjkmn$, $afhjkmn$, $adhjkmn$, $adfhkmn$, $adfhjmn$, and $adfhjkn$.
\begin{enumerate}
\item The facet $adfhjkm$ is also a facet of the simplex $\sigma_{11}$.
\item The facet $adfjkmn$ is also a facet of the simplex $\sigma_{15}$.
\item The facet $dfhjkmn$ is contained in the facet of $B_4$ defined by $x_{4,1}=0$.
\item The facet $afhjkmn$ is also a facet of the simplex $\sigma_{20}$.
\item The facet $adhjkmn$ is also a facet of the simplex $\sigma_{21}$.
\item The facet $adfhkmn$ is contained in the facet of $B_4$ defined by $x_{2,3}=0$.
\item The facet $adfhjmn$ is contained in the facet of $B_4$ defined by $x_{2,2}=0$.
\item The facet $adfhjkn$ is contained in the facet of $B_4$ defined by $x_{3,2}=0$.
\end{enumerate}

The facets of the simplex $\sigma_{17}$ are $aefhilm$, $efhilmn$, $afhilmn$, $aehilmn$, $aefilmn$, $aefhlmn$, $aefhimn$, and $aefhiln$.
\begin{enumerate}
\item The facet $aefhilm$ is also a facet of the simplex $\sigma_{12}$.
\item The facet $efhilmn$ is contained in the facet of $B_4$ defined by $x_{1,2}=0$.
\item The facet $afhilmn$ is also a facet of the simplex $\sigma_{18}$.
\item The facet $aehilmn$ is contained in the facet of $B_4$ defined by $x_{3,1}=0$.
\item The facet $aefilmn$ is contained in the facet of $B_4$ defined by $x_{3,4}=0$.
\item The facet $aefhlmn$ is contained in the facet of $B_4$ defined by $x_{2,3}=0$.
\item The facet $aefhimn$ is contained in the facet of $B_4$ defined by $x_{2,2}=0$.
\item The facet $aefhiln$ is contained in the facet of $B_4$ defined by $x_{4,3}=0$.
\end{enumerate}

The facets of the simplex $\sigma_{18}$ are $afhiklm$, $afhilmn$, $fhiklmn$, $ahiklmn$, $afiklmn$, $afhklmn$, $afhikmn$, and $afhikln$.
\begin{enumerate}
\item The facet $afhiklm$ is also a facet of the simplex $\sigma_{12}$.
\item The facet $afhilmn$ is also a facet of the simplex $\sigma_{17}$.
\item The facet $fhiklmn$ is contained in the facet of $B_4$ defined by $x_{1,2}=0$.
\item The facet $ahiklmn$ is also a facet of the simplex $\sigma_{19}$.
\item The facet $afiklmn$ is contained in the facet of $B_4$ defined by $x_{3,4}=0$.
\item The facet $afhklmn$ is contained in the facet of $B_4$ defined by $x_{2,3}=0$.
\item The facet $afhikmn$ is also a facet of the simplex $\sigma_{20}$.
\item The facet $afhikln$ is also a facet of the simplex $\sigma_{22}$.
\end{enumerate}

The facets of the simplex $\sigma_{19}$ are $adhiklm$, $ahiklmn$, $dhiklmn$, $adiklmn$, $adhklmn$, $adhilmn$, $adhikmn$, and $adhikln$.
\begin{enumerate}
\item The facet $adhiklm$ is also a facet of the simplex $\sigma_{14}$.
\item The facet $ahiklmn$ is also a facet of the simplex $\sigma_{18}$.
\item The facet $dhiklmn$ is contained in the facet of $B_4$ defined by $x_{2,4}=0$.
\item The facet $adiklmn$ is contained in the facet of $B_4$ defined by $x_{1,3}=0$.
\item The facet $adhklmn$ is contained in the facet of $B_4$ defined by $x_{2,3}=0$.
\item The facet $adhilmn$ is contained in the facet of $B_4$ defined by $x_{3,1}=0$.
\item The facet $adhikmn$ is also a facet of the simplex $\sigma_{21}$.
\item The facet $adhikln$ is also a facet of the simplex $\sigma_{31}$.
\end{enumerate}

The facets of the simplex $\sigma_{20}$ are $afhijkm$, $afhjkmn$, $afhikmn$, $fhijkmn$, $ahijkmn$, $afijkmn$, $afhijmn$, and $afhijkn$.
\begin{enumerate}
\item The facet $afhijkm$ is also a facet of the simplex $\sigma_{11}$.
\item The facet $afhjkmn$ is also a facet of the simplex $\sigma_{16}$.
\item The facet $afhikmn$ is also a facet of the simplex $\sigma_{18}$.
\item The facet $fhijkmn$ is contained in the facet of $B_4$ defined by $x_{1,2}=0$.
\item The facet $ahijkmn$ is also a facet of the simplex $\sigma_{21}$.
\item The facet $afijkmn$ is contained in the facet of $B_4$ defined by $x_{3,4}=0$.
\item The facet $afhijmn$ is contained in the facet of $B_4$ defined by $x_{2,2}=0$.
\item The facet $afhijkn$ is also a facet of the simplex $\sigma_{22}$.
\end{enumerate}

The facets of the simplex $\sigma_{21}$ are $adhijkm$, $adhjkmn$, $adhikmn$, $ahijkmn$, $dhijkmn$, $adijkmn$, $adhijmn$, and $adhijkn$.
\begin{enumerate}
\item The facet $adhijkm$ is also a facet of the simplex $\sigma_{11}$.
\item The facet $adhjkmn$ is also a facet of the simplex $\sigma_{16}$.
\item The facet $adhikmn$ is also a facet of the simplex $\sigma_{19}$.
\item The facet $ahijkmn$ is also a facet of the simplex $\sigma_{20}$.
\item The facet $dhijkmn$ is contained in the facet of $B_4$ defined by $x_{2,4}=0$.
\item The facet $adijkmn$ is contained in the facet of $B_4$ defined by $x_{1,3}=0$.
\item The facet $adhijmn$ is contained in the facet of $B_4$ defined by $x_{2,2}=0$.
\item The facet $adhijkn$ is also a facet of the simplex $\sigma_{32}$.
\end{enumerate}

The facets of the simplex $\sigma_{22}$ are $afhijkl$, $afhikln$, $afhijkn$, $fhijkln$, $ahijkln$, $afijkln$, $afhjkln$, and $afhijln$.
\begin{enumerate}
\item The facet $afhijkl$ is also a facet of the simplex $\sigma_{7}$.
\item The facet $afhikln$ is also a facet of the simplex $\sigma_{18}$.
\item The facet $afhijkn$ is also a facet of the simplex $\sigma_{20}$.
\item The facet $fhijkln$ is contained in the facet of $B_4$ defined by $x_{1,2}=0$.
\item The facet $ahijkln$ is also a facet of the simplex $\sigma_{30}$.
\item The facet $afijkln$ is contained in the facet of $B_4$ defined by $x_{3,4}=0$.
\item The facet $afhjkln$ is contained in the facet of $B_4$ defined by $x_{3,2}=0$.
\item The facet $afhijln$ is contained in the facet of $B_4$ defined by $x_{4,3}=0$.
\end{enumerate}

The facets of the simplex $\sigma_{23}$ are $acdfijk$, $bcdfijk$, $abdfijk$, $abcfijk$, $abcdijk$, $abcdfjk$, $abcdfik$, and $abcdfij$.
\begin{enumerate}
\item The facet $acdfijk$ is also a facet of the simplex $\sigma_{10}$.
\item The facet $bcdfijk$ is contained in the facet of $B_4$ defined by $x_{3,3}=0$.
\item The facet $abdfijk$ is also a facet of the simplex $\sigma_{27}$.
\item The facet $abcfijk$ is contained in the facet of $B_4$ defined by $x_{2,1}=0$.
\item The facet $abcdijk$ is contained in the facet of $B_4$ defined by $x_{1,3}=0$.
\item The facet $abcdfjk$ is contained in the facet of $B_4$ defined by $x_{3,2}=0$.
\item The facet $abcdfik$ is also a facet of the simplex $\sigma_{24}$.
\item The facet $abcdfij$ is contained in the facet of $B_4$ defined by $x_{2,2}=0$.
\end{enumerate}

The facets of the simplex $\sigma_{24}$ are $abcdfgi$, $acdfgik$, $abcdfik$, $bcdfgik$, $abdfgik$, $abcfgik$, $abcdgik$, and $abcdfgk$.
\begin{enumerate}
\item The facet $abcdfgi$ is also a facet of the simplex $\sigma_{1}$.
\item The facet $acdfgik$ is also a facet of the simplex $\sigma_{3}$.
\item The facet $abcdfik$ is also a facet of the simplex $\sigma_{23}$.
\item The facet $bcdfgik$ is contained in the facet of $B_4$ defined by $x_{3,3}=0$.
\item The facet $abdfgik$ is also a facet of the simplex $\sigma_{28}$.
\item The facet $abcfgik$ is contained in the facet of $B_4$ defined by $x_{2,1}=0$.
\item The facet $abcdgik$ is contained in the facet of $B_4$ defined by $x_{4,2}=0$.
\item The facet $abcdfgk$ is contained in the facet of $B_4$ defined by $x_{3,2}=0$.
\end{enumerate}

The facets of the simplex $\sigma_{25}$ are $afghijk$, $bfghijk$, $abghijk$, $abfhijk$, $abfgijk$, $abfghjk$, $abfghik$, and $abfghij$.
\begin{enumerate}
\item The facet $afghijk$ is also a facet of the simplex $\sigma_{7}$.
\item The facet $bfghijk$ is contained in the facet of $B_4$ defined by $x_{3,3}=0$.
\item The facet $abghijk$ is also a facet of the simplex $\sigma_{26}$.
\item The facet $abfhijk$ is also a facet of the simplex $\sigma_{27}$.
\item The facet $abfgijk$ is contained in the facet of $B_4$ defined by $x_{2,1}=0$.
\item The facet $abfghjk$ is contained in the facet of $B_4$ defined by $x_{3,2}=0$.
\item The facet $abfghik$ is also a facet of the simplex $\sigma_{28}$.
\item The facet $abfghij$ is contained in the facet of $B_4$ defined by $x_{4,3}=0$.
\end{enumerate}

The facets of the simplex $\sigma_{26}$ are $aghijkl$, $abghijk$, $bghijkl$, $abhijkl$, $abgijkl$, $abghjkl$, $abghikl$, and $abghijl$.
\begin{enumerate}
\item The facet $aghijkl$ is also a facet of the simplex $\sigma_{7}$.
\item The facet $abghijk$ is also a facet of the simplex $\sigma_{25}$.
\item The facet $bghijkl$ is contained in the facet of $B_4$ defined by $x_{2,4}=0$.
\item The facet $abhijkl$ is also a facet of the simplex $\sigma_{30}$.
\item The facet $abgijkl$ is contained in the facet of $B_4$ defined by $x_{2,1}=0$.
\item The facet $abghjkl$ is contained in the facet of $B_4$ defined by $x_{3,2}=0$.
\item The facet $abghikl$ is also a facet of the simplex $\sigma_{29}$.
\item The facet $abghijl$ is contained in the facet of $B_4$ defined by $x_{4,3}=0$.
\end{enumerate}

The facets of the simplex $\sigma_{27}$ are $adfhijk$, $abdfijk$, $abfhijk$, $bdfhijk$, $abdhijk$, $abdfhjk$, $abdfhik$, and $abdfhij$.
\begin{enumerate}
\item The facet $adfhijk$ is also a facet of the simplex $\sigma_{11}$.
\item The facet $abdfijk$ is also a facet of the simplex $\sigma_{23}$.
\item The facet $abfhijk$ is also a facet of the simplex $\sigma_{25}$.
\item The facet $bdfhijk$ is contained in the facet of $B_4$ defined by $x_{3,3}=0$.
\item The facet $abdhijk$ is also a facet of the simplex $\sigma_{32}$.
\item The facet $abdfhjk$ is contained in the facet of $B_4$ defined by $x_{3,2}=0$.
\item The facet $abdfhik$ is also a facet of the simplex $\sigma_{28}$.
\item The facet $abdfhij$ is contained in the facet of $B_4$ defined by $x_{2,2}=0$.
\end{enumerate}

The facets of the simplex $\sigma_{28}$ are $abdfghi$, $adfghik$, $abdfgik$, $abfghik$, $abdfhik$, $bdfghik$, $abdghik$, and $abdfghk$.
\begin{enumerate}
\item The facet $abdfghi$ is also a facet of the simplex $\sigma_{2}$.
\item The facet $adfghik$ is also a facet of the simplex $\sigma_{4}$.
\item The facet $abdfgik$ is also a facet of the simplex $\sigma_{24}$.
\item The facet $abfghik$ is also a facet of the simplex $\sigma_{25}$.
\item The facet $abdfhik$ is also a facet of the simplex $\sigma_{27}$.
\item The facet $bdfghik$ is contained in the facet of $B_4$ defined by $x_{3,3}=0$.
\item The facet $abdghik$ is also a facet of the simplex $\sigma_{29}$.
\item The facet $abdfghk$ is contained in the facet of $B_4$ defined by $x_{3,2}=0$.
\end{enumerate}

The facets of the simplex $\sigma_{29}$ are $adghikl$, $abghikl$, $abdghik$, $bdghikl$, $abdhikl$, $abdgikl$, $abdghkl$, and $abdghil$.
\begin{enumerate}
\item The facet $adghikl$ is also a facet of the simplex $\sigma_{6}$.
\item The facet $abghikl$ is also a facet of the simplex $\sigma_{26}$.
\item The facet $abdghik$ is also a facet of the simplex $\sigma_{28}$.
\item The facet $bdghikl$ is contained in the facet of $B_4$ defined by $x_{2,4}=0$.
\item The facet $abdhikl$ is also a facet of the simplex $\sigma_{31}$.
\item The facet $abdgikl$ is contained in the facet of $B_4$ defined by $x_{4,2}=0$.
\item The facet $abdghkl$ is contained in the facet of $B_4$ defined by $x_{3,2}=0$.
\item The facet $abdghil$ is contained in the facet of $B_4$ defined by $x_{3,1}=0$.
\end{enumerate}

The facets of the simplex $\sigma_{30}$ are $ahijkln$, $abhijkl$, $bhijkln$, $abijkln$, $abhjkln$, $abhikln$, $abhijln$, and $abhijkn$.
\begin{enumerate}
\item The facet $ahijkln$ is also a facet of the simplex $\sigma_{22}$.
\item The facet $abhijkl$ is also a facet of the simplex $\sigma_{26}$.
\item The facet $bhijkln$ is contained in the facet of $B_4$ defined by $x_{2,4}=0$.
\item The facet $abijkln$ is contained in the facet of $B_4$ defined by $x_{1,3}=0$.
\item The facet $abhjkln$ is contained in the facet of $B_4$ defined by $x_{3,2}=0$.
\item The facet $abhikln$ is also a facet of the simplex $\sigma_{31}$.
\item The facet $abhijln$ is contained in the facet of $B_4$ defined by $x_{4,3}=0$.
\item The facet $abhijkn$ is also a facet of the simplex $\sigma_{32}$.
\end{enumerate}

The facets of the simplex $\sigma_{31}$ are $adhikln$, $abdhikl$, $abhikln$, $bdhikln$, $abdikln$, $abdhkln$, $abdhiln$, and $abdhikn$.
\begin{enumerate}
\item The facet $adhikln$ is also a facet of the simplex $\sigma_{19}$.
\item The facet $abdhikl$ is also a facet of the simplex $\sigma_{29}$.
\item The facet $abhikln$ is also a facet of the simplex $\sigma_{30}$.
\item The facet $bdhikln$ is contained in the facet of $B_4$ defined by $x_{2,4}=0$.
\item The facet $abdikln$ is contained in the facet of $B_4$ defined by $x_{1,3}=0$.
\item The facet $abdhkln$ is contained in the facet of $B_4$ defined by $x_{3,2}=0$.
\item The facet $abdhiln$ is contained in the facet of $B_4$ defined by $x_{3,1}=0$.
\item The facet $abdhikn$ is also a facet of the simplex $\sigma_{32}$.
\end{enumerate}

The facets of the simplex $\sigma_{32}$ are $adhijkn$, $abdhijk$, $abhijkn$, $abdhikn$, $bdhijkn$, $abdijkn$, $abdhjkn$, and $abdhijn$.
\begin{enumerate}
\item The facet $adhijkn$ is also a facet of the simplex $\sigma_{21}$.
\item The facet $abdhijk$ is also a facet of the simplex $\sigma_{27}$.
\item The facet $abhijkn$ is also a facet of the simplex $\sigma_{30}$.
\item The facet $abdhikn$ is also a facet of the simplex $\sigma_{31}$.
\item The facet $bdhijkn$ is contained in the facet of $B_4$ defined by $x_{2,4}=0$.
\item The facet $abdijkn$ is contained in the facet of $B_4$ defined by $x_{1,3}=0$.
\item The facet $abdhjkn$ is contained in the facet of $B_4$ defined by $x_{3,2}=0$.
\item The facet $abdhijn$ is contained in the facet of $B_4$ defined by $x_{2,2}=0$.
\end{enumerate}

\chapter[A Counter-example to the Hirsch Conjecture]{A Counter-example to the Hirsch Conjecture}\label{appendix:CounterHirsch}

\setcounter{theorem}{0} 

Right before this dissertation was ready for submission, on the 10th of May 2010, Santos (see~\cite{Santos:CounterHirsch}) announced his construction of the first counter-example to the Hirsch Conjecture (see Conjecture~\ref{conjecture:hirsch}). In~\cite{Santos:CounterHirsch}, Santos describes the construction of a $43$-dimensional polytope $P$ with $86$ facets whose diameter is strictly greater than $43$. Santos' construction in~\cite{Santos:CounterHirsch} begins with a prismatoid.
\begin{definition}
A \defn{prismatoid} is a polytope $Q$ with two distinguished facets $F_1$ and $F_2$ (called the \defn{base facets}) so that every vertex of $Q$ is contained in exactly one of $F_1$ or $F_2$.
\end{definition}
Cubes and cross-polytopes are examples of prismatoids, and any pair of disjoint facets can be the base facets. The distance between the base facets in the dual graph $G^\Delta(Q)$ of a prismatoid $Q$ is the main combinatorial property Santos analyzes in~\cite{Santos:CounterHirsch}. More precisely,
\begin{definition}
Let $Q$ be a $d$-dimensional prismatoid with base facets $F_1$ and $F_2$, and let $G^\Delta(Q)$ be the polar graph of $Q$. We say that the prismatoid $Q$ \defn{satisfies the $d$-step property} (with respect to its base facets $F_1$ and $F_2$) if the distance from $F_1$ to $F_2$ in $G^\Delta(Q)$ is at most $d$. Otherwise, we say that $Q$ \defn{does not satisfy the $d$-step property}.
\end{definition}
In~\cite{Santos:CounterHirsch}, Santos constructs a $5$-dimensional prismatoid $Q$ with $48$ vertices that does not satisfy the $d$-step property. The distance between the base facets $F_1$ and $F_2$ in the dual graph $G^\Delta(Q)$ of $Q$ is strictly greater than five:
\begin{theorem}[Santos~\cite{Santos:CounterHirsch}]\label{theorem:SantosPrismatoidProperties}
There is a $5$-dimensional prismatoid $Q$ with $48$ vertices that does not satisfy the $d$-step property. The distance between the base facets $F_1$ and $F_2$ of $Q$ in the dual graph $G^\Delta(Q)$ of $Q$ is exactly six.
\end{theorem}
The data files needed to verify Theorem~\ref{theorem:SantosPrismatoidProperties} are available on the web (see~\cite{Kim:SantosHirsch}). The counterexample to the Hirsch Conjecture is obtained from the $5$-dimensional Santos prismatoid $Q$ by repeatedly applying a polytopal construction:
\begin{theorem}[Santos~\cite{Santos:CounterHirsch}]\label{theorem:reduceasimpliciality}
Let $n > 2d$. If there is a $d$-dimensional prismatoid $Q$ with $n$ vertices that does not satisfy the $d$-step property (with respect to its base facets), then there is a $(d+1)$-dimensional prismatoid $\widetilde{Q}$ with $n+1$ vertices that does not satisfy the $d$-step property (with respect to its base facets).
\end{theorem}
Theorem~\ref{theorem:reduceasimpliciality} reduces the \defn{asimpliciality} $s = n - 2d$ of a prismatoid. Repeated application of the theorem yields a prismatoid $Z$ whose asimpliciality $s=0$, and thus the resulting prismatoid $Z$ has $n=2d$. Santos applies this theorem $38$ times to the $5$-dimensional prismatoid $Q$ with $48$ vertices, obtaining a $43$-dimensional prismatoid $Z$ with $86$ vertices that does not satisfy the $d$-step property with respect to its base facets $F_1$ and $F_2$. That is to say, the distance between $F_1$ and $F_2$ in the dual graph $G^\Delta(Z)$ of $Z$ is strictly greater than $43$. The polar $P = Z^\Delta$ of $Z$ is a $43$-dimensional polytope with $86$ facets that does not satisfy the Hirsch Conjecture, since the distance between the vertices that are dual to $F_1$ and $F_2$ in the graph $G(P)$ of $P$ is strictly greater than $43$.

The monumental result leads to many new questions and open problems (and renews interest in the Polynomial Diameter Conjecture), including the following:
\begin{openproblem}
What is the smallest dimension $d$ for which the Hirsch Conjecture fails?
\end{openproblem}
\begin{openproblem}
Are there $d$-polytopes with $n$ facets whose diameter is greater than $\frac{3}{2}(n-d)$? Are there $d$-polytopes with $n$ facets whose diameter is greater than $(n-d)^2$?
\end{openproblem}
\begin{openproblem}
Are there $4$-dimensional prismatoids without the $d$-step property? (It is easy to prove that $3$-dimensional prismatoids satisfy the $d$-step property.)
\end{openproblem}
\begin{openproblem}
Find explicit coordinates for the non-Hirsch $43$-dimensional polytope derived from Santos' prismatoid $Q$ and Theorem~\ref{theorem:reduceasimpliciality}. According to Santos, the polytope is expected to have between $10^9$ and $10^{12}$ vertices.
\end{openproblem}
       
   \backmatter

\end{document}